\documentclass{cmslatex}
\usepackage[paperwidth=7in, paperheight=10in, margin=.875in]{geometry}
 \usepackage[backref,colorlinks,linkcolor=red,anchorcolor=green,citecolor=blue]{hyperref}
\usepackage{amsfonts,amssymb}
\usepackage{amsmath}
\usepackage{graphicx}
\usepackage{cite}
\usepackage{enumerate}
\renewcommand{\theequation}{\arabic{section}.\arabic{equation}}

\usepackage{mathrsfs} 
\usepackage{epsfig} 
   \allowdisplaybreaks
\begin{document}
 \title{Discrete energy estimates for the $abcd$-systems\thanks{received on \today, and accepted date (The correct dates will be entered by the editor).}}


          \author{Cosmin Burtea\thanks{Université Paris-Diderot,  Institut de Math\'ematiques de Jussieu-Paris Rive Gauche, UMR 7586, Paris, France, (cburtea@math.univ-paris-diderot.fr). \url{http://webusers.imj-prg.fr/\~cosmin.burtea/}}
          \and Cl\'{e}mentine Court\`{e}s\thanks{Institut de Math\'ematiques de Toulouse, UMR 5219, Universit\'e de Toulouse, CNRS, INSA, F-31077 Toulouse, France, (clementine.courtes@math.univ-toulouse.fr). \url{http://www.math.univ-toulouse.fr/\~ccourtes}}}

         \pagestyle{myheadings} \markboth{Discrete energy estimates for the $abcd$-systems}{Cosmin Burtea and Cl\'ementine Court\`es} \maketitle

          \begin{abstract}
               In this article, we propose finite volume schemes for the $abcd$-systems and we establish stability and error estimates. The order of accuracy depends on the so-called BBM-type dispersion coefficients $b$ and $d$. If $bd>0$, the numerical schemes are $O(\Delta t+(\Delta x)^2)$ accurate, while if $bd=0$, we obtain an $O(\Delta t+\Delta x)$ -order of convergence. The analysis covers a broad range of the parameters $a,b,c,d$. In the second part of the paper, numerical experiments validating the theoretical results as well as head-on collision of traveling waves are investigated.
          \end{abstract}
\begin{keywords}  system $abcd$; numerical convergence; error estimates
\end{keywords}

 \begin{AMS} 65M12; 35Q35
\end{AMS}
          \section{Introduction}\label{intro}
          
Consider a layer of incompressible, irrotational, perfect fluid flowing through a channel with flat bottom represented by the plane:%
\[
\{\left(  x,y,z\right)  :z=-h\},
\]
with $h>0$, the depth of the channel. We assume that the fluid at rest occupies all the region $\{\left(  x,y,z\right)  : -h \leq z \leq 0\}$. We suppose also that the free surface resulting after a perturbation of the rest state can be described by the graph of some function. Let $A$ be the maximum amplitude of the wave and $l$ a typical wavelength.
Furthermore, consider:%
\[
\epsilon=\frac{A}{h},\text{ }\mu=\left(  \frac{h}{l}\right)  ^{2},\text{
}S=\frac{\epsilon}{\mu}.
\]
The Boussinesq regime
\[
\epsilon\ll1,\text{ }\mu\ll1\text{ and }S\approx1
\]
describes small amplitude, long wavelength water waves. In order to study waves that fit into the Boussinesq regime, Bona, Chen and Saut,
\cite{BonaChenSaut2002}, derived the so-called $abcd$-systems:

\begin{equation}
\left\{
\begin{array}
[c]{l}%
\left(  I-\mu b\Delta\right)  \partial_{t}\eta+\operatorname{div}%
V+a\mu\operatorname{div}\Delta V+\epsilon\operatorname{div}\left(  \eta
V\right)  =0,\\
\left(  I-\mu d\Delta\right)  \partial_{t}V+\nabla\eta+c\mu\nabla
\Delta\eta+\epsilon V\cdot\nabla V=0,
\end{array}
\right.  \label{eq1}%
\end{equation}
with $I$ the identity operator.
In the system \eqref{eq1}, $\eta=\eta\left(  t,x\right)  \in\mathbb{R}$ represents the deviation of the free surface from the rest position, while $V=V\left(  t,x\right)  \in\mathbb{R}^{n}$ is the fluid velocity. The above family of systems is derived in \cite{BonaChenSaut2002} from the classical water waves problem by a formal series expansion and by neglecting the second and higher order terms. In fact, one may regard the zeros on the right hand side of \eqref{eq1} as the second order terms (i.e. having one of $\epsilon^{2}$, $\epsilon\mu$ or $\mu^{2}$ as a prefactor) neglected in the modeling process in order to establish \eqref{eq1}. In \cite{BonaColinLannes2005}, Bona, Colin and Lannes prove rigorously that the systems \eqref{eq1} approximate the water waves problem in the sense that the error estimate between the solution of \eqref{eq1} and the solution of the water waves system at time $t$ is of order $O\left( \epsilon^{2}t\right)  $ (see also \cite{Lannes2013livre}).

The parameters $a,b,c,d$ are restricted by the following relation:%
\begin{equation}
a+b+c+d=\frac{1}{3}. \label{suma}%
\end{equation}
Nevertheless, if surface tension is taken into account then, the previous relation rewrites 
$$a+b+c+d=\frac{1}{3}-\tau$$
where $\tau\geq0$ is the Bond number which characterizes the surface tension parameter, see \cite{Daripa_Dash_2003} or \cite{SautWangXu2015}. 

 In \cite{Chazel2007}, models taking into account more general topographies of
the bottom of the channel are obtained. One has to furthermore distinguish between two
different regimes: small respectively strong topography variations. Time-changing bottom-topographies are considered in \cite{Chen2003}. For a systematic study of approximate models for the water waves problem along with their rigorous justification, we refer the reader to the work of Lannes \cite{Lannes2013livre}. We point out that the only values of $n$ for which \eqref{eq1} is physically relevant are $n=1,2$.

The $abcd$ systems are well-posed locally in time in the following cases:
\begin{align}
a  &  \leq0,\text{ }c\leq0,\text{ }b\geq0,\text{ }d\geq0\label{1111}\\
\text{or }a  &  =c\geq0\text{ and }b\geq0,\text{ }d\geq0, \label{2222}%
\end{align}
see for instance Bona, Chen and Saut \cite{BonaChenSaut2004}, Anh
\cite{Anh2010} and Linares, Pilod and Saut \cite{LinaresPilodSaut2012}. Global
existence is known to hold true in dimension $1$ for the "classical"
Boussinesq system:
\[
a=b=c=0,\text{ }d>0,
\]
which was studied by Amick in \cite{Amick1984} and Schonbek in
\cite{Schonbek1981} and for the so-called Bona-Smith systems:%
\[
b=d>0,\text{ }a\leq0,~c<0,
\]
assuming some smallness condition on the initial data, see
\cite{BonaChenSaut2004} and the work of Bona and Smith \cite{BonaSmith1976}. In the remaining cases, the problem is still open.

Lower bounds on the time of existence of solutions of systems \eqref{eq1} in terms of the physical parameters are obtained in \cite{SautXu2012}, \cite{SautWangXu2015}, \cite{MingSautZhang2012}, \cite{Burtea2016a} for initial data lying in Sobolev classes, respectively in \cite{Burtea2016b} for initial data manifesting nontrivial values at infinity.

In the following lines, let us recall some important \textit{numerical results} obtained for the systems \eqref{eq1}. 
 
One of the earliest papers in this sense is the work of Peregrine \cite{Peregrine1966}, who numerically investigates undular bore propagation using the ``classical'' Boussinesq system corresponding to the values $a=b=c=0$, $d=1/3$.

In \cite{BonaChen1998}, Bona and Chen study the boundary value problem for the BBM-BBM case (which corresponds to $a=c=0,$ $b=d=1/6$). Using an integral reformulation of this problem, they construct a numerical scheme, which is proven to be fourth order accurate in both time and space. They employ it in order to track down generalized solitary waves and to study the collision between this numerical generalized solitary waves.

In \cite{BonaDougalisMitsotakis2007} and \cite{BonaDougalisMitsotakis2008}, Bona, Dougalis and Mitsotakis investigate the periodic value problem for the so-called KdV-KdV case, which corresponds to $b=d=0,$ $a=c=1/6$. They propose a numerical scheme constructed by using an implicit Runge-Kutta method for the time discretization and a Galerkin method with periodic splines scheme for the space discretization. They use this scheme in order to track down generalized solitary wave solutions and they simulate head-on collision of two such waves.

In \cite{AntonopoulusDugalisMitsotakis2009}, Antonopoulus, Dougalis and Mitsotakis study the periodic initial value problem for a large class of values of the $a,b,c,d$ parameters. More precisely, they obtain error estimates for semi-discretization schemes (in space) obtained via Galerkin approximation. In \cite{AntonopoulusDugalisMitsotakis2010}, the same authors obtain error estimates for a full-discretization scheme of the initial value problem with non-homogeneous Dirichlet and reflection boundary conditions for the Bona-Smith system (corresponding to the case $a=0$, $b=d>0$, $c\leq0$). They use their numerical algorithm to study soliton interactions. 

 In \cite{AntonopoulosDougalis2012}, (see also \cite{AntonopoulosDougalis2012n} for more details) Antonopoulos and Dougalis obtain error estimates for a semi-discrete and fully-discrete Galerkin-finite element scheme for the initial boundary value problems (ibvp) for the ``classical'' Boussinesq system.

 Regarding the two dimensional case \cite{DougalisMitsotakisSaut2007}, Dougalis, Mitsotakis and Saut study a space semi-discretization scheme using a Galerkin method. In \cite{Chen2009}, Chen, using a formal second-order semi-implicit Crank-Nicolson scheme along with spectral method, studies the $2D$ case of the BBM-BBM system for $3$D water waves over an uneven bottom.

 A very recent result of Bona and Chen \cite{BonaChen2016} provides numerical evidence of finite time blow-up for the BBM-BBM system. The blow-up phenomena seems to occur on head-on collision of some particular traveling waves solutions of the BBM system.

\subsection{Statement of the main results.}
To the authors knowledge, \textit{fully-discretized schemes} for various ibvp for $abcd$-systems \textit{along with error estimates} are available in only four cases :
\begin{itemize}
\item the BBM-BBM system $a=c=0$, $b=d>0$ treated by Bona and Chen in \cite{BonaChen1998};
\item the Bona-Smith system  $a=0$, $b=d>0$, $c\leq0$ in \cite{AntonopoulusDugalisMitsotakis2010};
\item the ``classical'' Boussinesq system $a=b=c=0$, $d>0$, obtained in \cite{AntonopoulosDougalis2012};
\item the shallow water case $a=b=c=d=0$, see for instance \cite{AntonopoulosDougalis2016}.
\end{itemize}

In this paper, we implement fully-discrete finite volume schemes, and we prove convergence and stability estimates, using a discrete variant of the "natural" energy functional associated to these systems. One practical advantage of our approach is that it allows us to treat explicitly the nonlinear terms.

Let us recall that the Cauchy problem associated with the one dimensional
$abcd$-systems reads:%
\begin{equation}
\left\{
\begin{array}
[c]{l}%
\left(  I-b\partial_{xx}^{2}\right)  \partial_{t}\eta+\left(  I+a\partial
_{xx}^{2}\right)  \partial_{x}u+\partial_{x}\left(  \eta u\right)  =0,\\
\left(  I-d\partial_{xx}^{2}\right)  \partial_{t}u+\left(  I+c\partial
_{xx}^{2}\right)  \partial_{x}\eta+\frac{1}{2}\partial_{x}u^{2}=0,\\
\eta_{|t=0}=\eta_{0},\text{ }u_{|t=0}=u_{0}.
\end{array}
\right.  \tag{$S_{abcd}$} \label{1neliniar_cont}%
\end{equation}
We will treat the cases where the parameters verify:
\begin{equation}
a\leq0,\text{ }c\leq0,\text{ }b\geq0,\text{ }d\geq0,\label{11111}%
\end{equation}
excluding the five cases:%
\begin{equation}
\left\{
\begin{array}
[c]{l}%
a=b=0,\ d>0,\ c<0,\\
a=b=c=d=0,\\
a=d=0,\ b>0,\ c<0,\\
a=b=d=0,\ c<0,\\
b=d=0,\ c<0,\ a<0.
\end{array}
\right.  \label{excluding}%
\end{equation}
\begin{remark}
The Shallow Water case $a=b=c=d=0$ corresponds to a classical non-symmetric hyperbolic system. A systematic method to construct appropriate semi-discrete schemes is presented in \cite{Tadmor_2003}. More recently, this case has been studied in \cite{AntonopoulosDougalis2016}. The authors construct a finite-volume-Galerkin numerical scheme, where discrete estimates in $L^2\times L^2$ are obtained.
\end{remark}

Given $s\in\mathbb{R}$, we will consider the following set of indices:%

\begin{equation}
\left\{
\begin{array}
[c]{l}%
s_{bc}=s+\operatorname*{sgn}\left(  b\right)  -\operatorname*{sgn}\left(  c\right)  ,\\
s_{ad}=s+\operatorname*{sgn}\left(  d\right)  -\operatorname*{sgn}\left(  a\right)  ,
\end{array}
\right.  \label{relatie}%
\end{equation}
where the sign function $\operatorname*{sgn}$ is given by:%
\[
\operatorname*{sgn}\left(  x\right)  =\left\{
\begin{array}
[c]{rrr}%
1 & \text{if} & x>0,\\
0 & \text{if} & x=0,\\
-1 & \text{if} & x<0.
\end{array}
\right.
\]

We will consider the set of indices defined by
\eqref{relatie}. The notation, $\mathcal{C}_{T}(H^{s_{bc}}\times H^{s_{ad}})$
stands for the space of continuous functions $(\eta,u)$ on $\left[  0,T\right]  $ with
values in the space $H^{s_{bc}}\left(  \mathbb{R}\right)  \times H^{s_{ad}%
}\left(  \mathbb{R}\right)  $. Let us recall in the following lines, the
existence result of regular solution that can be found for instance in 
\cite{Burtea2016a} and \cite{SautWangXu2015}:

\begin{theorem}
\label{Teorema_existenta_abcd}Consider $a,c\leq0$ and $b,d\geq0$ excluding the
cases $\left(  \text{\ref{excluding}}\right)  $. Also, consider an integer $s$
such that%
\[
s>\frac{5}{2}-\operatorname*{sgn}(b+d),\text{ }%
\]
and $s_{bc}$, $s_{ad}$ defined by \eqref{relatie}. Let us consider $\left(
\eta_{0},u_{0}\right)  \in H^{s_{bc}}\left(  \mathbb{R}\right)  \times
H^{s_{ad}}\left(  \mathbb{R}\right)  $. Then, there exists a positive time $T$
and a unique solution%
\[
\left(  \eta,u\right)  \in\mathcal{C}_{T}(H^{s_{bc}}\times H^{s_{ad}})\text{ }%
\]
of \eqref{1neliniar_cont}.
\end{theorem}

\begin{remark}
Note that energy $\mathcal{E}_s$ is an energy for the system  \eqref{1neliniar_cont}, see for exemple \cite{Burtea2016a},
\begin{multline*}
\mathcal{E}_s(\eta,u)=||\eta||_{H^s}^2+(b-c)||\partial_x \eta||_{H^s}^2+(-c)b||\partial_x^2\eta||_{H^s}^2\\+||u||_{H^s}^2+(d-a)||\partial_x u||_{H^s}^2+(-a)d||\partial_x^2u||_{H^s}^2.
\end{multline*}
\end{remark}

Consider $\Delta t$ and $\Delta x$ two positive real numbers. We endow
$\ell^{2}\left(  \mathbb{Z}\right)  $ with the following scalar product and
norm%
\[
\left\langle v,w\right\rangle :=\Delta x\sum_{j\in\mathbb{Z}}v_{j}w_{j}\text{,
}\left\Vert v\right\Vert _{\ell_{\Delta}^{2}}:=\left\langle v,v\right\rangle
^{\frac{1}{2}}.
\]
We sometimes denote this space $\ell^{2}_{\Delta}\left(  \mathbb{Z}\right)  $.\\
For all $v=\left(  v_{j}\right)  _{j\in\mathbb{Z}}\in\ell^{\infty}\left(
\mathbb{Z}\right)  $, we introduce the spatial shift operators:%
\[
\left(  S_{\pm}v\right)  _{j}:=v_{j\pm1}.
\]
If $v=\left(  v_{j}\right)  _{j\in\mathbb{Z}}\in\ell^{\infty}\left(
\mathbb{Z}\right)  $, we denote by $D_{+}v,$ $D_{-}v,$ $Dv$ the discrete
derivation operators:
\[
\left\{
\begin{array}
[c]{c}%
D_{+}v=\frac{1}{\Delta x}\left(  S_{+}v-v\right)  ,\\
\\
D_{-}v=\frac{1}{\Delta x}\left(  v-S_{-}v\right)  ,\\
\\
Dv=\frac{1}{2}\left(  D_{+}v+D_{-}v\right)  =\frac{1}{2\Delta x}\left(
S_{+}v-S_{-}v\right)  .
\end{array}
\right.
\]
Also, we consider the following discrete energy functional%
\begin{align}
  \mathcal{E}\left(  e,f\right)  \overset{def.}{=}&\left\Vert e\right\Vert
_{\ell_{\Delta}^{2}}^{2}+\left(  b-c\right)  \left\Vert D_{+}e\right\Vert
_{\ell_{\Delta}^{2}}^{2}+b\left(  -c\right)  \left\Vert D_{+}D_{-}e\right\Vert
_{\ell_{\Delta}^{2}}^{2}\nonumber\\
  +&\left\Vert f\right\Vert _{\ell_{\Delta}^{2}}^{2}+(d-a)\left\Vert
D_{+}f\right\Vert _{\ell_{\Delta}^{2}}^{2}+d\left(  -a\right)  \left\Vert
D_{+}D_{-}f\right\Vert _{\ell_{\Delta}^{2}}^{2}. \label{energy_discrete}%
\end{align}

We will now discuss the main results of the paper. From an initial datum $\left(  \eta_{0},u_{0}\right)  \in H^{s_{bc}}(\mathbb{R})\times H^{s_{ad}}(\mathbb{R})$ with $s$ large enough, there exists a solution $(\eta,u)$ of \eqref{1neliniar_cont} which enables to define
$\eta_{\Delta}^{n},u_{\Delta}^{n}\in\ell_{\Delta}^{2}\left(  \mathbb{Z}%
\right)  $ given by:%
\begin{equation}
\left\{
\begin{array}
[c]{c}%
\eta_{\Delta j}^{n}=\frac{1}{[\inf(t^{n+1},T)-t^n]\Delta x}\int_{t_{n}}^{\inf(t^{n+1},T)}%
\int_{x_{j}}^{x_{j+1}}\eta\left(  s,y\right)  dyds,\\
\\
u_{\Delta j}^{n}=\frac{1}{[\inf(t^{n+1},T)-t^n]\Delta x}\int_{t_{n}}^{\inf(t^{n+1},T)}\int_{x_{j}%
}^{x_{j+1}}u\left(  s,y\right)  dyds,
\end{array}
\right.  \ \ \ \ \ \text{for\ all\ } j\in\mathbb{Z}, \text{\ and\ for\ }n\geq1.\label{troncature}%
\end{equation}
Moreover, we define
\begin{equation}
\left\{
\begin{array}
[c]{c}%
\eta_{\Delta j}^{0}=\frac{1}{\Delta x}\int_{x_{j}}^{x_{j+1}}\eta_{0}\left(
y\right)  dy,\\
\\
u_{\Delta j}^{0}=\frac{1}{\Delta x}\int_{x_{j}}^{x_{j+1}}u_{0}\left( y \right)
dy,
\end{array}
\right.  \ \ \ \ \ \text{for\ all\ } j\in\mathbb{Z}.\label{conditie_ini_oct}%
\end{equation}

The results of the paper are gathered in two cases according to the values of the parameters $b$ and $d$ in \eqref{1neliniar_cont}.

\subsubsection*{The case when $b>0$ and $d>0$}

The assumption $b,d>0$ assures an $\ell^{2}$-control on the discrete
derivatives: this makes it possible to implement an energy method that mimics
the one from the continuous case. Moreover, it allows us to close the
estimates even without considering numerical viscosity. 

We consider the following
numerical scheme:%
\begin{equation}
\left\{
\begin{array}
[c]{l}%
\frac{1}{\Delta t}\left(  I-bD_{+}D_{-}\right)  (\eta^{n+1}-\eta^{n})+\left(  I+aD_{+}D_{-}\right)  D\left( (1-\theta) u^{n}+\theta u^{n+1}\right)  +D\left(
\eta^{n}u^{n}\right)  =0,\\
\\
\frac{1}{\Delta t}\left(  I-dD_{+}D_{-}\right)  (u^{n+1}-u^{n})+\left(  I+cD_{+}D_{-}\right)  D\left((1- \theta) \eta^{n}+\theta\eta^{n+1}\right)
+\frac{1}{2}D\left(  \left(  u^{n}\right)  ^{2}\right)  =0,
\end{array}
\right.  \label{schema_bd>0}%
\end{equation}
for all $n\geq0$ and $\theta=\frac{1}{2}$ or $\theta=1$, with the discrete initial datum
\begin{equation}\label{conditie_ini}
\left(  \eta^{0},u^{0}\right)  =\left(  \eta_{\Delta}^{0},u_{\Delta
}^{0}\right).
\end{equation}

The convergence error is defined as:%
\begin{equation}
e^{n}=\eta^{n}-\eta_{\Delta}^{n}\text{\ \ \ \  and\ \ \ \  }f^{n}=u^{n}-u_{\Delta}^{n}%
.\label{convergence_error1}%
\end{equation}
\begin{remark}
For the linear part of the system, when $\theta=\frac{1}{2}$, we consider the Crank-Nicolson discretization whereas for $\theta=1$, we consider the implicit discretization.
\end{remark}\\

We are now in the position of stating our first result.

\begin{theorem}
\label{Teorema_Numerica1} Let $a\leq0$, $b>0,$ $c\leq0$ and $d>0$. Consider
$s>6$,  $\left(  \eta_{0},u_{0}\right)  \in H^{s_{bc}}\left(
\mathbb{R}\right)  \times H^{s_{ad}}\left(  \mathbb{R}\right)  $ and $T>0$ such that $\left(
\eta,u\right)  \in \mathcal{C}_{T}(H^{s_{bc}}\times H^{s_{ad}})$ is the solution on $[0,T]$ of the
system \eqref{1neliniar_cont}. Let $N\in\mathbb{N}$,
there exists a positive
constant $\delta_{0}$ depending on $\sup\limits_{t\in\left[  0,T\right]
}\left\Vert \left(  \eta(t),u(t)\right)  \right\Vert _{H^{s_{bc}}\times
H^{s_{ad}}}$ such that if the number of time steps $N$ and the space
discretization step $\Delta x$ are chosen such that%
\[
\Delta t=T/N\leq\delta_{0}\text{ and }\Delta x\leq\delta_{0},%
\]
if we consider the numerical scheme \eqref{schema_bd>0} with $\theta=\frac{1}{2}$ or $\theta=1$ along with the initial data \eqref{conditie_ini} as well as the approximation $\left(
\eta_{\Delta}^{n},u_{\Delta}^{n}\right)  _{n\in\overline{1,N}}$, defined by
\eqref{troncature}-\eqref{conditie_ini_oct}, where $\overline{1,N}=\{1,...N\}$, then, the numerical scheme \eqref{schema_bd>0} is first order convergent in time and second order convergent in space i.e. the convergence
error defined in \eqref{convergence_error1}
satisfies:
\begin{equation}
\sup_{n\in\overline{0,N}}\mathcal{E}\left(  e^{n},f^{n}\right)  \leq
C_{abcd}\left\{ \left(  \Delta t\right)  ^{2}+\left(  \Delta x\right)
^{4}\right\}  ,\label{order_numeric1}%
\end{equation}
where $C_{abcd}$ depends on the parameters $a$,$b$,$c$,$d$, on
$\sup\limits_{t\in\left[  0,T\right]  }\left\Vert \left(  \eta(t),u(t)\right)
\right\Vert _{H^{s_{bc}}\times H^{s_{ad}}}$ and on $T$.
\end{theorem}

\begin{remark}
Note that no Courant-Friedrichs-Lewy-type condition (CFL-type condition hereafter) is needed. This is due to the regularity properties of the operators $\left(I-bD_+D_-\right)^{-1}$ and $\left(I-dD_+D_-\right)^{-1}$ which ensure stability even if a centered scheme is used to discretize the hyperbolic part of the system.
\end{remark}

\begin{remark}\label{asymptotic_preserving_scheme_10_NOv}
If we consider the Cauchy problem with small parameter $\epsilon$:
\begin{equation*}
\left\{
\begin{array}
[c]{l}%
\left(  I-b\epsilon \partial_{xx}^{2}\right)  \partial_{t}\eta+\left(  I+a\epsilon\partial
_{xx}^{2}\right)  \partial_{x}u+\epsilon\partial_{x}\left(  \eta u\right)  =0,\\
\left(  I-d\epsilon\partial_{xx}^{2}\right)  \partial_{t}u+\left(  I+c\epsilon\partial
_{xx}^{2}\right)  \partial_{x}\eta+\frac{\epsilon}{2}\partial_{x}u^{2}=0,\\
\eta_{|t=0}=\eta_{0},\text{ }u_{|t=0}=u_{0}.
\end{array}
\right.
\end{equation*}
The previous scheme is transformed into
\begin{small}
\begin{equation*}
\left\{
\begin{array}
[c]{l}%
\frac{1}{\Delta t}\left(  I-b\epsilon D_{+}D_{-}\right)  (\eta^{n+1}-\eta^{n})+\left(  I+a\epsilon D_{+}D_{-}\right)  D\left( (1-\theta) u^{n}+\theta u^{n+1}\right)  +\epsilon D\left(
\eta^{n}u^{n}\right)  =0,\\
\\
\frac{1}{\Delta t}\left(  I-d\epsilon D_{+}D_{-}\right)  (u^{n+1}-u^{n})+\left(  I+c\epsilon D_{+}D_{-}\right)  D\left((1- \theta) \eta^{n}+\theta\eta^{n+1}\right)
+\frac{\epsilon}{2}D\left(  \left(  u^{n}\right)  ^{2}\right)  =0,
\end{array}
\right. %
\end{equation*}
\end{small}
for all $n\geq0$ and $\theta=\frac{1}{2}$ or $\theta=1$. The conditions on the time step and space discretization in Theorem \ref{Teorema_Numerica1} become then $\frac{\Delta t}{\epsilon}\leq \delta_0$ and $\frac{\Delta x}{\epsilon}\leq \delta_0$ and the energy inequality rewrites
\begin{equation*}
\sup_{n\in\overline{0,N}}\mathcal{E}\left(  e^{n},f^{n}\right)  \leq
C_{abcd}\left\{  \left(\frac{  \Delta t}{\epsilon}\right)  ^{2}+\left(  \frac{\Delta x}{\epsilon}\right)
^{4}\right\},
\end{equation*}
which impose a strong restriction on $\Delta x$ and $\Delta t$.

 In order to avoid this inconvenience, one should design an \emph{asymptotic preserving} scheme, for which the error estimates are uniform with respect to $\epsilon$. In such a scheme, both limits $\epsilon\rightarrow0$ and $\Delta x, \Delta t\rightarrow 0$ may be commuted without affecting the accuracy of the scheme. Moreover, in such a scheme, the limit scheme when $\epsilon\rightarrow 0$ is consistent with the limit continuous system when $\epsilon\rightarrow0$. In the case of the $abcd$-systems, the limit system when $\epsilon \rightarrow0$ is the acoustic wave system.

 However, designing an asymptotic preserving scheme is not the aim of this paper and we will focus only on System \eqref{1neliniar_cont} with $\epsilon=1$.

\end{remark}

\subsubsection*{The case when $bd=0$}

We consider the following numerical scheme:%
\begin{equation}
\left\{
\begin{array}
[c]{l}%
\frac{1}{\Delta t}\left(  I-bD_{+}D_{-}\right)  (\eta^{n+1}-\eta^{n})+\left(
I+aD_{+}D_{-}\right)  D\left(  u^{n+1}\right)  +D\left(  \eta^{n}u^{n}\right)
\\
=\frac{1}{2}\left(  1-\operatorname*{sgn}\left(  b\right)  \right)  \tau
_{1}\Delta xD_{+}D_{-}\left(  \eta^{n}\right)  ,\\
\\
\frac{1}{\Delta t}\left(  I-dD_{+}D_{-}\right)  (u^{n+1}-u^{n})+\left(
I+cD_{+}D_{-}\right)  D\left(  \eta^{n+1}\right)  +\frac{1}{2}D\left(  \left(
u^{n}\right)  ^{2}\right) \\
=\frac{1}{2}\left(  1-\operatorname*{sgn}\left(  d\right)  \right)  \tau
_{2}\Delta xD_{+}D_{-}\left(  u^{n}\right)  ,
\end{array}
\right.  \label{schema_bd=0}%
\end{equation}
for all $n\geq0$, with
\begin{equation}\label{conditie_initiala2}
\left(  \eta^{0},u^{0}\right)  =\left(  \eta_{\Delta}^{0},u_{\Delta
}^{0}\right)  .
\end{equation}

\begin{remark}
For the case $bd=0$, we consider only an implicit time discretization for the linear part of the system.
\end{remark}\\

The convergence error is defined by $\left(\text{\ref{convergence_error1}}\right)  $. We are now in the position of
stating our second main result.

\begin{theorem}
\label{Teorema_Numerica2} Let $a\leq0$, $b\geq0,$ $c\leq0$ and $d\geq0$ with
$bd=0$, excluding the cases \eqref{excluding}.
Consider $s>8$, $\left(  \eta_{0},u_{0}\right)  \in H^{s_{bc}}\left(
\mathbb{R}\right)  \times H^{s_{ad}}\left(  \mathbb{R}\right)  $ and $T>0$ such that $\left(
\eta,u\right)  \in \mathcal{C}_{T}(H^{s_{bc}}\times H^{s_{ad}})$ is the solution on $[0,T]$ of the
system \eqref{1neliniar_cont}.

Choose
$\alpha>0,\ \tau_{1}>0$ and $\tau_{2}>0$ such
that:
\[
\left\Vert u\right\Vert _{L^{\infty}_tL^{\infty}_x}+\alpha\leq\tau_{1}%
\text{ and }\left\Vert u\right\Vert _{L^{\infty}_tL^{\infty}_x}%
+\alpha\leq\tau_{2}.
\]

There exists $\delta_0>0$ (depending on $\alpha
,\ \tau_{1}$, $\tau_{2}$ and on $\sup\limits_{t\in\left[
0,T\right]  }\left\Vert \left(  \eta(t),u(t)\right)  \right\Vert _{H^{s_{bc}%
}\times H^{s_{ad}}}$) such that if the number of time steps $N\in\mathbb{N}$ and the space
discretization step are chosen in order to
verify%
\[
\Delta t=T/N\leq\delta_{0}\text{, }\Delta x\leq\delta_{0},
\]
and%
\begin{equation}\label{CFL_cond_oct}
\max\{\left(  1-\operatorname*{sgn}(b)\right)  \tau_{1},\left(
1-\operatorname*{sgn}(d)\right)  \tau_{2}\}\Delta t\leq \Delta x,
\end{equation}
if we consider the numerical scheme \eqref{schema_bd=0} along with the initial data \eqref{conditie_initiala2} with the numerical viscosities $\tau_1$ and $\tau_2$ as well as the approximation
$\left(  \eta_{\Delta}^{n},u_{\Delta}^{n}\right)  _{n\in\overline{1,N}}$
defined by \eqref{troncature}, then, the numerical scheme \eqref{schema_bd=0} is first order convergent i.e. the convergence error
defined in \eqref{convergence_error1} satisfies:
\[
\sup_{n\in\overline{0,N}}\mathcal{E}\left(  e^{n},f^{n}\right)  \leq
C_{abcd}\left(  \Delta x\right)  ^{2},
\]
where $C_{abcd}$ depends on the parameters $a$,$b$,$c$,$d$, on
$\sup\limits_{t\in\left[  0,T\right]  }\left\Vert \left(  \eta(t),u(t)\right)
\right\Vert _{H^{s_{bc}}\times H^{s_{ad}}}$ and on $T$.
\end{theorem}

\begin{remark}
In this case, one of the two equations of \eqref{schema_bd=0} does not contain the operator $(I-bD_+D_-)^{-1}$ or $(I-dD_+D_-)^{-1}$. Artificial viscosity together with a hyperbolic CFL-type condition are thus needed in order to stabilize the numerical scheme. We use a centered scheme for the first-order derivatives combined with a Rusanov-type diffusion coefficient and a CFL-type condition (which, of course, corresponds to Relation \eqref{CFL_cond_oct}).
\end{remark}

\begin{remark} In order to have less diffusive schemes, we may in fact update the Rusanov coefficient at each time step i.e. take $\tau_i=\tau_i^n$ depending on $n$ such that 
$$||u_{\Delta}^n||_{\ell^{\infty}}+\alpha\leq \tau_i^n \text{\ \ \ with\ \ \ } i\in\{1,2\}.$$
\end{remark}

\begin{remark}
As mentioned in Remark \ref{asymptotic_preserving_scheme_10_NOv}, our numerical scheme \eqref{schema_bd=0} is not suitable if we consider System \eqref{1neliniar_cont} with the small parameter $\epsilon$. Indeed, the energy inequality changes in this case in 
\[
\sup_{n\in\overline{0,N}}\mathcal{E}\left(  e^{n},f^{n}\right)  \leq
C_{abcd}\left( \frac{ \Delta x}{\epsilon}\right)  ^{2},
\]
provided $\frac{\Delta x}{\epsilon}\leq \delta_0$ and $\frac{\Delta t}{\epsilon}\leq \delta_0$. 

\end{remark}

Our results owe much to the technics developed recently by Court\`{e}s, Lagouti\`{e}re and Rousset in \cite{CourtesLagoutiereRousset2017}. Let us give some more details. In order to study the Korteweg-de-Vries equation%
\[
\partial_{t}v+v\partial_{x}v+\partial_{xxx}^{3}v=0,
\]
they employ the following $\theta$-scheme, with $\theta\in[0,1]$,%
\[
\frac{1}{\Delta t}\left(  v^{n+1}-v^{n}\right)  +D_{+}D_{+}D_{-}%
((1-\theta)v^{n}+\theta v^{n+1})+D\left(  \frac{(v^{n})^{2}}{2}\right)
=\frac{\tau\Delta x}{2}D_{+}D_{-}v^{n}.
\]
With the aim to study the order of convergence, they consider the finite
volume discrete operators:%
\[
\left(  v_{\Delta}\right)^{n}  _{j}=\frac{1}{[\inf(t^{n+1},T)-t^n]\Delta x}\int_{t^{n}%
}^{\inf(t^{n+1}, T)}\int_{x_{j}}^{x_{j+1}}v\left(  s,y\right)  dsdy
\]
and the convergence error%
\[
e^{n}=v^{n}-v_{\Delta}^{n},
\]
which obeys the following equation:%
\[
e^{n+1}+\theta\Delta tD_{+}D_{+}D_{-}(e^{n+1})=\mathscr{B}e^{n},
\]
with%
\begin{small}
\begin{align*}
\mathscr{B}e^{n}  &  =e^{n}-\left(  1-\theta\right)  \Delta tD_{+}D_{+}%
D_{-}(e^{n})-\Delta tD\left(  e^{n}v_{\Delta}^{n}\right)   -\Delta tD\left(  \frac{(e^{n})^{2}}{2}\right)  +\frac{\tau\Delta
x\Delta t}{2}D_{+}D_{-}e^{n}-\Delta t\epsilon^{n},
\end{align*}
\end{small}
where $\epsilon^{n}$ is the consistency error. In order to establish $\ell^{2}_{\Delta}%
$-estimates, they show that under the CFL condition%
\[
\left\{
\begin{array}
[c]{l}%
\tau>\left\Vert v_{\Delta}
^{n}\right\Vert _{\ell^{\infty}},\\
\left[\tau+\frac{1}{2}\right]\Delta t<\Delta x,\\
\left(  1-2\theta\right)  \Delta t<\frac{\left(  \Delta x\right)  ^{3}}{4},
\end{array}
\right.
\]
the following holds true:%
\begin{equation}
\left\Vert e^{n+1}+\theta\Delta tD_{+}D_{+}D_{-}(e^{n+1})\right\Vert _{\ell_{\Delta}^{2}}^{2}\leq
\Delta tC\left\Vert \epsilon^{n}\right\Vert _{\ell_{\Delta}^{2}}^{2}+\left(  1+C\Delta
t\right)  \left\Vert e^{n}+\theta\Delta tD_{+}D_{+}D_{-}(e^{n})\right\Vert
_{\ell_{\Delta}^{2}}^{2}, \label{estimation}%
\end{equation}
provided that $\left\Vert e^{n}\right\Vert _{\ell^{\infty}}$ is sufficiently
small. Thus, they are able to use the discrete Gronwall lemma and close their
estimates. The proof of $\left(  \text{\ref{estimation}}\right)  $ is rather
technical and tricky. Loosely speaking, using some clever identities when computing
$\left\Vert \mathscr{B}e^{n}\right\Vert _{\ell_{\Delta}^{2}}^{2}$, they get
several negative terms wich, under the above CFL and smallness conditions, are
used to balance the "bad" terms.

In Section \ref{Burgers_type_estimates} we study a discrete operator which
appears in hyperbolic systems and we provide a bound for its $\ell_{\Delta}%
^{2}$-norm, the proof being essentially inspired from \cite{CourtesLagoutiereRousset2017}. As a consequence, we establish a higher-order estimate which proves crucial in the analysis of some of the $abcd$-systems. Among the systems in view, we distinguish three situations:

\begin{itemize}
\item when $bd>0$, establishing energy estimates for the convergence error
can be done by imitating the approach from the continuous case. The structure
of the equations provides enough control such that we do not need to impose
numerical viscosity or CFL-type conditions.

\item when at least one of the weakly dispersive operators does not appear, we
have to work only with estimations established in Section
\ref{Burgers_type_estimates} (see for instance the case $a<0$,\ $b=c=d=0$) either

\item combine the technics of Section \ref{Burgers_type_estimates} with
"continuous-type" estimates like those established for the case $bd>0$
(see for instance the case $a=b=c=0,$ $d>0$).
\end{itemize}

In order to illustrate the theoretical order of convergence (see Theorem \ref{Teorema_Numerica1} and Theorem \ref{Teorema_Numerica2}), we compare the numerical solutions computed with our schemes with the exact traveling wave solutions which were computed by Chen in \cite{Chen1998} and by Bona and Chen in \cite{BonaChen1998}.

Finally, we use our results in order to study exact traveling waves interactions. Our experiments are inspired from \cite{AntonopoulosDougalis2012}, \cite{AntonopoulusDugalisMitsotakis2009}, \cite{BonaChen1998} and \cite{BonaChen2016}. We perform two such numerical experiments. Recently, in \cite{BonaChen2016} Bona and Chen pointed out that finite time blow-up seems
to occur at the head-on collision of the two exact solutions:%
\[
\left\{
\begin{array}
[c]{l}%
\eta_{\pm}\left(  t,x\right)  =\frac{15}{2}\operatorname{sech}^{2}\left(
\frac{3}{\sqrt{10}}\left(  x-x_{0}\mp\frac{5}{2}t\right)  \right)  -\frac
{45}{4}\operatorname{sech}^{4}\left(  \frac{3}{\sqrt{10}}\left(  x-x_{0}%
\mp\frac{5}{2}t\right)  \right)  ,\\
\\
u_{\pm}\left(  t,x\right)  =\pm\frac{15}{2}\operatorname{sech}^{2}\left(
\frac{3}{\sqrt{10}}(x-x_{0}\mp\frac{5}{2}t)\right)  .
\end{array}
\right.
\]
 In order to build confidence in our codes, we repeat their experiment, in the same conditions and we show that we
obtain roughly the same results. We observed that the BBM-BBM system is not the only one having traveling waves of the above form. In particular, we provide numerical evidence that a blow-up phenomenon might occur on the head-on collision of the traveling waves for the case when the parameters are $a=c=d=0$, $b=3/5$.  

The plan of the rest of this paper is the following. In the next section, we
give a list of the main notations and identities that we use all along in this
manuscript. Section \ref{section b,d>0} is devoted to the proof of Theorem
\ref{Teorema_Numerica1}. The proof of Theorem \ref{Teorema_Numerica2} is more
involved as we cannot treat in a uniformly manner all the cases when $bd=0$.
First, we establish in Section \ref{Burgers_type_estimates} some estimates for
discrete operators appearing in hyperbolic systems. The rest of Section
\ref{section bd=0} is dedicated to the proof corresponding to the different
values of the $abcd$ parameters appearing in Theorem \ref{Teorema_Numerica2}.
Finally, in Sections \ref{Simulari} and \ref{travel_waves_collision}, we present our numerical simulations.

\subsection{Notations.\label{Notations_numerica}}
In the following we present the main notations that will be used throughout the
rest of the paper. Let us fix $\Delta x>0$. For all $v=\left(  v_{j}\right)
_{j\in\mathbb{Z}}\in\ell^{\infty}\left(  \mathbb{Z}\right)  $, we introduce the
spatial shift operators:%
\[
\left(  S_{\pm}v\right)  _{j}:=v_{j\pm1},
\]
For $v=\left(  v_{j}\right)  _{j\in\mathbb{Z}},$ $w=\left(  w_{j}\right)
_{j\in\mathbb{Z}}\in\ell^{2}\left(  \mathbb{Z}\right)  $, we define the product
operator:%
\[
vw=\left(  v_{j}w_{j}\right)  _{j\in\mathbb{Z}}\in\ell^{1}\left(
\mathbb{Z}\right)  .
\]
Also, we denote by%
\[
v^{2}=vv.
\]

\noindent We list below some basic formulas, whose proofs can be found in \cite{CourtesLagoutiereRousset2017}.\\
\noindent The following identities describe the derivation law of a product, for $v, w\in\ell^{2}\left(  \mathbb{Z}\right)  $ %
\begin{align}
D\left(  vw\right)   &  =vDw+\frac{1}{2}\left(  S_{+}wD_{+}v+S_{-}%
wD_{-}v\right),  \label{product_rule1}\\
D\left(  vw\right)   &  =S_{+}vDw+DvS_{-}w,\label{product_rule2}\\
D_{+}(vw) &  =S_{+}vD_{+}w+wD_{+}v.\label{product_rule3}%
\end{align}
We observe that we dispose of the following basic integration by parts rules, for $w,v\in\ell^2_{\Delta}(\mathbb{Z})$:%
\begin{align}
\left\langle D_{+}v,w\right\rangle  &  =-\left\langle v,D_{-}w\right\rangle
\label{integr_by_part_1},\\
\left\langle Dv,w\right\rangle  &  =-\left\langle v,Dw\right\rangle .\label{integr_by_part_2}
\end{align}
In particular, we see that, for $v\in\ell^2_{\Delta}(\mathbb{Z})$:
\[
\left\langle Dv,v\right\rangle =0,
\]
and
\[
\left\langle v,D_{+}v\right\rangle =-\frac{\Delta x}{2}||D_{+}v||_{\ell
_{\Delta}^{2}}^{2}.
\]
More elaborate integration by parts identities are the following, for $w,v\in\ell^2_{\Delta}(\mathbb{Z})$,%
\begin{equation}
\left\langle v,D\left(  vw\right)  \right\rangle =\frac{1}{2}\left\langle
D_{+}w,vS_{+}v\right\rangle \label{IPP1},%
\end{equation}
respectively%
\begin{equation}
\left\langle D_{+}D_{-}v,D\left(  vw\right)  \right\rangle =-\frac{1}{\Delta
x^{2}}\left\langle D_{+}w,vS_{+}v\right\rangle +\frac{1}{\Delta x^{2}%
}\left\langle Dw,S_{-}vS_{+}v\right\rangle .\label{IPP2}%
\end{equation}
When $v=w\in\ell^2_{\Delta}(\mathbb{Z})$ in the previous equation, some additional simplifications occur and one has
\begin{equation}
\left\langle D_{+}D_{-}v,D\left(  v^2\right)  \right\rangle =\frac{1}{3}\left\langle D_{+}v, (D_+v)^2\right\rangle-\frac{4}{3
}\left\langle Dv, (Dv)^2\right\rangle .\label{IPP2BiS}%
\end{equation}
Also, it holds true that, for $v\in\ell^2_{\Delta}(\mathbb{Z})$,%
\begin{equation}
\left\langle v,D\left(  v^{2}\right)  \right\rangle =-\frac{\Delta x^2}%
{6}\left\langle D_+v,\left(  D_{+}v\right)  ^{2}\right\rangle ,\label{IPP3}%
\end{equation}
\begin{equation}\label{Norm_D_v_carre}
||D(v^2)||_{\ell^2_{\Delta}}^2=\left\langle (Dv)^2, \left(\frac{S_+v+S_-v}{2}\right)^2 \right\rangle
\end{equation}
and
\begin{equation}
||D_{+}D_{-}v||_{\ell_{\Delta}^{2}}^{2}=\frac{4}{\Delta x^{2}}||D_{+}%
v||_{\ell_{\Delta}^{2}}^{2}-\frac{4}{\Delta x^{2}}||Dv||_{\ell_{\Delta}^{2}%
}^{2}.\label{IPP4}%
\end{equation}
Finally, for $v,w\in\ell^2_{\Delta}(\mathbb{Z})$, we recall Lemma 5 of \cite{CourtesLagoutiereRousset2017}
\begin{equation}
||D(vw)||_{\ell_{\Delta}^{2}}^{2}\leq \left\{||w||_{\ell^{\infty}}^2+\Delta t||D_+w||_{\ell^{\infty}}^2\right\}||Dv||_{\ell^2_{\Delta}}^2+\left\{\frac{1}{\Delta t}||w||^2_{\ell^{\infty}}+\frac{3}{4}||D_+w||_{\ell^{\infty}}^2\right\}||v||_{\ell^2_{\Delta}}^2\label{IPP5}%
\end{equation}
and Lemma 8 of \cite{CourtesLagoutiereRousset2017}
\begin{equation}
\left\langle D(vw), D(v^2)\right\rangle\leq 2||v||_{\ell^{\infty}}||w||_{\ell^{\infty}}||Dv||_{\ell^{2}_{\Delta}}^2
-\frac{8\Delta x^2}{3}\left\langle Dw, \left(Dv\right)^3\right\rangle-\frac{2}{3}\left\langle DDw,v^3\right\rangle\label{EQ_4}.\end{equation}
In the following, we will frequently use the notations:%
\begin{equation}
\left\{
\begin{array}
[c]{c}%
e^{n}=\eta^{n}-\eta_{\Delta}^{n},\text{\ \ \ \ \ }E^{\pm}\left(  n\right)  =e^{n+1}\pm
e^{n},\\
f^{n}=u^{n}-u_{\Delta}^{n},\text{\ \ \ \ \ }F^{\pm}\left(  n\right)  =f^{n+1}\pm f^{n}.
\end{array}
\right.  \label{eroare}%
\end{equation}

We end this section with the following discrete version of Gronwall's lemma, a
proof of which can be found in \cite{Clark1987}:

\begin{lemma}
\label{Granwall_discret_numerica}Let $v=\left(  v^{n}\right)  _{n\in
\mathbb{N}},$ $a=\left(  a^{n}\right)  _{n\in\mathbb{N}},$ $b=\left(
b^{n}\right)  _{j\in\mathbb{N}}$ be sequences of real numbers with $b^{n}%
\geq0$ for all $n\in\mathbb{N}$ which satisfies%
\[
v^{n}\leq a^{n}+\sum_{j=0}^{n-1}b^{j}v^{j},
\]
for all $n\in\mathbb{N}$. Then, for any $n\in\mathbb{N}$, we have:%
\[
v^{n}\leq\max_{k\in\overline{1,n}}a^{k}\prod_{j=0}^{n-1}\left(  1+b^{j}%
\right)  .
\]

\end{lemma}

\section{The proof of the main results}

This section is devoted to the proofs of Theorems \ref{Teorema_Numerica1} and \ref{Teorema_Numerica2}.

\subsection{The proof of Theorem \ref{Teorema_Numerica1}.\label{section b,d>0}}
Let us recall the energy functional:%
\begin{align}
  \mathcal{E}\left(  e,f\right)  \overset{def.}{=}&\left\Vert e\right\Vert
_{\ell_{\Delta}^{2}}^{2}+\left(  b-c\right)  \left\Vert D_{+}e\right\Vert
_{\ell_{\Delta}^{2}}^{2}+b\left(  -c\right)  \left\Vert D_{+}D_{-}e\right\Vert
_{\ell_{\Delta}^{2}}^{2}\nonumber\\
  +&\left\Vert f\right\Vert _{\ell_{\Delta}^{2}}^{2}+(d-a)\left\Vert
D_{+}f\right\Vert _{\ell_{\Delta}^{2}}^{2}+d\left(  -a\right)  \left\Vert
D_{+}D_{-}f\right\Vert _{\ell_{\Delta}^{2}}^{2}.\label{E=mc^2}%
\end{align}

\noindent  For all $n\geq0$, we will
consider $\left(  \epsilon_{1}^{n},\epsilon_{2}^{n}\right)  \in\left(\ell_{\Delta
}^{2}(\mathbb{Z})\right)^2$ the consistency error defined as:%
\begin{equation}
\left\{
\begin{array}
[c]{l}%
\frac{1}{\Delta t}\left(  I-bD_{+}D_{-}\right)  (\eta_{\Delta}^{n+1}%
-\eta_{\Delta}^{n})+\frac{1}{2}\left(  I+aD_{+}D_{-}\right)  D\left(
u_{\Delta}^{n}+u_{\Delta}^{n+1}\right)  +D\left(  \eta_{\Delta}^{n}u_{\Delta
}^{n}\right)  =\epsilon_{1}^{n},\\
\\
\frac{1}{\Delta t}\left(  I-dD_{+}D_{-}\right)  (u_{\Delta}^{n+1}-u_{\Delta
}^{n})+\frac{1}{2}\left(  I+cD_{+}D_{-}\right)  D\left(  \eta_{\Delta}%
^{n}+\eta_{\Delta}^{n+1}\right)  +\frac{1}{2}D\left(  \left(  u_{\Delta}%
^{n}\right)  ^{2}\right)  =\epsilon_{2}^{n}.
\end{array}
\right.  \label{consistency_error1}%
\end{equation}

\begin{proposition}\label{_prop_1_avcd}
For all $n\in\overline{0,N-1}$, there exist two constants $C_{1}$ and $C_{2}$ depending on $a, b, c, d$, on the
$\ell^{\infty}$-norm of $\left(  \eta_{\Delta}^{n},u_{\Delta}^{n}\right)  $
and $\left(  D(\eta_{\Delta}^{n}),D(u_{\Delta}^{n})\right)  $
and proportional to $\max\left\{||e^n||_{\ell^{\infty}}, ||f^n||_{\ell^{\infty}}\right\}$, such that, 
\begin{equation}
\mathcal{E}\left(  e^{n+1},f^{n+1}\right)  \leq 2\Delta t\left\Vert \epsilon^{n}\right\Vert _{\ell^{2}}^{2}+(1+\Delta tC_{1})
\mathcal{E}\left(  e^{n},f^{n}\right)  +\Delta tC_{2}\mathcal{E}\left(  e^{n+1},f^{n+1}\right),\label{ineg_inducti_totala}%
\end{equation}
with $(e^n,f^n)$ the two convergence errors defined by \eqref{eroare} and $
\left\Vert \epsilon^{n}\right\Vert _{\ell_{\Delta}^{2}}=\max\left\{
\left\Vert \epsilon_{1}^{n}\right\Vert _{\ell_{\Delta}^{2}},\left\Vert
\epsilon_{2}^{n}\right\Vert _{\ell_{\Delta}^{2}}\right\},
$ defined by \eqref{consistency_error1}.\\
\end{proposition}

\begin{proof}
As we announced in the introduction, we will establish energy estimates
imitating the approach from the continuous case. \\
\textbf{For the Crank-Nicolson case ($\theta=\frac{1}{2}$).}\ \ \ \ 
Using the notations introduced in $\left(  \text{\ref{convergence_error1}%
}\right)  $\ and $\left(  \text{\ref{eroare}}\right)  $, we see that the
equations governing the convergence error $\left(  e^{n},f^{n}\right)  $ are
the following:
\begin{equation}
\left\{
\begin{array}
[c]{r}%
\left(  I-bD_{+}D_{-}\right)  (E^{-}\left(  n\right)  )+\frac{\Delta t}%
{2}\left(  I+aD_{+}D_{-}\right)  D\left(  F^{+}\left(  n\right)  \right)
+\Delta tD\left(  e^{n}u_{\Delta}^{n}\right)  \hspace*{2cm}\\
+\Delta tD\left(  \eta_{\Delta}^{n}f^{n}\right)  +\Delta tD\left(  e^{n}%
f^{n}\right)  =-\Delta t\epsilon_{1}^{n},\hspace*{-0cm}\\
\\
\left(  I-dD_{+}D_{-}\right)  (F^{-}\left(  n\right)  )+\frac{\Delta t}%
{2}\left(  I+cD_{+}D_{-}\right)  D\left(  E^{+}\left(  n\right)  \right)
+\Delta tD\left(  f^{n}u_{\Delta}^{n}\right)  \hspace*{2cm}\\
+\frac{\Delta t}{2}D\left(  \left(  f^{n}\right)  ^{2}\right)  =-\Delta
t\epsilon_{2}^{n}.\hspace*{-0cm}
\end{array}
\right.  \label{abcd_discret3}%
\end{equation}
Let $n\in\overline{0,N-1}$
and observe that by multiplying the first equation of \eqref{abcd_discret3} by $\left(  I+cD_{+}%
D_{-}\right)  E^{+}\left(  n\right)  $, the second by $\left(  I+aD_{+}%
D_{-}\right)  F^{+}\left(  n\right)  $ and adding up the results, we find
that
\begin{align}
& \left\langle(I-bD_+D_-)(E^-(n)), (I+cD_+D_-)E^+(n)\right\rangle\nonumber\\
&+\frac{\Delta t}{2}\left\langle(I+aD_+D_-)D(F^+(n)),  (I+cD_+D_-)E^+(n)\right\rangle\nonumber\\
&+\left\langle(I-dD_+D_-)(F^-(n)), (I+aD_+D_-)F^+(n)\right\rangle\nonumber\\
&+\frac{\Delta t}{2}\left\langle(I+cD_+D_-)D(E^+(n)),  (I+aD_+D_-)F^+(n)\right\rangle \nonumber\\
  &=-\Delta t\left\langle \epsilon_{1}^{n},\left(  I+cD_{+}D_{-}\right)
E^{+}\left(  n\right)  \right\rangle -\Delta t\left\langle \epsilon_{2}%
^{n},\left(  I+aD_{+}D_{-}\right)  F^{+}\left(  n\right)  \right\rangle
\nonumber\\
&  -\Delta t\left\langle D\left(  e^{n}u_{\Delta}^{n}\right)  +D\left(
\eta_{\Delta}^{n}f^{n}\right)  ,\left(  I+cD_{+}D_{-}\right)  E^{+}\left(
n\right)  \right\rangle \nonumber\\
&  -\Delta t\left\langle D\left(  e^{n}f^{n}\right)  ,\left(  I+cD_{+}%
D_{-}\right)  E^{+}\left(  n\right)  \right\rangle \nonumber\\
&  -\Delta t\left\langle D\left(  f^{n}u_{\Delta}^{n}\right)  ,\left(
I+aD_{+}D_{-}\right)  F^{+}\left(  n\right)  \right\rangle \label{b,d>0}\\
&  -\frac{\Delta t}{2}\left\langle D\left(  \left(  f^{n}\right)  ^{2}\right)
,\left(  I+aD_{+}D_{-}\right)  F^{+}\left(  n\right)  \right\rangle
\overset{not.}{=}\sum_{i=1}^{5}T_{i}.\text{ }\nonumber
\end{align}

We begin by treating the left hand side of \eqref{b,d>0}. Notice that
\begin{equation*}
\left\langle\mathcal{L}\left(E^-(n)\right),\mathcal{L}\left(E^+(n)\right)\right\rangle=||\mathcal{L}e^{n+1}||_{\ell^2_{\Delta}}^2-||\mathcal{L}e^{n}||_{\ell^2_{\Delta}}^2.
\end{equation*}
with $\mathcal{L}$ any linear operator. With this in mind together with Relations \eqref{integr_by_part_1} and \eqref{integr_by_part_2}, it gives, 
\begin{align}
& \left\langle(I-bD_+D_-)(E^-(n)), (I+cD_+D_-)E^+(n)\right\rangle\nonumber\\
&+\frac{\Delta t}{2}\left\langle(I+aD_+D_-)D(F^+(n)),  (I+cD_+D_-)E^+(n)\right\rangle\nonumber\\
&+\left\langle(I-dD_+D_-)(F^-(n)), (I+aD_+D_-)F^+(n)\right\rangle\nonumber\\
&+\frac{\Delta t}{2}\left\langle(I+cD_+D_-)D(E^+(n)),  (I+aD_+D_-)F^+(n)\right\rangle\nonumber\\
&=\mathcal{E}(e^{n+1}, f^{n+1})-\mathcal{E}(e^n,f^n).\label{LHS_abcd}
\end{align}

\noindent $\bullet$ Let us now focus on $T_{1}$. We recall that
\[
\left\Vert \epsilon^{n}\right\Vert _{\ell_{\Delta}^{2}}=\max\left\{
\left\Vert \epsilon_{1}^{n}\right\Vert _{\ell_{\Delta}^{2}},\left\Vert
\epsilon_{2}^{n}\right\Vert _{\ell_{\Delta}^{2}}\right\}  ,
\]
and we write that, thanks to Cauchy-Schwarz inequality (we recall that $c\leq 0$ and $a\leq 0$)
\begin{align*}
 & -\Delta t\left\langle \epsilon_{1}^{n},\left(  I+cD_{+}D_{-}\right)
E^{+}\left(  n\right)  \right\rangle -\Delta t\left\langle \epsilon_{2}%
^{n},\left(  I+aD_{+}D_{-}\right)  F^{+}\left(  n\right)  \right\rangle
  \\&\leq\Delta t\left\Vert \epsilon_{1}^{n}\right\Vert _{\ell_{\Delta}^{2}%
}\left(  \left\Vert E^{+}\left(  n\right)  \right\Vert _{\ell_{\Delta}^{2}%
}-c\left\Vert D_{+}D_{-}E^{+}\left(  n\right)  \right\Vert _{\ell_{\Delta}%
^{2}}\right)   \\
&+\Delta t\left\Vert \epsilon_{2}^{n}\right\Vert _{\ell_{\Delta}^{2}}\left(
\left\Vert F^{+}\left(  n\right)  \right\Vert _{\ell_{\Delta}^{2}}-a\left\Vert
D_{+}D_{-}F^{+}\left(  n\right)  \right\Vert _{\ell_{\Delta}^{2}}\right).
\end{align*}
By applied Young inequality, we recover the $\ell^2_{\Delta}$-norm of $e^n$ and $f^n$.
\begin{align}
 &-\Delta t\left\langle \epsilon_{1}^{n},\left(  I+cD_{+}D_{-}\right)
E^{+}\left(  n\right)  \right\rangle -\Delta t\left\langle \epsilon_{2}%
^{n},\left(  I+aD_{+}D_{-}\right)  F^{+}\left(  n\right)  \right\rangle\nonumber\\
&  \leq2\Delta t\left\Vert \epsilon\right\Vert _{\ell^{2}_{\Delta}}^{2}+\Delta t\left(
\left\Vert e^{n}\right\Vert _{\ell^{2}_{\Delta}}^{2}+\left\Vert f^{n}\right\Vert
_{\ell^{2}_{\Delta}}^{2}+c^{2}||D_{+}D_{-}e^{n}||_{\ell_{\Delta}^{2}}^{2}+a^{2}%
||D_{+}D_{-}f^{n}||_{\ell_{\Delta}^{2}}^{2}\right)  \nonumber\\
& +\Delta t\left(  \left\Vert e^{n+1}\right\Vert _{\ell^{2}_{\Delta}}^{2}+\left\Vert
f^{n+1}\right\Vert _{\ell^{2}_{\Delta}}^{2}+c^{2}||D_{+}D_{-}e^{n+1}||_{\ell_{\Delta
}^{2}}^{2}+a^{2}||D_{+}D_{-}f^{n+1}||_{\ell_{\Delta}^{2}}^{2}\right)
\nonumber\\
&  \leq2\Delta t\left\Vert \epsilon^{n}\right\Vert _{\ell^{2}_{\Delta}}^{2}+\Delta
t\max\left\{  1,\frac{-c}{b},\frac{-a}{d}\right\}  \mathcal{E}\left(
e^{n},f^{n}\right)  +\Delta t\max\left\{  1,\frac{-c}{b},\frac{-a}%
{d}\right\}  \mathcal{E}\left(  e^{n+1},f^{n+1}\right)  .\label{bdpoz12}%
\end{align}
$\bullet$ Let us treat $T_{2}$. Using $\left(  \text{\ref{product_rule2}}\right)  $, we
first write%
\begin{small}
\begin{align*}
  \Delta t\left\Vert D\left(  e^{n}u_{\Delta}^{n}\right)  \right\Vert
_{\ell^{2}_{\Delta}} &=\Delta t\left\Vert S_{-}e^{n}Du_{\Delta}^{n}+S_{+}u_{\Delta}^{n}%
De^{n}\right\Vert _{\ell^{2}_{\Delta}}\\
&  \leq\Delta t\left\Vert Du_{\Delta}^{n}\right\Vert _{\ell^{\infty}%
}\left\Vert S_{-}e^{n}\right\Vert _{\ell^{2}_{\Delta}}+\Delta t\left\Vert
S_{+}u_{\Delta}^{n}\right\Vert _{\ell^{\infty}}\left\Vert De^{n}\right\Vert
_{\ell^{2}_{\Delta}}\\
&  \leq\Delta t\max\left\{  1,\frac{1}{\sqrt{b-c}}\right\}  \left(  \left\Vert
Du_{\Delta}^{n}\right\Vert _{\ell^{\infty}}\left\Vert e^{n}\right\Vert
_{\ell^{2}_{\Delta}}+\left\Vert u_{\Delta}^{n}\right\Vert _{\ell^{\infty}}\sqrt
{b-c}\left\Vert De^{n}\right\Vert _{\ell^{2}_{\Delta}}\right)  \\
&  \leq\Delta t\max\left\{  \left\Vert u_{\Delta}^{n}\right\Vert
_{\ell^{\infty}},\left\Vert Du_{\Delta}^{n}\right\Vert _{\ell^{\infty}%
}\right\}  \max\left\{  1,\frac{1}{\sqrt{b-c}}\right\}  \left(  \left\Vert
e^{n}\right\Vert _{\ell^{2}_{\Delta}}+\sqrt{b-c}\left\Vert De^{n}\right\Vert _{\ell
^{2}_{\Delta}}\right)  .
\end{align*}
\end{small}
Proceeding in a similar fashion with the other terms we arrive, thanks to the Cauchy-Schwarz inequality, at (we recall that $c\leq0$)%
\begin{multline*} -\Delta t\left\langle D\left(  e^{n}u_{\Delta}^{n}\right)  +D\left(
\eta_{\Delta}^{n}f^{n}\right)  ,\left(  I+cD_{+}D_{-}\right)  E^{+}\left(
n\right)  \right\rangle  \\\leq\Delta t\left(  \left\Vert D\left(  e^{n}u_{\Delta}^{n}\right)
\right\Vert _{\ell^{2}_{\Delta}}+\left\Vert D\left(  \eta_{\Delta}^{n}f^{n}\right)
\right\Vert _{\ell^{2}_{\Delta}}\right)  \left(  \left\Vert E^{+}\left(  n\right)
\right\Vert _{\ell^{2}_{\Delta}}-c\left\Vert D_{+}D_{-}E^{+}\left(  n\right)
\right\Vert _{\ell^{2}_{\Delta}}\right).
\end{multline*}
By Definition \eqref{eroare} of $E^+(n)$, one has
\begin{align*}
& \left\Vert E^{+}\left(  n\right)
\right\Vert _{\ell^{2}_{\Delta}}-c\left\Vert D_{+}D_{-}E^{+}\left(  n\right)
\right\Vert _{\ell^{2}_{\Delta}}\\
&\leq ||e^{n+1}||_{\ell^2_{\Delta}}+||e^n||_{\ell^2_{\Delta}}-c||D_+D_-(e^{n+1})||_{\ell^2_{\Delta}}-c||D_+D_-(e^{n})||_{\ell^2_{\Delta}}\\
&\leq \max\left\{1, \sqrt{-c/b}\right\}2\left[\sqrt{\mathcal{E}(e^{n+1}, f^{n+1})}+\sqrt{\mathcal{E}(e^{n}, f^{n})}\right].
\end{align*}
These together give
\begin{small}
\begin{align}
-\Delta t\left\langle D\left(  e^{n}u_{\Delta}^{n}\right)  +D\left(
\eta_{\Delta}^{n}f^{n}\right)  ,\left(  I+cD_{+}D_{-}\right)  E^{+}\left(
n\right)  \right\rangle
&  \leq\Delta tC_{1,1}\mathcal{E}\left(  e^{n},f^{n}\right)  +\Delta
tC_{2,1}\mathcal{E}\left(  e^{n+1},f^{n+1}\right)  \label{bdpoz2}%
\end{align}
\end{small}
where $C_{1,1}$ and $C_{2,1}$ can be written, for example, as a numerical
constants multiplied with:
\begin{align}
&\max\left\{  1,\sqrt{-\frac{c}{b}}\right\}  \left[ \max\left\{  \left\Vert
u_{\Delta}^{n}\right\Vert _{\ell^{\infty}},\left\Vert Du_{\Delta}%
^{n}\right\Vert _{\ell^{\infty}}\right\}  \max\left\{  1,\frac{1}{\sqrt{b-c}}\right\}\right.\nonumber\\
&\left.+\max\left\{  \left\Vert \eta_{\Delta}^{n}\right\Vert _{\ell^{\infty}%
},\left\Vert D\eta_{\Delta}^{n}\right\Vert _{\ell^{\infty}}\right\}
\max\left\{  1,\frac{1}{\sqrt{d-a}}\right\}  \right]. \label{proportional1}
\end{align}
$\bullet$ We treat $T_{4}$ in same spirit as above in order to obtain that
\begin{equation}
-\Delta t\left\langle D\left(  f^{n}u_{\Delta}^{n}\right)  ,\left(
I+aD_{+}D_{-}\right)  F^{+}\left(  n\right)  \right\rangle \leq\Delta
tC_{1,2}\mathcal{E}\left(  e^{n},f^{n}\right)  +\Delta tC_{2,2}\mathcal{E}%
\left(  e^{n+1},f^{n+1}\right),  \label{bdpoz3}%
\end{equation}
where $C_{1,2}$ and $C_{2,2}$\ are multiples of :
\begin{equation}
\max\left\{  \left\Vert u_{\Delta}^{n}\right\Vert _{\ell^{\infty}},\left\Vert
Du_{\Delta}^{n}\right\Vert _{\ell^{\infty}}\right\}  \max\left\{
1,1/\sqrt{d-a}\right\}  \max\left\{  1,\sqrt{-a/d}\right\}
.\label{proportional2}%
\end{equation}
$\bullet$ In order to treat $T_{3}$, we first observe that%
\begin{align*}
&\left\Vert D\left(  e^{n}f^{n}\right)  \right\Vert _{\ell^{2}_{\Delta}} \\
&  =\left\Vert
S_{-}e^{n}Df^{n}+S_{+}f^{n}De^{n}\right\Vert _{\ell^{2}_{\Delta}}\\
&\leq\left\Vert
e^{n}\right\Vert _{\ell^{\infty}}\left\Vert Df^{n}\right\Vert _{\ell^{2}_{\Delta}%
}+\left\Vert f^{n}\right\Vert _{\ell^{\infty}}\left\Vert De^{n}\right\Vert
_{\ell^{2}_{\Delta}}\\
&  \leq\max\left\{  1/\sqrt{b-c},1/\sqrt{d-a}\right\}  \left(||e^n||_{\ell^{\infty}}  \sqrt
{d-a}\left\Vert Df^{n}\right\Vert _{\ell^{2}_{\Delta}}+||f^n||_{\ell^{\infty}}\sqrt{b-c}\left\Vert
De^{n}\right\Vert _{\ell^{2}_{\Delta}}\right)  .
\end{align*}
Thus, we obtain%
\begin{align}
-\Delta t\left\langle D\left(  e^{n}f^{n}\right)  ,\left(
I+cD_{+}D_{-}\right)  E^{+}\left(  n\right)  \right\rangle  \leq\Delta tC_{1,3}\mathcal{E}\left(  e^{n},f^{n}\right)  +\Delta
tC_{2,3}\mathcal{E}\left(  e^{n+1},f^{n+1}\right),  \label{bdpoz4}%
\end{align}
where $C_{1,3}$ respectively $C_{2,3}$ are proportional with
\begin{equation}
\max\left\{  1/\sqrt{b-c},1/\sqrt{d-a}\right\}  \max\left\{  1,\sqrt{-c/b}\right\}\max\left\{||e^n||_{\ell^{\infty}}, ||f^n||_{\ell^{\infty}}\right\}
\text{.}\label{proportional3}%
\end{equation}
$\bullet$ The same holds for $T_{5}$, namely, we get that%
\begin{equation}
-\frac{\Delta t}{2}\left\langle D\left(  \left(  f^{n}\right)  ^{2}\right)
,\left(  I+aD_{+}D_{-}\right)  F^{+}\left(  n\right)  \right\rangle \leq\Delta
tC_{1,4}\mathcal{E}\left(  e^{n},f^{n}\right)  +\Delta tC_{2,4}\mathcal{E}%
\left(  e^{n+1},f^{n+1}\right),  \label{bdpoz5}%
\end{equation}
where $C_{1,4}$ respectively $C_{2,4}$ are proportional with
\begin{equation}
\frac{1}{\sqrt{d-a}}\max\left\{  1,\sqrt{-a/d}\right\}||f^n||_{\ell^{\infty}}  \text{.}%
\label{proportional4}%
\end{equation}
Gathering the informations from $\left(  \text{\ref{b,d>0}}\right)  $, \eqref{LHS_abcd},
$\left(  \text{\ref{bdpoz12}}\right)  $, $\left(  \text{\ref{bdpoz2}}\right)
$, $\left(  \text{\ref{bdpoz3}}\right)  $, $\left(  \text{\ref{bdpoz4}%
}\right)  $ and $\left(  \text{\ref{bdpoz5}}\right)  $ we obtain the existence
of two constants $C_{1}$ and $C_{2}$ that depend on $a,b,c,d,$ and the
$\ell^{\infty}$-norm of $\left(  \eta_{\Delta}^{n},u_{\Delta}^{n}\right)  $
and $\left(  D\eta_{\Delta}^{n},Du_{\Delta}^{n}\right)  $ (dependence which we
can track using relations $\left(  \text{\ref{proportional1}}\right)  $,
$\left(  \text{\ref{proportional2}}\right)  $, $\left(
\text{\ref{proportional3}}\right)  $ respectively $\left(
\text{\ref{proportional4}}\right)  $) and that are proportional to $\max\{||e^n||_{\ell^{\infty}}, ||f^n||_{\ell^{\infty}}\}$ such that
\[
\mathcal{E}\left(  e^{n+1},f^{n+1}\right)  -\mathcal{E}\left(  e^{n}%
,f^{n}\right)  \leq2\Delta t\left\Vert \epsilon^{n}\right\Vert _{\ell^{2}_{\Delta}}%
^{2}+\Delta tC_{1}\mathcal{E}\left(  e^{n},f^{n}\right)  +\Delta
tC_{2}\mathcal{E}\left(  e^{n+1},f^{n+1}\right).
\]

\textbf{For the implicit case ($\theta=1$).}\ \ \ \ 
The equations governing the convergence error $(e^n, f^n)$ are the following
\begin{equation}\label{SYS_1_10_NOv}
\left\{
\begin{split}
&(I-bD_+D_-)e^{n+1}+\Delta t(I+aD_+D_-)Df^{n+1}\\
&=(I-bD_+D_-)e^{n}-\Delta tD(e^nu_{\Delta}^n)-\Delta tD(\eta_{\Delta}^nf^n)-\Delta tD(e^nf^n)-\Delta t\epsilon_1^n,\\
&\\
&(I-dD_+D_-)f^{n+1}+\Delta t(I+cD_+D_-)De^{n+1}\\
&=(I-dD_+D_-)f^{n}-\Delta tD(f^nu_{\Delta}^n)-\frac{\Delta t}{2}D((f^n)^2)-\Delta t\epsilon_2^n.
\end{split}
\right.
\end{equation}

For that case, we fix the discret energy 
\begin{equation}\label{discret_energy_10_NOv}
\begin{split}
\mathcal{E}(e,f)&=(-c)d||(I-bD_+D_-)e||_{\ell^2_{\Delta}}^2+(-a)b||(I-dD_+D_-)f||_{\ell^2_{\Delta}}^2\\
&=(-c)d||e||_{\ell^2_{\Delta}}^2+2b(-c)d||D_+e||_{\ell^2_{\Delta}}^2+b^2(-c)d||D_+D_-e||_{\ell^2_{\Delta}}^2\\
&+(-a)b||f||_{\ell^2_{\Delta}}^2+2(-a)bd||D_+f||_{\ell^2_{\Delta}}^2+(-a)bd^2||D_+D_-f||_{\ell^2_{\Delta}}^2.
\end{split}
\end{equation}
This energy is of course equivalent to the one from \eqref{E=mc^2}.\\

Let us multiply the first equation of \eqref{SYS_1_10_NOv} by $\sqrt{-cd}$ and the second one of \eqref{SYS_1_10_NOv} by $\sqrt{-ab}$. The sum of $\ell^2_{\Delta}$-norm gives in that case
\begin{equation}\label{Left_right_hand_side_10_NOv}
\begin{split}
&-cd||(I-bD_+D_-)e^{n+1}+\Delta t(I+aD_+D_-)Df^{n+1}||^2_{\ell^2_{\Delta}}\\
&-ab||(I-dD_+D_-)f^{n+1}+\Delta t(I+cD_+D_-)De^{n+1}||_{\ell^2_{\Delta}}^2\\
&=-cd||(I-bD_+D_-)e^{n}-\Delta tD(e^nu_{\Delta}^n)-\Delta tD(\eta_{\Delta}^nf^n)-\Delta tD(e^nf^n)-\Delta t\epsilon_1^n||^2_{\ell^2_{\Delta}}\\
&-ab||(I-dD_+D_-)f^{n}-\Delta tD(f^nu_{\Delta}^n)-\frac{\Delta t}{2}D((f^n)^2)-\Delta t\epsilon_2^n||^2_{\ell^2_{\Delta}}.
\end{split}
\end{equation}
The left hand side of the previous equality gives
\begin{equation*}
\begin{split}
&-cd||(I-bD_+D_-)e^{n+1}+\Delta t(I+aD_+D_-)Df^{n+1}||^2_{\ell^2_{\Delta}}\\
&-ab||(I-dD_+D_-)f^{n+1}+\Delta t(I+cD_+D_-)De^{n+1}||_{\ell^2_{\Delta}}^2\\
&=-cd||(I-bD_+D_-)e^{n+1}||_{\ell^2_{\Delta}}^2-2cd\Delta t\left\langle (I-bD_+D_-)e^{n+1}, (I+aD_+D_-)Df^{n+1}\right\rangle\\
&-cd\Delta t^2||(I+aD_+D_-)Df^{n+1}||^2_{\ell^2_{\Delta}}-ab||(I-dD_+D_-)f^{n+1}||_{\ell^2_{\Delta}}^2\\
&-2ab\Delta t\left\langle (I-dD_+D_-)f^{n+1}, (I+cD_+D_-)De^{n+1}\right\rangle-ab\Delta t^2||(I+cD_+D_-)De^{n+1}||^2_{\ell^2_{\Delta}}.
\end{split}
\end{equation*}
For both cross products, it gives
\begin{equation*}
\begin{split}
&-2cd\Delta t\left\langle (I-bD_+D_-)e^{n+1}, (I+aD_+D_-)Df^{n+1}\right\rangle\\
&-2ab\Delta t\left\langle (I-dD_+D_-)f^{n+1}, (I+cD_+D_-)De^{n+1}\right\rangle\\
&= 2(-c)d\Delta t\left\langle e^{n+1}, Df^{n+1}\right\rangle+2(-a)(-c)d\Delta t\left\langle De^{n+1}, D_+D_-f^{n+1}\right\rangle\\
&+2b(-c)d\Delta t\left\langle De^{n+1}, D_+D_-f^{n+1}\right\rangle+2(-a)b(-c)d\Delta t\left\langle D_+D_-e^{n+1}, D_+D_-Df^{n+1}\right\rangle\\
&+2(-a)b\Delta t\left\langle f^{n+1}, De^{n+1}\right\rangle+2(-a)b(-c)\Delta t\left\langle Df^{n+1}, D_+D_-e^{n+1}\right\rangle\\
&+2(-a)bd\Delta t\left\langle Df^{n+1}, D_+D_-e^{n+1}\right\rangle+2(-a)b(-c)d\Delta t\left\langle D_+D_-f^{n+1}, D_+D_-De^{n+1}\right\rangle.
\end{split}
\end{equation*}
Thanks to integration by parts, Young's inequality together with Cauchy-Schwarz inequality, the previous equality simplifies into 
\begin{equation*}
\begin{split}
&-2cd\Delta t\left\langle (I-bD_+D_-)e^{n+1}, (I+aD_+D_-)Df^{n+1}\right\rangle\\
&-2ab\Delta t\left\langle (I-dD_+D_-)f^{n+1}, (I+cD_+D_-)De^{n+1}\right\rangle\\
&\geq - (-c)d\Delta t|| e^{n+1}||_{\ell^2_{\Delta}}^2 -(-a)b\Delta t|| f^{n+1}||_{\ell^2_{\Delta}}^2\\&-\left[(-a)(-c)d+b(-c)d+(-a)b\right]\Delta t|| De^{n+1}||_{\ell^2_{\Delta}}^2\\
&- \left[(-c)d+(-a)b(-c)+(-a)bd\right]\Delta t||Df^{n+1}||_{\ell^2_{\Delta}}^2\\
&-\left[(-a)b(-c)+(-a)bd\right]\Delta t|| D_+D_-e^{n+1}||_{\ell_{\Delta}^2}^2\\
&-\left[(-a)(-c)d +b(-c)d\right]\Delta t||D_+D_-f^{n+1}||_{\ell^2_{\Delta}}^2.
\end{split}
\end{equation*}

The left hand side of \eqref{Left_right_hand_side_10_NOv} becomes
\begin{equation*}
\begin{split}
&-cd||(I-bD_+D_-)e^{n+1}+\Delta t(I+aD_+D_-)Df^{n+1}||^2_{\ell^2_{\Delta}}\\
&-ab||(I-dD_+D_-)f^{n+1}+\Delta t(I+cD_+D_-)De^{n+1}||_{\ell^2_{\Delta}}^2\\
&\geq (-c)d||(I-bD_+D_-)e^{n+1}||_{\ell^2_{\Delta}}^2+(-a)b||(I-dD_+D_-)f^{n+1}||^2_{\ell^2_{\Delta}}\\
&- (-c)d\Delta t|| e^{n+1}||_{\ell^2_{\Delta}}^2 -(-a)b\Delta t|| f^{n+1}||_{\ell^2_{\Delta}}^2-\left[(-a)(-c)d+b(-c)d+(-a)b\right]\Delta t|| De^{n+1}||_{\ell^2_{\Delta}}^2\\
&- \left[(-c)d+(-a)b(-c)+(-a)bd\right]\Delta t||Df^{n+1}||_{\ell^2_{\Delta}}^2\\
&-\left[(-a)b(-c)+(-a)bd\right]\Delta t|| D_+D_-e^{n+1}||_{\ell_{\Delta}^2}^2
-\left[(-a)(-c)d +b(-c)d\right]\Delta t||D_+D_-f^{n+1}||_{\ell^2_{\Delta}}^2\\
&+(-c)d\Delta t^2||(I+aD_+D_-)Df^{n+1}||^2_{\ell^2_{\Delta}}+(-a)b\Delta t^2||(I+cD_+D_-)De^{n+1}||^2_{\ell^2_{\Delta}}.
\end{split}
\end{equation*}
Due to the definition of the energy \eqref{discret_energy_10_NOv}, one has
\begin{equation}\label{eq_fin_LHS_10_NOv}
\begin{split}
&-cd||(I-bD_+D_-)e^{n+1}+\Delta t(I+aD_+D_-)Df^{n+1}||^2_{\ell^2_{\Delta}}\\
&-ab||(I-dD_+D_-)f^{n+1}+\Delta t(I+cD_+D_-)De^{n+1}||_{\ell^2_{\Delta}}^2\\
&\geq \mathcal{E}(e^{n+1}, f^{n+1})- C_1\Delta t\mathcal{E}(e^{n+1}, f^{n+1}),
\end{split}
\end{equation}
with $$C_1=\max\left\{1, \frac{-a}{2b}+\frac{1}{2}+\frac{-a}{2(-c)d}, \frac{-c}{2(-a)b}+\frac{-c}{2d}+\frac{1}{2}, \frac{-a}{bd}+\frac{-a}{b(-c)}, \frac{-c}{bd}+\frac{-c}{(-a)d}\right\}.$$

Let us now focus on the right hand side of \eqref{Left_right_hand_side_10_NOv}. The triangular inequality together with Young's inequality give the existence of a constant $C_0$ independent of $\Delta t$, such that
\begin{equation*}
\begin{split}
&-cd||(I-bD_+D_-)e^{n}-\Delta tD(e^nu_{\Delta}^n)-\Delta tD(\eta_{\Delta}^nf^n)-\Delta tD(e^nf^n)-\Delta t\epsilon_1^n||^2_{\ell^2_{\Delta}}\\
&-ab||(I-dD_+D_-)f^{n}-\Delta tD(f^nu_{\Delta}^n)-\frac{\Delta t}{2}D((f^n)^2)-\Delta t\epsilon_2^n||^2_{\ell^2_{\Delta}}\\
&\leq (-c)d||(I-bD_+D_-)e^n||^2_{\ell^2_{\Delta}}(1+C_0\Delta t)\\
&+(-c)dC_0(\Delta t+\Delta t^2)||D(e^nu_{\Delta}^n)||_{\ell^2_{\Delta}}^2+(-c)dC_0(\Delta t+\Delta t^2)||D(\eta_{\Delta}^nf^n)||_{\ell^2_{\Delta}}^2\\
&+(-c)dC_0(\Delta t+\Delta t^2)||D(e^nf^n)||_{\ell^2_{\Delta}}^2+(-c)dC_0(\Delta t+\Delta t^2)||\epsilon_1^n||_{\ell^2_{\Delta}}^2\\
&+(-a)b||(I-dD_+D_-)f^n||^2_{\ell^2_{\Delta}}(1+C_0\Delta t)+(-a)b\,C_0(\Delta t+\Delta t^2)||D(f^nu_{\Delta}^n)||_{\ell^2_{\Delta}}^2\\
&+(-a)b\,C_0(\Delta t+\Delta t^2)||D((f^n)^2)||_{\ell^2_{\Delta}}^2+(-a)b\,C_0(\Delta t+\Delta t^2)||\epsilon_2^n||_{\ell^2_{\Delta}}^2.
\end{split}
\end{equation*}
Since 
$$||D(e^nf^n)||_{\ell^2_{\Delta}}^2\leq 2||e^n||_{\ell^{\infty}}^2||Df^n||_{\ell^2_{\Delta}}^2+2||f^n||_{\ell^{\infty}}^2||De^n||_{\ell^2_{\Delta}}^2,$$
it holds
\begin{equation}\label{eq_1_10_NOv}
||D(e^nf^n)||_{\ell^2_{\Delta}}^2\leq \max\left\{\frac{1}{(-a)bd}, \frac{1}{b(-c)d}\right\}\max\left\{||e^n||_{\ell^{\infty}}^2, ||f^n||_{\ell^{\infty}}^2\right\}\mathcal{E}(e^n,f^n).
\end{equation}

\noindent The same holds for other terms to obtain
\begin{equation}\label{eq_2_10_NOv}
||D(e^nu_{\Delta}^n)||_{\ell^2_{\Delta}}^2\leq \max\left\{\frac{2}{(-c)d},\frac{1}{b(-c)d}\right\}\max\left\{||u_{\Delta}^n||_{\ell^{\infty}}^2, ||Du_{\Delta}^n||_{\ell^{\infty}}^2\right\}\mathcal{E}(e^n,f^n),
\end{equation}
and
\begin{equation}\label{eq_3_10_NOv}
||D(\eta_{\Delta}^nf^n)||_{\ell^2_{\Delta}}^2\leq \max\left\{\frac{2}{(-a)b},\frac{1}{(-a)bd}\right\}\max\left\{||\eta_{\Delta}^n||_{\ell^{\infty}}^2, ||D\eta_{\Delta}^n||_{\ell^{\infty}}^2\right\}\mathcal{E}(e^n,f^n),
\end{equation}
and
\begin{equation}\label{eq_4_10_NOv}
||D(u_{\Delta}^nf^n)||_{\ell^2_{\Delta}}^2\leq \max\left\{\frac{2}{(-a)b},\frac{1}{(-a)bd}\right\}\max\left\{||u_{\Delta}^n||_{\ell^{\infty}}^2, ||Du_{\Delta}^n||_{\ell^{\infty}}^2\right\}\mathcal{E}(e^n,f^n),
\end{equation}
and finally 
\begin{equation}\label{eq_5_10_NOv}
||D((f^n)^2)||_{\ell^2_{\Delta}}^2\leq  ||f^n||_{\ell^{\infty}}^2 \frac{2}{(-a)bd}\mathcal{E}(e^n,f^n).
\end{equation}

Thus, there exists a constant $C_2$  (which can be tracked by \eqref{eq_1_10_NOv}-\eqref{eq_5_10_NOv}), such that 
\begin{equation}\label{eq_fin_RHS_10_NOv}
\begin{split}
&-cd||(I-bD_+D_-)e^{n}-\Delta tD(e^nu_{\Delta}^n)-\Delta tD(\eta_{\Delta}^nf^n)-\Delta tD(e^nf^n)-\Delta t\epsilon_1^n||^2_{\ell^2_{\Delta}}\\
&-ab||(I-dD_+D_-)f^{n}-\Delta tD(f^nu_{\Delta}^n)-\frac{\Delta t}{2}D((f^n)^2)-\Delta t\epsilon_2^n||^2_{\ell^2_{\Delta}}\\
&\leq \mathcal{E}(e^n,f^n)(1+C_0\Delta t)+C_2(\Delta t+\Delta t^2)\mathcal{E}(e^n,f^n)+(-c)dC_0(\Delta t+\Delta t^2)||\epsilon_1^n||_{\ell^2_{\Delta}}^2\\
&+(-a)b\,C_0(\Delta t+\Delta t^2)||\epsilon_2^n||_{\ell^2_{\Delta}}^2.
\end{split}
\end{equation}

Gathering the informations from \eqref{eq_fin_LHS_10_NOv} and \eqref{eq_fin_RHS_10_NOv}, there exists constants $C_3$ and $C_4$ such that 
$$(1-C_1\Delta t)\mathcal{E}(e^{n+1}, f^{n+1})\leq (1+C_3\Delta t)\mathcal{E}(e^{n}, f^{n})+C_4\Delta t||\epsilon^n||_{\ell^2_{\Delta}}^2,$$
with $||\epsilon^n||_{\ell^2_{\Delta}}^2=\max\left\{||\epsilon_1^n||_{\ell^2_{\Delta}}^2, ||\epsilon_2^n||_{\ell^2_{\Delta}}^2\right\}.$
Proposition \ref{_prop_1_avcd} is a straighforward consequence.

\end{proof}

\begin{proof}({\bf Proof of Theorem \ref{Teorema_Numerica1}}.)
Let us arbitrary fix $n\in\overline{0,N-1}$. Suppose the strong induction hypothesis%
\begin{equation}
\left\Vert e^{k}\right\Vert _{\ell^{\infty}}\leq1\text{ and }\left\Vert
f^{k}\right\Vert _{\ell^{\infty}}\leq1,\label{inductie}%
\end{equation}
for all $k\in\overline{0,n}$.\\
This is obviously true for $n=0$, since $e^0_j=f^0_j$, for all $j\in\mathbb{Z}$. Let us prove that $||e^{n+1}||_{\ell^{\infty}}\leq 1$ and $||f^{n+1}||_{\ell^{\infty}}\leq 1$.
Inequality \eqref{ineg_inducti_totala} is thus available for any $k\in\overline{0,n}$ and constants $C_{1}$ and $C_{2}$ may be upper bounded by $C_3$ and $C_4$ independent of $\left\Vert e^{k}\right\Vert _{\ell^{\infty}}$ and $\left\Vert f^{k}\right\Vert _{\ell^{\infty}} .$
One has, for all $k\in\overline{0,n}$
\begin{equation*}
\mathcal{E}\left(  e^{k+1},f^{k+1}\right)  \leq\frac{2\Delta t}{1-\Delta
tC_{4}}\left\Vert \epsilon^{k}\right\Vert _{\ell^{2}_{\Delta}}^{2}+\left(
1+\frac{\Delta t\left(  C_{3}+C_{4}\right)  }{1-\Delta tC_{4}}\right)
\mathcal{E}\left(  e^{k},f^{k}\right)  .
\end{equation*}
Namely, it becomes
\begin{equation*}
\mathcal{E}\left(  e^{k+1},f^{k+1}\right) -\mathcal{E}\left(  e^{k},f^{k}\right) \leq\frac{2\Delta t}{1-\Delta
tC_{4}}\left\Vert \epsilon^{k}\right\Vert _{\ell^{2}_{\Delta}}^{2}+\frac{\Delta t\left(  C_{3}+C_{4}\right)  }{1-\Delta tC_{4}}
\mathcal{E}\left(  e^{k},f^{k}\right)  .
\end{equation*}
Thus, taking the sum of all these inequalities, and noticing that $e^0=f^0=0$, we end up with%
\[
\mathcal{E}\left(  e^{n+1},f^{n+1}\right)  \leq\frac{2\left(  n+1\right)
\Delta t}{1-\Delta tC_{4}}\max_{k\in\overline{0,n}}\left\Vert \epsilon
^{k}\right\Vert _{\ell^{2}_{\Delta}}^{2}+\frac{\Delta t\left(  C_{3}+C_{4}\right)
}{1-\Delta tC_{4}}\sum_{k=1}^{n}\mathcal{E}\left(  e^{k},f^{k}\right)  .
\]
Applying the discrete Gr\"{o}nwall lemma \ref{Granwall_discret_numerica} and
using the fact that the consistency error is first-order accurate in time and
second-order accurate in space, see Appendix \ref{Section_consistency}, we get%
\begin{align*}
\mathcal{E}\left(  e^{n+1},f^{n+1}\right)   &  \leq\frac{2\left(  n+1\right)
\Delta t}{1-\Delta tC_{4}}\exp\left(  \frac{\left(  n+1\right)  \Delta
t\left(  C_{3}+C_{4}\right)  }{1-\Delta tC_{4}}\right)  \max_{k\in
\overline{0,n}}\left\Vert \epsilon^{k}\right\Vert _{\ell^{2}_{\Delta}}^{2}\\
&  \leq\frac{TC\left(  \eta_{0},u_{0}\right)  }{1-\Delta tC_{4}}\exp\left(
\frac{T\left(  C_{3}+C_{4}\right)  }{1-\Delta tC_{4}}\right)  \left\{  \left(
\Delta x\right)  ^{4}+\left(  \Delta t\right)  ^{2}\right\}  ,
\end{align*}
where $C\left(  \eta_{0},u_{0}\right)  $ is some constant depending on the
initial data $\left(  \eta_{0},u_{0}\right)  $. Thus, for $\Delta x,\Delta t$
small enough and using the inequality:%
\[
\left\Vert e^{n+1}\right\Vert _{\ell^{\infty}}\leq C\left\Vert e^{n+1}\right\Vert
_{\ell_{\Delta}^{2}}^{\frac{1}{2}}\left\Vert D_{+}e^{n+1}\right\Vert
_{\ell_{\Delta}^{2}}^{\frac{1}{2}},
\]
with $C$ a constant, we get that for sufficient small $\Delta x$ and $\Delta t$ :
\[
||e^{n+1}||_{\ell^{\infty}}\leq 1\text{\ \ \ and\ \ \ }||f^{n+1}||_{\ell^{\infty}}\leq 1.
\]
We can assure that the inductive hypothesis $\left(  \text{\ref{inductie}%
}\right)  $ holds for all $k\in\overline{0,n+1}$.

Obviously, this allows to
close the estimates and provide the desired bound. This concludes the proof of Theorem \ref{Teorema_Numerica1}.
\end{proof}

\subsection{The proof of Theorem \ref{Teorema_Numerica2}.\label{section bd=0}}
As announced in the introduction, in this section, we aim at providing a proof
for our second main result. As opposed to the previous result, the proof of
Theorem \ref{Teorema_Numerica2} is rather sensitive to the different values of
the $abcd$ parameters.

\noindent We recall that we will treat the case where the parameters verify:
\begin{equation*}
a\leq0,\text{ }c\leq0,\text{ }b\geq0,\text{ }d\geq0,\ \text{and}\ bd=0,%
\end{equation*}
excluding the five cases:%
\begin{equation*}
\left\{
\begin{array}
[c]{l}%
a=b=0,\ d>0,\ c<0,\\
a=b=c=d=0,\\
a=d=0,\ b>0,\ c<0,\\
a=b=d=0,\ c<0,\\
b=d=0,\ c<0,\ a<0.
\end{array}
\right.  %
\end{equation*}
\noindent For all $n\geq0$, we will
consider $\left(  \epsilon_{1}^{n},\epsilon_{2}^{n}\right)  \in\left(\ell_{\Delta
}^{2}(\mathbb{Z})\right)^2$ the consistency error defined as:%
\begin{equation}
\left\{
\begin{array}
[c]{l}%
\frac{1}{\Delta t}\left(  I-bD_{+}D_{-}\right)  (\eta_{\Delta}^{n+1}%
-\eta_{\Delta}^{n})+\left(  I+aD_{+}D_{-}\right)  D\left(  u_{\Delta}%
^{n+1}\right)  +D\left(  \eta_{\Delta}^{n}u_{\Delta}^{n}\right) \\
=\epsilon_{1}^{n}+\frac{1}{2}\left(  1-\operatorname*{sgn}(b)\right)  \tau
_{1}\Delta xD_{+}D_{-}\left(  \eta_{\Delta}^{n}\right)  ,\\
\\
\frac{1}{\Delta t}\left(  I-dD_{+}D_{-}\right)  (u_{\Delta}^{n+1}-u_{\Delta
}^{n})+\left(  I+cD_{+}D_{-}\right)  D\left(  \eta_{\Delta}^{n+1}\right)
+\frac{1}{2}D\left(  \left(  u_{\Delta}^{n}\right)  ^{2}\right) \\
=\epsilon_{2}^{n}+\frac{1}{2}\left(  1-\operatorname*{sgn}(d)\right)  \tau
_{2}\Delta xD_{+}D_{-}\left(  u^{n}_{\Delta}\right)  .
\end{array}
\right.  \label{consistency_error2}%
\end{equation}

In this section, we only detail the derivation of the energy inequality for $\mathcal{E}$, the equivalent of \eqref{ineg_inducti_totala}. This inequality is summarized as follows.\\

\begin{proposition}\label{Prop_enegry_EST_AbcD}
Assume $||e^n||_{\ell^{\infty}}, ||f^n||_{\ell^{\infty}}, |\operatorname*{sgn}(a)|||D_+(f)^n||_{\ell^{\infty}}\leq \Delta x^{\frac{1}{2}-\gamma}$, with $\gamma\in(0,\frac{1}{2})$, for all $n\in\overline{0,N}$. Then the following energy estimate holds true, for $n\in\overline{0,N-1}$
\begin{equation}\label{eq_Prop_enegry_EST_AbcD}
\begin{split}
&\left(1-\max\left\{\operatorname*{sgn}(b), \operatorname*{sgn}(d)\right\}C\Delta t\right)\mathcal{E}\left(e^{n+1}, f^{n+1}\right)\leq (1+C\Delta t)\mathcal{E}\left(e^n, f^n\right)\\
&+\left(\Delta t+\max\left\{|\operatorname*{sgn}(c)|,1-\operatorname*{sgn}(b)\right\}\Delta t^2\right)C||\epsilon^n_1||_{\ell^2_{\Delta}}^2\\
&+\left(\Delta t+\max\left\{|\operatorname*{sgn}(a)|,1-\operatorname*{sgn}(d)\right\}\Delta t^2\right)C||\epsilon^n_2||_{\ell^2_{\Delta}}^2\\
&+\left(\Delta t+\Delta t^2\right)\max\left\{|\operatorname*{sgn}(cd)|, |\operatorname*{sgn}(c)|\left(1-\max\left\{\operatorname*{sgn}(b), \operatorname*{sgn}(d)\right\}\right)\right\}C||D_+\left(\epsilon_1\right)^n||_{\ell^2_{\Delta}}^2\\
&+\left(\Delta t+\Delta t^2\right)\max\left\{|\operatorname*{sgn}(ab)|, |\operatorname*{sgn}(a)|\left(1-\max\left\{\operatorname*{sgn}(b), \operatorname*{sgn}(d)\right\}\right)\right\}C||D_+\left(\epsilon_2\right)^n||_{\ell^2_{\Delta}}^2,
\end{split}
\end{equation}
with $C$ a positive constant depending on $||\eta_{\Delta}^n||_{\ell^{\infty}}, ||D(\eta_{\Delta})^n||_{\ell^{\infty}}, ||u_{\Delta}^n||_{\ell^{\infty}}$, $||D(u_{\Delta})^n||_{\ell^{\infty}}$ and $||D_+D(u_{\Delta})^n||_{\ell^{\infty}}$.\\
\end{proposition}

In order to close the estimates and ensure the convergence proof (as the one made in Subsection \ref{section b,d>0}, for the case $b>0$ and $d>0$) we perform as usual an induction hypothesis on the smallness of $||e^n||_{\ell^{\infty}}, ||f^n||_{\ell^{\infty}}, ||D_+(f)^n||_{\ell^{\infty}}$ according to the cases. It is sufficient to assume by induction
\begin{equation}\label{HYP_AbcD}
||e^n||_{\ell^{\infty}}, ||f^n||_{\ell^{\infty}}, |\mathrm{sgn}(a)|||D_+(f)^n||_{\ell^{\infty}}\leq\Delta x^{\frac{1}{2}-\gamma}, \text{\ \ \ with\ }\gamma\in(0,\frac{1}{2}).
\end{equation}
Hypothesis \eqref{HYP_AbcD} is sufficient to assure the hypothesis of Proposition \ref{Prop_enegry_EST_AbcD}. The energy estimate \eqref{eq_Prop_enegry_EST_AbcD} is thus satisfied and the convergence rate (Theorem \ref{Teorema_Numerica2}) is a consequence of the discrete strong Gr\"onwall inequality, Lemma \ref{Granwall_discret_numerica} and the study of the consistency error \eqref{consistency_error2} detailed in Appendix \ref{Section_consistency} (all the previous guidelines are detailed in Subsection \ref{section b,d>0} for the case $b>0$ and $d>0$).

First, we establish a technical result that interfers
in a crucial manner in establishing the \emph{a priori} estimates.

\subsubsection{Burgers-type estimates.\label{Burgers_type_estimates}}
Let us state the first result of this subsection:

\begin{proposition}
Let $u\in\ell^{\infty}(\mathbb{Z})$ such that $(D_+(u)_j)_{j\in\mathbb{Z}} \in\ell^{\infty}(\mathbb{Z})$ and $\lambda\geq0$. Fix $\alpha>0$ and $\tau$ such that%
\[
\left\Vert u\right\Vert _{\ell^{\infty}}+\alpha<\tau.
\]
Then, there exists a sufficiently small positive number $\delta
_{0}$ such that the following holds true. Consider two positive
reals $\Delta t,\Delta x$ such that
\[
\frac{\text{ }\tau\Delta t}{\Delta x}\leq1\text{,\ \ \ \  }\Delta x\leq
\delta_{0}%
\]

and $a\in\ell^{2}\left(  \mathbb{Z}\right)  $, such that%
\[
\lambda\left\Vert a\right\Vert _{\ell^{\infty}}\leq\Delta x^{\frac{1}{2}-\gamma}, \text{\ with\ }\gamma\in(0,\frac{1}{2}).
\]
Then, there exists a positive constant $C$ depending on the $\ell^{\infty}%
$-norm of $u$ and $D_{+}\left(  u\right)  $ such that
\[
\left\Vert a-\Delta tD\left(  \left(  u+\lambda \frac{a}{2}\right)  a\right)
+\frac{\tau}{2}\Delta x\Delta tD_{+}D_{-}a\right\Vert _{\ell_{\Delta}^{2}}%
\leq\left(  1+C\Delta t\right)  \left\Vert a\right\Vert _{\ell_{\Delta}^{2}}.
\]
\label{Burgers1}
\end{proposition}

\begin{proof}
We define
\[
\mathcal{B}a=a-\Delta tD\left(  a\left(  u+\lambda\frac{a}{2}\right)  \right)
+\frac{\tau}{2}\Delta x\Delta tD_{+}D_{-}a.
\]
We compute the $\ell_{\Delta}^{2}$-norm of $\mathcal{B}a$ :
\begin{equation}\label{norm_of_Ba}
\begin{split}
||\mathcal{B}a||_{\ell_{\Delta}^{2}}^{2}-   ||a||_{\ell_{\Delta}^{2}}%
^{2}&-\Delta t^{2}||D(au)||_{\ell_{\Delta}^{2}}^{2}%
-\Delta t^{2}\lambda^2||D\left(  \frac{a^{2}}%
{2}\right)  ||_{\ell_{\Delta}^{2}}^{2} -\frac{\tau^{2}}{4}\Delta t^{2}\Delta x^{2}||D_{+}D_{-}a||_{\ell
_{\Delta}^{2}}^{2}\\
=&-2\Delta t\left\langle a,D\left(  au+\lambda\frac{a^{2}}{2}\right)
\right\rangle +\tau\Delta x\Delta t\left\langle a,D_{+}D_{-}a\right\rangle \\
+&2\lambda\Delta t^{2}\left\langle D\left(
au\right)  ,D\left(  \frac{a^{2}}{2}\right)  \right\rangle -\tau\Delta x\Delta
t^{2}\left\langle D\left(  au\right)  ,D_{+}D_{-}a\right\rangle\\
-&\tau\lambda\Delta x\Delta
t^{2}\left\langle D\left(  \frac{a^{2}}{2}\right)  ,D_{+}D_{-}a\right\rangle\overset{not.}{=}\sum_{i=1}^3R_i.
\end{split}
\end{equation}
$\bullet$ For $R_1$, Relations \eqref{IPP1}, \eqref{IPP3} and \eqref{integr_by_part_1}
 give
\begin{align*}
&-2\Delta t\left\langle a,D\left(  au+\lambda\frac{a^{2}}{2}\right)
\right\rangle +\tau\Delta x\Delta t\left\langle a,D_{+}D_{-}a\right\rangle\\
&=-\Delta t\left\langle D_{+}u,aS_{+}a\right\rangle
+\frac{\Delta x^{2}\Delta t}{6}\lambda\left\langle D_{+}a,\left(  D_{+}a\right)
^{2}\right\rangle -\tau\Delta t\Delta x||D_{+}a||_{\ell_{\Delta}^{2}}^{2}\\
&\leq \Delta t||D_+u||_{\ell^{\infty}}||a||_{\ell^2_{\Delta}}^2+\frac{\Delta x^{2}\Delta t}{6}\lambda\left\langle D_{+}a,\left(  D_{+}a\right)
^{2}\right\rangle -\tau\Delta t\Delta x||D_{+}a||_{\ell_{\Delta}^{2}}^{2}.
\end{align*}
$\bullet$ For the $R_2$-term, one has, thanks to Relations
\eqref{EQ_4} and \eqref{IPP2}, \newline%
\begin{align*}
&2\lambda\Delta t^{2}\left\langle D\left(
au\right)  ,D\left(  \frac{a^{2}}{2}\right)  \right\rangle -\tau\Delta x\Delta
t^{2}\left\langle D\left(  au\right)  ,D_{+}D_{-}a\right\rangle\\
& \leq2\lambda\Delta t^{2}||a||_{\ell^{\infty}%
}||u||_{\ell^{\infty}}||Da||_{\ell_{\Delta}^{2}}^{2}-\frac{8\Delta t^{2}\Delta
x^{2}}{3}\lambda\left\langle Du,\left(  Da\right)  ^{3}\right\rangle -\frac{2\Delta
t^{2}}{3}\lambda\left\langle DDu,a^{3}\right\rangle \\
& +\frac{\tau\Delta t^{2}}{\Delta x}||D_{+}u||_{\ell^{\infty}}||a||_{\ell
_{\Delta}^{2}}^{2}+\frac{\tau\Delta t^{2}}{\Delta x}||Du||_{\ell^{\infty}%
}||a||_{\ell_{\Delta}^{2}}^{2}\\
&\leq 2\lambda\Delta t^{2}||a||_{\ell^{\infty}%
}||u||_{\ell^{\infty}}||Da||_{\ell_{\Delta}^{2}}^{2} +\frac{8\Delta x^2\Delta t^2}{3}\lambda||Du||_{\ell^{\infty}}||Da||_{\ell^{\infty}}||Da||_{\ell^2_{\Delta}}^2\\
&+\frac{2\Delta t^2}{3\Delta x}\lambda||a||_{\ell^2_{\Delta}}^2||Du||_{\ell^{\infty}}||a||_{\ell^{\infty}} +\frac{\tau\Delta t^{2}}{\Delta x}||D_{+}u||_{\ell^{\infty}}||a||_{\ell
_{\Delta}^{2}}^{2}+\frac{\tau\Delta t^{2}}{\Delta x}||Du||_{\ell^{\infty}%
}||a||_{\ell_{\Delta}^{2}}^{2}.
\end{align*}
$\bullet$ Eventually, for $R_3$, one has, thanks to \eqref{IPP2BiS}
\begin{small}
\[
-\tau\lambda\Delta x\Delta
t^{2}\left\langle D\left(  \frac{a^{2}}{2}\right)  ,D_{+}D_{-}a\right\rangle=-\lambda\frac{\Delta t^{2}\Delta
x\tau}{6}\left\langle D_{+}a,\left(  D_{+}a\right)  ^{2}\right\rangle
+\frac{2\Delta t^{2}\Delta x\tau}{3}\lambda\left\langle Da,\left(  Da\right)
^{2}\right\rangle.
\]
\end{small}
For the left hand side of the $\ell^2_{\Delta}$-norm of $\mathcal{B}a$, Equation \eqref{norm_of_Ba}, we know, thanks to \eqref{IPP5}
\begin{small}
\[
\Delta t^{2}||D\left(  au\right)  ||_{\ell_{\Delta}^{2}}^{2}\leq\Delta
t^{2}\left\{  ||u||_{\ell^{\infty}}^{2}+\Delta t||D_{+}u||_{\ell^{\infty}}%
^{2}\right\}  ||Da||_{\ell_{\Delta}^{2}}^{2}+\Delta t\left\{  ||u||_{\ell
^{\infty}}^{2}+\frac{3\Delta t}{4}||D_{+}u||_{\ell^{\infty}}^{2}\right\}
||a||_{\ell_{\Delta}^{2}}^{2}.
\]
\end{small}
Thanks to \eqref{IPP4}, one has
\[
\frac{\tau^{2}}{4}\Delta t^{2}\Delta x^{2}||D_{+}D_{-}a||_{\ell_{\Delta}^{2}%
}^{2}=\tau^{2}\Delta t^{2}||D_{+}a||_{\ell_{\Delta}^{2}}^{2}-\tau^{2}\Delta
t^{2}||Da||_{\ell_{\Delta}^{2}}^{2},
\]
and, thanks to \eqref{Norm_D_v_carre}, 
\[
\Delta t^2\lambda^2||D\left(\frac{a^2}{2}\right)||_{\ell^2_{\Delta}}^2\leq \Delta t^2\lambda^2||Da||_{\ell^2_{\Delta}}^2||a||_{\ell^{\infty}}^2.
\]
Finally, we gather all these results
\[%
\begin{split}
&  ||\mathcal{B}a||_{\ell_{\Delta}^{2}}^{2}\leq||a||_{\ell_{\Delta}^{2}}%
^{2}\left\{  1+\Delta t||D_{+}u||_{\ell^{\infty}}+\Delta t||u||_{\ell^{\infty
}}^{2}+\frac{3\Delta t^{2}}{4}||D_{+}u||_{\ell^{\infty}}^{2}+\frac{2\Delta
t^{2}}{3\Delta x}\lambda||Du||_{\ell^{\infty}}||a||_{\ell^{\infty}}\right.  \\
&  \left.  +\frac{\tau\Delta t^{2}}{\Delta x}||D_{+}u||_{\ell^{\infty}}%
+\frac{\tau\Delta t^{2}}{\Delta x}||Du||_{\ell^{\infty}}\right\}
+\frac{\Delta x^{2}\Delta t}{6}\lambda\left\langle D_{+}a,\left(  D_{+}a\right)
^{2}\right\rangle -\tau\Delta t\Delta x||D_{+}a||_{\ell_{\Delta}^{2}}^{2}\\
&  -\frac{\Delta t^{2}\Delta x\tau}{6}\lambda\left\langle D_{+}a,\left(
D_{+}a\right)  ^{2}\right\rangle +\tau^{2}\Delta t^{2}||D_{+}a||_{\ell
_{\Delta}^{2}}^{2}+\Delta t^{2}||Da||_{\ell_{\Delta}^{2}}^{2}\left\{
||u||_{\ell^{\infty}}^{2}+\Delta t||D_{+}u||_{\ell^{\infty}}^{2}\right.  \\
&  \left.  +2\lambda||a||_{\ell^{\infty}}||u||_{\ell^{\infty}}+\frac{8\Delta x^{2}%
}{3}\lambda||Du||_{\ell^{\infty}}||Da||_{\ell^{\infty}}+\lambda^2||a||_{\ell^{\infty}}%
^{2}+\frac{2\tau\Delta x}{3}\lambda||Da||_{\ell^{\infty}}-\tau^{2}\right\}.
\end{split}
\]
Thus
\[%
\begin{split}
&  ||\mathcal{B}a||_{\ell_{\Delta}^{2}}^{2}\leq||a||_{\ell_{\Delta}^{2}}%
^{2}\left\{  1+\Delta t||D_{+}u||_{\ell^{\infty}}+\Delta t||u||_{\ell^{\infty
}}^{2}+\frac{3\Delta t^{2}}{4}||D_{+}u||_{\ell^{\infty}}^{2}+\frac{2\Delta
t^{2}}{3\Delta x}\lambda||Du||_{\ell^{\infty}}||a||_{\ell^{\infty}}\right.  \\
&  \left.  +\frac{\tau\Delta t^{2}}{\Delta x}||D_{+}u||_{\ell^{\infty}}%
+\frac{\tau\Delta t^{2}}{\Delta x}||Du||_{\ell^{\infty}}\right\}\\
&
+\left\langle \frac{\Delta x^{2}\Delta t}{6}\lambda D_{+}a-\tau\Delta t\Delta
x\mathbf{1}-\frac{\Delta t^{2}\Delta x\tau}{6}\lambda D_{+}a+\tau^{2}\Delta
t^{2}\mathbf{1},\left(  D_{+}a\right)  ^{2}\right\rangle \\
&  +\Delta t^{2}||Da||_{\ell_{\Delta}^{2}}^{2}\left\{  ||u||_{\ell^{\infty}%
}^{2}+\Delta t||D_{+}u||_{\ell^{\infty}}^{2}+2\lambda||a||_{\ell^{\infty}}%
||u||_{\ell^{\infty}}+\frac{8\Delta x}{3}\lambda||Du||_{\ell^{\infty}}||a||_{\ell
^{\infty}}%
\right.\\
&\left.+\lambda^2||a||_{\ell^{\infty}}^{2}+\frac{2\tau}{3}\lambda||a||_{\ell^{\infty}}-\tau^{2}\right\}  ,
\end{split}
\]
where%
\[
\mathbf{1=}\left(  ...1,1,1...\right)  .
\]
However,
\[
\frac{\Delta x^{2}\Delta t}{6}\lambda D_{+}a-\tau\Delta t\Delta x\mathbf{1}%
-\frac{\Delta t^{2}\Delta x\tau}{6}\lambda D_{+}a+\tau^{2}\Delta t^{2}\mathbf{1}%
=\left(  \frac{\Delta x}{6}\lambda D_{+}a-\tau\mathbf{1}\right)  \left(  \Delta
x-\tau\Delta t\right) \Delta t .
\]
By hypothesis, $\tau\Delta t\leq\Delta x$ and $ ||u||_{\ell^{\infty}}+\alpha<\tau$,
thus for $\delta_0$ such that 
\begin{equation}\label{condition1_sur_delta_0}
\delta_0\leq \left(3||u||_{\ell^{\infty}}\right)^{\frac{2}{1-2\gamma}},
\end{equation}
 one has $$\frac{\Delta x}{6}\lambda D_{+}a\leq
\frac{\lambda ||a||_{\ell^{\infty}}}{3}\leq\frac{\Delta x^{\frac{1}{2}-\gamma}}{3}%
\leq \frac{\delta_0^{\frac{1}{2}-\gamma}}{3}\leq||u||_{\ell^{\infty}}<\tau.$$ This implies
\[
\left\langle \frac{\Delta x^{2}\Delta t}{6}\lambda D_{+}a-\tau\Delta t\Delta
x\mathbf{1}-\frac{\Delta t^{2}\Delta x\tau}{6}\lambda D_{+}a+\tau^{2}\Delta
t^{2}\mathbf{1},\left(  D_{+}a\right)  ^{2}\right\rangle \leq0.
\]
In the same way, for $\Delta t$ and $\Delta x$ small enough and $||a||_{\ell^{\infty}}$ small enough, for $\delta_0$ satisfying
\begin{equation}\label{condition2_sur_delta_0}
||D_+(u)||_{\ell^{\infty}}^2\frac{\delta_0}{\tau}+2\delta_0^{\frac{1}{2}-\gamma}||u||_{\ell^{\infty}}+\frac{8}{3}||D(u)||_{\ell^{\infty}}\delta_0^{\frac{3}{2}-\gamma}+\delta_0^{1-2\gamma}+\frac{2\tau}{3}\delta_0^{\frac{1}{2}-\gamma}\leq \alpha^2,
\end{equation}

\noindent the condition
$ ||u||_{\ell^{\infty}}+\alpha\leq \tau$ implies%
\begin{multline*}
||u||_{\ell^{\infty}}^{2}+\Delta t||D_{+}u||_{\ell^{\infty}}^{2}%
+2\lambda||a||_{\ell^{\infty}}||u||_{\ell^{\infty}}+\frac{8\Delta x}{3}\lambda%
||Du||_{\ell^{\infty}}||a||_{\ell^{\infty}}\\+\lambda^2||a||_{\ell^{\infty}}^{2}%
+\frac{2\tau}{3}\lambda||a||_{\ell^{\infty}}-\tau^{2}\leq0.
\end{multline*}
Proposition \ref{Burgers1} results from the fact that there exists $C_1$ such that $\sqrt{1+C\Delta t}\leq 1+C_1\Delta t$ with 
\begin{multline*}
C=||D_{+}u||_{\ell^{\infty}}+||u||_{\ell^{\infty}}^{2}+\frac{3\Delta t}%
{4}||D_{+}u||_{\ell^{\infty}}^{2}+\frac{2\Delta t}{3\Delta x}\lambda||Du||_{\ell
^{\infty}}||a||_{\ell^{\infty}}\\+\frac{\tau\Delta t}{\Delta x}||D_{+}%
u||_{\ell^{\infty}}+\frac{\tau\Delta t}{\Delta x}||Du||_{\ell^{\infty}}.
\end{multline*}
The upper bound $\delta_0$ must be chosen such that Conditions \eqref{condition1_sur_delta_0} and \eqref{condition2_sur_delta_0} be satisfied.
\end{proof}

The next result is an immediate consequence of the preceding one.
\begin{proposition}
\label{Burgers2}Consider $u\in\ell^{\infty}\left(  \mathbb{Z}\right)  $ such that $(D_+(u)_j)_{j\in\mathbb{Z}}\in\ell^{\infty}(\mathbb{Z})$ and 
$\lambda\geq0$. Fix $\tau$ such
that%
\[
\left\Vert u\right\Vert _{\ell^{\infty}}<\tau.
\]
Then, there exist two sufficiently small positive numbers $\delta
_{0},\delta_{1}$ such that the following holds true. Consider two positive
reals $\Delta t,\Delta x$ such that
\[
\frac{\text{ }\tau\Delta t}{\Delta x}\leq1\text{,\ \ \  }\Delta x\leq
\delta_{0}%
\]
and $a\in\ell^{2}\left(  \mathbb{Z}\right)  $, $b\in\ell^{\infty}\left(
\mathbb{Z}\right)  $ such that%
\[
\lambda\left\Vert a\right\Vert _{\ell^{\infty}}\leq \Delta x^{\frac{1}{2}-\gamma}\text{,\ with\ }\gamma\in(0,\frac{1}{2}),\ \ \ \left\Vert b\right\Vert
_{\ell^{\infty}}\leq\delta_{1}\text{\ and\ }\left\Vert D_{+}b\right\Vert
_{\ell^{\infty}}\leq1.
\]
Then, there exists a positive constant $C$ depending on the $\ell^{\infty}%
$-norm of $u$ and $D_{+}\left(  u\right)  $ such that:
\begin{equation}
\left\Vert a-\Delta tD\left(  a\left(  u+b+\lambda \frac{a}{2}\right)  \right)
+\frac{\tau}{2}\Delta x\Delta tD_{+}D_{-}a\right\Vert _{\ell_{\Delta}^{2}}%
\leq\left(  1+C\Delta t\right)  \left\Vert a\right\Vert _{\ell_{\Delta}^{2}}.
\label{B21}%
\end{equation}
\end{proposition}

\begin{proof}
Let us consider $\tau>\left\Vert u\right\Vert _{\ell^{\infty}}$. Let us
suppose that $\delta_{1}$ is chosen small enough such that%
\[
\left\Vert u+b\right\Vert _{\ell^{\infty}}\leq\delta_{1}+\left\Vert
u\right\Vert _{\ell^{\infty}}<\tau.
\]
Then, taking a smaller $\Delta t$ and $\Delta x$ if neccesary, we may apply
Proposition \ref{Burgers1} with $u+b$ instead of $u$ in order to establish the estimate $\left(
\text{\ref{B21}}\right)  $.
\end{proof}

\begin{proposition}
\label{Burgers3}Consider $u\in\ell^{\infty}\left(  \mathbb{Z}\right)  $ such that $(D_+(u)_j)_{j\in\mathbb{Z}} \in\ell^{\infty}(\mathbb{Z})$ and 
$\lambda\geq0$. Fix $\tau$ such
that%
\[
\left\Vert u\right\Vert _{\ell^{\infty}}<\tau.
\]
Then, there exist two sufficiently small positive numbers $\delta
_{0},\delta_{1}$ such that the following holds true. Consider two positive
reals $\Delta t,\Delta x$ such that
\[
\frac{\text{ }\tau\Delta t}{\Delta x}\leq1\text{,\ \ \  }\Delta x\leq
\delta_{0}%
\]
and $a\in\ell^{2}\left(  \mathbb{Z}\right)  $, $b\in\ell^{\infty}\left(
\mathbb{Z}\right)  $ such that%
\begin{align}
\lambda\left\Vert a\right\Vert _{\ell^{\infty}},\lambda\left\Vert
D_{+}a\right\Vert _{\ell^{\infty}}\leq \Delta x^{\frac{1}{2}-\gamma}\text{,\ with\ }\gamma\in(0,\frac{1}{2}),\label{H_avant}\\
\left\Vert b\right\Vert _{\ell^{\infty}%
}<\delta_{1}\text{ and }\left\Vert D_{+}b\right\Vert _{\ell^{\infty}%
},\left\Vert D_{+}D_{-}\left(  b\right)  \right\Vert _{\ell^{\infty}}\left\Vert D_{+}D_{-}\left(  u\right)  \right\Vert _{\ell^{\infty}}\leq1.
\label{H}%
\end{align}
Then, there exist positive constants $C_{1}$,$C_{2}$ depending on the
$\ell^{\infty}$-norm of $u,$ $D_{+}\left(  u\right)  $ and $DD_{+}\left(
u\right)  $ such that:
\begin{equation*}
\begin{split}
\left\Vert D_{+}a-\Delta tD_{+}D\left(  a\left(  u+b+\lambda \frac{a}{2}\right)
\right)  +\frac{\tau}{2}\Delta x\Delta tD_{+}D_{+}D_{-}a\right\Vert
_{\ell_{\Delta}^{2}}\\\leq C_{1}\Delta t\left\Vert a\right\Vert _{\ell_{\Delta
}^{2}}+\left(  1+C_{2}\Delta t\right)  \left\Vert D_{+}a\right\Vert
_{\ell_{\Delta}^{2}}.
\end{split}
\end{equation*}

\end{proposition}

\begin{proof}
Let us observe that
\begin{align*}
&  D_{+}a-\Delta tD_{+}D\left(  a\left(  u+b+\lambda \frac{a}{2}\right)  \right)
+\frac{\tau}{2}\Delta x\Delta tD_{+}D_{+}D_{-}\left(  a\right) \\
&  =D_{+}a-\Delta tD\left(  D_{+}\left(  a\right)  \left(  u+b+\lambda
\frac{a}{2}\right)  \right)  -\Delta tD\left(  S_{+}\left(  a\right)  D_{+}\left(
u+b+\lambda \frac{a}{2}\right)  \right)   \\
&+\frac{\tau}{2}\Delta x\Delta tD_{+}D_{-}\left(  D_{+}\left(  a\right)
\right) \\
&  =D_{+}a-\Delta tD\left(  D_{+}\left(  a\right)  \left(  u+b+\frac{\lambda}{2}
S_{+}\left(  a\right)  +\lambda \frac{a}{2}\right)  \right)  +\frac{\tau}{2}\Delta
x\Delta tD_{+}D_{-}\left(  D_{+}\left(  a\right)  \right)  \\
&-\Delta tD\left(  S_{+}\left(  a\right)  D_{+}\left(  u\right)  \right)-\Delta tD\left(  S_{+}\left(  a\right)  D_{+}\left(  b\right)  \right).
\end{align*}
Owing to the hypothesis  $\left(  \text{\ref{H_avant}}\right)  $-$\left(  \text{\ref{H}}\right)  $ and Proposition
\ref{Burgers2} \ with $b+\frac{\lambda}{2} S_{+}a+\lambda \frac{a}{2}$ instead of $b$ and $0$ instead of
$\lambda$, we may choose $\delta_{0}$ small
enough which ensures the existence of a positive constant $C_{2}$ such that:%
\begin{equation}
\begin{split}
\left\Vert D_{+}a-\Delta tD\left(  D_{+}\left(  a\right)  \left(  u+b+\frac{\lambda}{2}
S_{+}\left(  a\right)  +\frac{\lambda}{2} a\right)  \right)  +\frac{\tau}{2}\Delta
x\Delta tD_{+}D_{-}\left(  D_{+}\left(  a\right)  \right)  \right\Vert
_{\ell_{\Delta}^{2}}\\
\leq\left(  1+C_{2}\Delta t\right)  \left\Vert
D_{+}\left(  a\right)  \right\Vert _{\ell_{\Delta}^{2}}. \label{B31}%
\end{split}
\end{equation}
Moreover, using the derivation formula \eqref{product_rule2}, it transpires
that%
\begin{align}
&  \left\Vert -\Delta tD\left(  S_{+}\left(  a\right)  D_{+}\left(  u\right)
\right)  -\Delta tD\left(  S_{+}\left(  a\right)  D_{+}\left(  b\right)
\right)  \right\Vert _{\ell_{\Delta}^{2}}\nonumber\\
&  \leq\Delta t\left(  \left\Vert DD_{+}\left(  u\right)  \right\Vert
_{\ell^{\infty}}+\left\Vert DD_{+}\left(  b\right)  \right\Vert _{\ell
^{\infty}}\right)  \left\Vert a\right\Vert _{\ell_{\Delta}^{2}}+\Delta
t\left(  \left\Vert D_{+}\left(  u\right)  \right\Vert _{\ell^{\infty}%
}+\left\Vert D_{+}\left(  b\right)  \right\Vert _{\ell^{\infty}}\right)
\left\Vert D_{+}\left(  a\right)  \right\Vert _{\ell_{\Delta}^{2}}.
\label{B32}%
\end{align}
The conclusion follows from estimates $\left(  \text{\ref{B31}}\right)  $ and
$\left(  \text{\ref{B32}}\right)  $.
\end{proof}

\subsubsection{The case $b=d=0$.} We distinguish many different settings for $a,b,c$ and $d$.\\
\textbf{The case $a<0$, $b=c=d=0$.}\ \ \ \ 
The convergence error satisfies:%

\begin{equation}\label{sys_a<0_b=c=d=0}
\left\{
\begin{array}
[c]{l}%
e^{n+1}+\Delta tDf^{n+1}+a\Delta tD_{+}D_{-}Df^{n+1}\\
=e^{n}-\Delta tD\left(  e^{n}u_{\Delta}^{n}\right)  -\Delta tD\left(
e^{n}f^{n}\right)  -\Delta tD\left(  \eta_{\Delta}^{n}f^{n}\right)
+\frac{\tau_{1}}{2}\Delta t\Delta xD_{+}D_{-}e^{n}-\Delta t\epsilon
_{1}^{n},\\
\\
f^{n+1}+\Delta tDe^{n+1}=f^{n}-\Delta tD\left(  f^{n}\left(  u_{\Delta}%
^{n}+\frac{1}{2}f^{n}\right)  \right)  +\frac{\tau_{2}}{2}\Delta t\Delta
xD_{+}D_{-}f^{n}-\Delta t\epsilon_{2}^{n},
\end{array}
\right.
\end{equation}
with
\[
\left\{
\begin{array}
[c]{l}%
\max\left\{  \tau_{1},\tau_{2}\right\}  \Delta t<\Delta x,\\
\left\Vert u_{\Delta}^{n}\right\Vert _{\ell^{\infty}}<\min\left\{  \tau
_{1}\text{ },\tau_{2}\right\}  ,
\end{array}
\right. \ \ \ \ n\in\overline{0,N}.
\]
We consider the energy functional%
\[
\mathcal{E}\left(  e,f\right)  \overset{def.}{=}\left\Vert e\right\Vert
_{\ell_{\Delta}^{2}}^{2}+\left\Vert f\right\Vert _{\ell_{\Delta}^{2}}%
^{2}+\left(  -a\right)  \left\Vert D_{+}f\right\Vert _{\ell_{\Delta}^{2}}^{2}.
\]

\begin{proof}{\bf (Proof of Proposition \ref{Prop_enegry_EST_AbcD} in the case $a<0$, $b=c=d=0$).}  In order to recover a $H^{1}%
$-type control for $f^{n}$ (necessary for the control of $\mathcal{E}(e,f)$), let us apply $\sqrt{-a}D_{+}$ to the second
equation of \eqref{sys_a<0_b=c=d=0}. We get that:
\begin{align}
 \sqrt{-a}D_{+}f^{n+1}+\Delta t\sqrt{-a}D_{+}De^{n+1} & =\sqrt{-a}D_{+}f^{n}-\sqrt{-a}\Delta tDD_{+}\left(  f^{n}\left(  u_{\Delta
}^{n}+\frac{1}{2}f^{n}\right)  \right) \nonumber \\
&  +\sqrt{-a}\frac{\tau_{2}}{2}\Delta t\Delta xD_{+}D_{+}D_{-}f^{n}%
-\sqrt{-a}\Delta tD_{+}\epsilon_{2}^{n}.\label{eq_en_plus_a<0_b=c=d=0}
\end{align}
\noindent We consider now the three equations system comprised of system \eqref{sys_a<0_b=c=d=0} with added equation \eqref{eq_en_plus_a<0_b=c=d=0}. As previously, we square the equations and add them together. Young's inequality enables us to obtain, with $C_0$ a constant 
\begin{small}
\begin{align}
&  \left\Vert e^{n+1}+\Delta t\left(  I+aD_{+}D_{-}\right)  Df^{n+1}%
\right\Vert _{\ell_{\Delta}^{2}}^{2}+\left\Vert f^{n+1}+\Delta tDe^{n+1}%
\right\Vert _{\ell_{\Delta}^{2}}^{2} \nonumber\\
&+\left\Vert \sqrt{-a}D_{+}f^{n+1}+\Delta t\sqrt{-a}D_{+}D(e^{n+1})\right\Vert
_{\ell_{\Delta}^{2}}^{2}\nonumber\\
&  \leq\left(  \Delta t+  \Delta t  ^{2}\right) C_0 \left\Vert
\epsilon_{1}^{n}\right\Vert _{\ell_{\Delta}^{2}}^{2}+\left(  \Delta t+
\Delta t  ^{2}\right) C_0 \left\Vert \epsilon_{2}^{n}\right\Vert
_{\ell_{\Delta}^{2}}^{2}+\left(  -a\right)  \left(  \Delta t+\left(  \Delta
t\right)  ^{2}\right) C_0 \left\Vert D_{+}(\epsilon_{2})^{n}\right\Vert
_{\ell_{\Delta}^{2}}^{2}\nonumber\\
& + \left(  1+C_0\Delta t\right) \left\Vert e^{n}-\Delta tD\left(
e^nu_{\Delta}^n\right)  -\Delta tD\left(  e^nf^n\right)  +\frac{\tau_{1}}{2}\Delta t\Delta xD_{+}D_{-}e^{n}\right\Vert _{\ell_{\Delta}^{2}}^{2}\nonumber\\
&+\left(  \Delta t+\left(  \Delta t\right)  ^{2}\right)  C_0\left\Vert D\left(\eta_{\Delta}^nf^n\right) \right\Vert _{\ell_{\Delta}^{2}}^{2}\nonumber\\
&+\left(  1+C_0\Delta t\right)  \left\Vert f^{n}-\Delta tD\left(  f^n\left(u_{\Delta}^n+\frac{1}{2}f^n\right)  \right)  +\frac{\tau_{2}}{2}\Delta t\Delta xD_{+}D_{-}f^{n}\right\Vert _{\ell_{\Delta}^{2}}^{2}\nonumber\\
&+\left(  1+C_0\Delta t\right)  (-a)\left\Vert D_{+}f^{n}-\Delta tDD_{+}\left(f^n\left(  u_{\Delta}^n+\frac{1}{2}f^n\right)  \right) +\frac{\tau_{2}}{2}\Delta t\Delta xD_{+}D_{+}D_{-}f^{n}\right\Vert _{\ell_{\Delta
}^{2}}^{2}.\label{a>0 1}
\end{align}
\end{small}
Let us consider the right hand side of \eqref{a>0 1}. Owing to Proposition \eqref{Burgers2} and Proposition
\eqref{Burgers1} we have that, thanks hypotheses of Proposition \ref{Prop_enegry_EST_AbcD} on $||e^n||_{\ell^{\infty}}, ||f^n||_{\ell^{\infty}},||D_+f^n||_{\ell^{\infty}}\leq\Delta x^{\frac{1}{2}-\gamma}$%
\begin{equation}
\left\Vert e^{n}-\Delta tD\left(  e^nu_{\Delta}^n\right)  -\Delta
tD\left(  e^nf^n\right)  +\frac{\tau_{1}}{2}\Delta t\Delta xD_{+}%
D_{-}e^{n}\right\Vert _{\ell_{\Delta}^{2}}^{2}\leq\left(  1+C_{1}\Delta
t\right)  \left\Vert e^{n}\right\Vert _{\ell_{\Delta}^{2}}^{2},\label{a>0 2}%
\end{equation}
respectively %
\begin{equation}
\left\Vert f^{n}-\Delta tD\left(  f^n\left(  u_{\Delta}^n+\frac{1}{2}%
f^n\right)  \right)  +\frac{\tau_{2}}{2}\Delta t\Delta xD_{+}D_{-}%
f^{n}\right\Vert _{\ell_{\Delta}^{2}}^{2}\leq\left(  1+C_{2}\Delta t\right)
\left\Vert f^{n}\right\Vert _{\ell_{\Delta}^{2}}^{2}.\label{a>0 3}%
\end{equation}
Proposition \ref{Burgers3} ensures the existence of
constants $C_{3},C_{4}$ such that:%
\begin{multline}
  \left(  -a\right)  \left\Vert D_{+}f^{n}-\Delta tD_{+}D\left(  f^n\left(
u_{\Delta}^n+\frac{f^n}{2}\right)  \right)  +\frac{\tau_{2}}{2}\Delta
t\Delta xD_{+}D_{+}D_{-}f^{n}\right\Vert _{\ell_{\Delta}^{2}}^{2}%
\label{a>0 4}\\
 \leq\left(  1+C_{3}\Delta t\right)  (-a)\left\Vert D_{+}f^n\right\Vert
_{\ell_{\Delta}^{2}}^{2}+C_{4}(-a)\Delta t\left\Vert f^{n}\right\Vert
_{\ell_{\Delta}^{2}}^{2}.
\end{multline}
Finally, we have, thanks Relation \eqref{product_rule2},%
\begin{align}
\left\Vert D\left(  \eta_{\Delta}^nf^n\right)  \right\Vert _{\ell_{\Delta
}^{2}}^{2} &  \leq2\left\Vert D  \eta^n_{\Delta}  \right\Vert
_{\ell^{\infty}}^{2}\left\Vert f^{n}\right\Vert _{\ell_{\Delta}^{2}}%
^{2}+2\left\Vert \eta_{\Delta}^{n}\right\Vert _{\ell^{\infty}}^{2}\left\Vert
D  f^n  \right\Vert _{\ell_{\Delta}^{2}}^{2}\nonumber\\
&  \leq2\max\left\{  \left\Vert D \eta_{\Delta}^n  \right\Vert
_{\ell^{\infty}}^{2},\frac{\left\Vert \eta_{\Delta}^{n}\right\Vert
_{\ell^{\infty}}^{2}}{-a}\right\}  \mathcal{E}\left(  e^{n},f^{n}\right)
.\label{a>0 5}%
\end{align}

\noindent Next, we compute the left hand side of \eqref{a>0 1}. First of all, we observe that
\begin{multline*}
  \left\Vert \sqrt{-a}D_{+}f^{n+1}+\Delta t\sqrt{-a}D_{+}De^{n+1}\right\Vert
_{\ell_{\Delta}^{2}}^{2}\\
  =-a\left\Vert D_{+}f^{n+1}\right\Vert _{\ell_{\Delta}^{2}}^{2}-a\left(
\Delta t\right)  ^{2}\left\Vert D_{+}De^{n+1}\right\Vert _{\ell_{\Delta}^{2}%
}^{2}+2a\Delta t\left\langle D_{+}D_{-}f^{n+1},De^{n+1}\right\rangle .
\end{multline*}
Moreover, one has
\begin{small}
\begin{align*}
&  \left\Vert e^{n+1}+\Delta t\left(  I+aD_{+}D_{-}\right)  Df^{n+1}%
\right\Vert _{\ell_{\Delta}^{2}}^{2}  +\left\Vert f^{n+1}+\Delta tDe^{n+1}\right\Vert _{\ell_{\Delta}^{2}}%
^{2}=||e^{n+1}||_{\ell^2_{\Delta}}^2\nonumber\\
&  +2a\Delta t\left\langle e^{n+1}, D_+D_-Df^{n+1}\right\rangle+\Delta t^2||(I+aD_+D_-)Df^{n+1}||_{\ell^2_{\Delta}}^2+||f^{n+1}||_{\ell^2_{\Delta}}^2+\Delta t^2||De^{n+1}||_{\ell^2_{\Delta}}^2.
\end{align*}
\end{small}

\noindent Eventually, we get, for the left hand side,%
\begin{align}
&  \left\Vert e^{n+1}+\Delta t\left(  I+aD_{+}D_{-}\right)  Df^{n+1}%
\right\Vert _{\ell_{\Delta}^{2}}^{2}+\left\Vert f^{n+1}+\Delta tDe^{n+1}%
\right\Vert _{\ell_{\Delta}^{2}}^{2}  \\
&+\left\Vert \sqrt{-a}D_{+}f^{n+1}+\Delta t\sqrt{-a}D_{+}De^{n+1}\right\Vert
_{\ell_{\Delta}^{2}}^{2}\nonumber\\
&  =\mathcal{E}\left(  e^{n+1},f^{n+1}\right)  + \Delta t^{2}\left\Vert \left(  I+aD_{+}D_{-}\right)  Df^{n+1}\right\Vert
_{\ell_{\Delta}^{2}}^{2}  + \Delta t  ^{2}\left\Vert De^{n+1}\right\Vert _{\ell
_{\Delta}^{2}}^{2}\\
&-a \Delta t ^{2}\left\Vert D_{+}%
De^{n+1}\right\Vert _{\ell_{\Delta}^{2}}^{2}\nonumber\\
&\geq\mathcal{E}\left(  e^{n+1},f^{n+1}\right) .\label{a<0_CaSE_1}
\end{align}
Gathering relations $\left(  \text{\ref{a>0 1}}\right)  $, $\left(
\text{\ref{a>0 2}}\right)  $, $\left(  \text{\ref{a>0 3}}\right)  $, $\left(
\text{\ref{a>0 4}}\right)  $, $\left(  \text{\ref{a>0 5}}\right)  $ and \eqref{a<0_CaSE_1} yields%
\begin{small}
\begin{align*}
\mathcal{E}\left(  e^{n+1},f^{n+1}\right)   &  \leq\left(  \Delta t+\left(
\Delta t\right)  ^{2}\right)  C_0\left(  \left\Vert \epsilon_{1}^{n}\right\Vert
_{\ell_{\Delta}^{2}}^{2}+\left\Vert \epsilon_{2}^{n}\right\Vert _{\ell
_{\Delta}^{2}}^{2}+\left(  -a\right)  \left\Vert D_{+}\epsilon_{2}%
^{n}\right\Vert _{\ell_{\Delta}^{2}}^{2}\right)   +\left(  1+C_1\Delta t\right)  \mathcal{E}\left(  e^{n},f^{n}\right).
\end{align*}
\end{small}
\end{proof}
\subsubsection{The case $b=0,d>0$.} Three configurations are studied in this subsubsection.\\
\textbf{The case $a=b=c=0,d>0$\label{classicAL_BOUSSinesq}.}\ \ \ \ 
In this case, without any difficulties, we are able to prove a more general
result. Indeed, we will show that the following general $\theta$-scheme:%
\begin{equation}
\left\{
\begin{array}
[c]{l}%
\frac{1}{\Delta t}(\eta^{n+1}-\eta^{n})+D\left(  (1-\theta)u^{n}+\theta
u^{n+1}\right)  +D\left(  \eta^{n}u^{n}\right)  =\frac{\tau_{1}\Delta x}%
{2}D_{+}D_{-}\left(  \eta^{n}\right)  ,\\
\\
\frac{1}{\Delta t}\left(  I-dD_{+}D_{-}\right)  (u^{n+1}-u^{n})+D\left(
(1-\theta)\eta^{n}+\theta\eta^{n+1}\right)  +\frac{1}{2}D\left(  \left(
u^{n}\right)  ^{2}\right)  =0
\end{array}
\right.
\end{equation}
is adapted for studying the classical Boussinesq system, with $\theta\in[0,1]$. The convergence error verifies:%

\begin{equation}
\left\{
\begin{array}
[c]{r}%
E^{-}\left(  n\right)  +\Delta tD\left(  \left(  1-\theta\right)  f^{n}+\theta
f^{n+1}\right)  +\Delta tD\left(  e^{n}u_{\Delta}^{n}\right)  +\Delta
tD\left(  \eta_{\Delta}^{n}f^{n}\right)  \\
+\Delta tD\left(  e^{n}f^{n}\right)  =\frac{\tau_{1}}{2}\Delta t\Delta
xD_{+}D_{-}\left(  e^{n}\right)  -\Delta t\epsilon_{1}^{n},\hspace*{-2cm}\\
\\
\left(  I-dD_{+}D_{-}\right)  F^{-}\left(  n\right)  +\Delta tD\left(  \left(
1-\theta\right)  e^{n}+\theta e^{n+1}\right)  +\Delta tD\left(  f^{n}%
u_{\Delta}^{n}\right)  \\
+\frac{\Delta t}{2}D\left(  \left(  f^{n}\right)  ^{2}\right)  =-\Delta
t\epsilon_{2}^{n}.\hspace*{-2cm}
\end{array}
\right.  \label{a=b=c=0,d>0}%
\end{equation}
Recall that in this case:%
\begin{equation}
\mathcal{E}\left(  e,f\right)  =\left\Vert e\right\Vert _{\ell_{\Delta}^{2}%
}^{2}+\left\Vert f\right\Vert _{\ell_{\Delta}^{2}}^{2}+d\left\Vert
D_{+}f\right\Vert _{\ell_{\Delta}^{2}}^{2}.
\end{equation}
\begin{proof}{\bf (Proof of Proposition \ref{Prop_enegry_EST_AbcD} in the case $a=b=c=0,d>0$).}
Multiply the second equation of $\left(  \text{\ref{a=b=c=0,d>0}}\right)  $
with $F^{+}\left(  n\right)  $ and proceeding as we did in Section
\ref{section b,d>0} (Identity \eqref{bdpoz5} with $a=0$), we obtain that there exist two constants $C_{1,1}$ and
$C_{1,2}$ depending on $d$, $\left\Vert u_{\Delta}^{n}\right\Vert
_{\ell^{\infty}}$, $\left\Vert Du_{\Delta}^{n}\right\Vert _{\ell^{\infty}}$
and $\theta$, such that:%
\begin{align}
&  \left\Vert f^{n+1}\right\Vert _{\ell_{\Delta}^{2}}^{2}+d\left\Vert
D_{+}f^{n+1}\right\Vert _{\ell_{\Delta}^{2}}^{2}-\left\Vert f^{n}\right\Vert
_{\ell_{\Delta}^{2}}^{2}-d\left\Vert D_{+}f^{n}\right\Vert _{\ell_{\Delta}%
^{2}}^{2}\nonumber\\
&  =-\left\langle F^{+}\left(  n\right)  ,\Delta tD\left(  \left(
1-\theta\right)  e^{n}+\theta e^{n+1}\right)  +\Delta tD\left(  f^{n}%
u_{\Delta}^{n}\right)  +\frac{\Delta t}{2}D\left(  \left(  f^{n}\right)
^{2}\right)  -\Delta t\epsilon_{2}^{n}\right\rangle\label{d>0 1}\\
&  \leq\Delta t\left\Vert \epsilon_{2}^{n}\right\Vert _{\ell_{\Delta}^{2}}%
^{2}+\Delta tC_{1,1}\mathcal{E}\left(  e^{n},f^{n}\right)  +\Delta
tC_{1,2}\mathcal{E}\left(  e^{n+1},f^{n+1}\right)  \nonumber \\
&-\Delta t(1-\theta)\left\langle De^n,F^+\left(n\right)\right\rangle-\theta\Delta t\left\langle De^{n+1},F^+(n)\right\rangle. \nonumber%
\end{align}
Notice that Relation \eqref{integr_by_part_1} combined with the Cauchy-Schwarz inequality, the Young's inequality and the upper bound $||D(.)||_{\ell^2_{\Delta}}\leq||D_+(.)||_{\ell^2_{\Delta}}$  simplifies the previous last term
\begin{equation*}
\begin{split}
\left\langle De^n,F^+\left(n\right)\right\rangle&\leq||e^n||_{\ell^2_{\Delta}}^2+\frac{1}{2}||Df^n||_{\ell^2_{\Delta}}^2+\frac{1}{2}||Df^{n+1}||_{\ell^2_{\Delta}}^2\\
&\leq \max\left\{1+\frac{1}{2d}\right\}\mathcal{E}\left(e^n,f^n\right)+\frac{1}{2d}\mathcal{E}\left(e^{n+1}, f^{n+1}\right).
\end{split}
\end{equation*}
A similar inequality holds true for $\left\langle De^{n+1}, F^+\left(n\right)\right\rangle$.\\
Next, we rewrite the first equation of $\left(  \text{\ref{a=b=c=0,d>0}%
}\right)  $ as
\begin{align*}
e^{n+1} &  =e^{n}-\Delta tD\left(  e^{n}\left(  u_{\Delta}^{n}+f^{n}\right)
\right)  +\frac{\tau_{1}}{2}\Delta t\Delta xD_{+}D_{-}\left(
e^{n}\right)    -\Delta tD\left(  \left(  1-\theta\right)  f^{n}+\theta f^{n+1}\right)\\
&
-\Delta tD\left(  \eta_{\Delta}^{n}f^{n}\right)  -\Delta t\epsilon_{1}^{n}.
\end{align*}
Thus, using the Proposition $\left(  \text{\ref{Burgers2}}\right)  $ and Relation \eqref{product_rule2}, we get
that%
\begin{align}
\left\Vert e^{n+1}\right\Vert _{\ell_{\Delta}^{2}} &  \leq\left\Vert
e^{n}-\Delta tD\left(  e^{n}\left(  u_{\Delta}^{n}+f^{n}\right)  \right)
+\frac{\tau_{1}}{2}\Delta t\Delta xD_{+}D_{-}\left(  e^{n}\right)
\right\Vert _{\ell_{\Delta}^{2}}\nonumber\\
&  +\Delta t\left\Vert D\left(  \left(  1-\theta\right)  f^{n}+\theta
f^{n+1}\right)  \right\Vert _{\ell_{\Delta}^{2}}+\Delta t\left\Vert D\left(
\eta_{\Delta}^{n}f^{n}\right)  \right\Vert _{\ell_{\Delta}^{2}}+\Delta
t\left\Vert \epsilon_{1}^{n}\right\Vert _{\ell_{\Delta}^{2}}\nonumber\\
&  \leq\Delta t\left\Vert \epsilon_{1}^{n}\right\Vert _{\ell_{\Delta}^{2}%
}+\left(  1+C\Delta t\right)  \left\Vert e^{n}\right\Vert _{\ell_{\Delta}^{2}%
}+\Delta t\left(  \left(  1-\theta\right)  \left\Vert D\left(  f^{n}\right)
\right\Vert _{\ell_{\Delta}^{2}}+\theta\left\Vert D\left(  f^{n+1}\right)
\right\Vert _{\ell_{\Delta}^{2}}\right)  \nonumber\\
&  +\Delta t\max\left\{  \left\Vert \eta_{\Delta}^{n}\right\Vert
_{\ell^{\infty}},\left\Vert D\left(  \eta_{\Delta}^{n}\right)  \right\Vert
_{\ell^{\infty}}\right\}  \left(  \left\Vert f^{n}\right\Vert _{\ell_{\Delta
}^{2}}+\left\Vert D\left(  f^{n}\right)  \right\Vert _{\ell_{\Delta}^{2}%
}\right)  \nonumber\\
&  \leq\Delta t\left\Vert \epsilon_{1}^{n}\right\Vert _{\ell_{\Delta}^{2}%
}+\left(  1+C\Delta t\right)  \left\Vert e^{n}\right\Vert _{\ell_{\Delta}^{2}%
}+\Delta tC_{1,2}\mathcal{E}^{\frac{1}{2}}\left(  e^{n},f^{n}\right)
+C_{2,2}\Delta t\mathcal{E}^{\frac{1}{2}}\left(  e^{n+1},f^{n+1}\right)
,\label{d>0 2}%
\end{align}
with $C_{1,2}$ and $C_{2,2}$ depending on $||\eta_{\Delta}^n||_{\ell^{\infty}}$ and $||D\eta_{\Delta}^n||_{\ell^{\infty}}$. From $\left(  \text{\ref{d>0 2}}\right)  $, we deduce that%
\begin{align}
\left\Vert e^{n+1}\right\Vert _{\ell_{\Delta}^{2}}^{2} &  \leq\left(  \Delta
t+\left(  \Delta t\right)  ^{2}\right) C_0 \left\Vert \epsilon_{1}^{n}\right\Vert
_{\ell_{\Delta}^{2}}^{2}+\left(  1+C_{1,3}\Delta t\right)  \left\Vert
e^{n}\right\Vert _{\ell_{\Delta}^{2}}^{2}\nonumber\\
&  +\left(  \Delta t+\left(  \Delta t\right)  ^{2}\right)  C_{1,3}%
\mathcal{E}\left(  e^{n},f^{n}\right)  +\left(  \Delta t+\left(  \Delta
t\right)  ^{2}\right)  C_{2,3}\mathcal{E}\left(  e^{n+1},f^{n+1}\right)
.\label{d>0 3}%
\end{align}
Adding up the estimates yields%
\[
\left(  1-C_{1}\Delta t\right)  \mathcal{E}\left(  e^{n+1}%
,f^{n+1}\right)  \leq\left(  \Delta t+\left(  \Delta t\right)  ^{2}\right)C_0
\left\Vert \epsilon_{1}^{n}\right\Vert _{\ell_{\Delta}^{2}}^{2}+\Delta
t\left\Vert \epsilon_{2}^{n}\right\Vert _{\ell_{\Delta}^{2}}^{2}+\left(
1+C_{2}\Delta t\right)  \mathcal{E}\left(  e^{n},f^{n}\right)  .
\]
\end{proof}

\textbf{The case $a<0,b=0,c=0,d>0$.}\ \ \ \ 
In this case, the convergence error satisfies:%

\[
\left\{
\begin{array}
[c]{l}%
e^{n+1}+\Delta tDf^{n+1}+a\Delta tD_{+}D_{-}Df^{n+1}\\
=e^{n}-\Delta tD\left(  e^{n}u_{\Delta}^{n}\right)  -\Delta tD\left(
e^{n}f^{n}\right)  -\Delta tD\left(  \eta_{\Delta}^{n}f^{n}\right)
+\frac{\tau_{1}}{2}\Delta t\Delta xD_{+}D_{-}e^{n}-\Delta t\epsilon
_{1}^{n},\\
\\
\left(  I-dD_{+}D_{-}\right)  f^{n+1}+\Delta tDe^{n+1}=\left(  I-dD_{+}D_{-}\right)  f^{n}-\Delta tD\left(  f^{n}\left(  u_{\Delta
}^{n}+\frac{1}{2}f^{n}\right)  \right)  -\Delta t\epsilon_{2}^{n}.
\end{array}
\right.
\]
In this section, we will work with the energy functional:%
\[
\mathcal{E}\left(  e,f\right)  =d\left\Vert e\right\Vert _{\ell_{\Delta}^{2}%
}^{2}+\left(  -a\right)  \left\Vert \left(  I-dD_{+}D_{-}\right)  f\right\Vert
_{\ell_{\Delta}^{2}}^{2},
\]
which is better adapted to the system corresponding to the particular values
of the parameters in view here. Of course, $\mathcal{E}$ is equivalent to the
energy from $\left(  \text{\ref{E=mc^2}}\right)  $. 

\begin{proof}{\bf (Proof of Proposition \ref{Prop_enegry_EST_AbcD} in the case $a<0,b=c=0,d>0$).}
By summing up the square of the $\ell_{\Delta}^{2}$-norm of the first equation
by $d$, the square of the $\ell_{\Delta}^{2}$ norm of the second one by
$(-a)$, we get that, thanks to Relations \eqref{integr_by_part_1} and \eqref{integr_by_part_2},
\begin{align*}
&  \mathcal{E}\left(  e^{n+1},f^{n+1}\right)  +2(d+a)\Delta t\left\langle
Df^{n+1},e^{n+1}\right\rangle +d\Delta t^2||Df^{n+1}||_{\ell^2_{\Delta}}^2\\
&+2\Delta t^2(-a)d||D_+Df^{n+1}||_{\ell^2_{\Delta}}^2+(-a)\Delta t^2||De^{n+1}||_{\ell^2_{\Delta}}^2\\
&  +\Delta t^2a^2d||D_+D_-Df^{n+1}||_{\ell^2_{\Delta}}^2\leq(\Delta t+\Delta t^{2})C_0d\left\Vert \epsilon_{1}^{n}\right\Vert
_{\ell_{\Delta}^{2}}^{2}+(\Delta t+\Delta t^{2})C_0\left(  -a\right)  \left\Vert
\epsilon_{2}^{n}\right\Vert _{\ell_{\Delta}^{2}}^{2}\\
&  +\left(  1+\Delta tC_0\right)  d\left\Vert e^{n}-\Delta tD\left(
e^{n}\left(  u_{\Delta}^{n}+f^{n}\right)  \right)  +\frac{\tau_{1}}%
{2}\Delta t\Delta xD_{+}D_{-}e^{n}\right\Vert _{\ell_{\Delta}^{2}}^{2}\\
&  +\left(  \Delta t+\Delta t^{2}\right) C_0 d\left\vert \left\vert D\left(
\eta_{\Delta}^{n}f^{n}\right)  \right\vert \right\vert _{\ell_{\Delta}^{2}%
}^{2}\\
&  +\left(  1+\Delta tC_0\right)  \left(  -a\right)  \left\Vert \left(
I-dD_{+}D_{-}\right)  f^{n}\right\Vert _{\ell_{\Delta}^{2}}^{2}\\
&  +\left(  \Delta t+\Delta t^{2}\right)C_0  \left(  -a\right)  \left\Vert
D\left(  f^{n}\left(  u_{\Delta}^{n}+\frac{1}{2}f^{n}\right)  \right)
\right\Vert _{\ell_{\Delta}^{2}}^{2}.%
\end{align*}
We notice that
\begin{align*}
2\Delta t(d+a)\left\langle f^{n+1},De^{n+1}\right\rangle  &  \geq-\Delta
t(d-a)||Df^{n+1}||_{\ell_{\Delta}^{2}}^{2}-\Delta t(d-a)||e^{n+1}%
||_{\ell_{\Delta}^{2}}^{2}\\
&  \geq-\Delta tC_1\mathcal{E}\left(  e^{n+1},f^{n+1}\right)  .
\end{align*}
Using Proposition \ref{Burgers2} we get that%
\begin{align*}
& \left(  1+\Delta tC_0\right)  d\left\Vert e^{n}-\Delta tD\left(  e^{n}\left(
u_{\Delta}^{n}+f^{n}\right)  \right)  +\frac{\tau_{1}}{2}\Delta t\Delta
xD_{+}D_{-}e^{n}\right\Vert _{\ell_{\Delta}^{2}}^{2} \leq\left(  1+C_2\Delta t\right)  d||e^n||_{\ell^2_{\Delta}}^2  .
\end{align*}
Also, we have, due to \eqref{product_rule2}
\[
\left\vert \left\vert D\left(  \eta_{\Delta}^{n}f^{n}\right)  \right\vert
\right\vert _{\ell_{\Delta}^{2}}^{2}\leq\left\vert \left\vert D\left(
\eta_{\Delta}^{n}\right)  \right\vert \right\vert _{\ell^{\infty}}^{2}%
||f^{n}||_{\ell_{\Delta}^{2}}^{2}+||\eta_{\Delta}^{n}||_{\ell^{\infty}}%
^{2}||D\left(  f^{n}\right)  ||_{\ell_{\Delta}^{2}}^{2}\leq C_3\mathcal{E}%
\left(  e^{n},f^{n}\right)  .
\]
Proceeding as above, we get that:%
\[
\left\Vert D\left(  f^{n}\left(  u_{\Delta}^{n}+\frac{1}{2}f^{n}\right)
\right)  \right\Vert _{\ell_{\Delta}^{2}}^{2}\leq\max\left\{||Du_{\Delta}^n||_{\ell^{\infty}}^2,||u_{\Delta}^n||_{\ell^{\infty}}^2, ||Df^n||_{\ell^{\infty}}^2,||f^n||_{\ell^{\infty}}^2\right\} C_4\mathcal{E}\left(
e^{n},f^{n}\right)  .
\]
Adding up the above estimate gives us:%
\begin{align*}
\left(  1-C_1\Delta t\right)  \mathcal{E}\left(  e^{n+1},f^{n+1}\right)   &
\leq\left(  \Delta t+(\Delta t)^{2}\right)  C_0d\left\Vert \epsilon_{1}%
^{n}\right\Vert _{\ell_{\Delta}^{2}}^{2}+\left(  \Delta t+(\Delta
t)^{2}\right)  C_0\left(  -a\right)  \left\Vert \epsilon_{2}^{n}\right\Vert
_{\ell_{\Delta}^{2}}^{2}\\
&  +\left(  1+C_5\Delta t\right)  \mathcal{E}\left(  e^{n},f^{n}\right)  .
\end{align*}
\end{proof}

\textbf{The case $a<0,b=0,c<0,d>0$.}\ \ \ \ 
In this case, the convergence error satisfies:%

\begin{equation}\label{syst_complet_a<0,b=0,c<0,d>0}
\hspace*{-0.5cm}\left\{
\begin{array}
[c]{l}%
e^{n+1}+\Delta tDf^{n+1}+a\Delta tD_{+}D_{-}Df^{n+1}\\
=e^{n}-\Delta tD\left(  e^{n}u_{\Delta}^{n}\right)  -\Delta tD\left(
e^{n}f^{n}\right)  -\Delta tD\left(  \eta_{\Delta}^{n}f^{n}\right)
+\frac{\tau_{1}}{2}\Delta t\Delta xD_{+}D_{-}e^{n}-\Delta t\epsilon
_{1}^{n},\\
\\
\left(  I-dD_{+}D_{-}\right)  f^{n+1}+\Delta tDe^{n+1}+c\Delta tD_{+}%
D_{-}De^{n+1}\\
=\left(  I-dD_{+}D_{-}\right)  f^{n}-\Delta tD\left(  f^{n}\left(  u_{\Delta
}^{n}+\frac{1}{2}f^{n}\right)  \right)  -\Delta t\epsilon_{2}^{n}.
\end{array}
\right.
\end{equation}
Again, in order to close the estimates and prove the convergence of the
scheme, we will be using the following energy functional:%
\[
\mathcal{E}\left(  e,f\right)  \overset{def.}{=}(-a)\left\Vert e\right\Vert
_{\ell_{\Delta}^{2}}^{2}+\left(  -cd\right)  \left\Vert D_{+}e\right\Vert
_{\ell_{\Delta}^{2}}^{2}+\left(  -a\right)  \left\Vert \left(  I-dD_{+}%
D_{-}\right)  f\right\Vert _{\ell_{\Delta}^{2}}^{2},
\]
which is equivalent to that one from $\left(  \text{\ref{E=mc^2}}\right)  $.
\begin{proof}{\bf (Proof of Proposition \ref{Prop_enegry_EST_AbcD} in the case $a<0,b=0,c<0,d>0$).} In order to derive a $H^{1}$-control on $e^{n+1}$, let us apply $\sqrt
{-cd}D_{+}$ \ in the first equation, to obtain%
\begin{align}
&  \sqrt{-cd}D_{+}e^{n+1}+\Delta t\sqrt{-cd}D_{+}Df^{n+1}+\sqrt{-cd}a\Delta tD_{+}%
D_{+}D_{-}Df^{n+1}  \nonumber\\
&=\sqrt{-cd}D_{+}e^{n}-\Delta t\sqrt{-cd}D_{+}D\left(  e^{n}u_{\Delta}%
^{n}\right) \nonumber \\
& -\Delta t\sqrt{-cd}D_{+}D\left(  e^{n}f^{n}\right)   +\frac{\tau_{1}}{2}\sqrt{-cd}\Delta t\Delta xD_{+}D_{+}D_{-}e^{n}\nonumber\\
&-\sqrt{-cd}\Delta tD_{+}D\left(  \eta_{\Delta}^{n}f^{n}\right)  -\sqrt
{-cd}\Delta tD_{+}\epsilon_{1}^{n}.\label{EQ_IN_ADDItion}
\end{align}
We focus now on the system composed of both equations of \eqref{syst_complet_a<0,b=0,c<0,d>0} with Equation \eqref{EQ_IN_ADDItion} in addition. We will multiply the first and second equations of \eqref{syst_complet_a<0,b=0,c<0,d>0} by $(-a)$ thereafter. As before, we square the three equalities to compute the $\ell^2_{\Delta}$-norm. Let us observe that%
\begin{align}
I_{1} &  =(-a)\left\Vert e^{n+1}+\Delta tDf^{n+1}+a\Delta tD_{+}D_{-}%
Df^{n+1}\right\Vert _{\ell_{\Delta}^{2}}^{2}\nonumber\\
&  =(-a)\left\Vert e^{n+1}\right\Vert _{\ell_{\Delta}^{2}}^{2}+(-a)\left(  \Delta
t\right)  ^{2}\left\Vert Df^{n+1}\right\Vert _{\ell_{\Delta}^{2}}^{2}%
-a^{3}\left(  \Delta t\right)  ^{2}\left\Vert D_{+}D_{-}Df^{n+1}\right\Vert
_{\ell_{\Delta}^{2}}^{2}\nonumber\\
&  +2a^2(\Delta t)^{2}\left\Vert D_{+}Df^{n+1}\right\Vert _{\ell_{\Delta}^{2}%
}^{2}+2(-a)\Delta t\left\langle e^{n+1},Df^{n+1}\right\rangle -2a^2\Delta
t\left\langle e^{n+1},D_{+}D_{-}Df^{n+1}\right\rangle ,\label{b=0_1}%
\end{align}
along with%
\begin{small}
\begin{align}
I_{2} &  =\left(  -cd\right)  \left\Vert D_{+}e^{n+1}+\Delta tD_{+}%
Df^{n+1}+a\Delta tD_{+}D_{+}D_{-}Df^{n+1}\right\Vert _{\ell_{\Delta}^{2}}%
^{2}\nonumber\\
&  =\left(  -cd\right)  \left\Vert D_{+}e^{n+1}\right\Vert _{\ell_{\Delta}%
^{2}}^{2}+\left(  -cd\right)  \left(  \Delta t\right)  ^{2}\left\Vert
D_{+}Df^{n+1}\right\Vert _{\ell_{\Delta}^{2}}^{2}+\left(  -cd\right)
a^{2}\left(  \Delta t\right)  ^{2}\left\Vert D_{+}D_{+}D_{-}Df^{n+1}%
\right\Vert _{\ell_{\Delta}^{2}}\nonumber\\
&  +2acd(\Delta t)^{2}\left\Vert D_{+}D_{-}Df^{n+1}\right\Vert _{\ell_{\Delta
}^{2}}^{2}-2cd\Delta t\left\langle D_{+}e^{n+1},D_{+}Df^{n+1}\right\rangle
\nonumber\\
&  -2acd\Delta t\left\langle D_{+}e^{n+1},D_{+}D_{+}D_{-}Df^{n+1}\right\rangle
,\label{b=0_2}%
\end{align}
\end{small}
and
\begin{small}
\begin{align}
I_{3} &  =\left(  -a\right)  \left\Vert \left(  I-dD_{+}D_{-}\right)
f^{n+1}+\Delta tDe^{n+1}+c\Delta tD_{+}D_{-}De^{n+1}\right\Vert _{\ell
_{\Delta}^{2}}^{2}\nonumber\\
&  =\left(  -a\right)  \left\Vert \left(  I-dD_{+}D_{-}\right)  f^{n+1}%
\right\Vert _{\ell_{\Delta}^{2}}^{2}+\left(  -a\right)  c^{2}\left(  \Delta
t\right)  ^{2}\left\Vert D_{+}D_{-}De^{n+1}\right\Vert _{\ell_{\Delta}^{2}%
}^{2}+(-a)\left(  \Delta t\right)  ^{2}\left\Vert De^{n+1}\right\Vert
_{\ell_{\Delta}^{2}}^2\nonumber\\
&  +2ac\left(  \Delta t\right)  ^{2}\left\Vert D_{+}De^{n+1}\right\Vert
_{\ell_{\Delta}^{2}}+2(-a)\Delta t\left\langle \left(  I-dD_{+}D_{-}\right)
f^{n+1},De^{n+1}\right\rangle \nonumber\\
&  -2ac\Delta t\left\langle f^{n+1},D_{+}D_{-}De^{n+1}\right\rangle
+2acd\Delta t\left\langle D_{+}D_{-}f^{n+1},D_{+}D_{-}De^{n+1}\right\rangle
.\label{b=0_3}%
\end{align}
\end{small}
Observe that the last term from $\left(  \text{\ref{b=0_2}}\right)  $ cancels
with the last term of $\left(  \text{\ref{b=0_3}}\right)  $, according to Relations \eqref{integr_by_part_1} and \eqref{integr_by_part_2}. The same is true for $2\Delta t\left\langle e^{n+1}, Df^{n+1}\right\rangle$ in \eqref{b=0_1} and $2\Delta t\left\langle f^{n+1}, De^{n+1}\right\rangle$ in \eqref{b=0_3}. Therefore, by integrating by part $2a\Delta t\left\langle e^{n+1}, D_+D_-Df^{n+1}\right\rangle$ in \eqref{b=0_1} and $-2ac\Delta t\left\langle f^{n+1}, D_+D_-De^{n+1}\right\rangle$ in \eqref{b=0_3}, it yieds
\begin{align*}
I_{1}+I_{2}+I_{3} &  \geq\mathcal{E}\left(  e^{n+1},f^{n+1}\right)  +2a^2\Delta
t\left\langle De^{n+1},D_{+}D_{-}f^{n+1}\right\rangle \\
&  -2cd\Delta t\left\langle D_{+}e^{n+1},D_{+}Df^{n+1}\right\rangle +2da\Delta
t\left\langle D_{+}D_{-}f^{n+1},De^{n+1}\right\rangle \\
&  -2ac\Delta t\left\langle D_{+}D_{-}f^{n+1},De^{n+1}\right\rangle .
\end{align*}
Young's inequality enables to lower bound the previous inequality:
\begin{small}
\begin{align*}
I_{1}+I_{2}+I_{3} &  \geq\mathcal{E}\left(  e^{n+1},f^{n+1}\right)-(-a^2-cd+d(-a)+ac)\Delta t\left\{||D_+e^{n+1}||_{\ell^2_{\Delta}}^2+||D_+D_-f^{n+1}||_{\ell^2_{\Delta}}^2\right\}\\
&\geq\left(  1-C_1\Delta t\right)  \mathcal{E}\left(  e^{n+1},f^{n+1}\right)
.
\end{align*}
\end{small}
Let us now focus on the right hand side of the squared equations.
Using Cauchy-Schwarz inequality and Proposition \ref{Burgers2}, we get that%
\begin{small}
\begin{align}
I_{1} &  =(-a)\left\Vert e^{n}-\Delta tD\left(  e^{n}u_{\Delta}%
^{n}\right)  -\Delta tD\left(  e^{n}f^{n}\right)  +\frac{\tau_{1}}%
{2}\Delta t\Delta xD_{+}D_{-}e^{n}-\Delta tD(\eta_{\Delta}^nf^n)-\Delta t\epsilon_1^n\right\Vert _{\ell_{\Delta}^{2}}^2\nonumber\\
&\leq (-a)\left\Vert e^{n}-\Delta tD\left(  e^{n}u_{\Delta}%
^{n}\right)  -\Delta tD\left(  e^{n}f^{n}\right)  +\frac{\tau_{1}}%
{2}\Delta t\Delta xD_{+}D_{-}e^{n}\right\Vert ^2_{\ell_{\Delta}^{2}}(1+C_0\Delta t)\nonumber\\
&  +(-a)(\Delta t+\Delta t^2)C_0\left\Vert D\left(  \eta_{\Delta}^{n}f^{n}\right)  \right\Vert^2
_{\ell_{\Delta}^{2}}+(-a)(\Delta t+\Delta t^2)C_0\left\Vert \epsilon_{1}^{n}\right\Vert^2
_{\ell_{\Delta}^{2}}\nonumber\\
&  \leq\left(  1+C\Delta t\right)  (-a)\left\Vert e^{n}\right\Vert ^2_{\ell^{2}%
}+(\Delta t+\Delta t^2)(-a)C_0\left\Vert D\left(  \eta_{\Delta}^{n}f^{n}\right)  \right\Vert^2
_{\ell_{\Delta}^{2}}+(\Delta t+\Delta t^2)(-a)C_0\left\Vert \epsilon_{1}^{n}\right\Vert^2
_{\ell_{\Delta}^{2}}.\label{acd1}%
\end{align}
\end{small}
Using Proposition \ref{Burgers3} and Young's inequality, we get %
\begin{align}
  &I_{2}\leq-cd\left\Vert D_{+}e^{n}-\Delta tD_{+}D\left(  e^{n}(u_{\Delta
}^{n}+f^n)\right)    +\frac{\tau_{1}%
}{2}\Delta t\Delta xD_{+}D_{+}D_{-}e^{n}-\Delta tD_+D(\eta_{\Delta}^nf^n)\right.\nonumber\\
&\left.-\Delta tD_+\left(\epsilon_1^n\right)\right\Vert _{\ell_{\Delta}^{2}%
}^2\nonumber\\
&\leq-cd(1+C_0\Delta t)\left\Vert D_{+}e^{n}-\Delta tD_{+}D\left(  e^{n}u_{\Delta
}^{n}\right)  -\Delta tD_{+}D\left(  e^{n}f^{n}\right)  +\frac{\tau_{1}%
}{2}\Delta t\Delta xD_{+}D_{+}D_{-}e^{n}\right\Vert^2 _{\ell_{\Delta}^{2}%
}\nonumber\\
&  -cd(\Delta t+\Delta t^2)C_0\left\Vert D_{+}D\left(  \eta_{\Delta}^{n}f^{n}\right)
\right\Vert ^2_{\ell_{\Delta}^{2}}-cd(\Delta t+\Delta t^2)C_0\left\Vert D_{+}%
\epsilon_{1}^{n}\right\Vert^2 _{\ell_{\Delta}^{2}}\nonumber\\
&  \leq-cd(\Delta t+\Delta t^2)C_0\left\Vert D_{+}\epsilon_{1}^{n}\right\Vert^2
_{\ell_{\Delta}^{2}}-cd\left(  \Delta tC_{1}\left\Vert e^{n}%
\right\Vert ^2_{\ell_{\Delta}^{2}}+\left(  1+C_{2}\Delta t\right)  \left\Vert
D_{+}e^{n}\right\Vert ^2_{\ell_{\Delta}^{2}}\right)  \nonumber\\
&  -cd(\Delta t+\Delta t^2)\max\left\{||\eta_{\Delta}^n||_{\ell^{\infty}}^2, ||D_+\eta_{\Delta}^n||_{\ell^{\infty}}^2, ||D_+D\eta_{\Delta}^n||_{\ell^{\infty}}^2\right\}\left(||f^n||_{\ell^2_{\Delta}}^2\right.\\
&\left.+||D_+f^n||_{\ell^2_{\Delta}}^2+||D_+Df^n||_{\ell^2_{\Delta}}^2\right)\nonumber\\
&  \leq-cd(\Delta t+\Delta t^2)C_0\left\Vert D_{+}\epsilon_{1}^{n}\right\Vert
_{\ell_{\Delta}^{2}}^2+\Delta tC_{3}\mathcal{E}\left(
e^{n},f^{n}\right) +(-cd)||D_+e^n||_{\ell^2_{\Delta}}^2. \label{acd3}%
\end{align}
Using once again the Young's inequality, it holds%
\begin{align}
  I_{3}&=-a\left\Vert \left(  I-dD_{+}D_{-}\right)
f^{n}-\Delta tD\left(f^n\left(u_{\Delta}^n+\frac{f^n}{2}\right)\right)-\Delta t\epsilon_2^n\right\Vert _{\ell_{\Delta}^{2}}^2\nonumber\\
&\leq -a(1+C_0\Delta t)\left\Vert \left(  I-dD_{+}D_{-}\right)
f^{n}\right\Vert^2 _{\ell_{\Delta}^{2}}\nonumber\\
&-a(\Delta t+\Delta t^2)C_0\left\Vert D\left(  f^{n}\left(  u_{\Delta}^{n}+\frac{1}{2}f^{n}\right)
\right)  \right\Vert ^2_{\ell_{\Delta}^{2}}-a(\Delta t+\Delta t^2)C_0\left\Vert
\epsilon_{2}^{n}\right\Vert^2 _{\ell_{\Delta}^{2}}\nonumber\\
&  \leq-a(\Delta t+\Delta t^2)C_0\left\Vert \epsilon_{2}^{n}\right\Vert ^2_{\ell_{\Delta
}^{2}}+C\Delta t \max\left\{1, ||u_{\Delta}^n||_{\ell^{\infty}}, ||D_+u_{\Delta}^n||_{\ell^{\infty}}\right\} \mathcal{E}\left(
e^{n},f^{n}\right) \nonumber\\
&+(-a)||(I-dD_+D_-)f^{n}||^2_{\ell^2_{\Delta}} .\label{acd2}%
\end{align}
From $\left(  \text{\ref{b=0_1}}\right)  $ - $\left(  \text{\ref{acd2}}\right)
$, we obtain%
\begin{align*}
\mathcal{E}\left(  e^{n+1},f^{n+1}\right)  \left(  1-C_{1}\Delta
t\right)   &  \leq\left(  \Delta t+\Delta t^{2}\right)  ((-a)\left\Vert
\epsilon_{1}^{n}\right\Vert _{\ell_{\Delta}^{2}}^{2}+\left(  -a\right)
\left\Vert \epsilon_{2}^{n}\right\Vert _{\ell_{\Delta}^{2}}^{2}%
+(-c)d\left\Vert D_{+}\epsilon_{1}^{n}\right\Vert _{\ell_{\Delta}^{2}}^{2}) \\
&+\left(  1+\Delta tC_{2}\right)  \mathcal{E}\left(  e^{n}%
,f^{n}\right).
\end{align*}
\end{proof}

\subsubsection{The case $d=0,b>0$.} Once again, three configurations are considered according to the parameters $a, b, c$ and $d$.\\
\textbf{The case $a=c=d=0,b>0$.}\ \ \ \ 
In this case, as for the classical Boussinesq system (the case $a=b=c=0,d>0$, page \pageref{classicAL_BOUSSinesq}) we derive estimates for a
more general scheme: the following $\theta$-scheme where the advection term is discretized according to a convex combination of $n$ and $n+1$%

\begin{equation}
\left\{
\begin{array}
[c]{l}%
\frac{1}{\Delta t}\left(  I-bD_{+}D_{-}\right)(\eta^{n+1}-\eta^{n})+D\left(  (1-\theta)u^{n}+\theta
u^{n+1}\right)  +D\left(  \eta^{n}u^{n}\right)  =0,\\
\\
\frac{1}{\Delta t}  (u^{n+1}-u^{n})+D\left(
(1-\theta)\eta^{n}+\theta\eta^{n+1}\right)  +\frac{1}{2}D\left(  \left(
u^{n}\right)  ^{2}\right)  =\frac{\tau_{2}\Delta x}{2}D_{+}D_{-}\left(
\eta^{n}\right)  .
\end{array}
\right.
\end{equation}
The convergence error verifies:%
\begin{equation}
\left\{
\begin{array}
[c]{l}%
\left(  I-bD_{+}D_{-}\right)E^{-}\left(  n\right)  +\Delta tD\left(  \left(  1-\theta\right)  f^{n}+\theta
f^{n+1}\right)  +\Delta tD\left(  e^{n}u_{\Delta}^{n}\right) \\ +\Delta
tD\left(  \eta_{\Delta}^{n}f^{n}\right)  +\Delta tD\left(  e^{n}f^{n}\right)
=  -\Delta
t\epsilon_{1}^{n},\\
\\
  F^{-}\left(  n\right)  +\Delta tD\left(  \left(
1-\theta\right)  e^{n}+\theta e^{n+1}\right)  +\Delta tD\left(  f^{n}%
u_{\Delta}^{n}\right)  +\frac{\Delta t}{2}D\left(  \left(  f^{n}\right)
^{2}\right)  \\=\frac{\tau_{2}}{2}\Delta t\Delta xD_{+}D_{-}\left(  e^{n}\right)-\Delta t\epsilon_{2}^{n}.
\end{array}
\right. \label{a=c=d=0,b>0}%
\end{equation}

\noindent We work with%
\begin{equation}
\mathcal{E}\left(  e,f\right)  =\left\Vert e\right\Vert _{\ell_{\Delta}^{2}%
}^{2}+b\left\Vert D_{+}\left(  e\right)  \right\Vert _{\ell_{\Delta}^{2}}%
^{2}+\left\Vert f\right\Vert _{\ell_{\Delta}^{2}}^{2}.
\end{equation}
\begin{proof}{\bf (Proof of Proposition \ref{Prop_enegry_EST_AbcD} in the case $a=c=d=0,b>0$).}
We multiply the first equation by $E^{+}\left(  n\right)  $ in order to get %
\begin{align}
&  \left(  \left\Vert e^{n+1}\right\Vert _{\ell^{2}}^{2}+b\left\Vert
D_{+}e^{n+1}\right\Vert _{\ell^{2}}^{2}\right)  -\left(  \left\Vert
e^{n}\right\Vert _{\ell^{2}}^{2}+b\left\Vert D_{+}e^{n}\right\Vert _{\ell^{2}%
}^{2}\right)  \nonumber\\
&  =-\Delta t\left\langle D\left(  \left(  1-\theta\right)  f^{n}+\theta
f^{n+1}\right)  +D\left(  e^{n}u_{\Delta}^{n}\right)  +D\left(  \eta_{\Delta
}^{n}f^{n}\right)  +D\left(  e^{n}f^{n}\right)  +\epsilon_{1}^{n},E^{+}\left(
n\right)  \right\rangle .\nonumber
\end{align}
Integration by parts \eqref{integr_by_part_1} gives
\begin{align}
& \left(  \left\Vert e^{n+1}\right\Vert _{\ell^{2}}^{2}+b\left\Vert
D_{+}e^{n+1}\right\Vert _{\ell^{2}}^{2}\right)  -\left(  \left\Vert
e^{n}\right\Vert _{\ell^{2}}^{2}+b\left\Vert D_{+}e^{n}\right\Vert _{\ell^{2}%
}^{2}\right)  \nonumber\\
&  =-\Delta t\left\langle \epsilon_{1}^{n},E^{+}\left(  n\right)
\right\rangle +\Delta t\left\langle \left(  1-\theta\right)  f^{n}+\theta
f^{n+1}+e^{n}u_{\Delta}^{n}+\eta_{\Delta}^{n}f^{n}+e^{n}f^{n},DE^{+}\left(
n\right)  \right\rangle \nonumber\\
&  \leq\Delta t\left\Vert \epsilon_{1}^{n}\right\Vert _{\ell_{\Delta}^{2}}%
^{2}+\Delta tC_{1}\mathcal{E}\left(  e^{n},f^{n}\right)  +\Delta
tC_{2}\mathcal{E}\left(  e^{n+1},f^{n+1}\right)  ,\label{b>0 1}%
\end{align}
with $C_1$ and $C_2$ two constants proportional to $$1+\frac{1}{b}\max\left\{||u_{\Delta}^n||_{\ell^{\infty}}, ||\eta_{\Delta}^n||_{\ell^{\infty}}\right\}.$$
Using the second equation of $\left(  \text{\ref{a=c=d=0,b>0}}\right)  $, the
triangle inequality along with Proposition \ref{Burgers1}, one obtains:%
\begin{align}
\left\Vert f^{n+1}\right\Vert _{\ell^{2}} &  \leq\left\Vert f^{n}-\Delta
tD\left(  f^{n}u_{\Delta}^{n}\right)  -\frac{\Delta t}{2}D\left(  \left(
f^{n}\right)  ^{2}\right)  +\frac{\tau_{2}}{2}\Delta t\Delta xD_{+}%
D_{-}\left(  f^{n}\right)  \right\Vert _{\ell^{2}}\nonumber\\
&  +\Delta t\left\Vert \epsilon_{2}^{n}\right\Vert _{\ell_{\Delta}^{2}}+\Delta
t\left(  1-\theta\right)  \left\Vert D\left(  e^{n}\right)  \right\Vert
_{\ell_{\Delta}^{2}}+\Delta t\theta\left\Vert D\left(  e^{n+1}\right)
\right\Vert _{\ell_{\Delta}^{2}}\nonumber\\
&  \leq\Delta t\left\Vert \epsilon_{2}^{n}\right\Vert _{\ell_{\Delta}^{2}%
}+\left(  1+C\Delta t\right)  \left\Vert f^{n}\right\Vert _{\ell_{\Delta}^{2}%
}+\Delta t\left(  1-\theta\right)  \left\Vert D\left(  e^{n}\right)
\right\Vert _{\ell_{\Delta}^{2}}+\Delta t\theta\left\Vert D\left(
e^{n+1}\right)  \right\Vert _{\ell_{\Delta}^{2}}.\label{b>0 2}%
\end{align}
Thus, by adding up estimate $\left(  \text{\ref{b>0 1}}\right)  $ with the
square of the estimate $\left(  \text{\ref{b>0 2}}\right)  $, we get that%
\[
\left(  1-C_3\Delta t\right)  \mathcal{E}\left(  e^{n+1},f^{n+1}\right)
\leq\Delta t\left\Vert \epsilon_{1}^{n}\right\Vert _{\ell_{\Delta}^{2}}%
^{2}+(\Delta t+\Delta t^{2})C_0\left\Vert \epsilon_{2}^{n}\right\Vert
_{\ell_{\Delta}^{2}}^{2}+\left(  1+C_4\Delta t\right)  \mathcal{E}\left(
e^{n},f^{n}\right)  .
\]
\end{proof}

\textbf{The case $a<0,b>0,c=d=0$.}\ \ \ \ 
The convergence error satisfies:%

\begin{equation}
\left\{
\begin{array}
[c]{l}%
\left(  I-bD_{+}D_{-}\right)  e^{n+1}+\Delta tDf^{n+1}+a\Delta tD_{+}%
D_{-}Df^{n+1}=\left(  I-bD_{+}D_{-}\right)  e^{n}\\
-\Delta tD\left(  e^{n}u_{\Delta}^{n}\right)  -\Delta tD\left(  e^{n}%
f^{n}\right)  -\Delta tD\left(  \eta_{\Delta}^{n}f^{n}\right)  -\Delta
t\epsilon_{1}^{n},\\
\\
f^{n+1}+\Delta tDe^{n+1}=f^{n}-\Delta tD\left(  f^{n}\left(  u_{\Delta}%
^{n}+\frac{1}{2}f^{n}\right)  \right)  +\frac{\tau_{2}}{2}\Delta t\Delta
xD_{+}D_{-}\left(  f^{n}\right)  -\Delta t\epsilon_{2}^{n}.
\end{array}
\right.  \label{ab1}%
\end{equation}
Consider the energy functional we will work with will be%
\[
\mathcal{E}\left(  e,f\right)  =\left\Vert e\right\Vert _{\ell_{\Delta}^{2}%
}^{2}+b\left\Vert D_{+}e\right\Vert _{\ell_{\Delta}^{2}}^{2}+\left\Vert
f\right\Vert _{\ell_{\Delta}^{2}}^{2}+\left(  -a\right)  \left\Vert
D_{+}f\right\Vert _{\ell_{\Delta}^{2}}^{2}.
\]

\begin{proof}{\bf (Proof of Proposition \ref{Prop_enegry_EST_AbcD} in the case $a<0,b>0,c=d=0$).} Let us multiply the first equation of \eqref{ab1} with $2e^{n+1}$ to obtain:%
\begin{align}
&  2\left\Vert e^{n+1}\right\Vert _{\ell_{\Delta}^{2}}^{2}+2b\left\Vert
D_{+}e^{n+1}\right\Vert _{\ell_{\Delta}^{2}}^{2}+2\Delta t\left\langle
Df^{n+1},e^{n+1}\right\rangle +2a\Delta t\left\langle D_{+}D_{-}%
Df^{n+1},e^{n+1}\right\rangle \nonumber\\
&  =2\left\langle \left(  I-bD_{+}D_{-}\right)  e^{n},e^{n+1}\right\rangle
-2\Delta t\left\langle D\left(  e^{n}u_{\Delta}^{n}\right)  +D\left(
e^{n}f^{n}\right)  +D\left(  \eta_{\Delta}^{n}f^{n}\right)  +\epsilon_{1}%
^{n},e^{n+1}\right\rangle \nonumber\\
&  \leq\left\Vert e^{n+1}\right\Vert _{\ell_{\Delta}^{2}}^{2}+b\left\Vert
D_{+}e^{n+1}\right\Vert _{\ell_{\Delta}^{2}}^{2}+\left\Vert e^{n}\right\Vert
_{\ell_{\Delta}^{2}}^{2}+b\left\Vert D_{+}e^{n}\right\Vert _{\ell_{\Delta}%
^{2}}^{2}\nonumber\\
&  +2\Delta t||D(e^nu_{\Delta}^n)+D(e^nf^n)+D(\eta_{\Delta}^nf^n)||_{\ell^2_{\Delta}}||e^{n+1}||_{\ell^2_{\Delta}}+\Delta t||\epsilon_1^n||_{\ell^2_{\Delta}}^2+\Delta t||e^{n+1}||_{\ell^2_{\Delta}}^2.\label{ab_cd=0}
\end{align}
The last inequality is due to Cauchy-Schwarz and Young's inequalities. Notice that 
$$||D(e^nu_{\Delta}^n)+D(e^nf^n)+D(\eta_{\Delta}^nf^n)||_{\ell^2_{\Delta}}\leq C_1\mathcal{E}^{1/2}(e^n,f^n),$$
with $C_1$ proportional to 
$$\max\left\{||u_{\Delta}^n||_{\ell^{\infty}}, ||Du_{\Delta}^n||_{\ell^{\infty}}, ||\eta_{\Delta}^n||_{\ell^{\infty}}, ||D\eta_{\Delta}^n||_{\ell^{\infty}}\right\} \max\left\{1, \frac{1}{b}, \frac{1}{-a}\right\}.$$
Thus, Equation \eqref{ab_cd=0} becomes
\begin{align}
&  \left\Vert e^{n+1}\right\Vert _{\ell_{\Delta}^{2}}^{2}+b\left\Vert
D_{+}e^{n+1}\right\Vert _{\ell_{\Delta}^{2}}^{2}+2\Delta t\left\langle
Df^{n+1},e^{n+1}\right\rangle +2a\Delta t\left\langle D_{+}D_{-}%
Df^{n+1},e^{n+1}\right\rangle \nonumber\\
&  \leq\left\Vert e^{n}\right\Vert _{\ell_{\Delta}^{2}}^{2}+b\left\Vert
D_{+}e^{n}\right\Vert _{\ell_{\Delta}^{2}}^{2}+C_1\Delta t  \mathcal{E}%
^{\frac{1}{2}}\left(  e^{n},f^{n}\right)   \left\Vert e^{n+1}\right\Vert
_{\ell_{\Delta}^{2}}    +\Delta t\left\Vert \epsilon_{1}%
^{n}\right\Vert^2 _{\ell_{\Delta}^{2}}+\Delta t  \left\Vert e^{n+1}\right\Vert^2
_{\ell_{\Delta}^{2}}.\label{ab2}%
\end{align}
Next, taking the square of the $\ell_{\Delta}^{2}$-norm of the second equation
of $\left(  \text{\ref{ab1}}\right)  $ and using Proposition \ref{Burgers2}, we
obtain%
\begin{align}
&  \left\Vert f^{n+1}\right\Vert _{\ell_{\Delta}^{2}}^{2}+2\Delta
t\left\langle f^{n+1},De^{n+1}\right\rangle +\left(  \Delta t\right)
^{2}\left\Vert De^{n+1}\right\Vert _{\ell_{\Delta}^{2}}^{2}\nonumber\\
&  \leq\left(  \Delta t+\left(  \Delta t\right)  ^{2}\right)C_0  \left\Vert
\epsilon_{2}^{n}\right\Vert _{\ell_{\Delta}^{2}}^{2}\nonumber\\
&+\left(  1+C_0\Delta
t\right)  \left\Vert f^{n}-\Delta tD\left(  f^{n}\left(  u_{\Delta}^{n}%
+\frac{1}{2}f^{n}\right)  \right)  +\frac{\tau_{2}}{2}\Delta t\Delta
xD_{+}D_{-}\left(  f^{n}\right)  \right\Vert _{\ell_{\Delta}^{2}}%
^{2}\nonumber\\
&  \leq\left(  \Delta t+\left(  \Delta t\right)  ^{2}\right) C_0 \left\Vert
\epsilon_{2}^{n}\right\Vert _{\ell_{\Delta}^{2}}^{2}+\left(  1+C_2\Delta
t\right)  \left\Vert f^{n}\right\Vert _{\ell_{\Delta}^{2}}^{2}.\label{ab3}%
\end{align}
When we will sum up Equations \eqref{ab2} and \eqref{ab3}, both terms $$2\Delta t\left\langle Df^{n+1},e^{n+1}\right\rangle \text{\ \ \ and\ \ \ } 2\Delta t\left\langle f^{n+1},De^{n+1}\right\rangle $$ will cancel each other, thanks to Relation \eqref{integr_by_part_1}. We have to cancel $2a\Delta t\left\langle D_+D_-Df^{n+1},e^{n+1}\right\rangle $ in \eqref{ab2} too. This is the aim of the following computation.\\
Applying $\sqrt{-a}D_{+}$ into the second equation of $\left(  \text{\ref{ab1}%
}\right)  $ and taking the square of the $\ell_{\Delta}^{2}$-norm of the
resulting equation yields (with Proposition \ref{Burgers3})
\begin{align}
&  \left(  -a\right)  \left\Vert D_{+}f^{n+1}\right\Vert _{\ell_{\Delta}^{2}%
}^{2}+2\left(  -a\right)  \Delta t\left\langle D_{+}f^{n+1},D_{+}%
De^{n+1}\right\rangle +\left(  -a\right)  \left(  \Delta t\right)
^{2}\left\Vert D_{+}De^{n+1}\right\Vert _{\ell_{\Delta}^{2}}^{2}\nonumber\\
&  \leq\left(  \Delta t+\left(  \Delta t\right)  ^{2}\right)C_0  \left(
-a\right)  \left\Vert D_{+}\epsilon_{2}^{n}\right\Vert _{\ell_{\Delta}^{2}%
}^{2}\nonumber\\
&  +\left(  1+C_0\Delta t\right)  \left(  -a\right)  \left\Vert D_{+}f^{n}-\Delta
tD_{+}D\left(  f^{n}\left(  u_{\Delta}^{n}+\frac{1}{2}f^{n}\right)  \right)
+\frac{\tau_{2}}{2}\Delta t\Delta xD_{+}D_{+}D_{-}\left(  f^{n}\right)
\right\Vert _{\ell_{\Delta}^{2}}^{2}\nonumber\\
&  \leq\left(  \Delta t+\left(  \Delta t\right)  ^{2}\right)  C_0\left(
-a\right)  \left\Vert D_{+}\epsilon_{2}^{n}\right\Vert _{\ell_{\Delta}^{2}%
}^{2}+\left(  1+C_3\Delta t\right)  \left(  -a\right)  \left\Vert D_{+}%
f^{n}\right\Vert _{\ell_{\Delta}^{2}}^{2}+\Delta tC_4\left(  -a\right)
\left\Vert f^{n}\right\Vert _{\ell_{\Delta}^{2}}^{2}.\label{ab4}%
\end{align}
Adding up the estimates $\left(  \text{\ref{ab2}}\right)  $, $\left(
\text{\ref{ab3}}\right)  $ and $\left(  \text{\ref{ab4}}\right)  $ leads to
\begin{align*}
(1-C_5\Delta t)\mathcal{E}\left(  e^{n+1},f^{n+1}\right)   &  \leq\Delta t\left\Vert
\epsilon_{1}^{n}\right\Vert _{\ell_{\Delta}^{2}}^{2}+\left(  \Delta t+\left(
\Delta t\right)  ^{2}\right)C_0  \left(  \left\Vert \epsilon_{2}^{n}\right\Vert
_{\ell_{\Delta}^{2}}^{2}+\left(  -a\right)  \left\Vert D_{+}\epsilon_{2}%
^{n}\right\Vert _{\ell_{\Delta}^{2}}^{2}\right) \\
& +\left(  1+C_6\Delta t\right)  \mathcal{E}\left(  e^{n},f^{n}\right)  .
\end{align*}
\end{proof}

\textbf{The case $a<0,b>0,c<0,d=0$.}\ \ \ \ 
Finally, in this case, the convergence error satisfies:%

\begin{equation}
\left\{
\begin{array}
[c]{l}%
\left(  I-bD_{+}D_{-}\right)  e^{n+1}+\Delta tDf^{n+1}+a\Delta tD_{+}%
D_{-}Df^{n+1}=\left(  I-bD_{+}D_{-}\right)  e^{n}\\
-\Delta tD\left(  e^{n}u_{\Delta}^{n}\right)  -\Delta tD\left(  e^{n}%
f^{n}\right)  -\Delta tD\left(  \eta_{\Delta}^{n}f^{n}\right)  -\Delta
t\epsilon_{1}^{n},\\
\\
f^{n+1}+\Delta tDe^{n+1}+c\Delta tD_{+}D_{-}De^{n+1}\\
=f^{n}-\Delta tD\left(
f^{n}\left(  u_{\Delta}^{n}+\frac{1}{2}f^{n}\right)  \right)  +\frac{\tau_{2}}{2}\Delta t\Delta xD_{+}D_{-}\left(  f^{n}\right)  -\Delta
t\epsilon_{2}^{n}.
\end{array}
\right.  \label{abc1}%
\end{equation}
In this case, we will use the energy functional:%
\[
\mathcal{E}\left(  e,f\right)  =\left(  -c\right)  \left\Vert \left(
I-bD_{+}D_{-}\right)  e\right\Vert _{\ell_{\Delta}^{2}}^{2}+\left\Vert
f\right\Vert _{\ell_{\Delta}^{2}}^{2}+\left(  -a\right)  b\left\Vert
D_{+}f\right\Vert _{\ell_{\Delta}^{2}}^{2}.
\]

\begin{proof}{\bf (Proof of Proposition \ref{Prop_enegry_EST_AbcD} in the case $a<0,b>0,c<0,d=0$).} Taking the square of the first
equations of $\left(  \text{\ref{abc1}}\right)  $ and multiplying the result
with $\left(  -c\right)  $ yields%
\begin{small}
\begin{align}
J_{1} &  =\left(  -c\right)  \left\Vert \left(  I-bD_{+}D_{-}\right)
e^{n+1}\right\Vert _{\ell_{\Delta}^{2}}^{2}+\left(  -c\right)  \left(  \Delta
t\right)  ^{2}\left\Vert Df^{n+1}\right\Vert _{\ell_{\Delta}^{2}}^{2}+\left(
-c\right)  a^{2}\left(  \Delta t\right)  ^{2}\left\Vert D_{+}D_{-}%
Df^{n+1}\right\Vert _{\ell_{\Delta}^{2}}^{2}\nonumber\\
&  +2ac\left(  \Delta t\right)  ^{2}\left\Vert D_{+}Df^{n+1}\right\Vert
_{\ell_{\Delta}^{2}}^{2}-2c\Delta t\left\langle \left(  I-bD_{+}D_{-}\right)
e^{n+1},Df^{n+1}\right\rangle \nonumber\\
&  -2ac\Delta t\left\langle e^{n+1},D_{+}D_{-}Df^{n+1}\right\rangle
+2acb\Delta t\left\langle D_{+}D_{-}e^{n+1},D_{+}D_{-}Df^{n+1}\right\rangle.\label{abc22}%
\end{align}
\end{small}
Taking the square of the $\ell_{\Delta}^{2}$-norm of the second equation of
$\left(  \text{\ref{abc1}}\right)  $ leads to%
\begin{align}
 J_{2}&=\left\Vert f^{n+1}\right\Vert _{\ell_{\Delta}^{2}}^{2}+\left(  \Delta
t\right)  ^{2}\left\Vert De^{n+1}\right\Vert _{\ell_{\Delta}^{2}}^{2}+\left(
c\Delta t\right)  ^{2}\left\Vert D_{+}D_{-}De^{n+1}\right\Vert _{\ell_{\Delta
}^{2}}^{2}-2c\left(  \Delta t\right)  ^{2}\left\Vert D_{+}De^{n+1}\right\Vert
_{\ell_{\Delta}^{2}}^{2}\nonumber\\
& + 2\Delta t\left\langle f^{n+1},De^{n+1}\right\rangle +2c\Delta t\left\langle
f^{n+1},D_{+}D_{-}De^{n+1}\right\rangle.\label{abc2}
\end{align}

Let us observe that the terms
\[
2\Delta t\left\langle f^{n+1},De^{n+1}\right\rangle \text{\ \ \ and\ \ \  }2c\Delta t\left\langle
f^{n+1},D_{+}D_{-}De^{n+1}\right\rangle\text{,\ \ \ in\ }\eqref{abc2}
\]
and 
\[
-2ac\Delta t\left\langle
e^{n+1},D_{+}D_{-}Df^{n+1}\right\rangle\text{\ \ \ and\ \ \  }-2c\Delta t \left\langle (I-bD_+D_-)e^{n+1}, Df^{n+1}\right\rangle \text{,\ \ \ in\ }\eqref{abc22}
\]

\noindent can be lower controlled by the energy $-\Delta tC\mathcal{E}(e^{n+1}, f^{n+1})$ (by using integration by parts \eqref{integr_by_part_1} and Cauchy-Schwarz inequality) such that they do not raise any issues. In
order to get rid of the term%
\begin{equation}
2acb\Delta t\left\langle D_{+}D_{-}e^{n+1},D_{+}D_{-}Df^{n+1}\right\rangle \text{\ \ \ in\ }\eqref{abc22}
\label{problematic}%
\end{equation}
we will apply $\sqrt{\left(  -a\right)  b}D_{+}$ into the second equation of
$\left(  \text{\ref{abc1}}\right)  $ and consider the square of the
$\ell_{\Delta}^{2}$-norm of the result. First, let us compute:
\begin{small}
\begin{align}
J_{3} &  =\left(  -a\right)  b\left\Vert D_{+}f^{n+1}+\Delta tD_{+}%
De^{n+1}+c\Delta tD_{+}D_{+}D_{-}De^{n+1}\right\Vert _{\ell_{\Delta}^{2}}%
^{2}\nonumber\\
&  =\left(  -a\right)  b\left\Vert D_{+}f^{n+1}\right\Vert _{\ell_{\Delta}%
^{2}}^{2}+\left(  -a\right)  b\left(  \Delta t\right)  ^{2}\left\Vert
D_{+}De^{n+1}\right\Vert _{\ell_{\Delta}^{2}}^{2}+\left(  -a\right)  b\left(
c\Delta t\right)  ^{2}\left\Vert D_{+}D_{+}D_{-}De^{n+1}\right\Vert
_{\ell_{\Delta}^{2}}^{2}\nonumber\\
&  +2\left(  -a\right)  b\Delta t\left\langle D_{+}f^{n+1},D_{+}%
De^{n+1}\right\rangle +2(-a)bc\Delta t\left\langle D_{+}f^{n+1},D_{+}%
D_{+}D_{-}De^{n+1}\right\rangle \nonumber\\
&  +2abc\left(  \Delta t\right)  ^{2}\left\Vert D_{+}D_{-}De^{n+1}\right\Vert
_{\ell_{\Delta}^{2}}^{2}.\label{abc_3}%
\end{align}
\end{small}
Of course the problematic term from $\left(  \text{\ref{problematic}}\right)
$ will cancel with
\[
-2abc\Delta t\left\langle D_{+}f^{n+1},D_{+}D_{+}D_{-}De^{n+1}\right\rangle
\]
appearing in $\left(  \text{\ref{abc_3}}\right)  $. The additional term $2\left(  -a\right)  b\Delta t\left\langle D_{+}f^{n+1},D_{+}D(e)^{n+1}\right\rangle$ in \eqref{abc_3} can be once again controlled by the energy $-\Delta tC\mathcal{E}(e^{n+1}, f^{n+1})$ (by using integrations by parts \eqref{integr_by_part_1} and Cauchy-Schwarz inequality).\\
Let us now interest to the right hand side.  For the first equation of $\left(  \text{\ref{abc1}}\right)  $, we obtain by Young inequality
\begin{align}
J_1&  \leq\left(  -c\right)  \left(  \Delta t+ \Delta t
^{2}\right) C_0 \left\Vert \epsilon_{1}^{n}\right\Vert _{\ell_{\Delta}^{2}}%
^{2}+\left(  1+C_0\Delta t\right)  \left(  -c\right)  \left\Vert \left(
I-bD_{+}D_{-}\right)  e^{n}\right\Vert _{\ell_{\Delta}^{2}}^{2}+C_1\Delta
t\mathcal{E}\left(  e^{n},f^{n}\right)  .\label{I1_abc}
\end{align}
In the previous inequality, we have upper bounded
$\Delta t||-D(e^nu^n_{\Delta})-D(e^nf^n)-D(\eta_{\Delta}^nf^n)||_{\ell^{2}_{\Delta }}^2$
by
$$\max\left\{||u_{\Delta}^n||_{\ell^{\infty}}, ||Du_{\Delta}^n||_{\ell^{\infty}}, ||\eta_{\Delta}^n||_{\ell^{\infty}} , ||D\eta_{\Delta}^n||_{\ell^{\infty}}, \right\}\max\left\{\frac{1}{(-a)b}, \frac{1}{2b(-c)}\right\}\mathcal{E}(e^n, f^n).$$
\noindent The second equation of $\left(  \text{\ref{abc1}}\right)  $ gives
\begin{small}
\begin{align}
J_2&  \leq \left(  \Delta t+ \Delta t  ^{2}\right) C_0 \left\Vert
\epsilon_{2}^{n}\right\Vert _{\ell_{\Delta}^{2}}^{2}\nonumber\\
&+\left(  1+C_0\Delta
t\right)  \left\Vert f^{n}-\Delta tD\left(  f^n\left(  u^n_{\Delta}
+\frac{1}{2}f^n\right)  \right)  +\frac{\tau_{2}}{2}\Delta t\Delta
xD_{+}D_{-}\left(  f^n\right)  \right\Vert _{\ell_{\Delta}^{2}}%
^{2}\nonumber\\
&  \leq\left(  \Delta t+ \Delta t ^{2}\right)C_0  \left\Vert
\epsilon_{2}^{n}\right\Vert _{\ell_{\Delta}^{2}}^{2}+\left(  1+C_1\Delta
t\right)  \left\Vert f^{n}\right\Vert _{\ell_{\Delta}^{2}}^{2}.\label{I2_abc}%
\end{align}
\end{small}

\noindent Moreover, using
Proposition $\left(  \text{\ref{Burgers3}}\right)  $, we obtain%
\begin{align}
J_{3} &  \leq\left(  -a\right)  b\left(  \Delta t+\left(  \Delta t\right)
^{2}\right)  C_0\left\Vert D_{+}\epsilon_{2}^{n}\right\Vert _{\ell_{\Delta}^{2}%
}^{2}+C_2\Delta t(-a)b\left\Vert f^{n}\right\Vert _{\ell_{\Delta}^{2}}%
^{2}\nonumber\\
&+\left(  1+C_3\Delta t\right)  \left(  -a\right)  b\left\Vert D_{+}%
f^{n}\right\Vert _{\ell_{\Delta}^{2}}^{2}\label{I3}.%
\end{align}
Putting together Estimates $\left(  \text{\ref{abc22}}\right)  $, $\left(
\text{\ref{abc2}}\right)  $ and $\left(  \text{\ref{abc_3}}\right)  $ with \eqref{I1_abc}, \eqref{I2_abc} and \eqref{I3} gives
\begin{align*}
  \left(  1-C_4\Delta t\right)  \mathcal{E}\left(  e^{n+1},f^{n+1}\right) &  \leq\left(  \Delta t+\Delta t^{2}\right) C_0 \left(  -c\right)  \left\Vert
\epsilon_{1}^{n}\right\Vert _{\ell_{\Delta}^{2}}^{2}\\
&+\left(  \Delta t+\left(
\Delta t\right)  ^{2}\right)  C_0\left(  \left\Vert \epsilon_{2}^{n}\right\Vert
_{\ell_{\Delta}^{2}}^{2}+\left(  -ab\right)  \left\Vert D_{+}\epsilon_{2}%
^{n}\right\Vert _{\ell_{\Delta}^{2}}^{2}\right)  \\
&  +\left(  1+C_5\Delta t\right)  \mathcal{E}\left(  e^{n},f^{n}\right)  .
\end{align*}
\end{proof}

\section{Experimental results\label{Simulari}}

In order to illustrate our theoretical results, we compare in this paragraph, known exact
solutions (\textit{i.e.} traveling waves) with the discrete solutions computed with the numerical schemes. We fix $[0,L]$ the space domain with $L=40$. Moreover, we use periodic boundary conditions. Those conditions are not absorbing boundary
conditions, which would mimic perfectly the behavior on $\mathbb{Z}$, but we fix the final time $T$ small enough and we take the initial
conditions localized enough in order to minimize boundary effects. \\

 In Figures \ref{fig_1}-\ref{fig_7}, we plot the exact and the numerical solutions which are computed with a space cell size $\Delta x=\frac{1}{2^6}$ and the time step $\Delta t=0.001$. \\
 The convergence results are gathered in Tables \ref{TABLE_CASES_A_B}-\ref{TABLE_CASES_E_F}. The computations are performed with a number of cells $J$ for the values $J\in\{640, 1280, 2560, 5120, 10240\}.$
 The time step is chosen to be $\Delta t^n=\frac{\Delta x}{||u_{\Delta}^n||_{\ell^{\infty}}}$ in order to verify the CFL-type condition. We perform computations up to the final time $T=2$ or $T=4$. In the case where $bd=0$, we chose the Rusanov coefficients $\tau_i=\tau_i^n$ verifying $\tau_1^n\geq||u_{\Delta}^n||_{\ell^{\infty}}$ and $\tau_2^n\geq||u_{\Delta}^n||_{\ell^{\infty}}$, when they are needed. 
 
\subsection{Linear case.}
We begin by a test for the linear case:

\begin{equation}
\left\{
\begin{array}
[c]{l}%
\left(  I-b\partial_{xx}^{2}\right)  \partial_{t}\eta+\left(  I+a\partial
_{xx}^{2}\right)  \partial_{x}u =0,\\
\left(  I-d\partial_{xx}^{2}\right)  \partial_{t}u+\left(  I+c\partial
_{xx}^{2}\right)  \partial_{x}\eta=0,\\
\eta_{|t=0}=\eta_{0},\text{ }u_{|t=0}=u_{0}.
\end{array}
\right.  \tag{$L_{abcd}$}
\end{equation}

More precisely, we take $a=-\frac{7}{30}$, $b=\frac{7}{15}$, $c=-\frac{2}{5}$ and $d=\frac{1}{2}$ and initialize the scheme with the initial datum
\[
\left\{
\begin{split}
&  \eta_0(x)=\frac{3}{8}\mathrm{sech}^{2}\left(  \frac{1}{2}\sqrt{\frac{5}{7}%
}\left(  x-\frac{L}{2}\right)  \right)  ,\\
&  u_0(x)=\frac{1}{2\sqrt{2}}\mathrm{sech}^{2}\left(  \frac{1}{2}\sqrt
{\frac{5}{7}}\left(  x-\frac{L}{2}\right)  \right)  .
\end{split}
\right.
\]
In this case, the discrete energy must be conserved if we use the Crank-Nicolson scheme (\textit{i.e.} $\theta=\frac{1}{2}$).
We perform computation up to the final time $T=2$ with space and time steps $\Delta x=\frac{1}{2^5}$ and $\Delta t=0.001$ with $\theta=\frac{1}{2}$. The conclusion is that the discrete energy is conserved, more or less up to a factor of order $10^{-11}$, as it is illustrated by Figure \ref{linear_ABCD}. 
\begin{figure}[h]
\centering\epsfig{figure=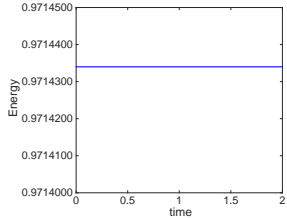,width=.3\linewidth}
\caption{Energy conservation in the linear case \label{linear_ABCD}}
\end{figure}

\subsection{Case $b>0$, $d>0$.}
This section is divided in two parts. In the first part, we will illustrate experimentally the rates of accuracy obtained in Theorem \ref{Teorema_Numerica1}. In the second part, we investigate in details the traveling-wave solutions.

\subsubsection*{Numerical convergence rates}
We first take a look at the BBM-BBM system with $a=0$, $b=\frac{1}{6}$, $c=0$,
$d=\frac{1}{6}$ for which two different exact solutions are known, see \cite{BonaChen1998} and\cite{Chen1998}. We perform a first experiment for the following family of exact solutions:
\[
(A)\left\{
\begin{split}
&  \eta(t,x)=\frac{15}{4}\left\{  -2+\mathrm{cosh}\left(  3\sqrt{\frac{2}{5}%
}\left(  x-\frac{L}{2}-\frac{5}{2}t\right)  \right)  \right\}  \mathrm{sech}%
^{4}\left(  \frac{3}{\sqrt{10}}\left(  x-\frac{L}{2}-\frac{5}{2}t\right)
\right)  ,\\
&  u(t,x)=\frac{15}{2}\mathrm{sech}^{2}\left(  \frac{3}{\sqrt{10}}\left(
x-\frac{L}{2}-\frac{5}{2}t\right)  \right)  .
\end{split}
\right.
\]
The results are represented in Figure \ref{fig_1}.

\begin{figure}[h]
\begin{minipage}[t]{1\linewidth}
\begin{minipage}[b]{.5\linewidth}
\centering\epsfig{figure=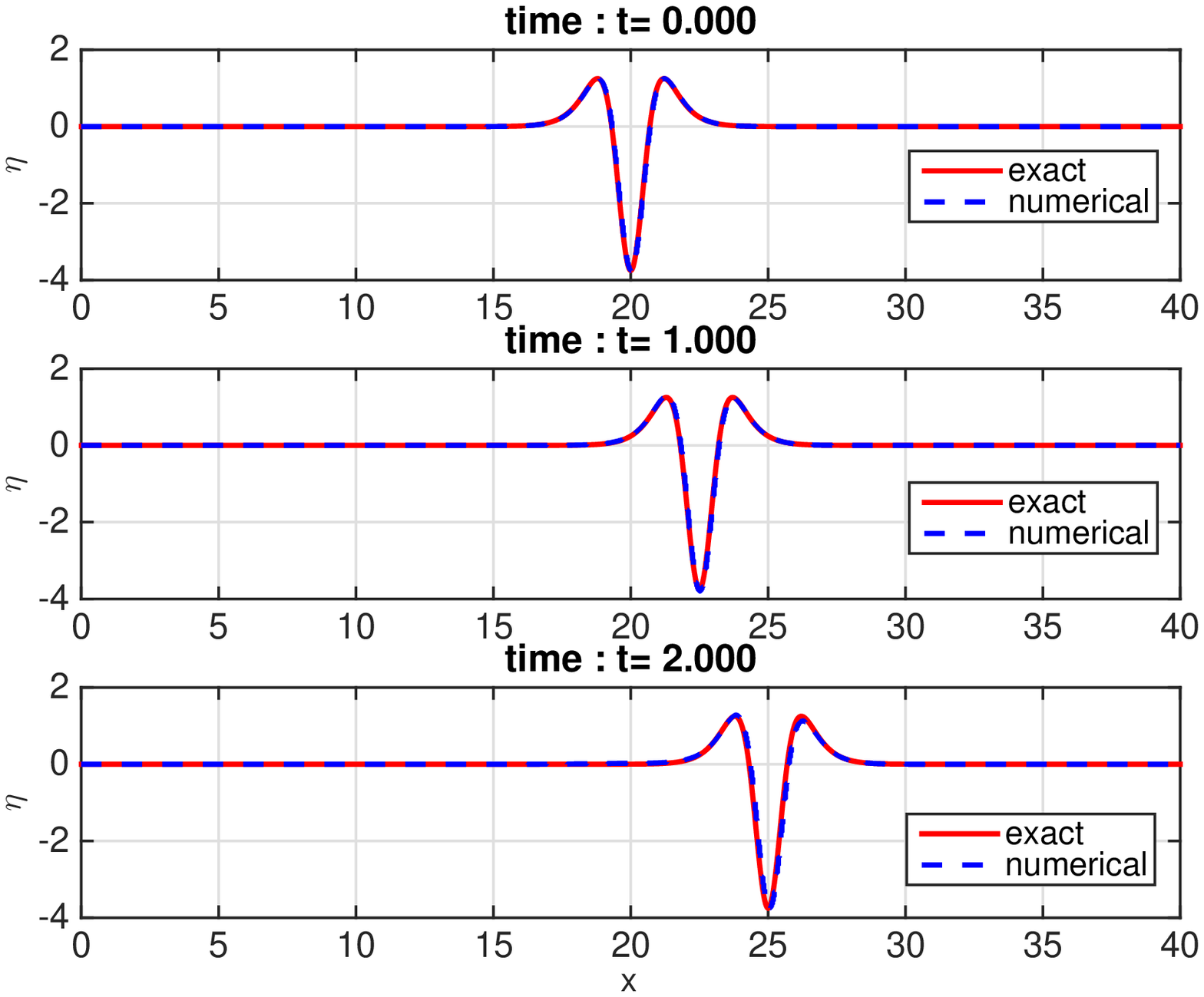,width=\linewidth}
\end{minipage}\hspace*{-0.2cm}
\begin{minipage}[b]{.5\linewidth}
\centering\epsfig{figure=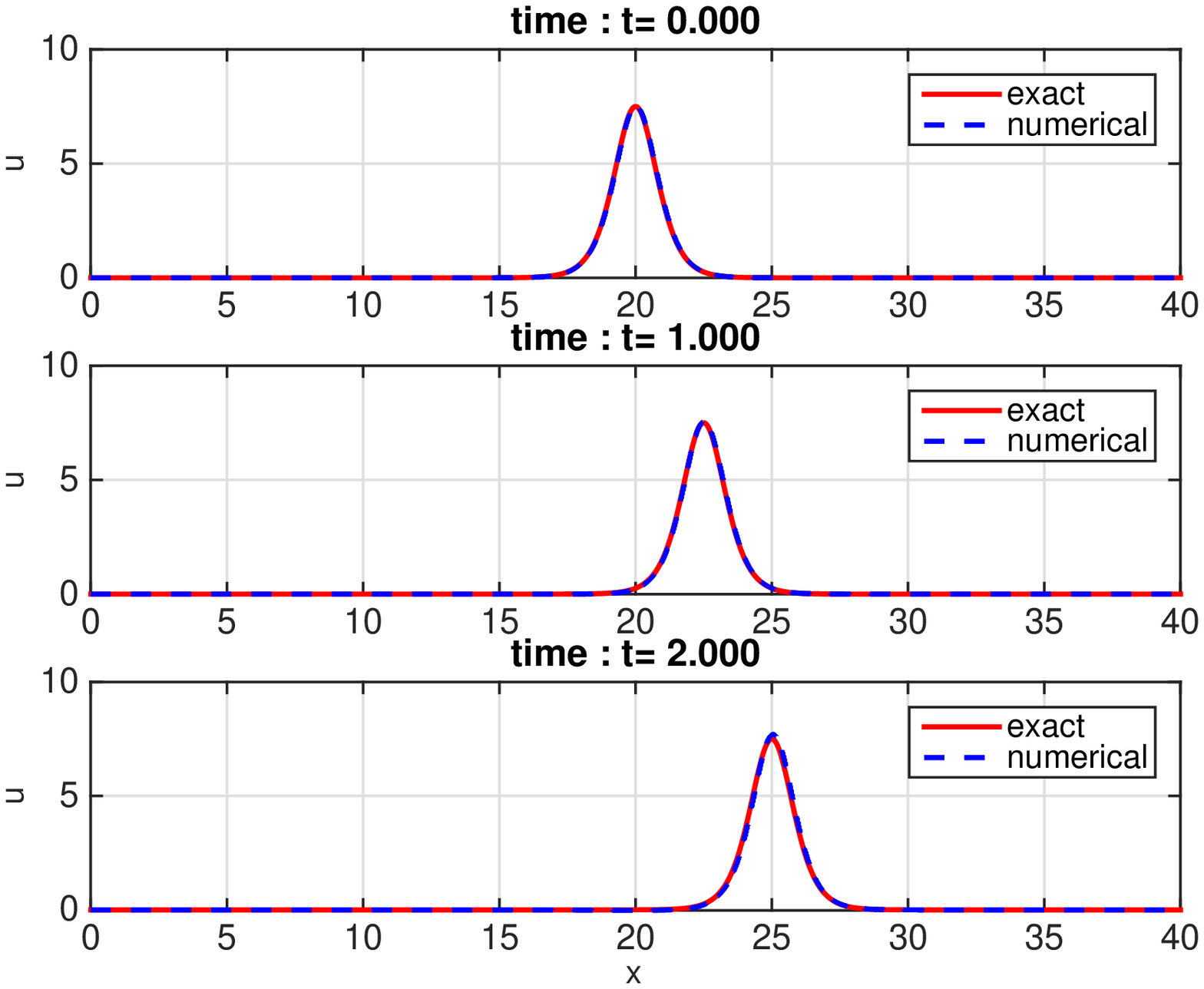,width=\linewidth}
\end{minipage}\vspace*{0.1cm}
\caption{Case $(A)$ where $a=0$, $b=\frac{1}{6}$, $c=0$, $d=\frac{1}{6}$ : results for $\eta$ (left) and $u$ (right) \label{fig_1}}
\end{minipage}
\end{figure}

We perform a second experiment for the following family of solutions:
\[
(B)\left\{
\begin{split}
&  \eta(t,x)=-1,\\
&  u(t,x)=C_{s}\left(1-\frac{\rho}{6}\right)+\frac{C_{s}\rho}{2}\mathrm{sech}^{2}\left(
\frac{\sqrt{\rho}}{2}\left(  x-\frac{L}{2}-C_{s}t\right)  \right)  ,
\end{split}
\right.
\]
with $C_{s}=2$ and $\rho=1.1$. The corresponding result can be found in Figure \ref{fig_2}.\\
\begin{figure}[h]
\begin{minipage}[t]{1\linewidth}
\begin{minipage}[b]{.5\linewidth}
\centering\epsfig{figure=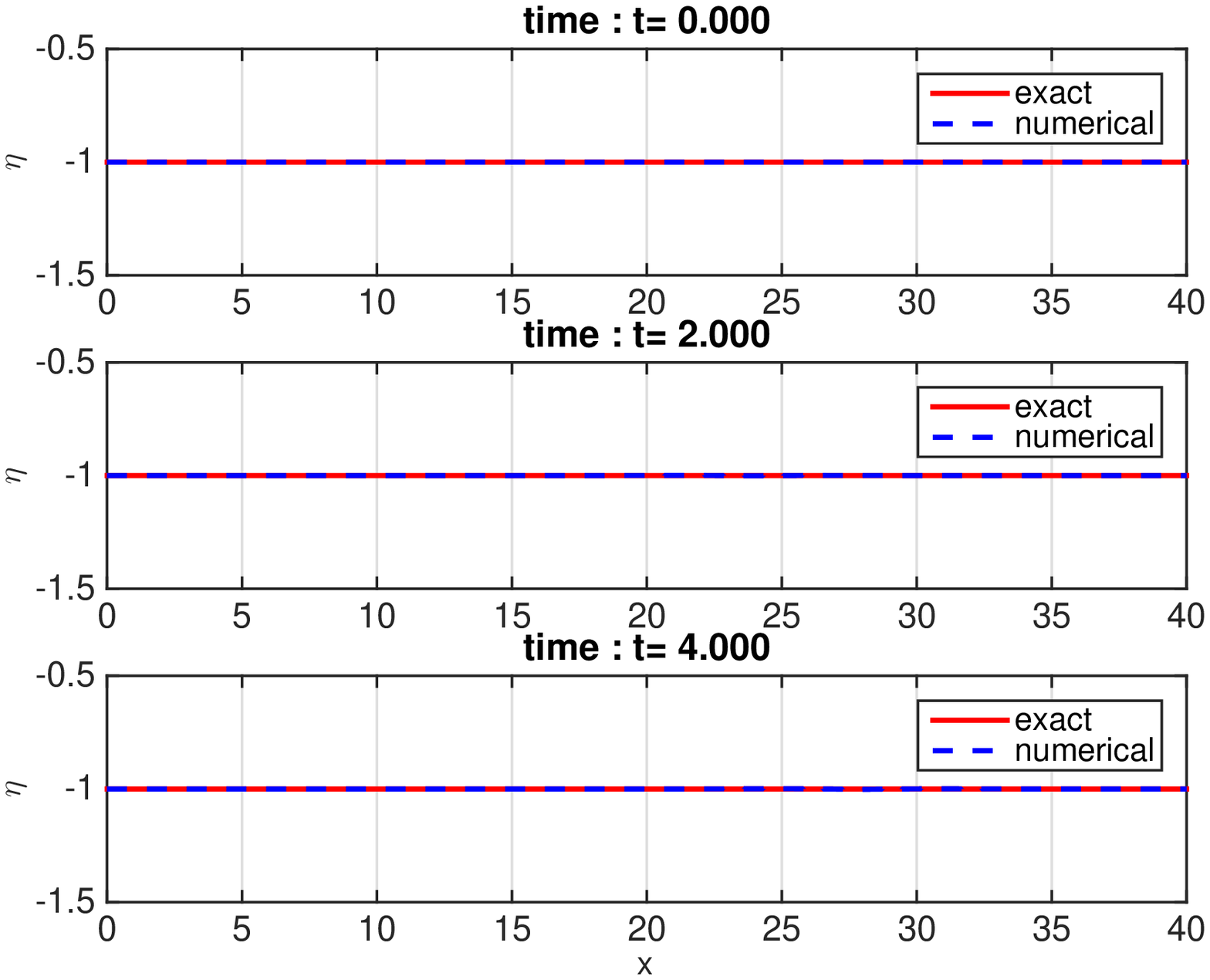,width=\linewidth}
\end{minipage}\hspace*{-0.2cm}
\begin{minipage}[b]{.5\linewidth}
\centering\epsfig{figure=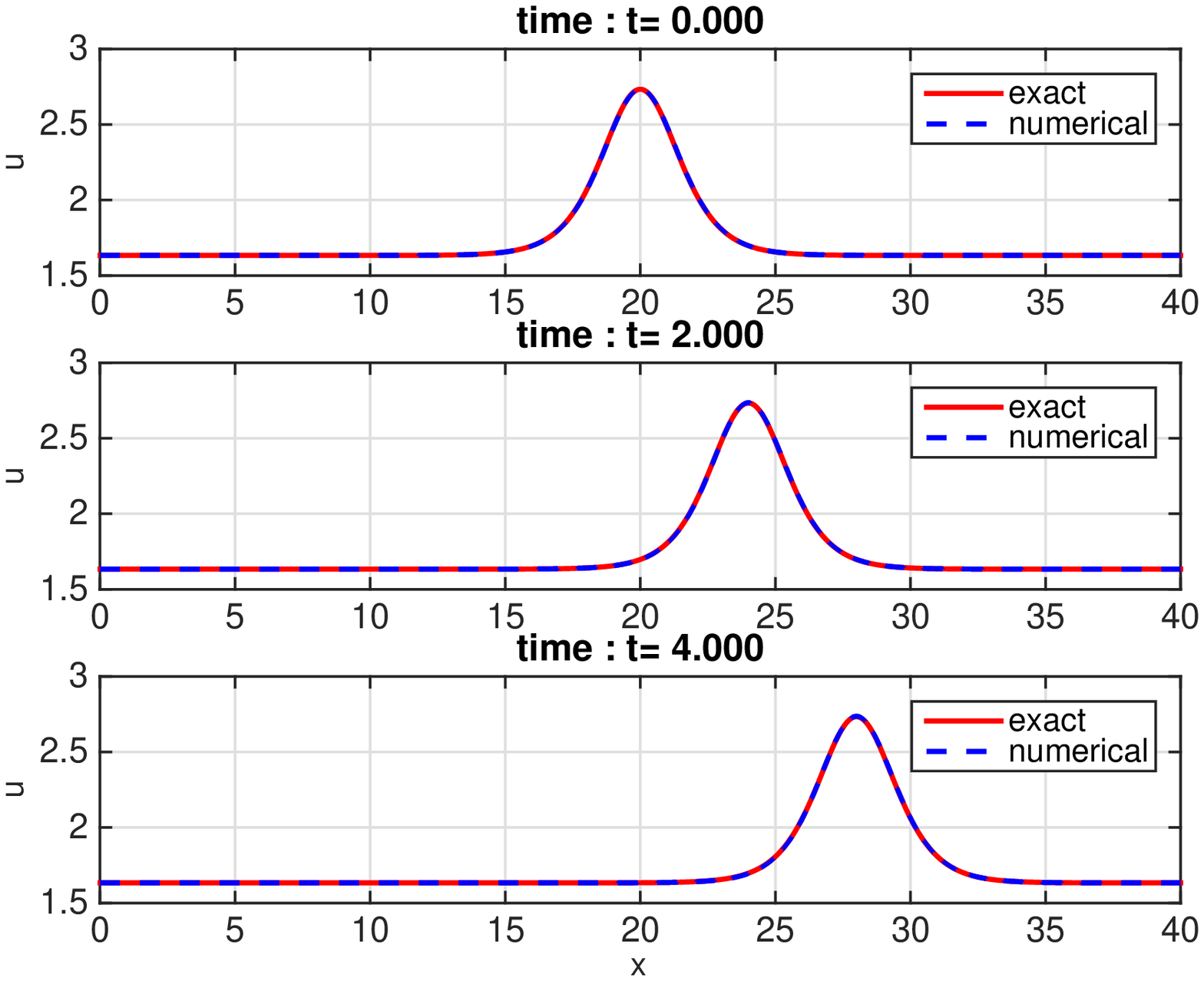,width=\linewidth}
\end{minipage}\vspace*{0.1cm}
\caption{Case $(B)$ where $a=0$, $b=\frac{1}{6}$, $c=0$, $d=\frac{1}{6}$ : results for $\eta$ (left) and $u$ (right) \label{fig_2}}
\end{minipage}
\end{figure}
In order to compute \textit{a numerical rate of convergence}, we perform computations with increasingly smaller space meshes.
The results are gathered in Table \ref{TABLE_CASES_A_B}.

\begin{table}[b]
\begin{center}
\hspace*{-0.7cm}
\begin{tabular}{|c|c|c|c|c|}
\hline
&\multicolumn{2}{c|}{Case $(A)$}&\multicolumn{2}{c|}{Case $(B)$}\\
\cline{2-5}
$\Delta x$&&&&\\
&energy error&exp. rate&energy error&exp. rate\\
&&&&\\
\hline
&&&&\\
$6.25000.10^{-2}$&$4.48993.10^{0\ \ }$&&$8.51815.10^{-2}$&\\
$3.12500.10^{-2}$&$2.05132.10^{0\ \ }$&$1.13270$&$4.14750.10^{-2}$&$1.03830$\\
$1.56250.10^{-2}$&$9.80969.10^{-1}$&$1.06450$&$2.04332.10^{-2}$&$1.02133$\\
$7.81250.10^{-3}$&$4.79738.10^{-1}$&$1.03181$&$1.01409.10^{-2}$&$1.01073$\\
$3.90625.10^{-3}$&$2.37234.10^{-1}$&$1.01580$&$5.05189.10^{-3}$&$1.00529$\\
&&&&\\
\hline
\end{tabular}
\caption{Experimental rates of convergence for $b,d\neq0$}
\label{TABLE_CASES_A_B}
\end{center}
\end{table}

\clearpage
A third example that we treat concerns the case when $a=-\frac{7}{30}$, $b=\frac{7}{15}$, $c=-\frac{2}{5}$ and $d=\frac{1}{2}$ which is discussed in \cite{Chen1998} for which the exact solution reads:
\[
(C)\left\{
\begin{split}
&  \eta(t,x)=\frac{3}{8}\mathrm{sech}^{2}\left(  \frac{1}{2}\sqrt{\frac{5}{7}%
}\left(  x-\frac{L}{2}-5\frac{\sqrt{2}}{6}t\right)  \right)  ,\\
&  u(t,x)=\frac{1}{2\sqrt{2}}\mathrm{sech}^{2}\left(  \frac{1}{2}\sqrt
{\frac{5}{7}}\left(  x-\frac{L}{2}-5\frac{\sqrt{2}}{6}t\right)  \right)  .
\end{split}
\right.
\]
The results are gathered
in Figure \ref{fig_4} and Table \ref{TABLE_CASES_C}.
\begin{figure}[h]
\begin{minipage}[t]{1\linewidth}
\begin{minipage}[b]{.5\linewidth}
\centering\epsfig{figure=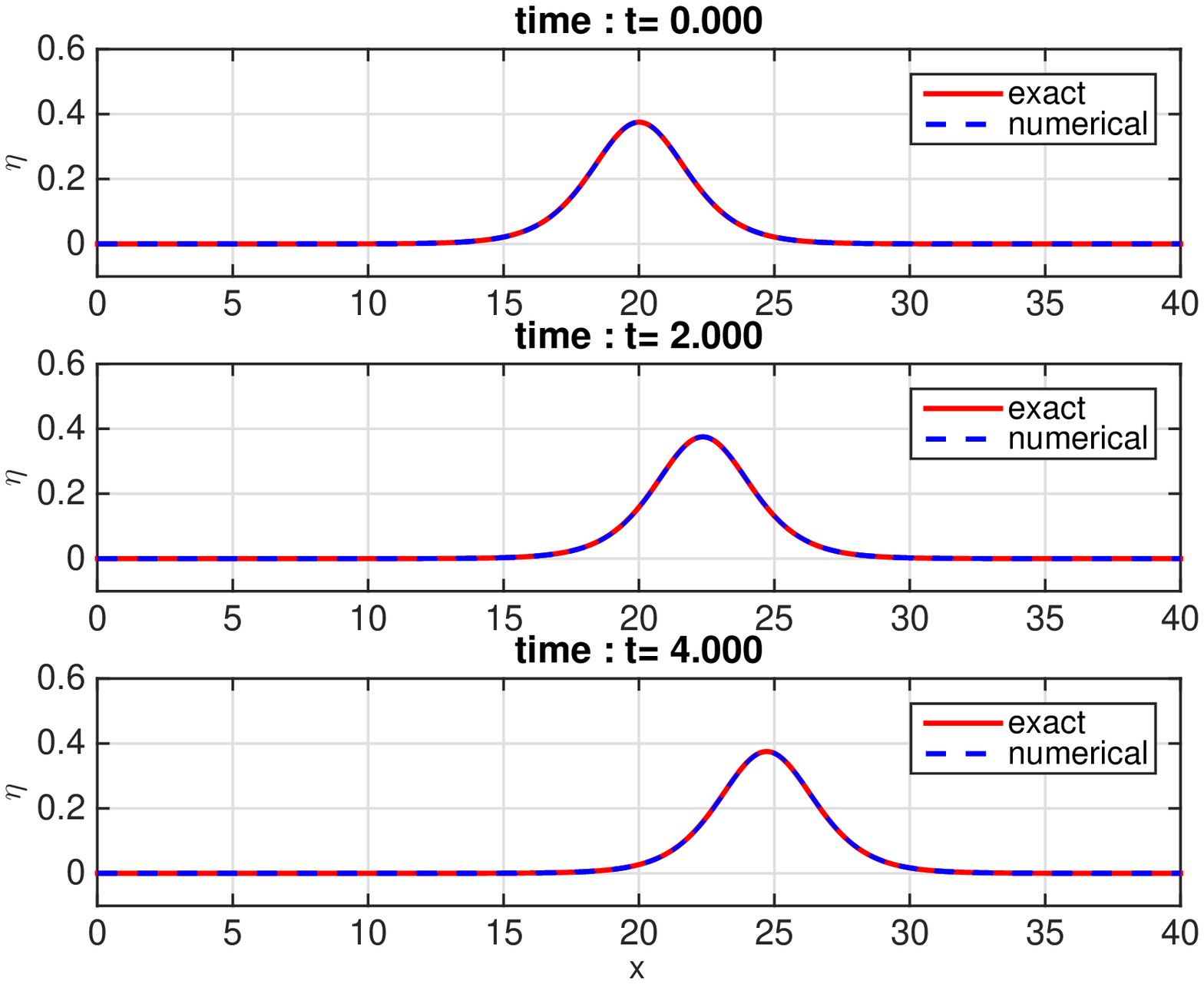,width=\linewidth}
\end{minipage}\hspace*{-0.2cm}
\begin{minipage}[b]{.5\linewidth}
\centering\epsfig{figure=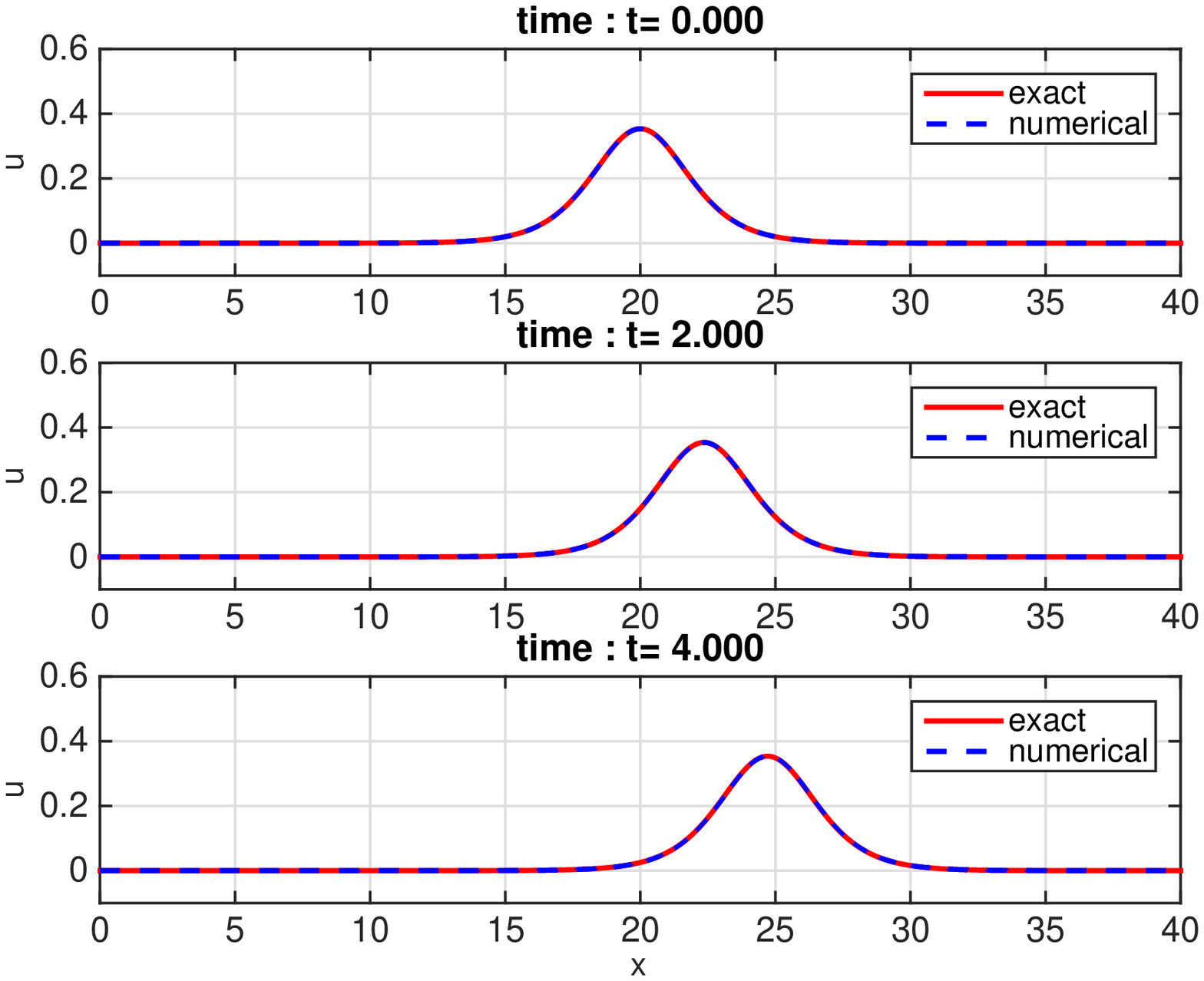,width=\linewidth}
\end{minipage}
\caption{Case $(C)$ where $a=-\frac{7}{30}$, $b=\frac{7}{15}$, $c=-\frac{2}{5}$, $d=\frac{1}{2}$ : results for $\eta$ (left) and $u$ (right) \label{fig_4}}
\end{minipage}
\end{figure}

\begin{table}[h]
\begin{center}
\hspace*{-0.7cm}
\begin{tabular}{|c|c|c|}
\hline
&\multicolumn{2}{c|}{Case $(C)$}\\
\cline{2-3}
$\Delta x$&&\\
&energy error&exp. rate\\
&&\\
\hline
&&\\
$6.25000.10^{-2}$&$2.27860.10^{-2}$&\\
$3.12500.10^{-2}$&$1.126019.10^{-2}$&$1.01692$\\
$1.56250.10^{-2}$&$5.612993.10^{-3}$&$1.00439$\\
$7.81250.10^{-3}$&$2.847910.10^{-3}$&$0.97887$\\
&&\\
\hline
\end{tabular}
\caption{Experimental rates of convergence for $a,b,c,d\neq0$}
\label{TABLE_CASES_C}
\end{center}
\end{table}

The fourth and last example that we present in this paragraph is the case when $a=0$, $b=\frac{1}{3}$, $c=-\frac{1}{3}$
and $d=\frac{1}{3}$, for which the exact solution writes:
\[
(D)\left\{
\begin{split}
&  \eta(t, x)=-1,\\
&  u(t, x)=(1-\frac{\rho}{3})C_{s}+C_{s}\rho\mathrm{sech}^{2}\left(
\frac{\sqrt{\rho}}{2}\left(  x-\frac{L}{2}-C_{s}t\right)  \right)  .
\end{split}
\right.
\]
We take $C_{s}=3$ and $\rho=2$ in order to perform our computations. The results are gathered in Figure \ref{fig_5} and Table \ref{TABLE_CASES_D}.
\begin{figure}[h]
\begin{minipage}[t]{1\linewidth}
\begin{minipage}[b]{.5\linewidth}
\centering\epsfig{figure=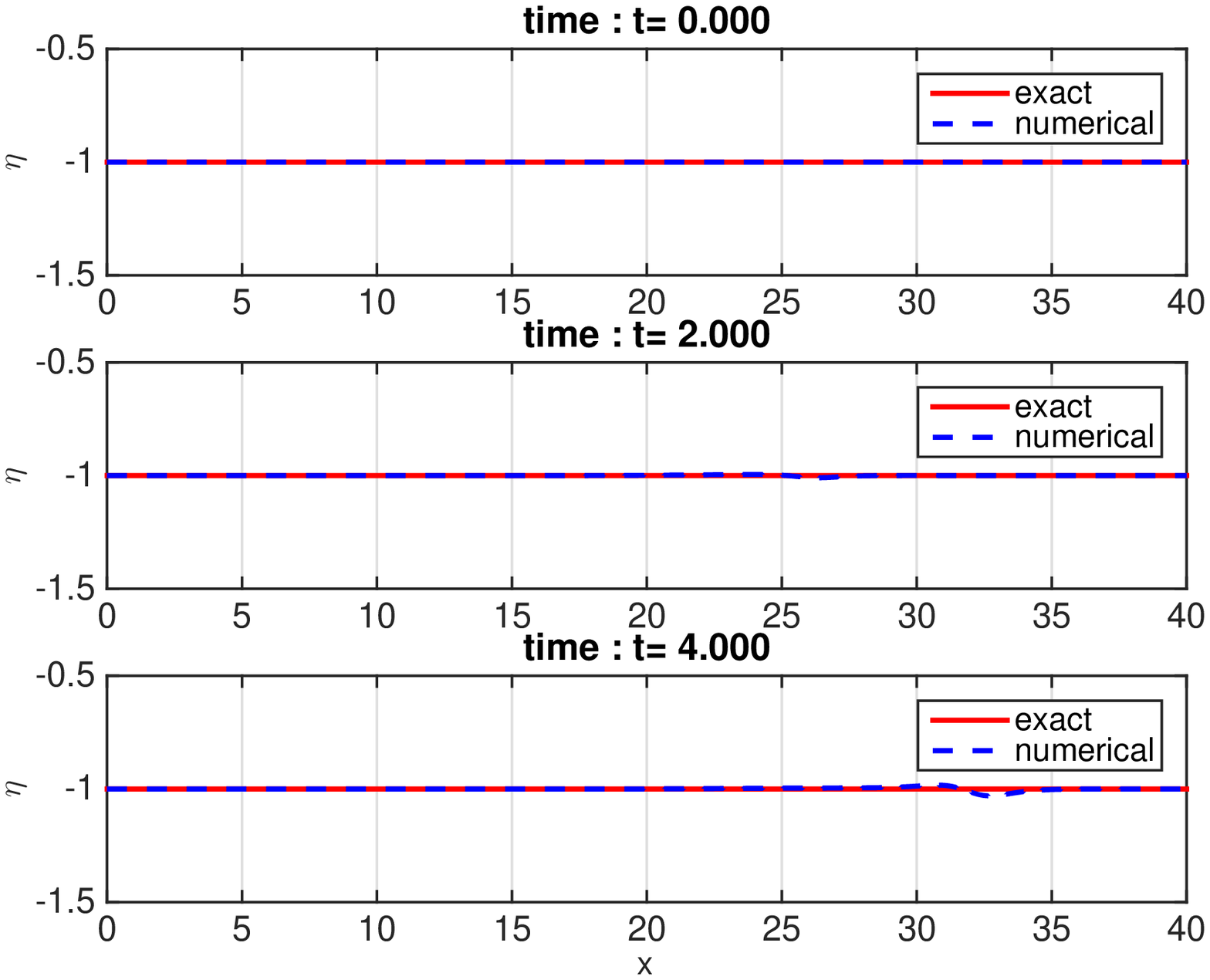,width=\linewidth}
\end{minipage}\hspace*{-0.2cm}
\begin{minipage}[b]{.5\linewidth}
\centering\epsfig{figure=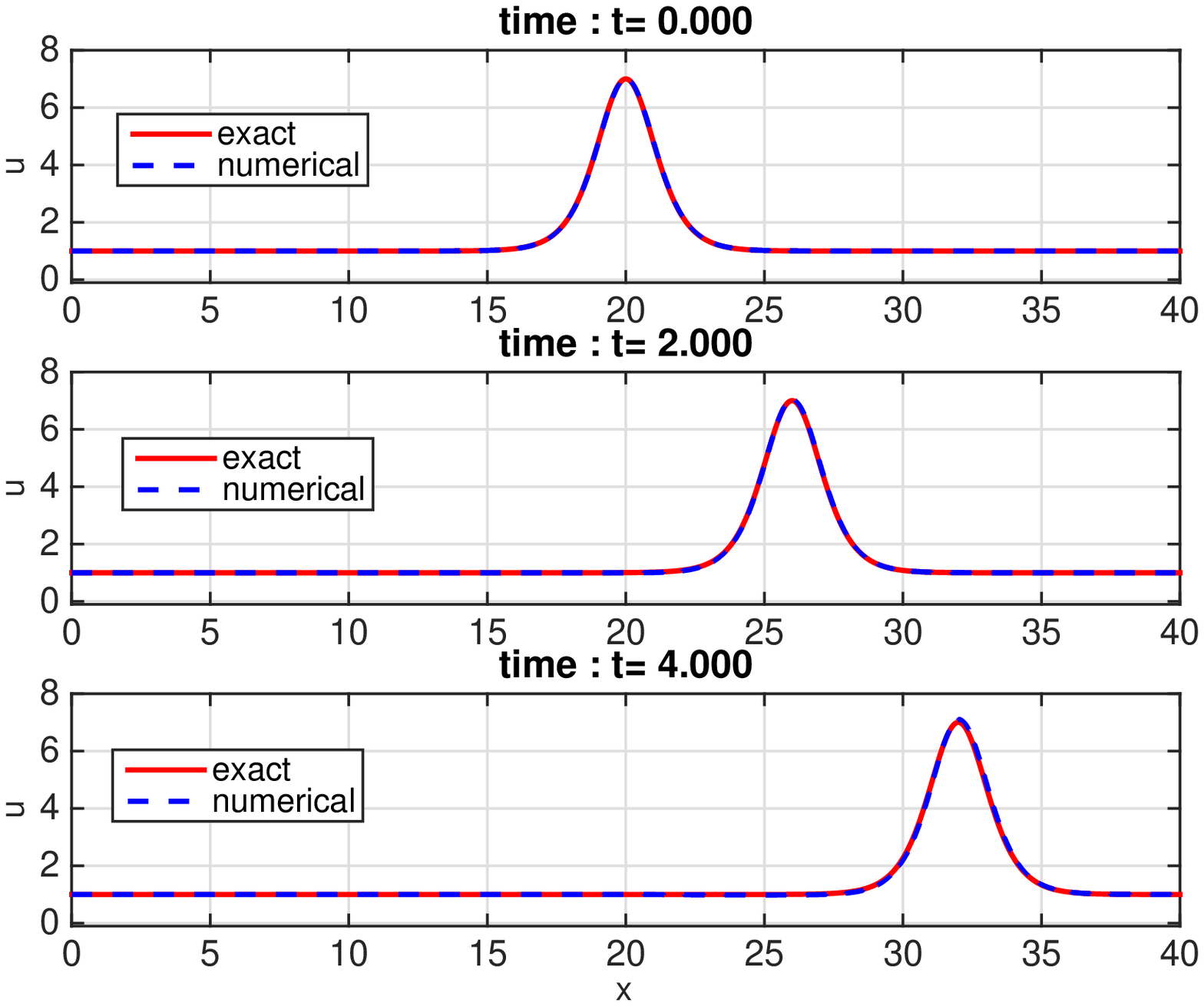,width=\linewidth}
\end{minipage}\vspace*{-0.2cm}
\caption{Case $(D)$ where $a=0$, $b=\frac{1}{3}$, $c=-\frac{1}{3}$, $d=\frac{1}{3}$ : results for $\eta$ (left) and $u$ (right) \label{fig_5}}
\end{minipage}
\end{figure}

\begin{table}[h]
\begin{center}
\hspace*{-0.7cm}
\begin{tabular}{|c|c|c|}
\hline
&\multicolumn{2}{c|}{Case $(D)$}\\
\cline{2-3}
$\Delta x$&&\\
&energy error&exp. rate\\
&&\\
\hline
&&\\
$6.25000.10^{-2}$&$6.39353.10^{-1}$&\\
$3.12500.10^{-2}$&$3.09159.10^{-1}$&$1.04826$\\
$1.56250.10^{-2}$&$1.52031.10^{-1}$&$1.02399$\\
$7.81250.10^{-3}$&$7.53884.10^{-2}$&$1.01195$\\
$3.90625.10^{-3}$&$3.75388.10^{-2}$&$1.00596$\\
&&\\
\hline
\end{tabular}
\caption{Experimental rates of convergence for $b,c,d\neq0$}
\label{TABLE_CASES_D}\vspace*{-0.5cm}
\end{center}
\end{table}

\begin{remark}
In these four examples, our numerical schemes do no contain artificial viscosity i.e. $\tau_1=0$ and $\tau_2=0$. As we explained in the introduction, the parameters $b,d$ enable us to control and stabilize the scheme as the results of Figures \ref{fig_1}-\ref{fig_5} show.\\
The schemes used to perform these experiments are $O(\Delta t+\Delta x^2)$ accurate so, if we take $\Delta t=\Delta x^2$, we should be able to observe a second order convergence rate. The results when performing the above experiments with $\Delta t=\Delta x^2$ are gathered in Table \ref{TABLE_SECOND_ORDER} and confirm our intuition.

\begin{table}[h]
\begin{center}
\hspace*{-0.7cm}
\begin{tabular}{|c|c|c|c|c|}
\hline
&\multicolumn{2}{c|}{Case $(A)$}&\multicolumn{2}{c|}{Case $(C)$}\\
\cline{2-5}
$\Delta x$&&&&\\
&energy error&exp. rate&energy error&exp. rate\\
&&&&\\
\hline
&&&&\\
$2.5000.10^{-1}$&$3.29137.10^{\ 0}$&&$2.52768.10^{-2}$&\\
$1.2500.10^{-1}$&$7.75742.10^{-1}$&$2.08504$&$6.08893.10^{-3}$&$2.05355$\\
$6.2500.10^{-2}$&$1.90828.10^{-1}$&$2.02330$&$1.50692.10^{-3}$&$2.01459$\\
$3.1250.10^{-2}$&$4.75112.10^{-2}$&$2.00594$&$3.81640.10^{-4}$&$1.98131$\\
&&&&\\
\hline
\end{tabular}
\caption{Experimental rates of convergence with $\Delta t=\Delta x^2$}
\label{TABLE_SECOND_ORDER}
\end{center}
\end{table}

\end{remark}
\subsubsection*{Behavior of the numerical scheme for traveling wave solutions} All the test cases studied previously are in fact traveling wave solutions and it is worth pushing forward our analysis for such particular solutions with the numerical schemes. First of all, since $bd\geq0$, the Rusanov coefficients $\tau_1$ and $\tau_2$ are taken equal to 0, which restricts the numerical diffusion of the scheme and provides a relatively good numerical solution in long temporal intervals, as seen in Figures \ref{cas_B_4_fig} and \ref{casC_3_bis_fig}. For test case $(B)$, we have chosen $\Delta x=0.01$ and a space domain $[0,L]=[0,122]$, whereas for test case $(C)$, we have chosen $\Delta x=0.05$ and $[0,L]=[0,400]$. 
In both simulations, the changes in amplitude are limited : the relative error on the maximum amplitude of $u$ is equal to $1.8732\%$ at $t=34$ for the case $(B)$ and  $1.4632\%$ at $t=100$ for the case $(C)$.

\begin{figure}[h!]
\begin{minipage}[t]{1\linewidth}
\begin{minipage}[b]{.5\linewidth}
\centering\epsfig{figure=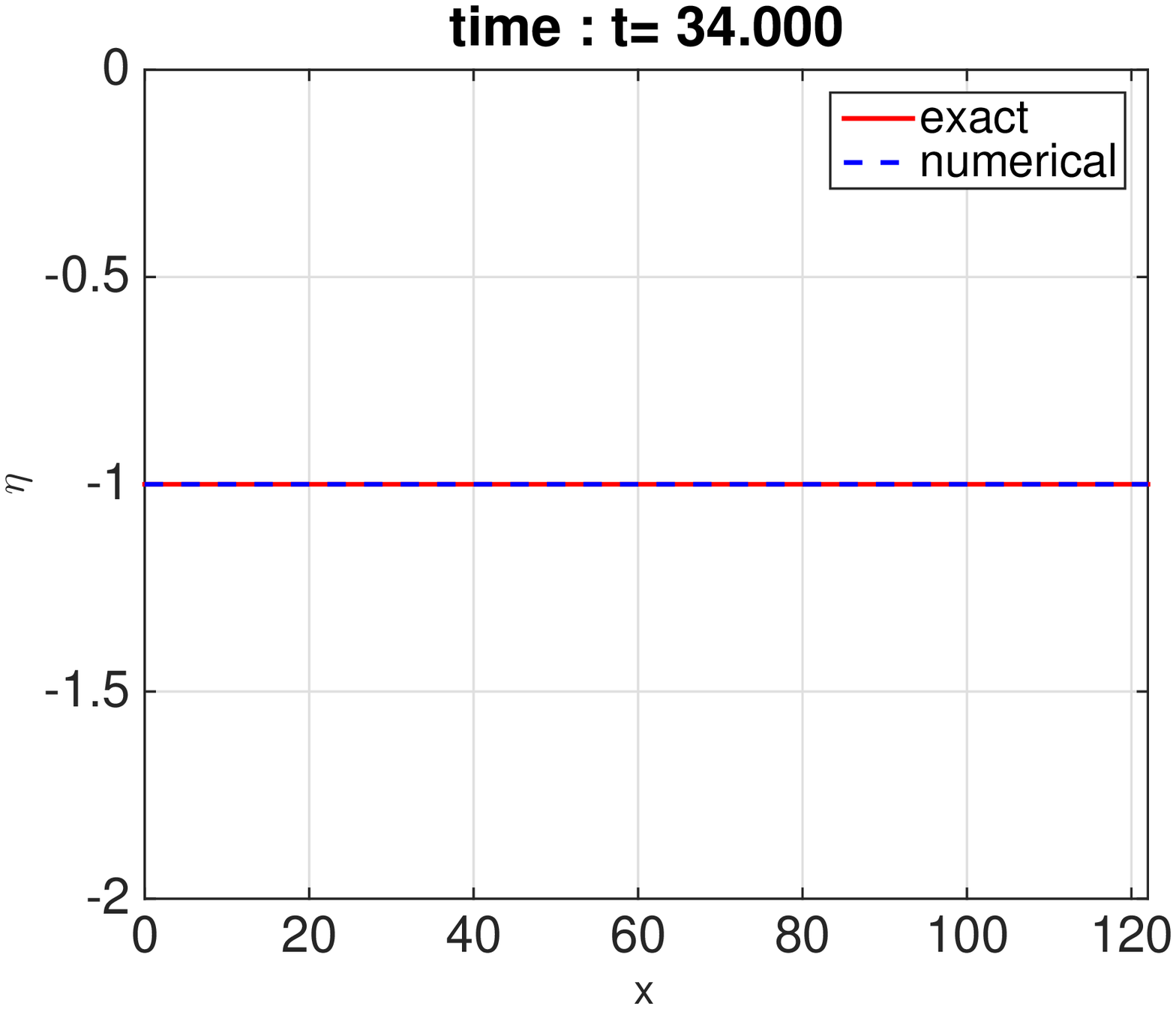,width=0.76\linewidth}
\end{minipage}\hspace*{-0.2cm}
\begin{minipage}[b]{.5\linewidth}
\centering\epsfig{figure=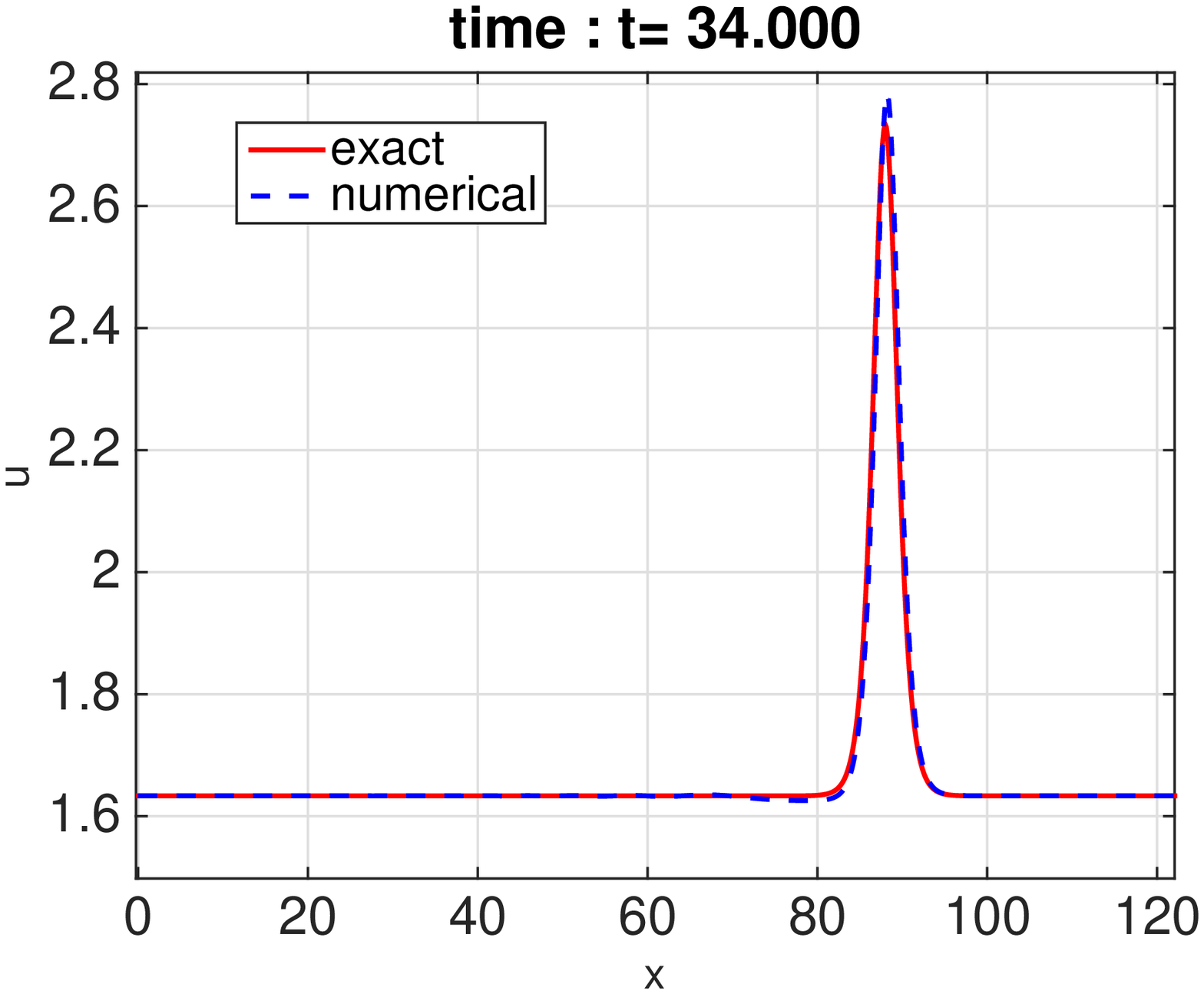,width=0.8\linewidth}
\end{minipage}
\caption{Long time behavior for case $(B)$ with $a=0$, $b=\frac{1}{6},$ $c=0$, $d=\frac{1}{6}$: results for $\eta$ (left) and $u$ (right) \label{cas_B_4_fig}}
\end{minipage}
\end{figure}

\begin{figure}[h!]
\begin{minipage}[t]{1\linewidth}
\begin{minipage}[b]{.5\linewidth}
\centering\epsfig{figure=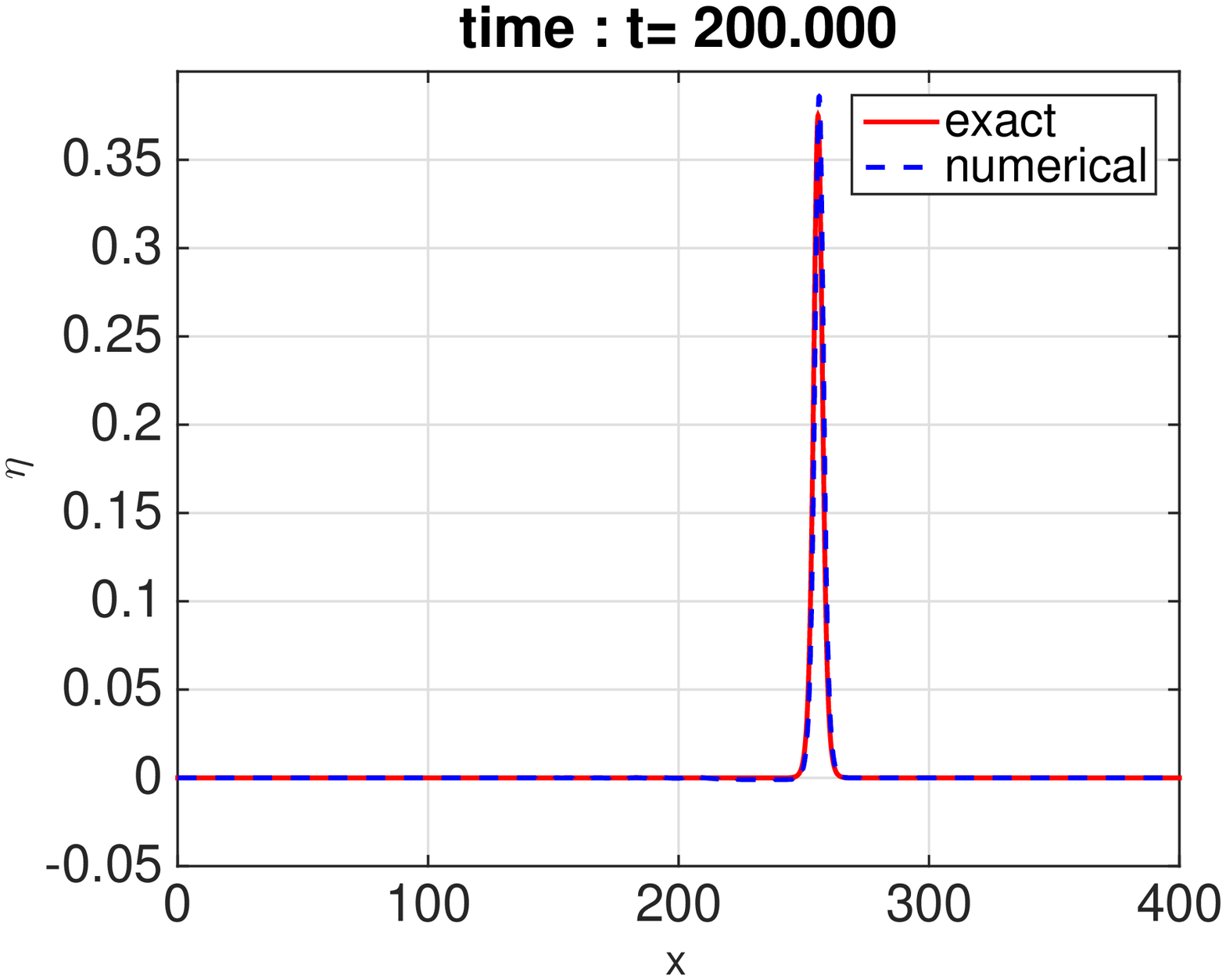,width=0.8\linewidth}
\end{minipage}\hspace*{-0.2cm}
\begin{minipage}[b]{.5\linewidth}
\centering\epsfig{figure=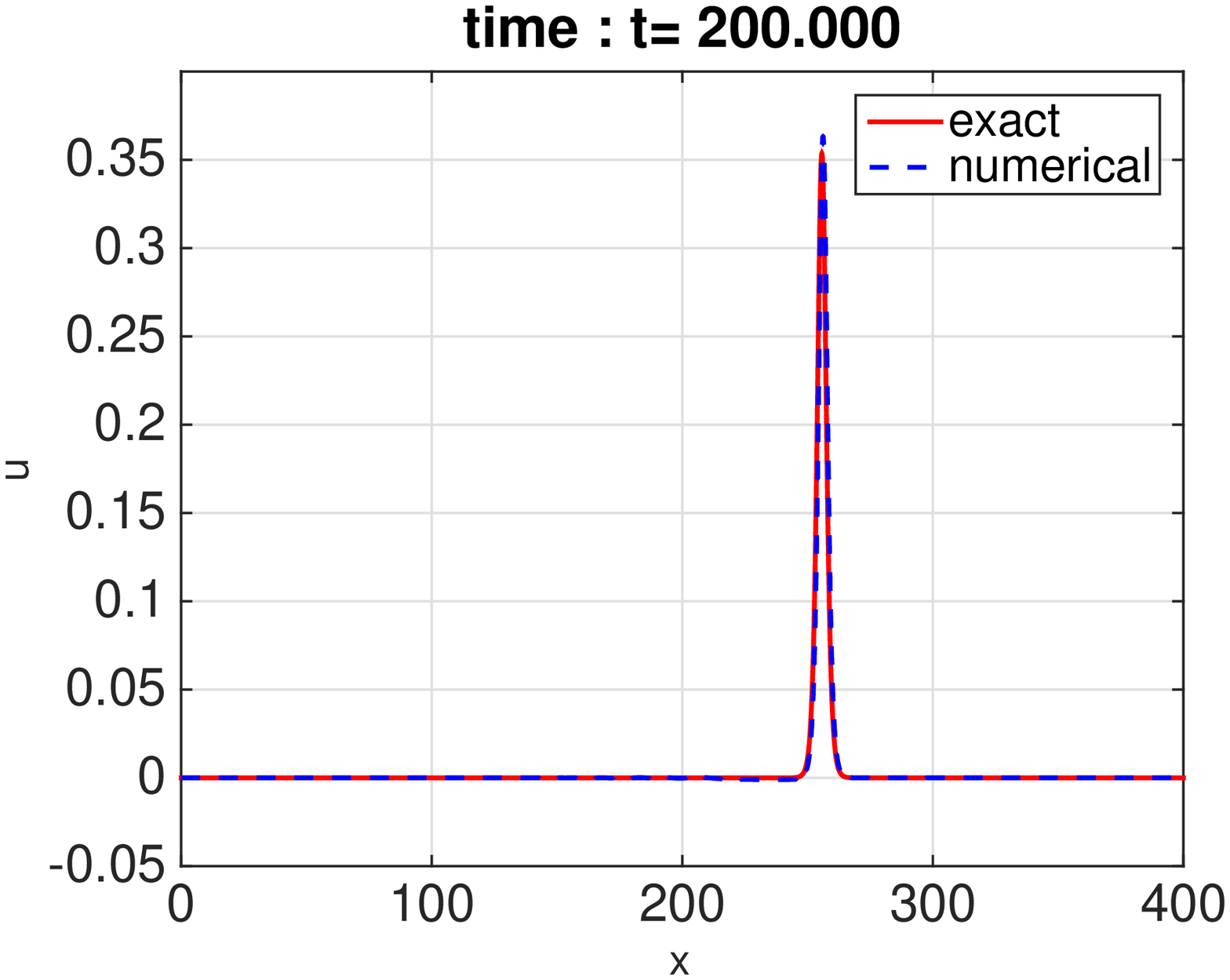,width=0.8\linewidth}
\end{minipage}
\caption{Long time behavior for case $(C)$ with $a=-\frac{1}{6}$, $b=c=0$, $d=\frac{1}{2}$: results for $\eta$ (left) and $u$ (right) \label{casC_3_bis_fig}}
\end{minipage}
\end{figure}

The oscillatory dispersive tails which may be visible in a zoom after the solitary wave are a numerical artifact. Thinner is the space mesh grid $\Delta x$, smaller is the amplitude of these oscillations. For instance, the amplitude of this tail is divided by approximately 2.15 when $\Delta x$ is halved, see Figure \ref{dispersive_tail_temps_long_fig}.

\begin{figure}[h!]
\begin{minipage}[t]{1\linewidth}
\begin{minipage}[b]{.5\linewidth}
\centering\epsfig{figure=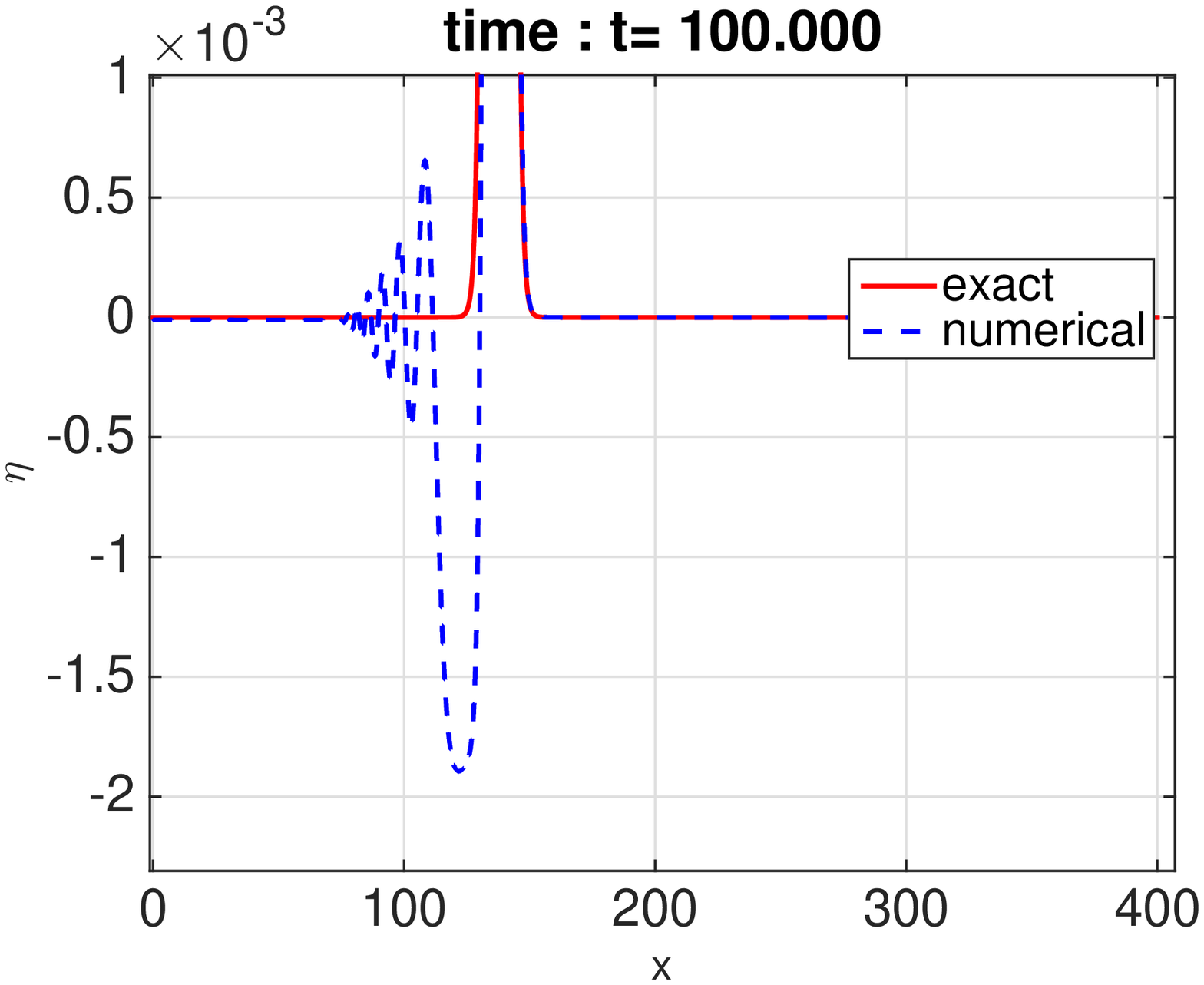,width=0.8\linewidth}
\end{minipage}\hspace*{-0.2cm}
\begin{minipage}[b]{.5\linewidth}
\centering\epsfig{figure=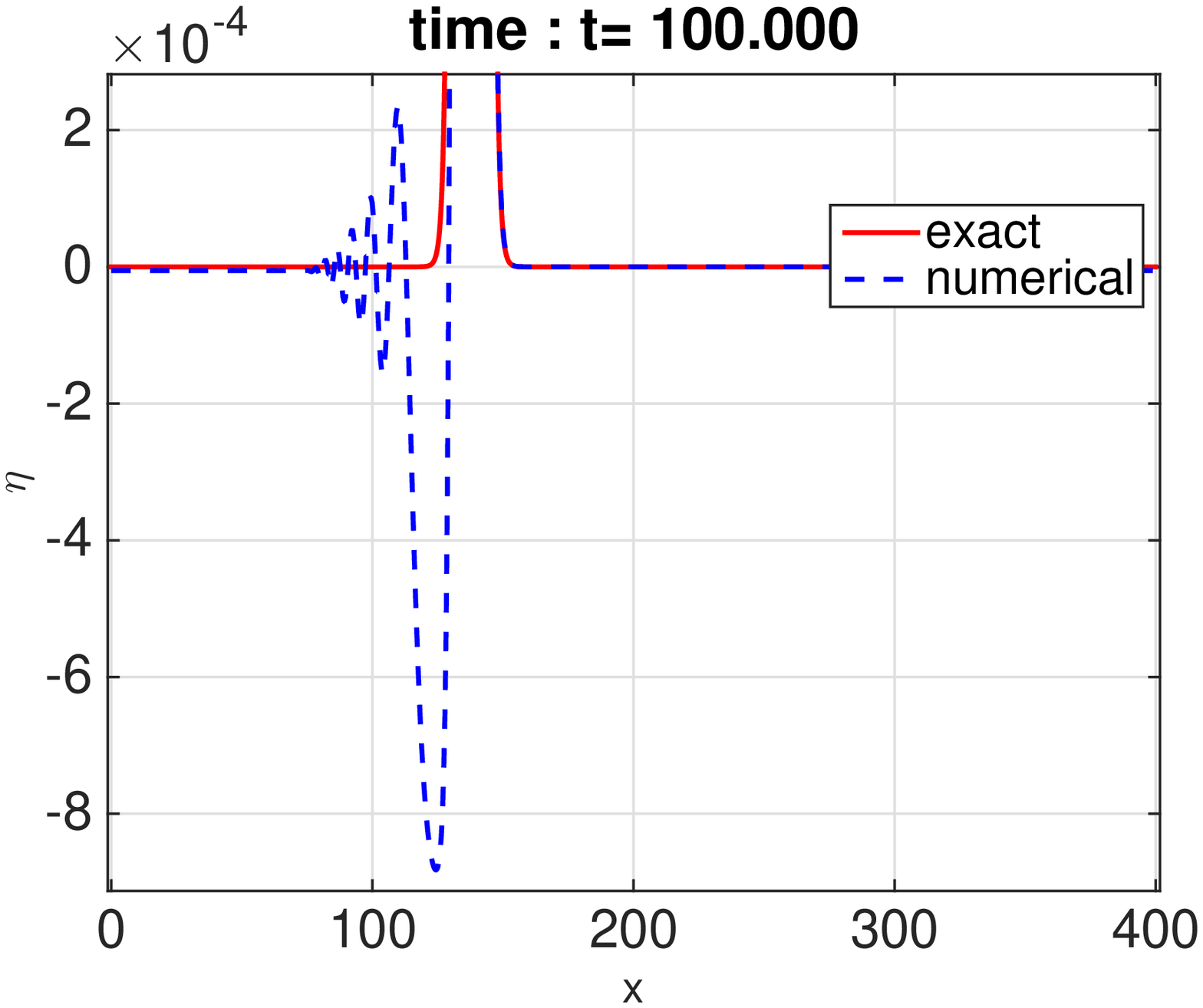,width=0.8\linewidth}
\end{minipage}
\caption{Dispersive tails for case $(C)$ due to a numerical artifact : results for $\Delta x=0.1$ (left) and $\Delta x=0.05$ (right) \label{dispersive_tail_temps_long_fig}}
\end{minipage}
\end{figure}

\subsection{Case $bd=0$.}\ \ \ We perform the same kind of numerical simulations : on the one hand, the validation of the convergence rates and on the other hand the behavior of the numerical schemes.
\subsubsection*{Numerical convergence rates}

In the following two examples, we have chosen $b=0$ while $d>0$. First, we consider the case when $a=b=c=0$, $d=\frac{1}{6}$. The exact
solution obtained in \cite{Chen1998} writes:
\[
(E)\left\{
\begin{split}
&  \eta(t, x)=-1,\\
&  u(t, x)=\left(  1-\frac{\rho}{6}\right)  C_{s}+\frac{C_{s}\rho}%
{2}\mathrm{sech}^{2}\left(  \frac{\sqrt{\rho}}{2}\left(  x-\frac{L}{2}%
-C_{s}t\right)  \right)  .
\end{split}
\right.
\]
Our computations are done with $C_{s}=1$ and $\rho=2$. We represent the numerical and the exact solutions in Figure \ref{fig_6}.
\begin{figure}[h]
\begin{minipage}[t]{1\linewidth}
\begin{minipage}[b]{.5\linewidth}
\centering\epsfig{figure=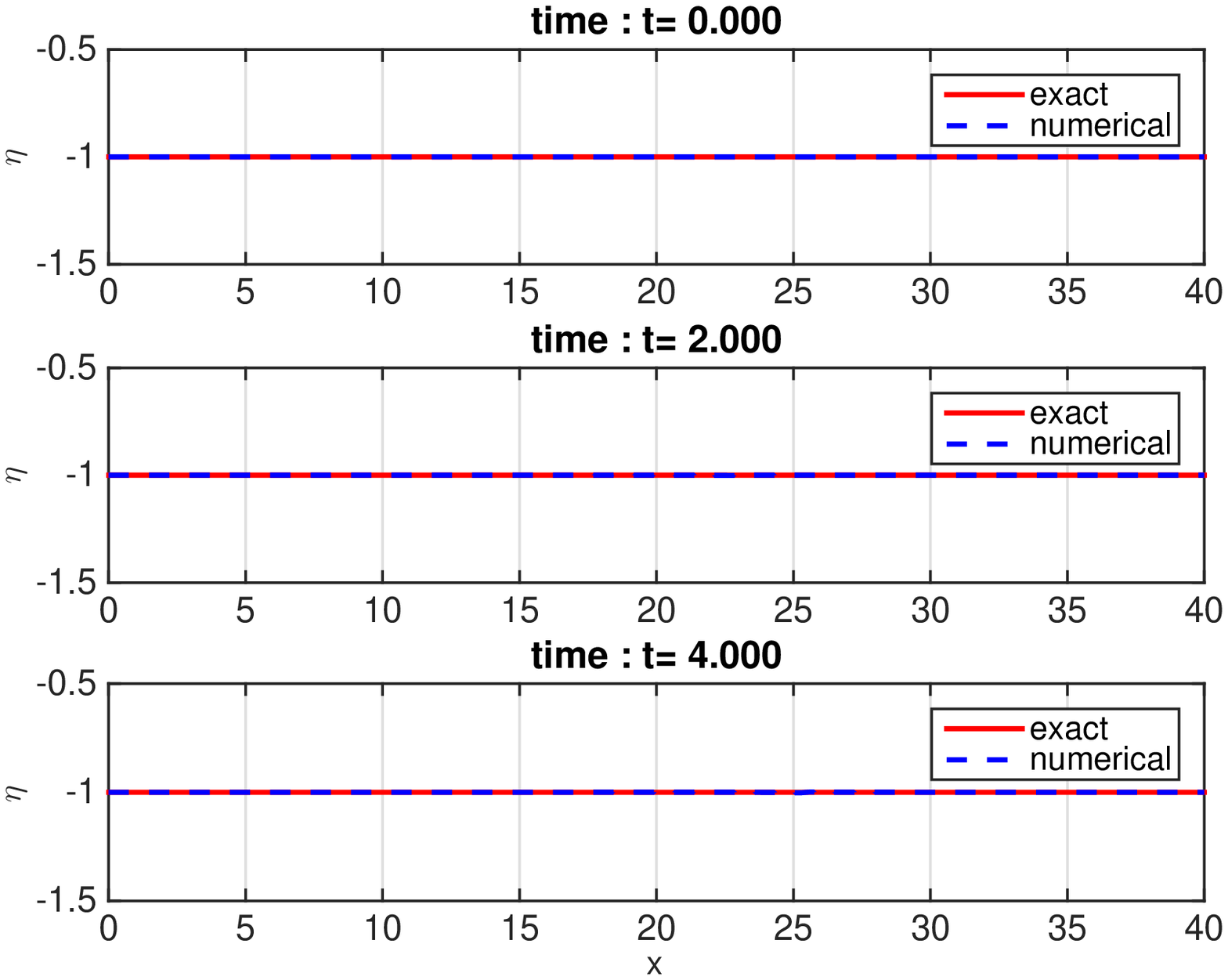,width=\linewidth}
\end{minipage}\hspace*{-0.2cm}
\begin{minipage}[b]{.5\linewidth}
\centering\epsfig{figure=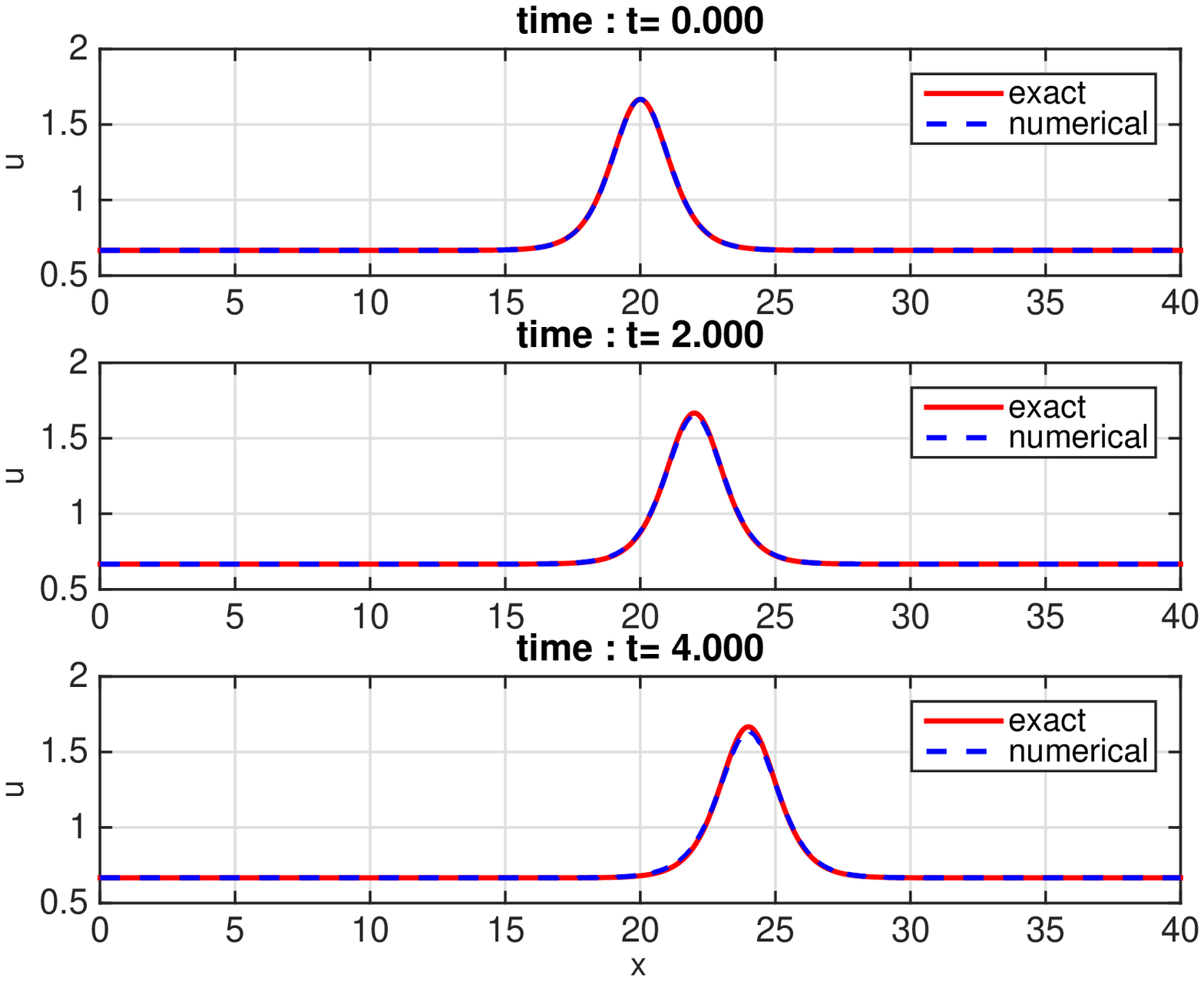,width=\linewidth}
\end{minipage}
\caption{Case $(E)$ where $a=b=c=0$, $d=\frac{1}{6}$ : results for $\eta$ (left) and $u$ (right) \label{fig_6}}
\end{minipage}
\end{figure}

In the second example, we treat the case
$a=-\frac{1}{6}$, $b=c=0$ and $d=\frac{1}{2}$, for which the exact
solution obtained in \cite{Chen1998} is:
\[
(F)\left\{
\begin{split}
&  \eta(t, x)=-\frac{7}{4}\mathrm{sech}^{2}\left(  \frac{\sqrt{7}}{2}\left(
x-\frac{L}{2}+\frac{1}{\sqrt{15}}t\right)  \right)  ,\\
&  u(t, x)=-\frac{7}{2}\sqrt{\frac{3}{5}}\mathrm{sech}^{2}\left(  \frac
{\sqrt{7}}{2}\left(  x-\frac{L}{2}+\frac{1}{\sqrt{15}}t\right)  \right)  .
\end{split}
\right.
\]
We obtain Figure \ref{fig_7}. 
\begin{figure}[h]
\begin{minipage}[t]{1\linewidth}
\begin{minipage}[b]{.5\linewidth}
\centering\epsfig{figure=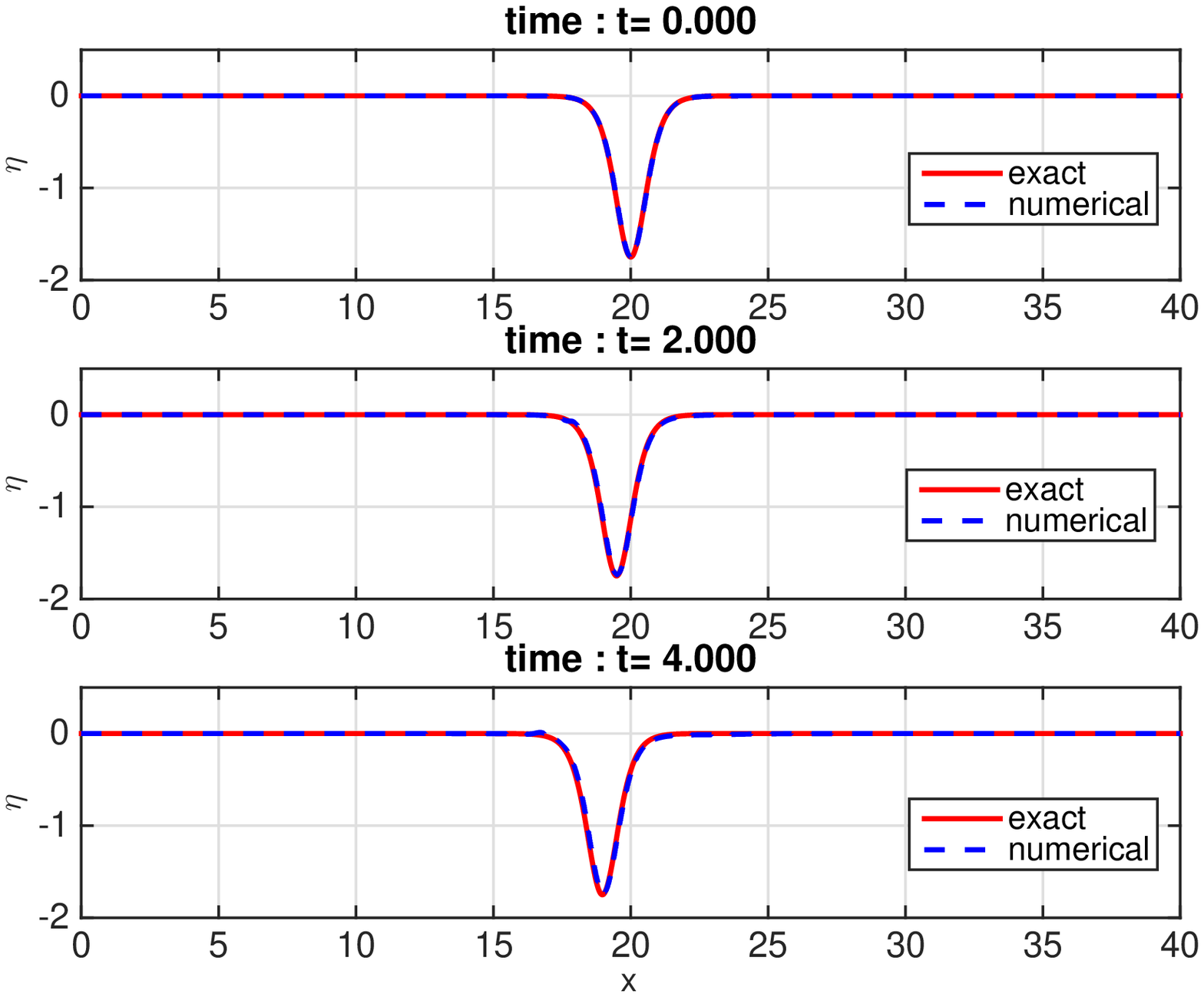,width=\linewidth}
\end{minipage}\hspace*{-0.2cm}
\begin{minipage}[b]{.5\linewidth}
\centering\epsfig{figure=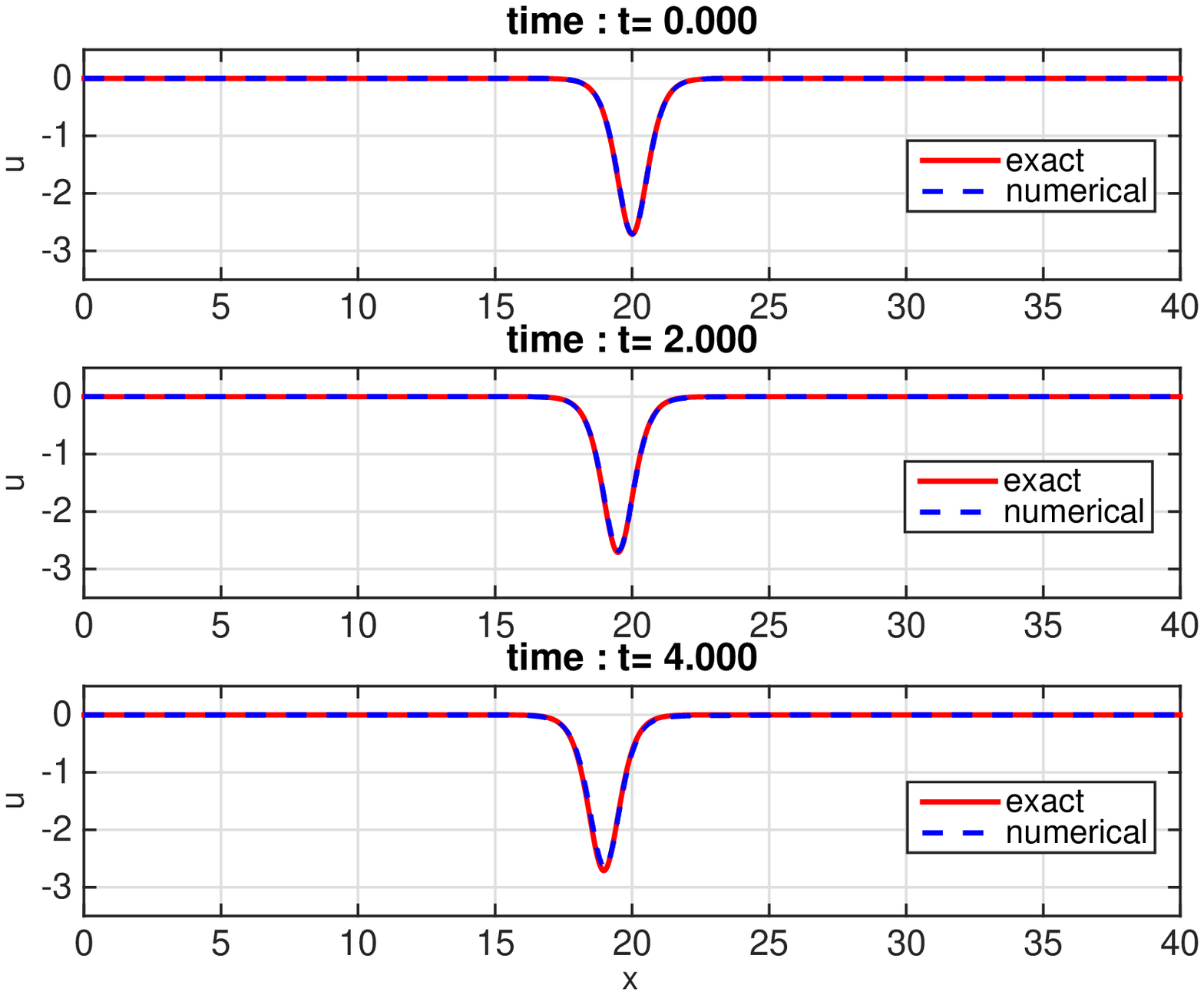,width=\linewidth}
\end{minipage}
\caption{Case $(F)$ where $a=-\frac{1}{6}$, $b=c=0$, $d=\frac{1}{2}$ : results for $\eta$ (left) and $u$ (right) \label{fig_7}}
\end{minipage}
\end{figure}

The experimental rates of convergence of the two previous cases with $b=0$ are gathered in Table \ref{TABLE_CASES_E_F}
\begin{table}[t]
\begin{center}
\begin{tabular}{|c|c|c|c|c|}
\hline
&\multicolumn{2}{c|}{Case $(E)$}&\multicolumn{2}{c|}{Case $(F)$}\\
\cline{2-5}
$\Delta x$&&&&\\
&energy error &exp. rate&energy error &exp. rate\\
&&&&\\
\hline
&&&&\\
$6.25000.10^{-2}$&$5.94214.10^{-2}$&&$3.62176.10^{-1}$&\\
$3.12500.10^{-2}$&$3.01052.10^{-2}$&$0.98097$&$1.92823.10^{-1}$&$0.93164$\\
$1.56250.10^{-2}$&$1.51573.10^{-2}$&$0.99000$&$1.00366.10^{-1}$&$0.94780$\\
$7.81250.10^{-3}$&$7.60581.10^{-3}$&$0.99483$&$5.16267.10^{-2}$&$0.95784$\\
$3.90625.10^{-3}$&$3.80985.10^{-3}$&$0.99737$&$2.63453.10^{-2}$&$0.96715$\\
&&&&\\
\hline
\end{tabular}
\caption{Experimental rates of convergence when $b=0$ and $d>0$}
\label{TABLE_CASES_E_F}
\end{center}
\end{table}

Observe that the first order convergence is recovered which, of course, is in accordance with the theoretical results.
\newpage
\subsubsection*{Behavior of the numerical scheme for traveling wave solutions} Because of the numerical diffusion, the scheme creates light shifts of the position and of the amplitude of the traveling wave solution in long time. We detail below these aspects for the test case $(F)$, where the numerical solution is computed with $\Delta x=0.005$ up to the final time $T=20$. Quantitative results concerning the relative errors in position and amplitude are gathered in Figure \ref{Cas_F_long_fig} (respectively in Table \ref{TABLE_LONG_BEHAVIOR_F}).
 \begin{figure}[h]
\begin{minipage}[t]{1\linewidth}
\begin{minipage}[b]{.5\linewidth}
\centering\epsfig{figure=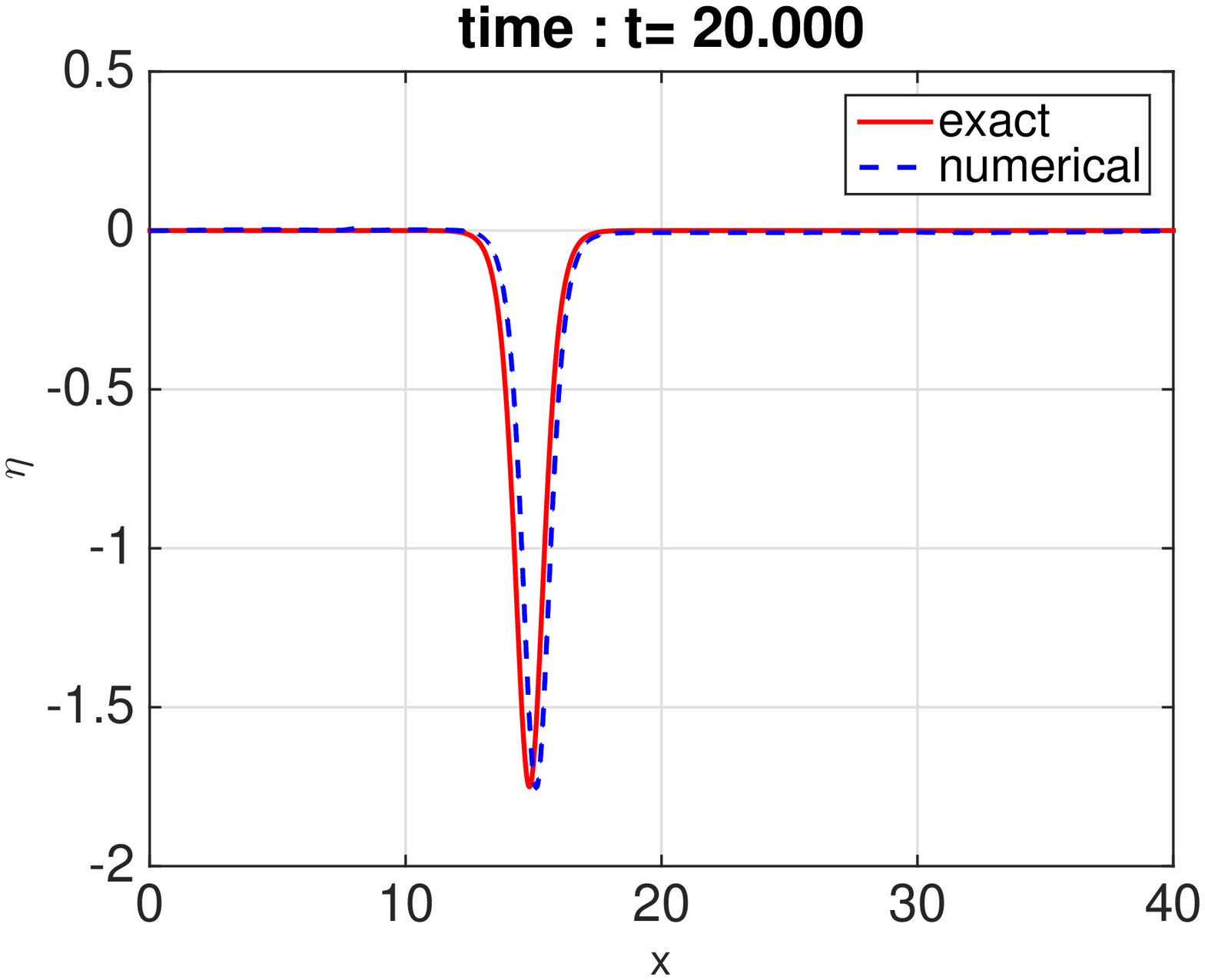,width=0.8\linewidth}
\end{minipage}\hspace*{-0.2cm}
\begin{minipage}[b]{.5\linewidth}
\centering\epsfig{figure=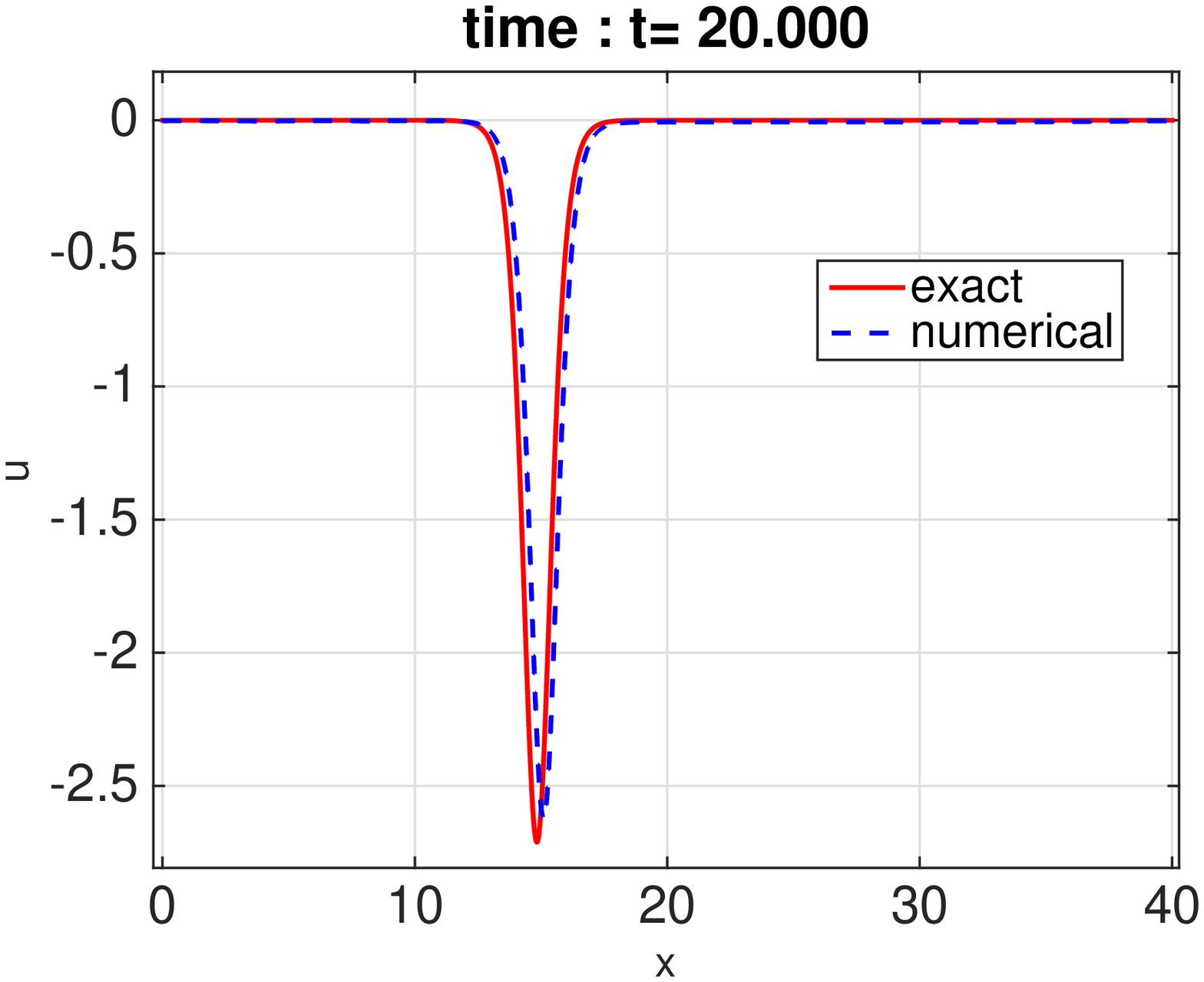,width=0.8\linewidth}
\end{minipage}\vspace*{-0.5cm}
\caption{Long time behavior for case $(F)$ with $a=-\frac{1}{6}$, $b=0$, $c=0$, $d=\frac{1}{2}$ : results for $\eta$ (left) and $u$ (right) \label{Cas_F_long_fig}}
\end{minipage}
\end{figure}

\begin{table}[h]
\begin{center}
\hspace*{-0.7cm}
\begin{tabular}{|c|c|c|c|c|c|c|}
\hline
&\multicolumn{3}{c|}{Position}&\multicolumn{3}{c|}{Amplitude}\\
\hline
&Numerical&Exact&Relative error&Numerical&Exact&Relative error\\
\hline
$\eta$&15.105&14.840&1.7857\%&1.7544&1.7500&0.2501\%\\
\hline
$u$&15.100&14.840&1.7520\%&2.6227&2.7111&3.2594\%\\
\hline
\end{tabular}
\caption{Relative errors on position and amplitude of the traveling wave in long time behavior $(T=20)$ for case $(F)$}
\label{TABLE_LONG_BEHAVIOR_F}
\end{center}
\end{table}

\newpage

\section{Traveling waves collision\label{travel_waves_collision}}

Recently, in \cite{BonaChen2016}, the authors simulate the collision of two
traveling waves moving in opposite directions in $[-L,L]$ in the BBM-BBM case
($a=0$, $b=\frac{1}{6}$, $c=0$, $d=\frac{1}{6}$). Motivated by their results,
we simulated the same phenomenon but for different values of the $abcd$ parameters. \\
Two behaviors are observable: 
\begin{itemize}
\item the collision leads to phase shifts and visible dispersive tails which follow each solitary wave,
\item the collision suggests a possible blow-up of the $L^\infty$-norm in finite time either on the density $\eta$ or on the derivative of $u$.
 \end{itemize}
Our numerical study is restricted to one type of traveling waves collision : the \emph{head-on} collision when the solitary waves travel in opposite directions and collide. We do not study the \emph{overtaking} collision (when the solitary waves propagate in the same direction) and we refer the reader for example to \cite{AntonopoulosDougalis2012} for a review of such collisions in the case $a=b=c=0, d=\frac{1}{3}$.

\subsection{Finite time blow-up.}
First, we used our numerical results and performed the same experiment
described in \cite{BonaChen2016} for ($a=0$, $b=\frac{1}{6}$, $c=0$, $d=\frac{1}{6}$). The initial condition is fixed to
\[
\left\{
\begin{split}
&  \eta(t,x)=\eta_{+}(t,x)+\eta_{-}(t,x),\\
&  u(t,x)=u_{+}(t,x)+u_{-}(t,x),
\end{split}
\right.
\]
with
\[(G)
\left\{
\begin{split}
&  \eta_{\pm}(t,x)=\frac{15}{2}\mathrm{sech}^{2}\left(  \frac{3}{\sqrt{10}%
}\left(  x-x_{\pm}\pm\frac{5}{2}t\right)  \right)  -\frac{45}{4}%
\mathrm{sech}^{4}\left(  \frac{3}{\sqrt{10}}\left(  x-x_{\pm}\pm\frac{5}%
{2}t\right)  \right)  ,\\
&  u_{\pm}(t,x)=\mp\frac{15}{2}\mathrm{sech}^{2}\left(  \frac{3}{\sqrt{10}}\left(
x-x_{\pm}\pm\frac{5}{2}t\right)  \right)  ,
\end{split}
\right.
\]
where $x_{\pm}=\pm\frac{L}{2}$.\newline The space domain is fixed at
$[-14,14]$ and initially, the traveling-waves are centered in $x_{+}=7$ and
$x_{-}=-7$. We choose the same space size and time step as in
\cite{BonaChen2016}, namely $\Delta x=0.02$ and $\Delta t=0.0001$. The
simulation suggests that a blow-up occurs while the explosion time appears to be
around $t=4.5$ as shown in Figure \ref{fig_9}, result which is very close to
the one obtained in \cite{BonaChen2016}. 
\begin{figure}[h]
\begin{minipage}[t]{1\linewidth}
\begin{minipage}[b]{.5\linewidth}
\centering\epsfig{figure=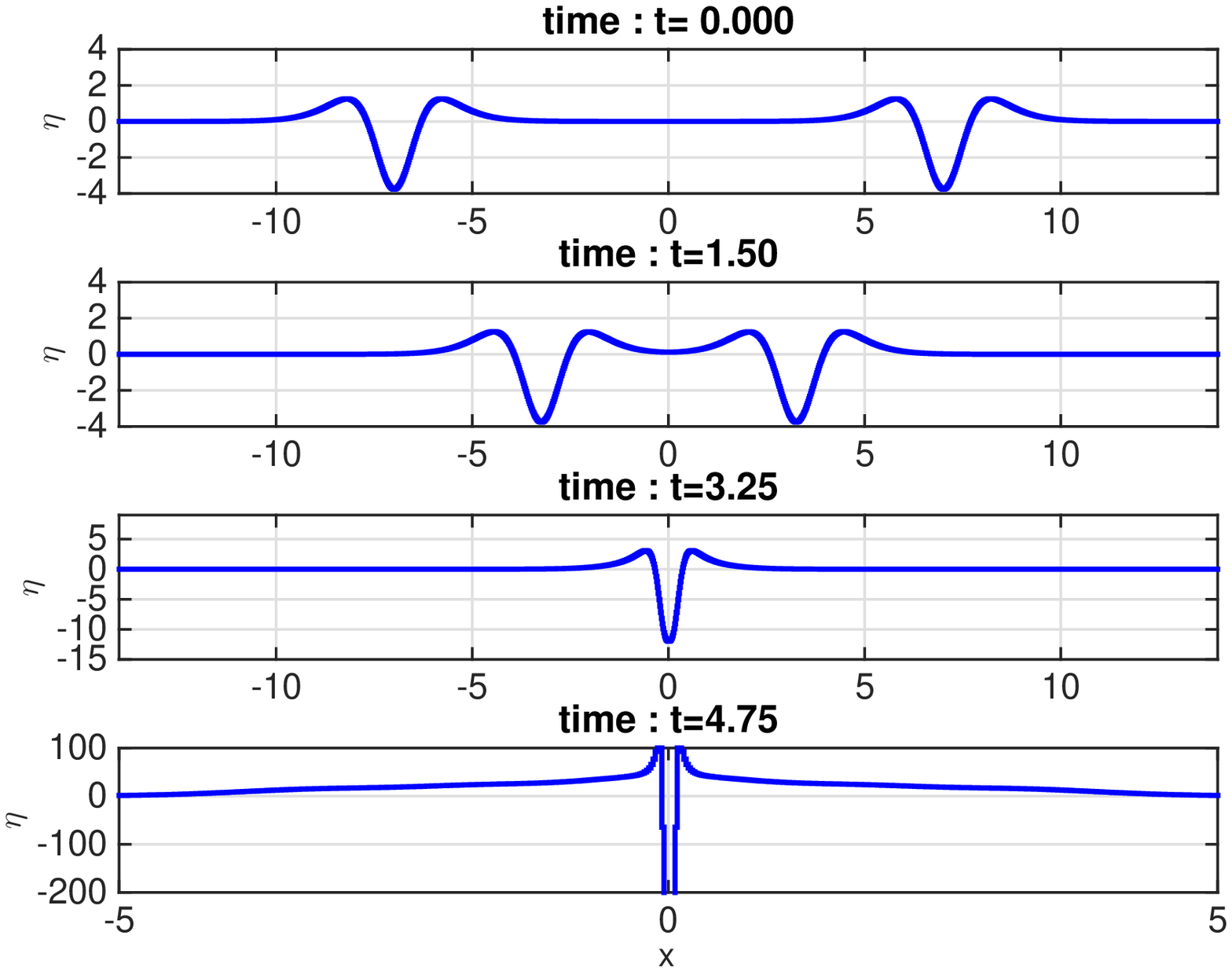,width=\linewidth}
\end{minipage}\hspace*{-0.2cm}
\begin{minipage}[b]{.5\linewidth}
\centering\epsfig{figure=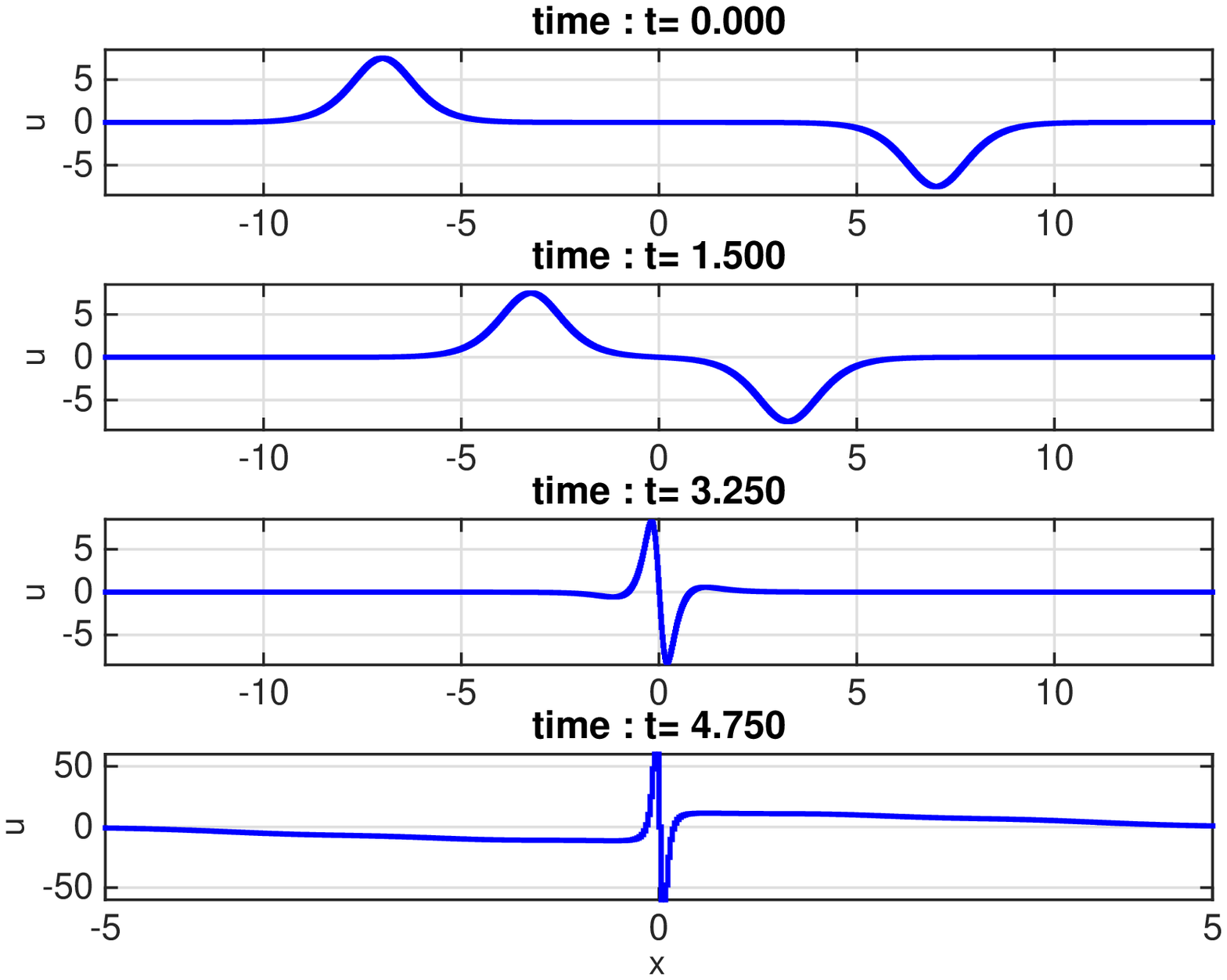,width=\linewidth}
\end{minipage}\vspace*{-0.4cm}
\caption{Explosion for case $(G)$ with $a=0$, $b=\frac{1}{6}$, $c=0$, $d=\frac{1}{6}$ : results for $\eta$ (left) and $u$ (right) \label{fig_9}}
\end{minipage}
\end{figure}

In a second time, we have observed that for the case $a=0, b=\frac{3}{5}, c=0, d=0$, solutions of a similar structure as those above are available. The experiment is performed with $\Delta t=0.001$, $\Delta x=0.01$ and taking the initial value
\[
\left\{
\begin{split}
&  \eta(0,x)=\eta_{+}(0,x)+\eta_{-}(0,x),\\
&  u(0,x)=u_{+}(0,x)+u_{-}(0,x),
\end{split}
\right.
\]
with\vspace*{-0.4cm}
\[(H)
\left\{
\begin{split}
&  \eta_{\pm}(t,x)=\frac{15}{2}\mathrm{sech}^{2}\left(  \frac{1}{2}\left(  x-x_{\pm}\mp\frac{\sqrt{10}}{2}t\right)  \right)-\frac{45}{4}\mathrm{sech}^{4}\left(  \frac{1}{2}\left(  x-x_{\pm}\mp\frac{\sqrt{10}}{2}t\right)  \right)  ,\\
&  u_{\pm}(t,x)=\pm\frac{3\sqrt{10}}{2}\mathrm{sech}^{2}\left(  \frac{1}%
{2}\left(  x-x_{\pm}\mp\frac{\sqrt{10}}{2}t\right)  \right)
,
\end{split}
\right.\vspace*{-0.2cm}
\]
where $x_{\pm}=\pm5$. The results are gathered in Figure \ref{fig_11_fin}. We interpret this figure as a possible blow-up of the $L^\infty$-norm on the derivative of $u$.
\begin{figure}[h!]
\begin{minipage}[t]{1\linewidth}
\begin{minipage}[b]{.5\linewidth}
\centering\epsfig{figure=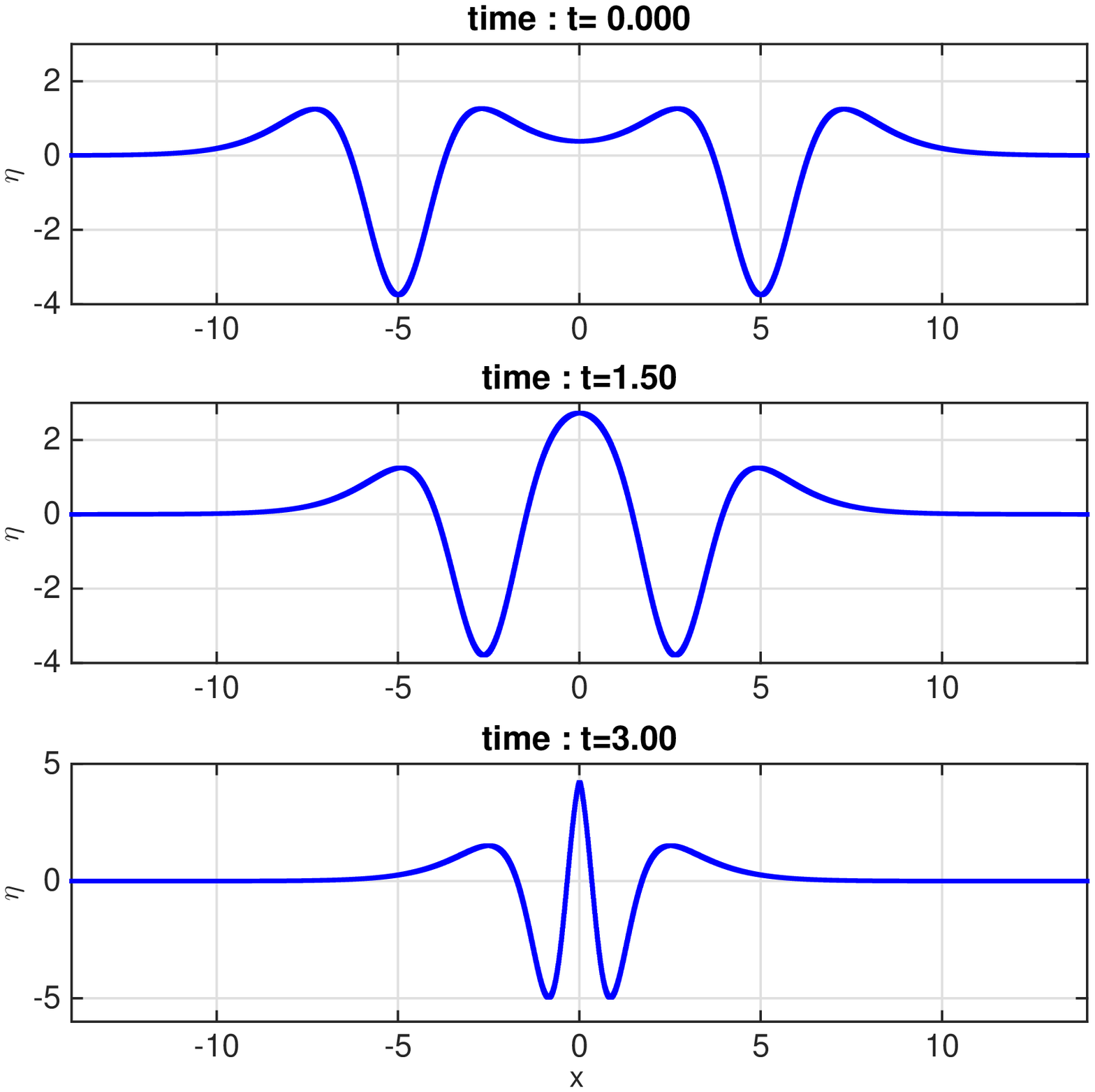,width=\linewidth}\vspace*{-0.3cm}
\centering\epsfig{figure=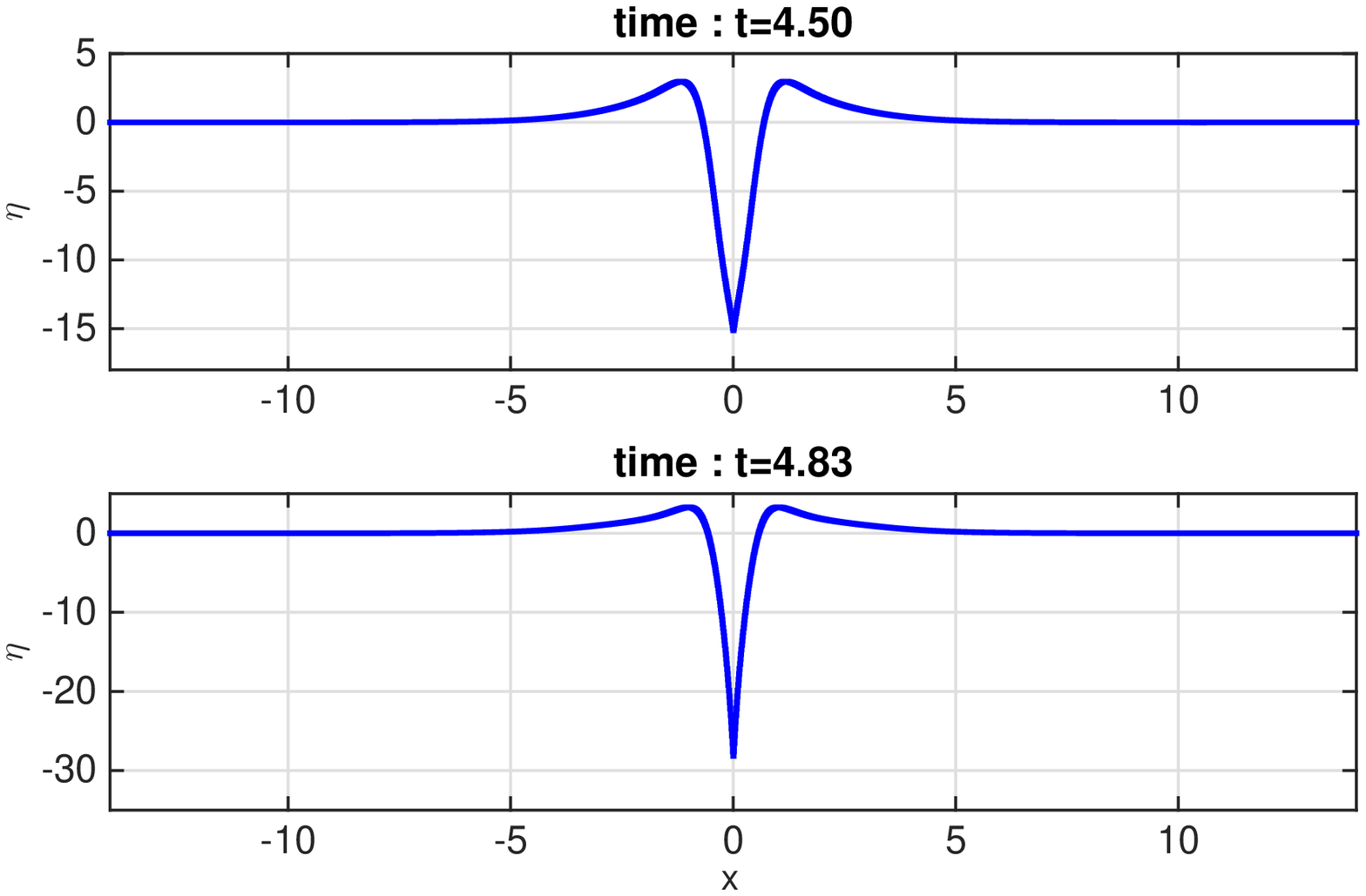,width=\linewidth}
\end{minipage}\hspace*{-0.2cm}\vspace*{-0.2cm}
\begin{minipage}[b]{.5\linewidth}
\centering\epsfig{figure=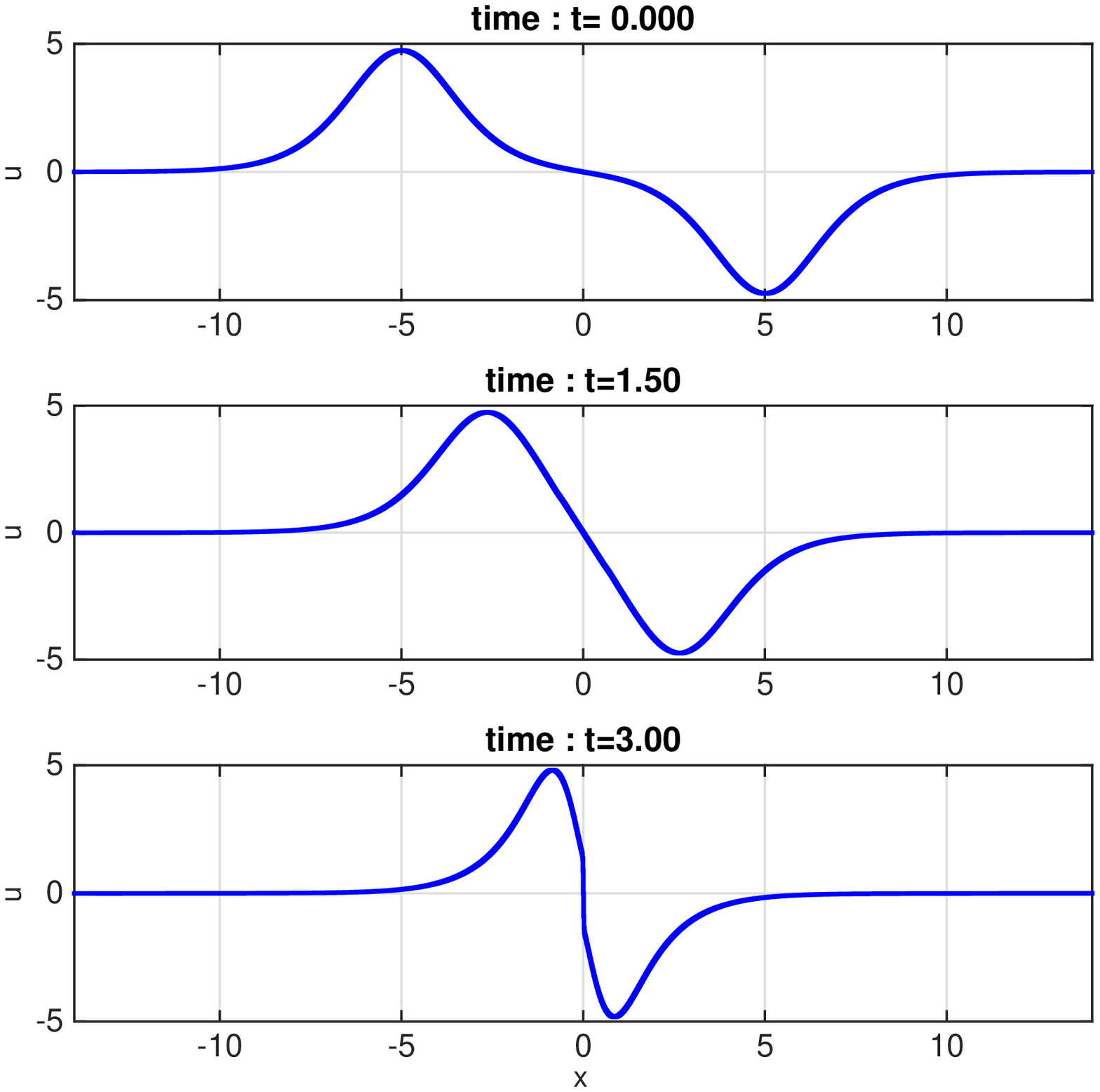,width=\linewidth}\vspace*{-0.3cm}
\centering\epsfig{figure=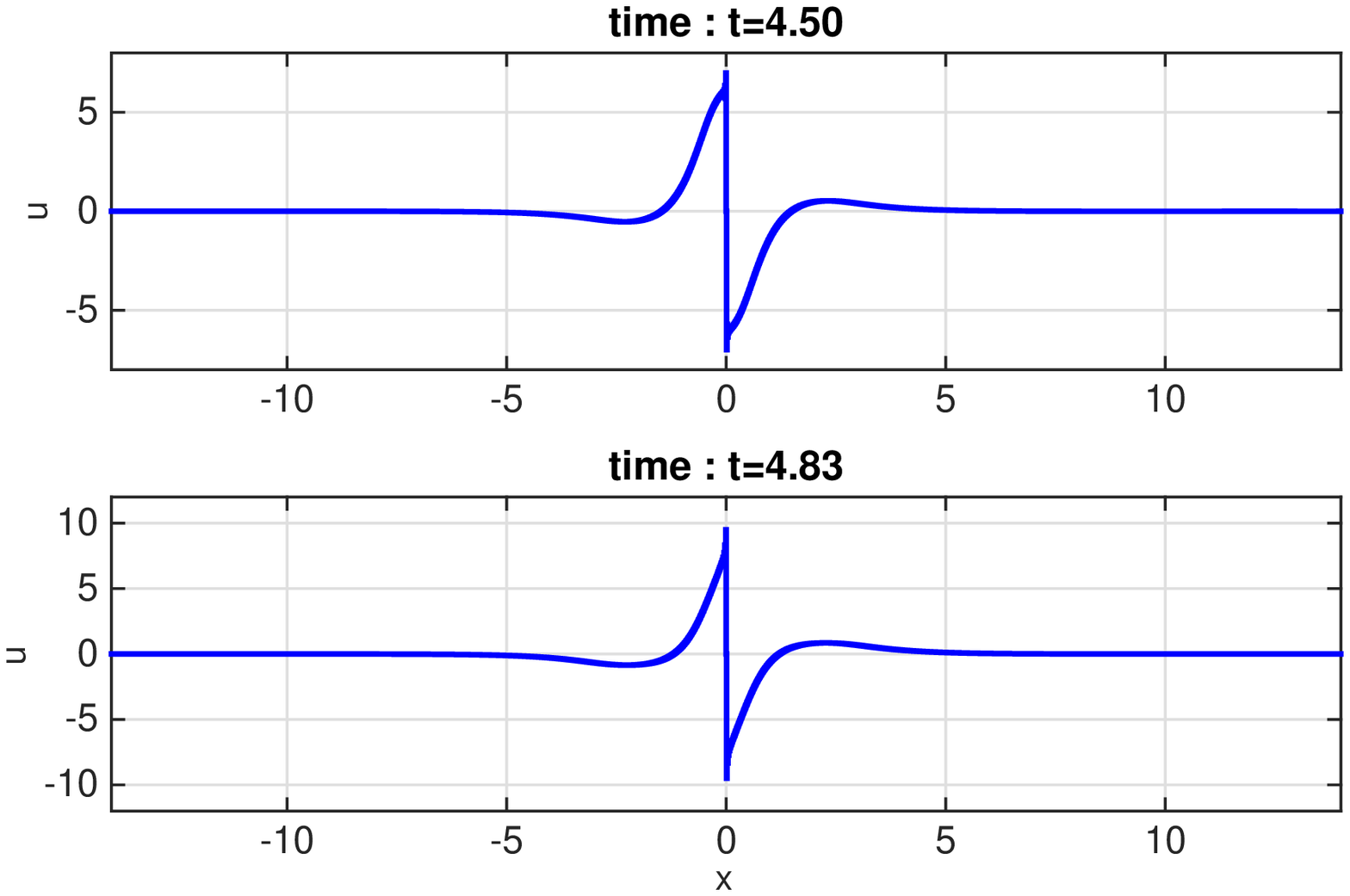,width=\linewidth}
\end{minipage}
\caption{Evidence suggesting the explosion of the derivative of $u$ for the case $(H)$ with $a=0$, $b=\frac{3}{5}$, $c=0$, $d=0$: results for $\eta$ (left) and $u$ (right) \label{fig_11_fin}}
\end{minipage}
\end{figure}

\newpage
\begin{remark}
The initial data leading to the ''numerical'' blow-up are not physical because the quantity $1+\eta$, which is the height of the water column, should be a positive quantity.
\end{remark}
\subsection{Head-on collision.}
In this subsection, we studied two cases where the collision of traveling waves leads to phase shifts and visible dispersive tails which follow each of the waves. 

For ($a=-\frac{7}{30},\ b=\frac{7}%
{15},\ c=-\frac{2}{5},\ d=\frac{1}{2}$), the
experiment is performed on the space domain $[-L,L]=[-14,14]$, the space- and time-meshes $\Delta x=0.01$, $\Delta t=0.001$ and the following initial value 
\[
\left\{
\begin{split}
&  \eta(0,x)=\eta_{+}(0,x)+\eta_{-}(0,x),\\
&  u(0,x)=u_{+}(0,x)+u_{-}(0,x),
\end{split}
\right.
\]
with
\[(I)
\left\{
\begin{split}
&  \eta_{\pm}(t,x)=\frac{3}{8}\mathrm{sech}^{2}\left(  \frac{1}{2}\sqrt
{\frac{5}{7}}\left(  x-x_{\pm}\mp\frac{5\sqrt{2}}{6}t\right)  \right)  ,\\
&  u_{\pm}(t,x)=\pm\frac{1}{2\sqrt{2}}\mathrm{sech}^{2}\left(  \frac{1}%
{2}\sqrt{\frac{5}{7}}\left(  x-x_{\pm}\mp\frac{5\sqrt{2}}{6}t\right)  \right)
,
\end{split}
\right.
\]
where $x_{\pm}=\pm\frac{L}{2}$. 
\begin{figure}[h]
\begin{minipage}[t]{1\linewidth}
\begin{minipage}[b]{.5\linewidth}
\centering\epsfig{figure=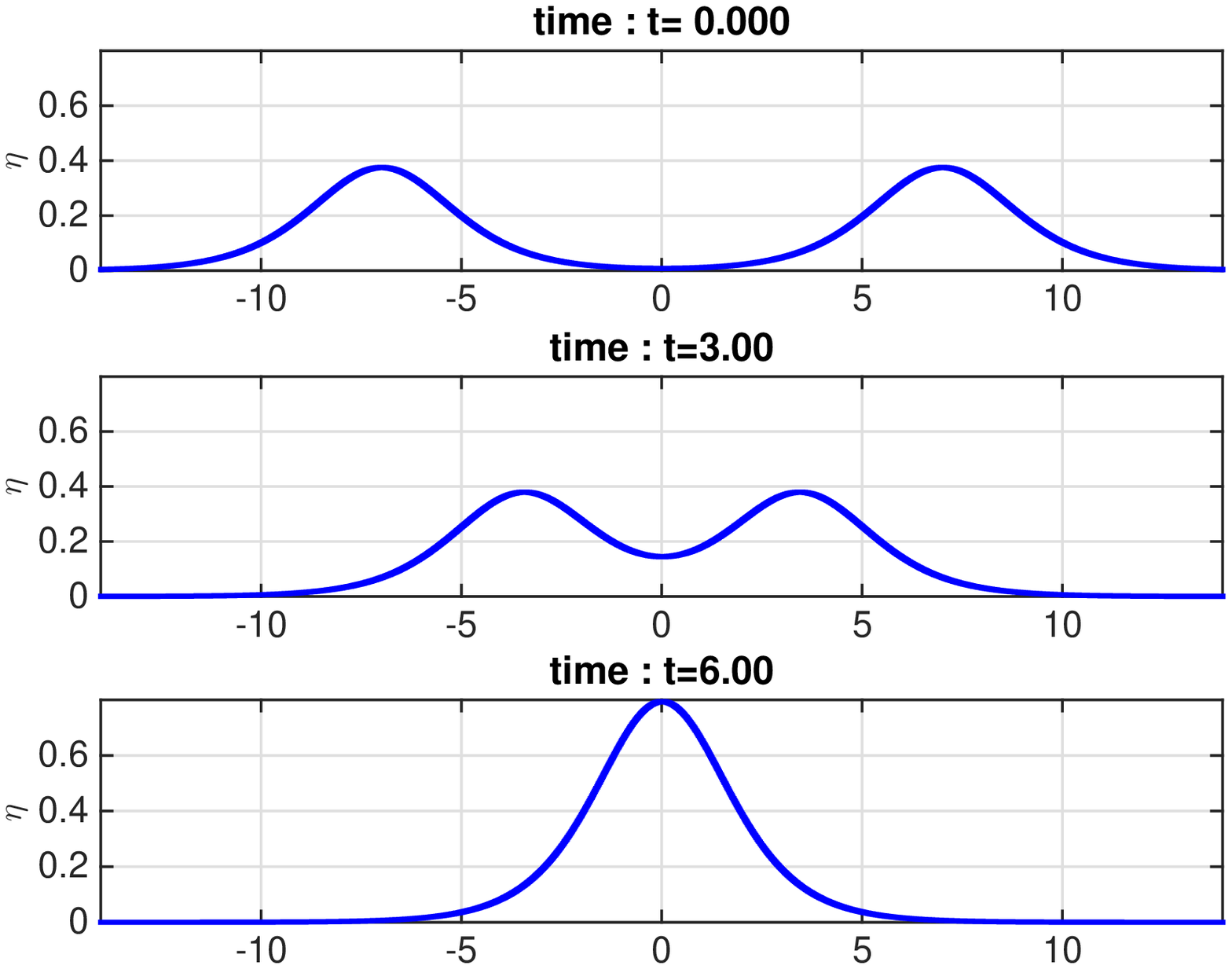,width=\linewidth}\vspace*{-0.4cm}
\centering\epsfig{figure=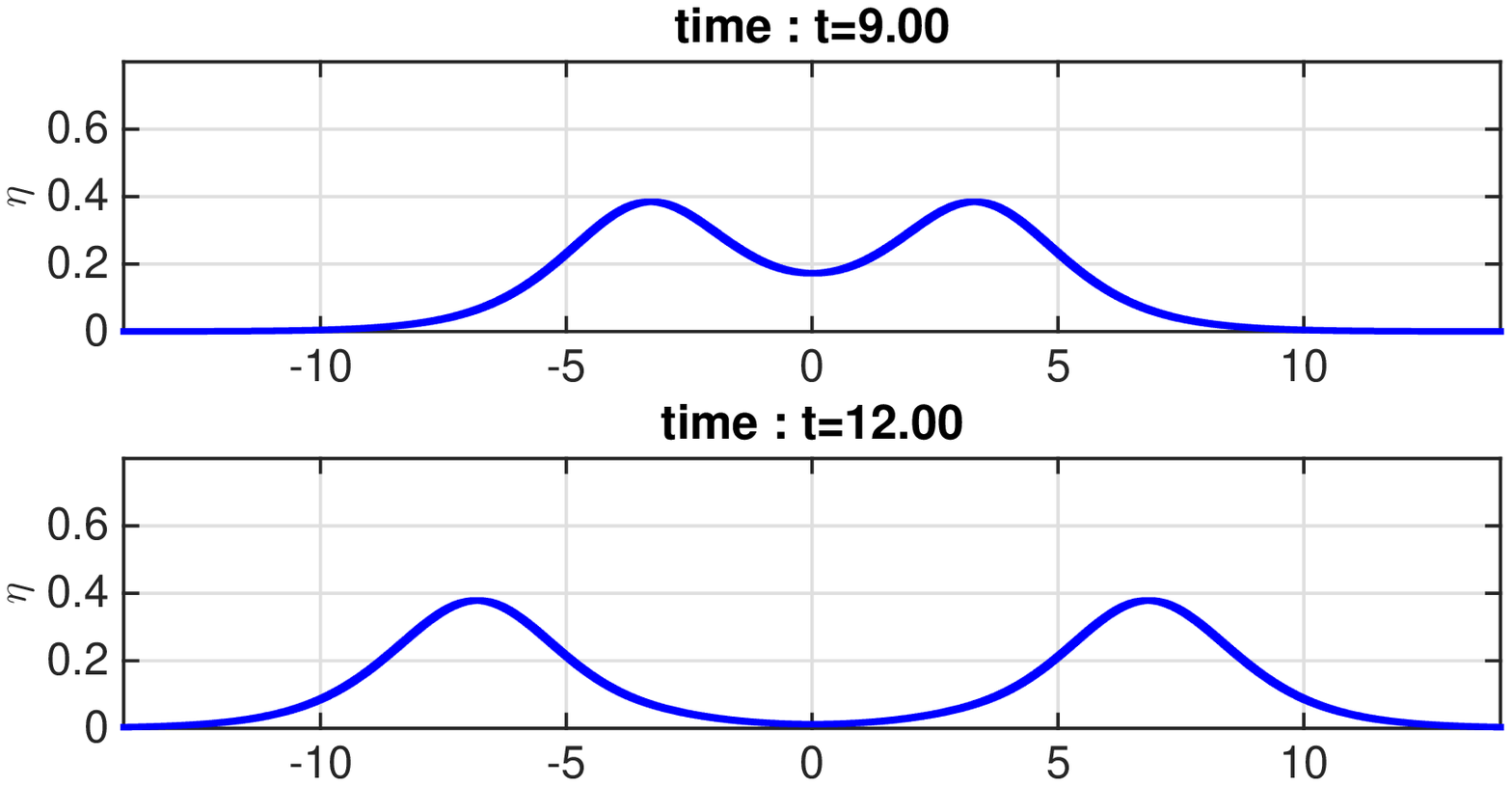,width=\linewidth}
\end{minipage}\hspace*{-0.2cm}
\begin{minipage}[b]{.5\linewidth}
\centering\epsfig{figure=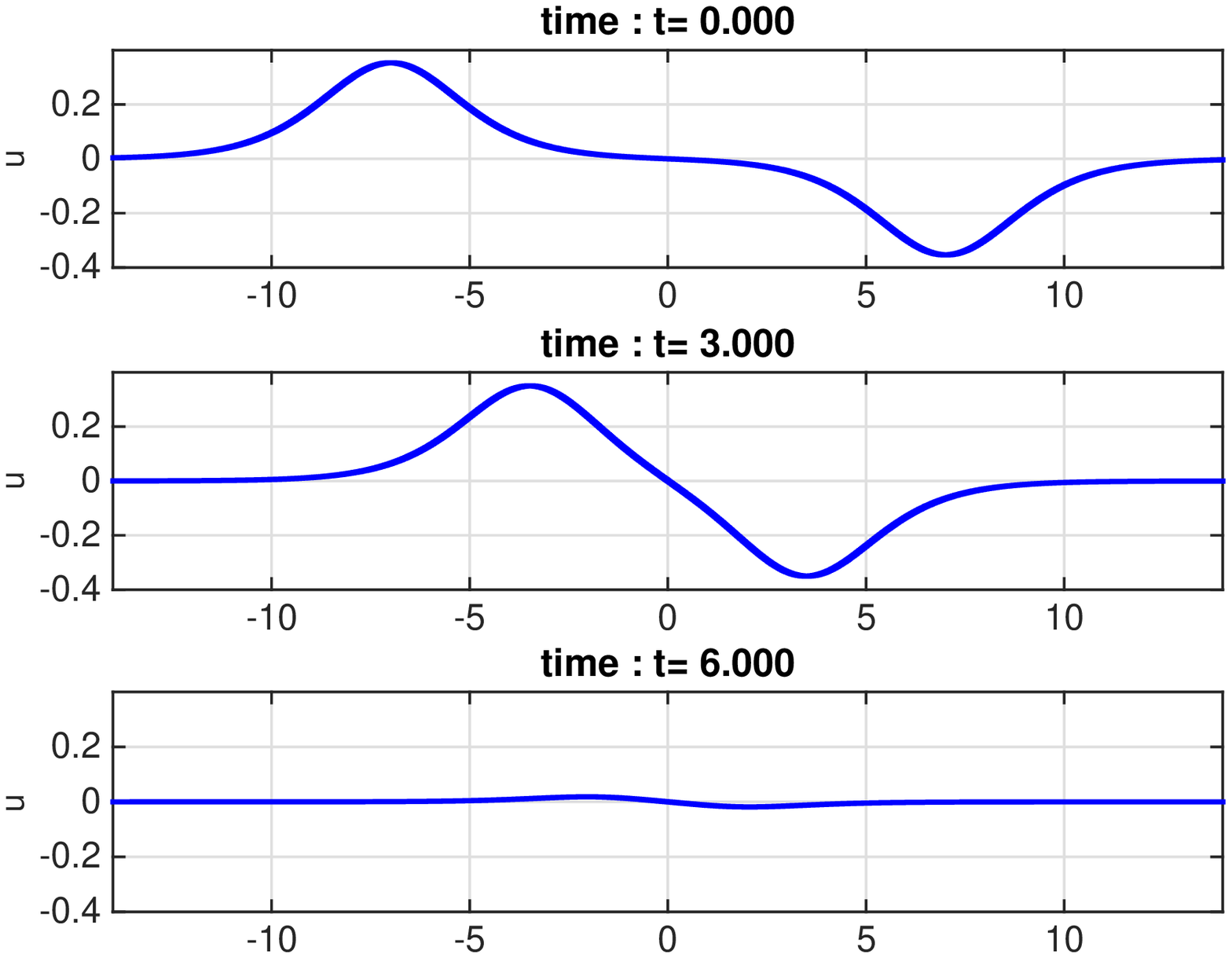,width=\linewidth}\vspace*{-0.4cm}
\centering\epsfig{figure=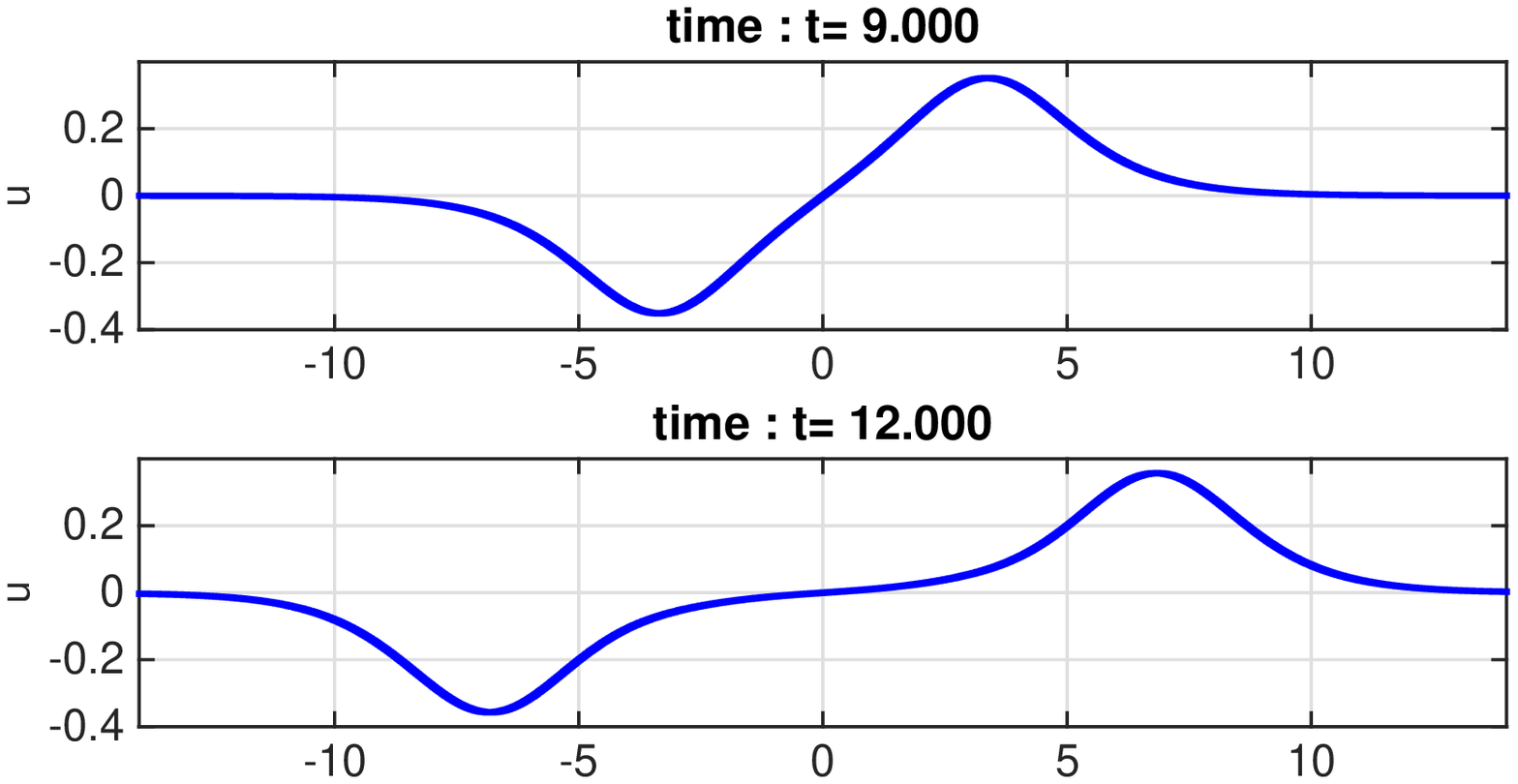,width=\linewidth}
\end{minipage}
\vspace*{-2cm}
\caption{No explosion for case $(I)$ with $a=-\frac{7}{30}$, $b=\frac{7}{15}$, $c=-\frac{2}{5}$, $d=\frac{1}{6}$ : results for $\eta$ (left) and $u$ (right) \label{fig_10}}
\end{minipage}
\end{figure}

Figure \ref{fig_10} is an overview of the situation, which clearly illustrates that traveling waves emerge from the collision (which occurs at $t\approx5.94$) without major changes. In reality, some local changes are expected to be produced after the interaction of the waves such as phase-shifts, slight changes in amplitude and dispersive tails. 
Figures \ref{Phase_shift_2_fig}-\ref{magnificent_fig} illustrate all this expected consequences. 

More precisely, Figure \ref{Phase_shift_2_fig} illustrates the phase-shift which arises from the collision. The paths of the maximum amplitude of both traveling waves are represented in the $(t,x)$-plane, with or without collision. As expected, the traveling waves are delayed by the head-on collision since the phase-shifts are in the opposite direction of their motion. These phase-shifts are characterized by a relative error in velocity equals to $4.8383\%$ for the traveling waves moving from the right to the left and a relative error in velocity equals to $4.1043\%$ for the traveling wave moving from the left to the right.

To visualize the other expected behaviors caused by the head-on collision, we have increased the final time of the numerical simulations. The parameters are now $L=200$, $x_{\pm}=\pm 50$ and the final time $T=150$. With these new parameters, the collision occurs at $t\approx42.43$.
A little modification in the amplitude of the traveling waves is expected from the head-on collision and numerical simulations match this expectation as shown in Figure \ref{amplitude_max_temps_fig}. Indeed, the maximum amplitude of $\eta$ for each traveling wave is lower after the collision than before. For instance, the relative error in maximum amplitude at $t=47$ is $1.1003\%$. It is worth noting also that the maximum amplitude at the collision ($t\approx 42.43$) is bigger than the sum of the maximum amplitudes of the two traveling waves before the collision. The maximum elevation during the collision is about $7.29\%$ bigger than the sum.

Eventually, the collision is inelastic and thus produces some oscillatory dispersive tails, visible in Figure \ref{temps_long_1_fig} which are magnified in the close up of Figure \ref{magnificent_fig}. We point out that this dispersive tail is not due to an instability of the numerical scheme because its amplitude remains constant, regardless of $\Delta t$ and $\Delta x$.

\begin{figure}[h]
\begin{minipage}[t]{1\linewidth}
\begin{minipage}[b]{.5\linewidth}
\centering\epsfig{figure=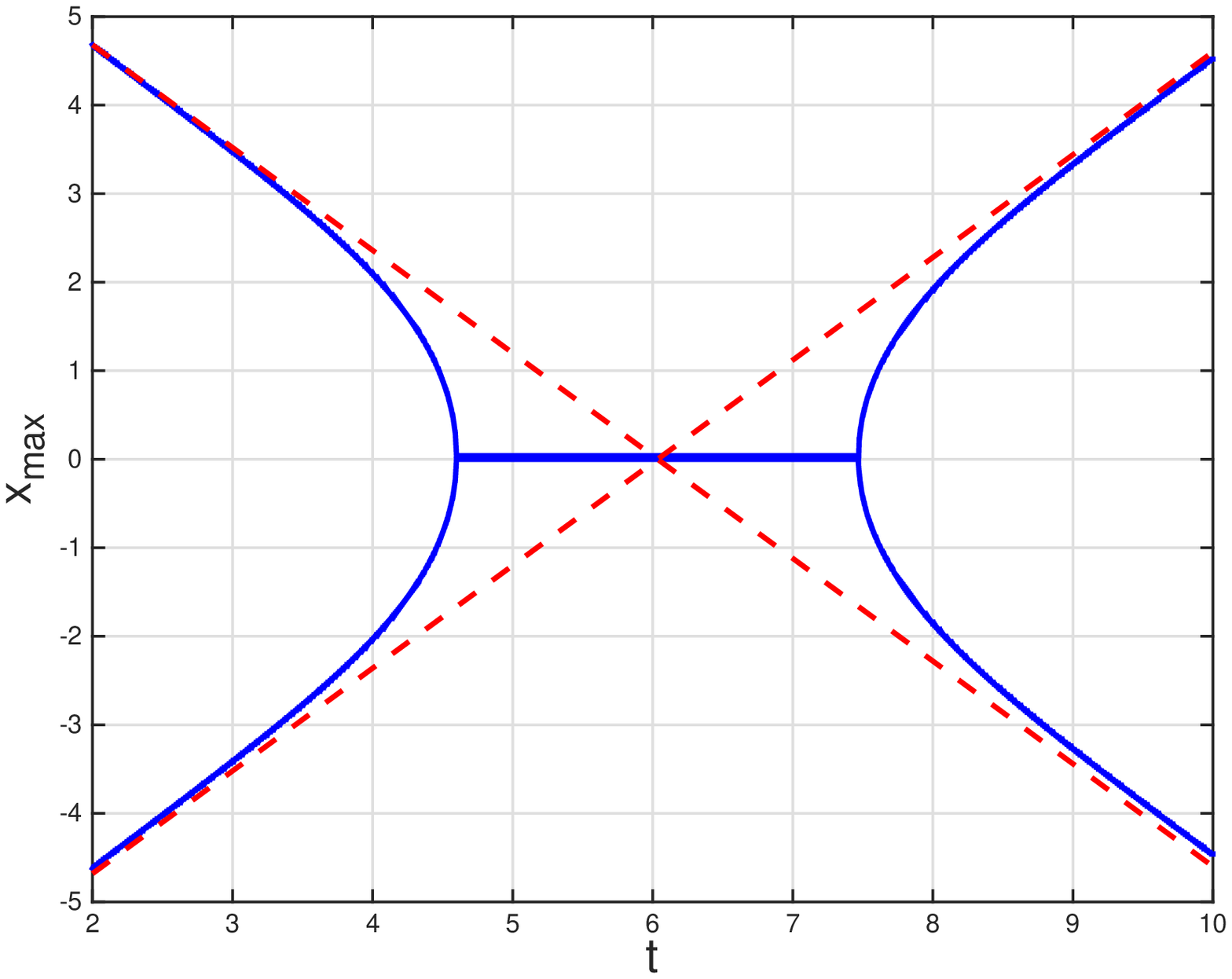,width=\linewidth}
\begin{footnotesize}\emph{Paths of the traveling waves in the $(t,x)$-plane}\end{footnotesize}
\end{minipage}\hspace*{-0.2cm}
\begin{minipage}[b]{.5\linewidth}
\centering\epsfig{figure=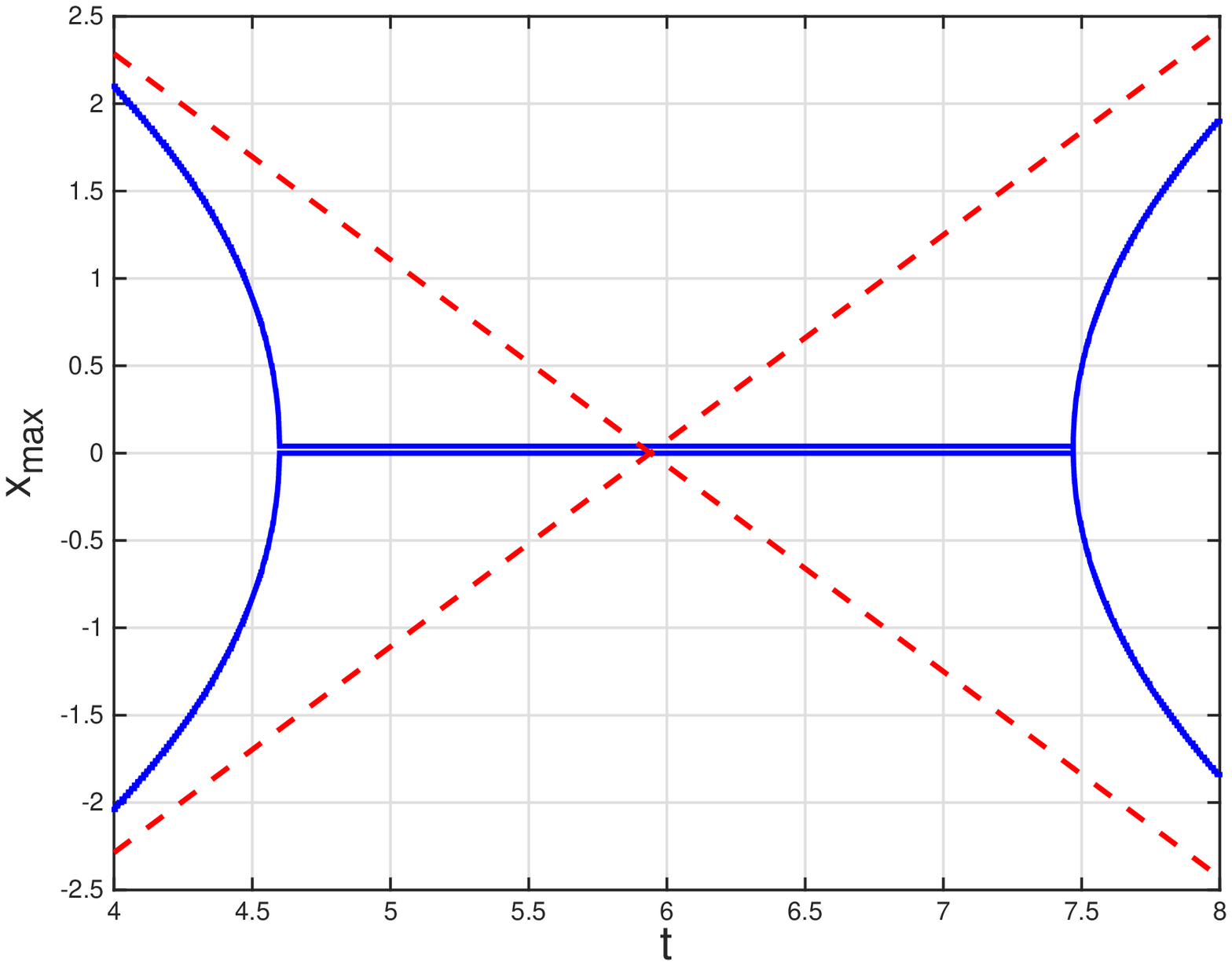,width=\linewidth}
\begin{footnotesize}\emph{Magnification around the collision time $(t\approx 5.94)$}\end{footnotesize}
\end{minipage}
\caption{The phase shift of the traveling waves caused by the head-on collision case $(I)$ (dashed line : both paths without collision, solid line : both paths with collision)\label{Phase_shift_2_fig}}
\end{minipage}
\end{figure}

\begin{figure}[h]
\begin{minipage}[t]{1\linewidth}
\begin{minipage}[b]{.5\linewidth}
\centering\epsfig{figure=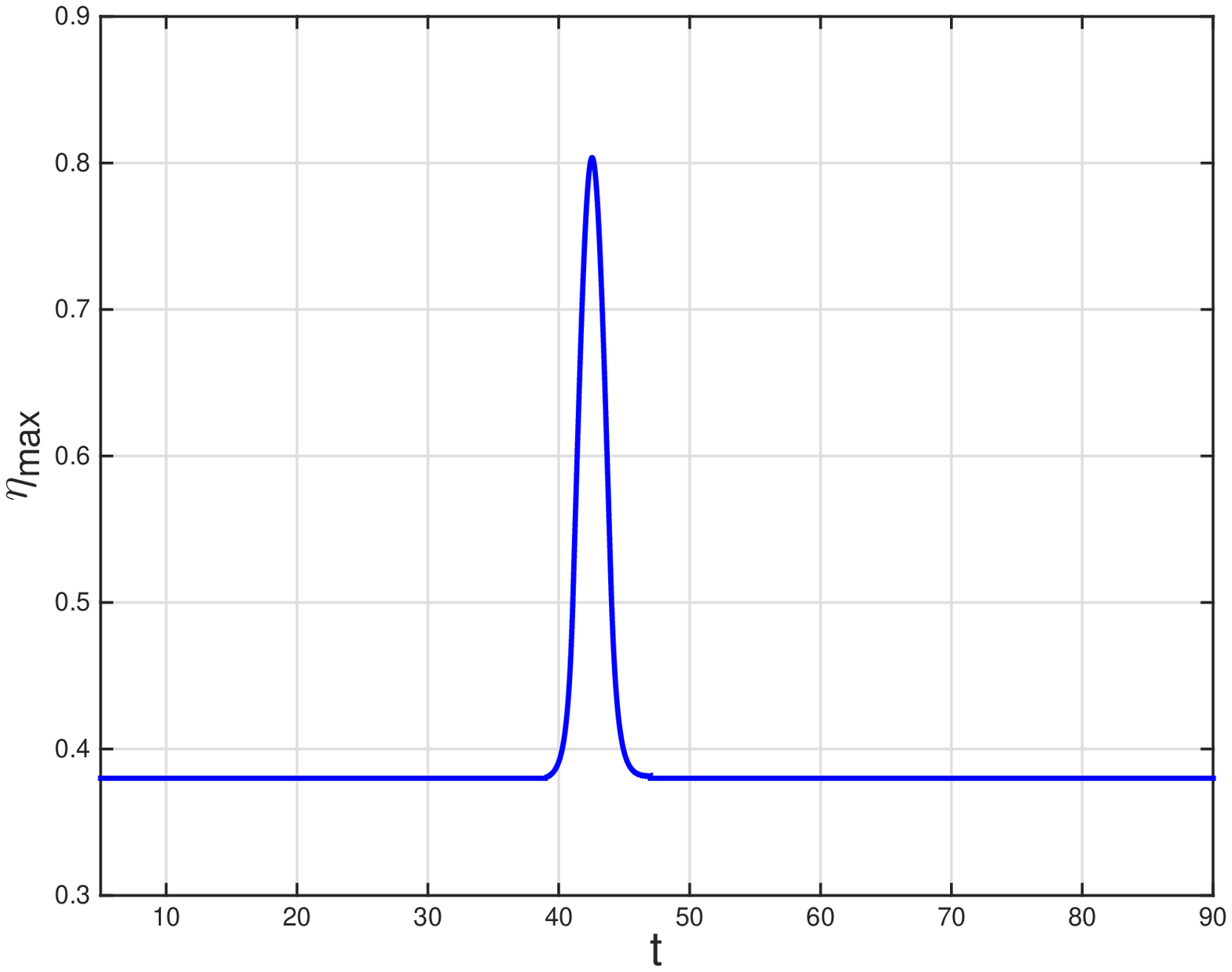,width=\linewidth}
\begin{footnotesize}\emph{Maximum amplitude of $\eta$ against time}\end{footnotesize}
\end{minipage}\hspace*{-0.2cm}
\begin{minipage}[b]{.5\linewidth}
\centering\epsfig{figure=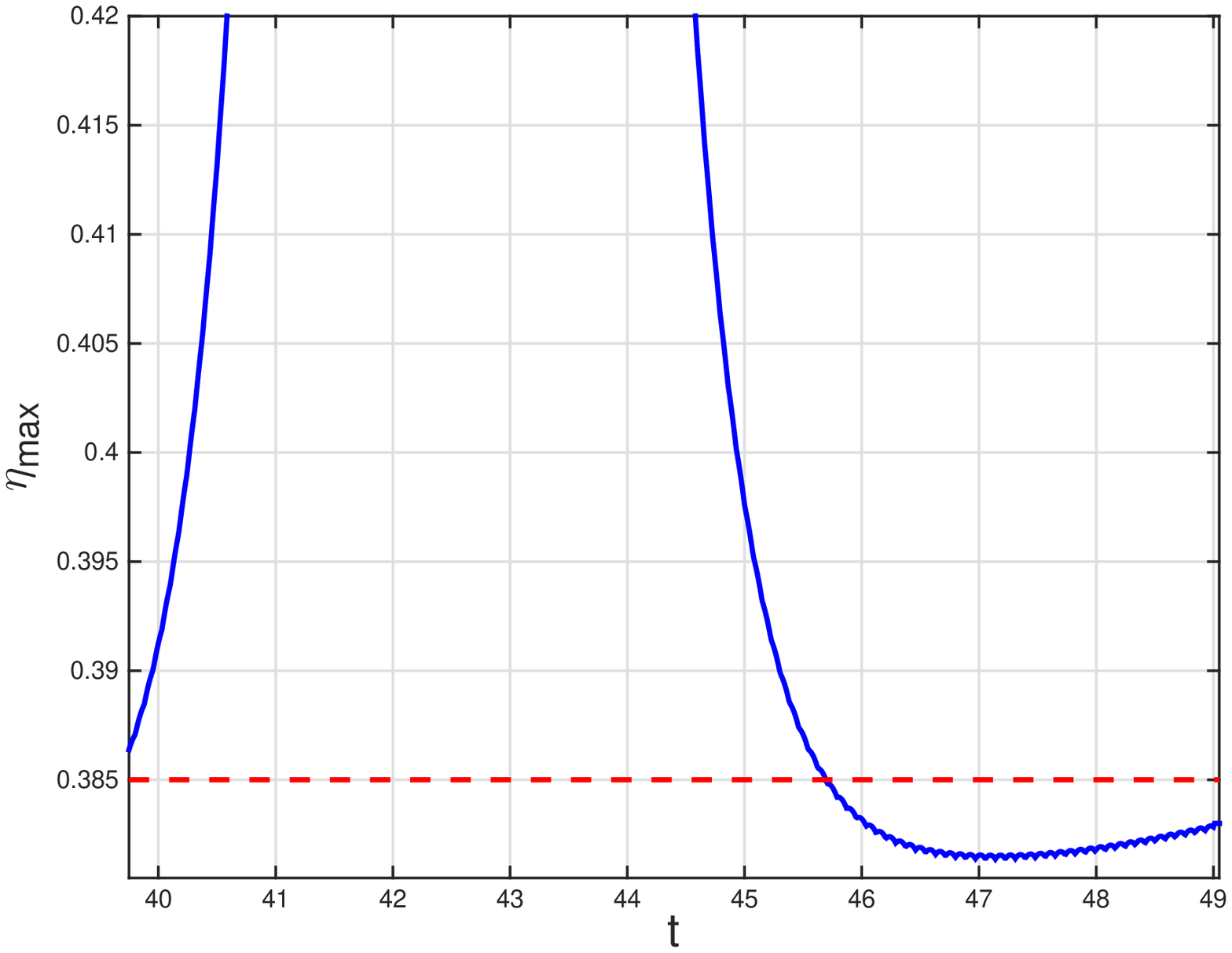,width=\linewidth}
\begin{footnotesize}\emph{Magnification around the collision time}\end{footnotesize}
\end{minipage}
\caption{Changes in amplitude of $\eta$ caused by the head-on collision case $(I)$\label{amplitude_max_temps_fig}}
\end{minipage}
\end{figure}

\begin{figure}[h]
\begin{minipage}[t]{1\linewidth}
\begin{minipage}[b]{.335\linewidth}
\centering\epsfig{figure=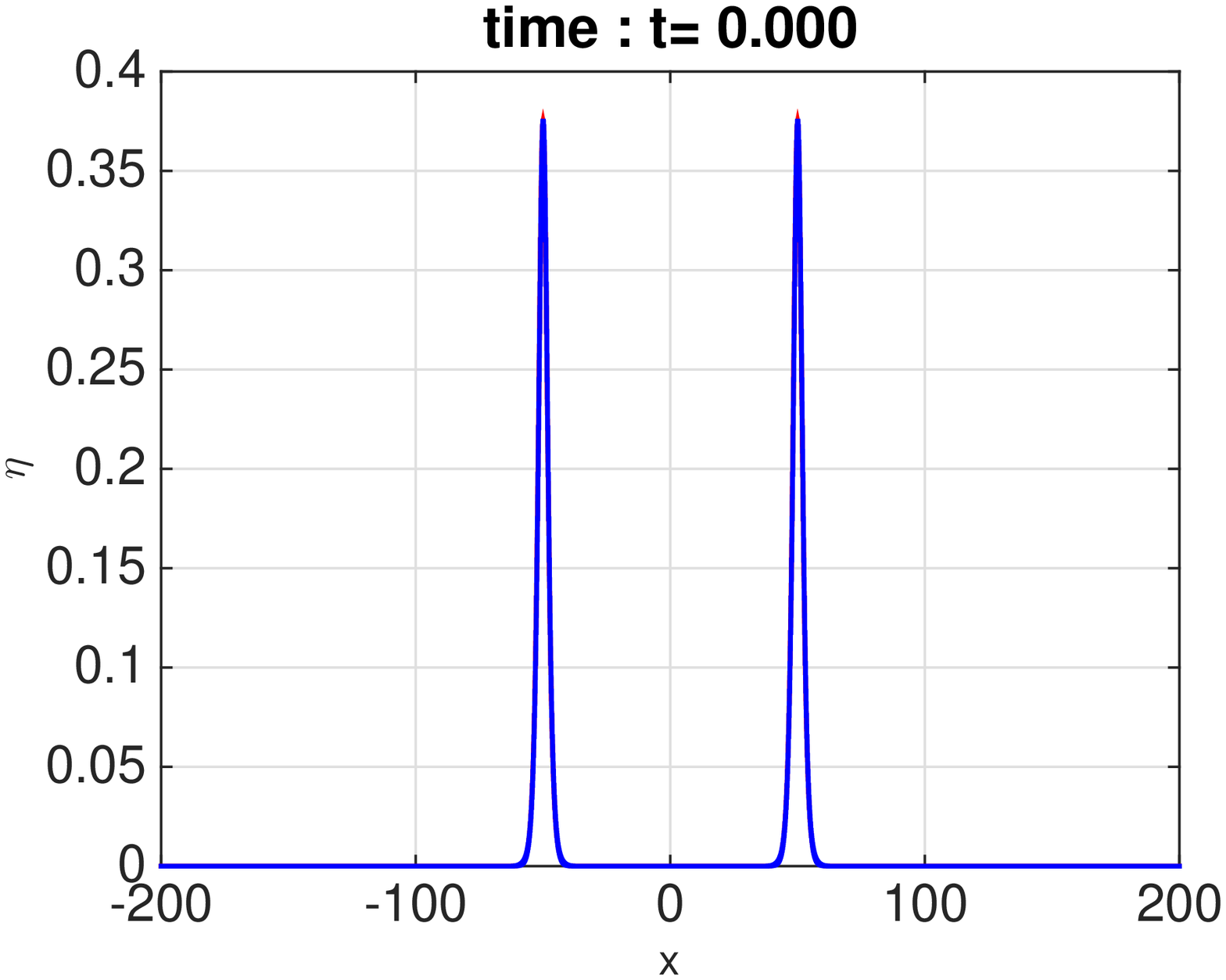,width=\linewidth}
\end{minipage}\hspace*{-0.2cm}
\begin{minipage}[b]{.335\linewidth}
\centering\epsfig{figure=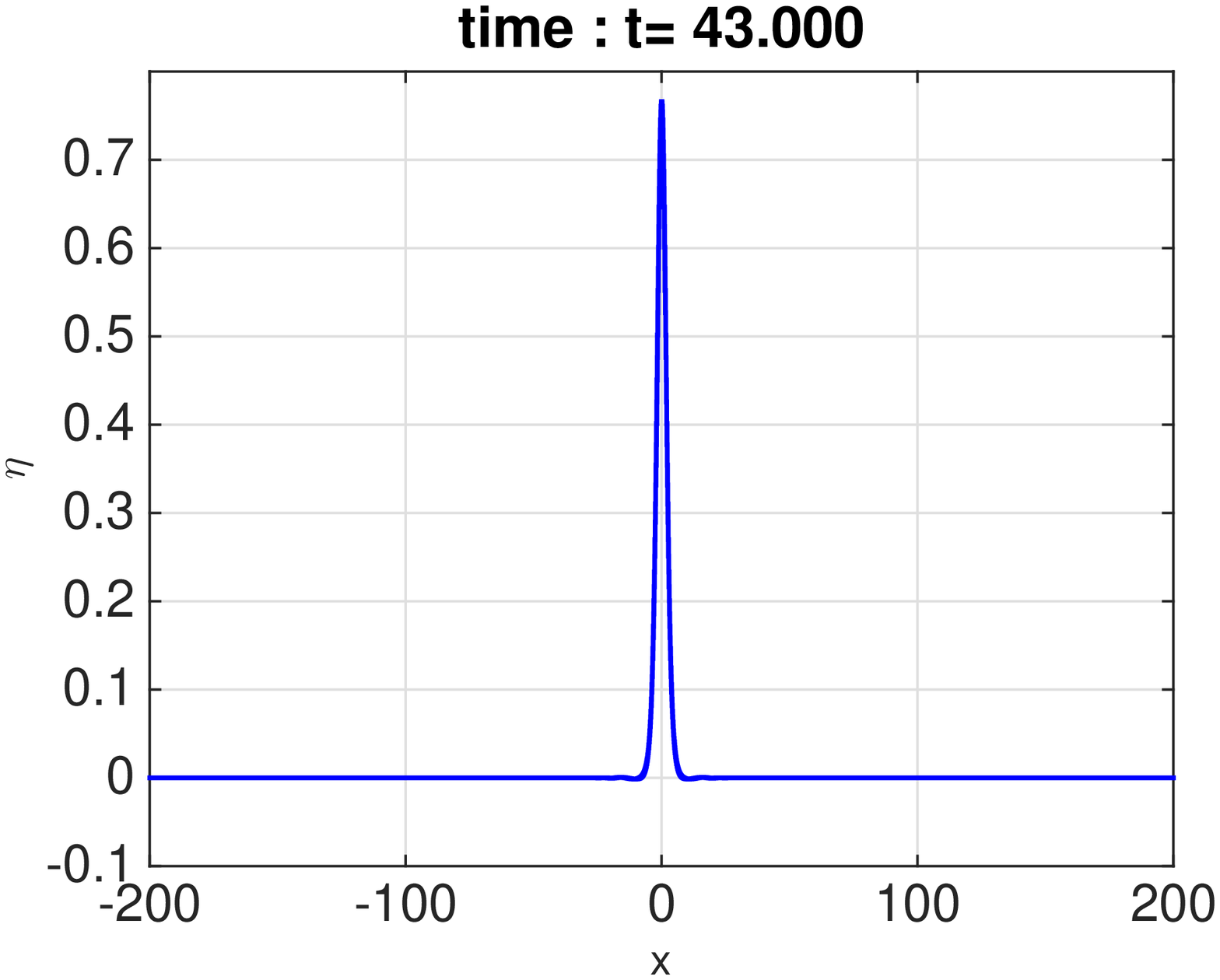,width=\linewidth}
\end{minipage}\hspace*{-0.2cm}
\begin{minipage}[b]{.335\linewidth}
\centering\epsfig{figure=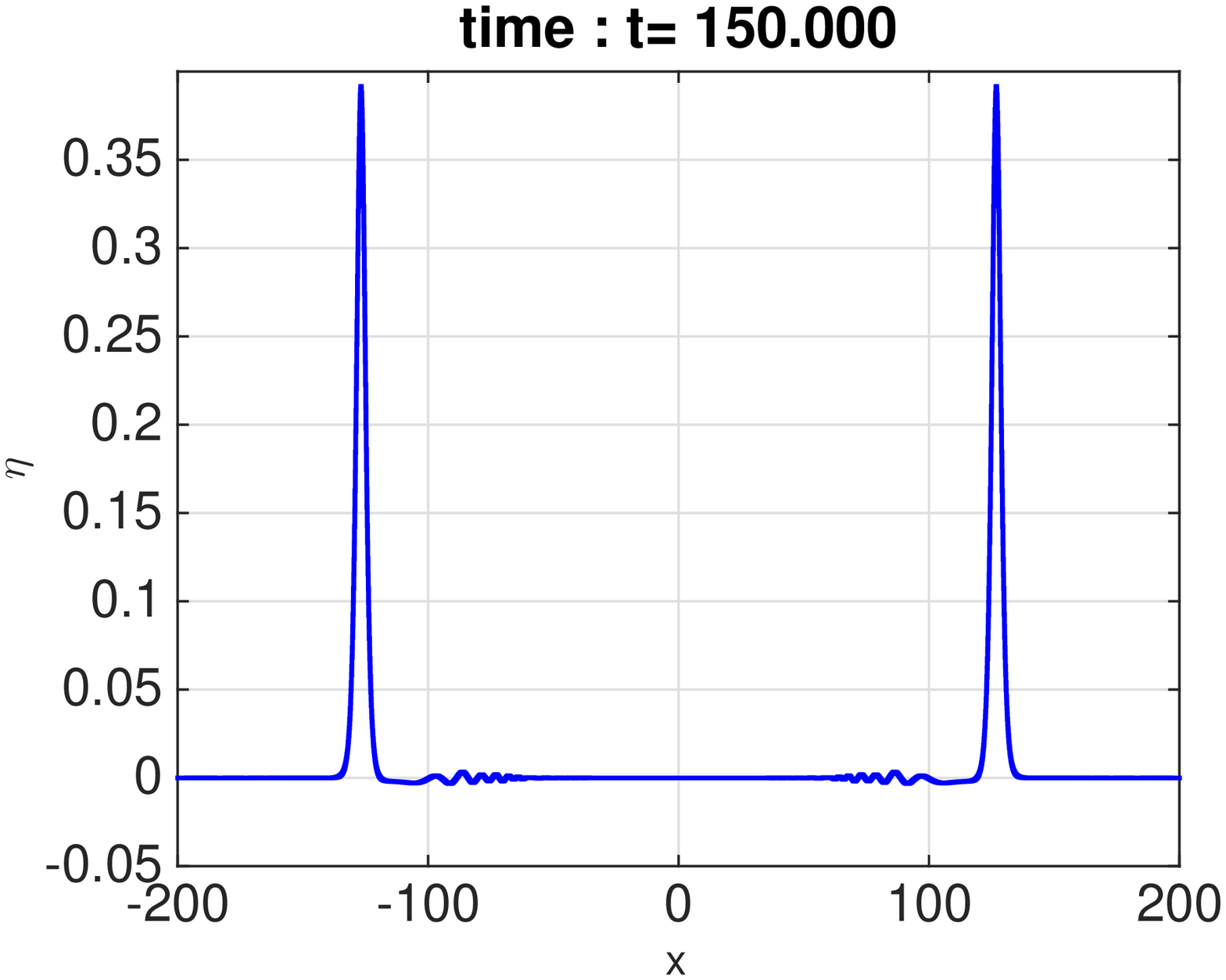,width=\linewidth}
\end{minipage}
\caption{Dispersive tails following the traveling waves and created by the head-on collision case $(I)$\label{temps_long_1_fig}}
\end{minipage}
\end{figure}

\begin{figure}[h!]
\begin{center}
\begin{minipage}[c]{0.4\linewidth}
\centering\epsfig{figure=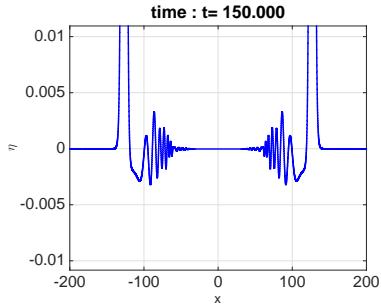,width=\linewidth}
\end{minipage}
\end{center}
\caption{Magnification of the dispersive tails for case $(I)$\label{magnificent_fig}}
\end{figure}

Last but not least, we have performed a numerical experiment to illustrate the head-on collision for $(a=0, b=\frac{1}{6}, c=0, d=\frac{1}{6})$. This test case is drawn from \cite{BonaChen1998}. The numerical parameters are fixed to $L=75$, $\Delta x=\frac{0.5}{32}$, $\Delta t=0.002$ and the initial value :
\begin{equation*}
\left\{
\begin{split}
&\eta(0,x)=\eta_+(0,x)+\eta_-(0,x),\\
&u(0,x)=u_+(0,x)+u_-(0,x),
\end{split}
\right.
\end{equation*}
with 
\begin{equation*}
(J)\left\{
\begin{split}
&\eta_{\pm}(0,x)=\frac{1}{2}\text{sech}^2\left(\frac{1}{2}\sqrt{\frac{6}{5}}\left(x-x_{\pm}\right)\right),\\
&u_{\pm}(0,x)=\pm\frac{1}{2}\text{sech}^2\left(\frac{1}{2}\sqrt{\frac{6}{5}}\left(x-x_{\pm}\right)\right)\mp\frac{1}{16}\text{sech}^4\left(\frac{1}{2}\sqrt{\frac{6}{5}}\left(x-x_{\pm}\right)\right),
\end{split}
\right.
\end{equation*}
where $x_{\pm}=\pm67$. The traveling waves collide at $t\approx54.88$ and emerge with only few changes in phase, amplitude and shape. Figure \ref{phase_shift_fig} illustrates for instance the phase-shift by a representation in the $(t,x)$-plane of the paths of the maximum amplitude of both waves with (solid line) or without (dashed line) the collision. Once again, a consequence of the collision is a delay of the wave propagation because the phase-shift is oriented in the opposite direction of the motion. The relative error on the velocity of the traveling wave moving from the left to the right is $0.1883\%$, whereas the relative error on the velocity of the traveling wave moving from the right to the left is equal to $0.2353\%$.

\begin{figure}[h]
\begin{center}
\begin{minipage}[c]{0.4\linewidth}
\centering\epsfig{figure=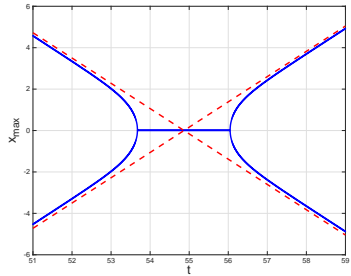,width=\linewidth}
\end{minipage}
\end{center}
\caption{Phase-shift due to the head-on collision case $(J)$ in the $(t,x)$-plane (dashed line : paths without collision and solid line : paths with collision)\label{phase_shift_fig}}
\end{figure}

The changes in maximum amplitude of $\eta$ are summarized in Figure \ref{amplitude_max_Bona_Chen_fig}. On one hand, both traveling waves have a smaller maximal elevation after the collision (the relative error of maximum amplitude at $t=62$ for instance is equal to $0.6538\%$). Secondly, the maximal elevation during the collision is bigger than the sum of both maximum amplitudes ($10.28\%$ bigger than the sum).
\begin{figure}[h]
\begin{minipage}[t]{1\linewidth}
\begin{minipage}[b]{.5\linewidth}
\centering\epsfig{figure=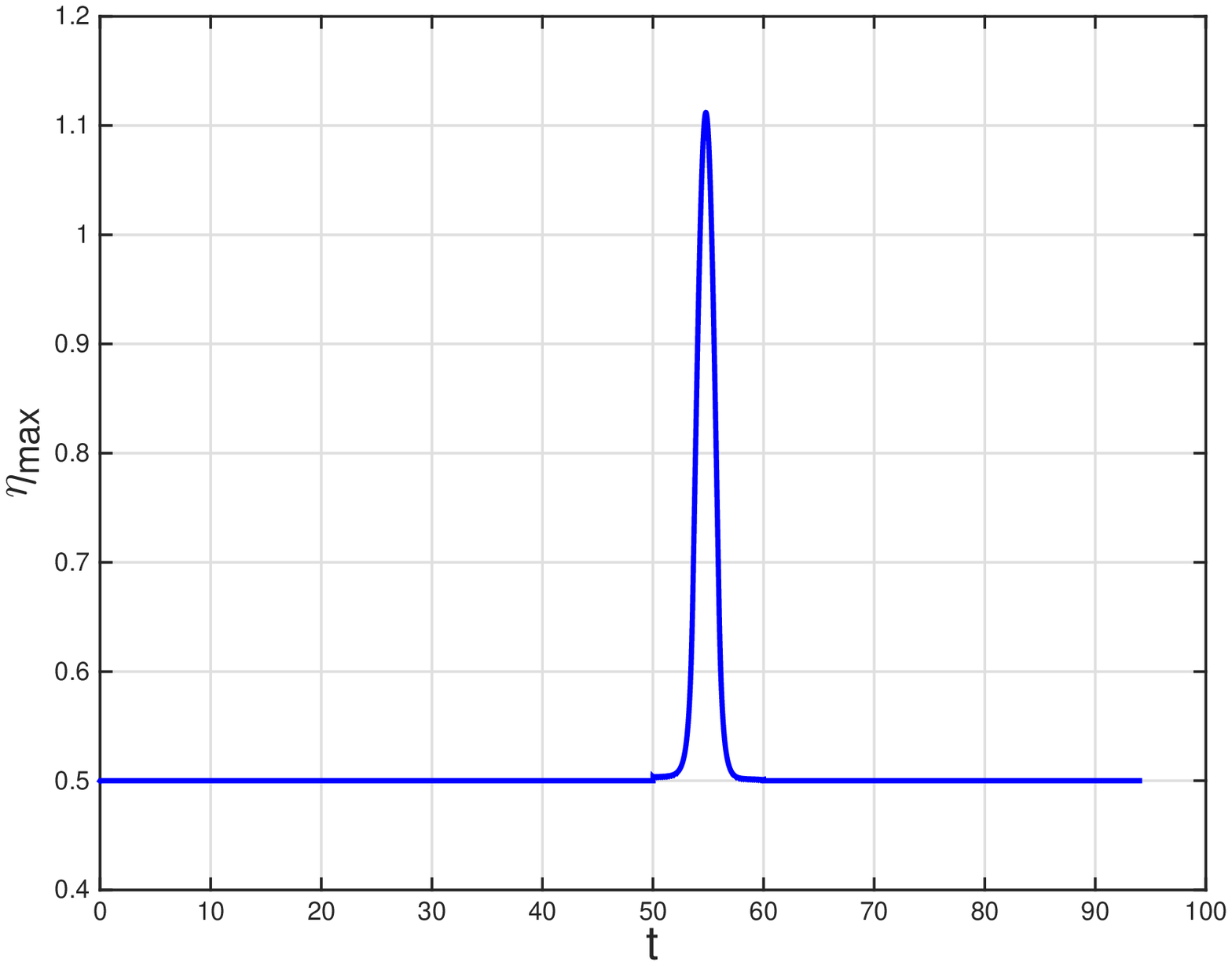,width=\linewidth}
\begin{footnotesize}\emph{Maximum amplitude of $\eta$ against time}\end{footnotesize}
\end{minipage}\hspace*{-0.2cm}
\begin{minipage}[b]{.5\linewidth}
\centering\epsfig{figure=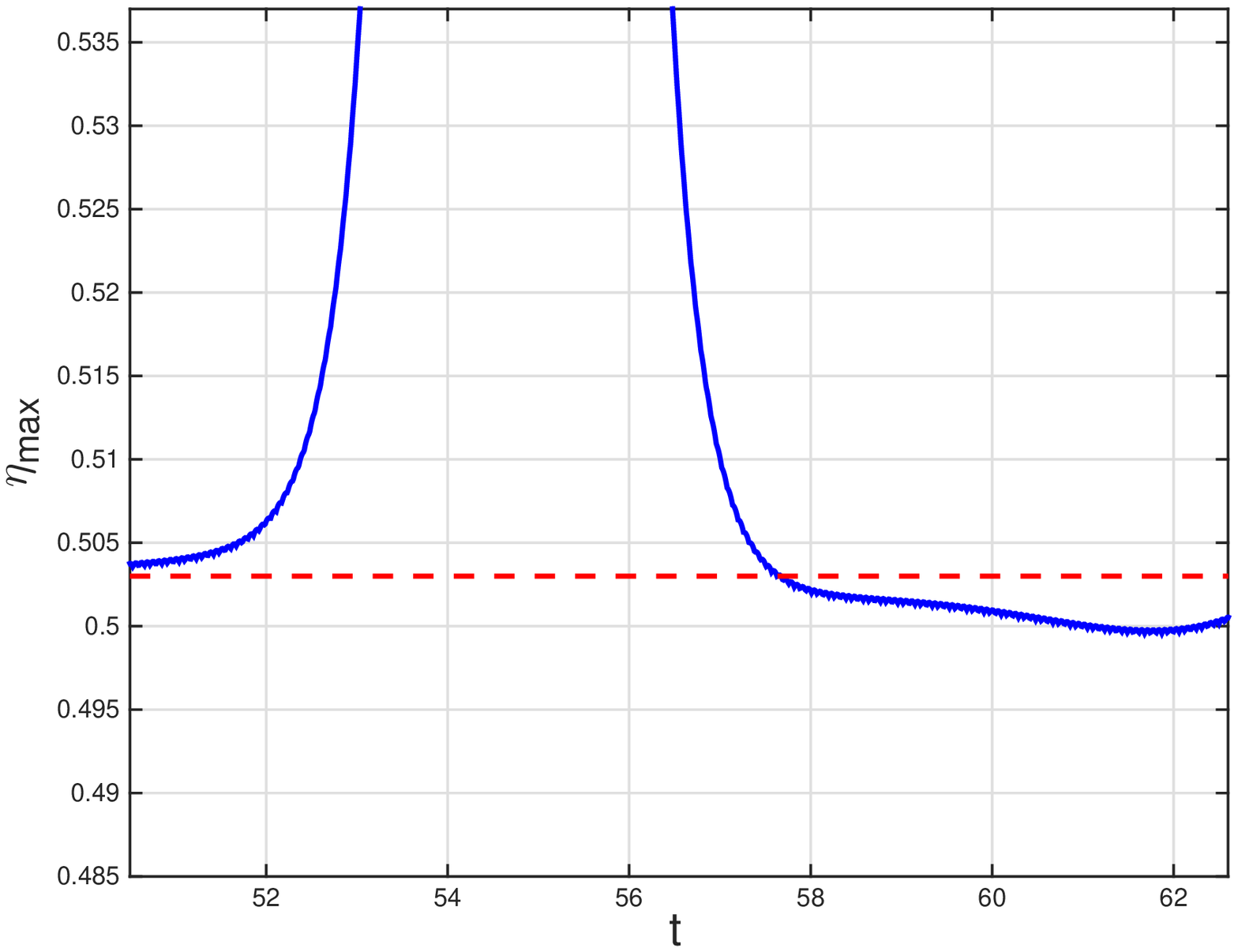,width=\linewidth}
\begin{footnotesize}\emph{Magnification around the collision time}\end{footnotesize}
\end{minipage}
\caption{Changes in amplitude of $\eta$ caused by the head-on collision case $(J)$\label{amplitude_max_Bona_Chen_fig}}
\end{minipage}
\end{figure}
 Eventually, the dispersive tails resulting from the head-on collision is illustrated in Figure \ref{Dispersive_tail_fig} at final time $T=80$.  
\begin{figure}[h]
\begin{minipage}[t]{1\linewidth}
\begin{minipage}[b]{.5\linewidth}
\centering\epsfig{figure=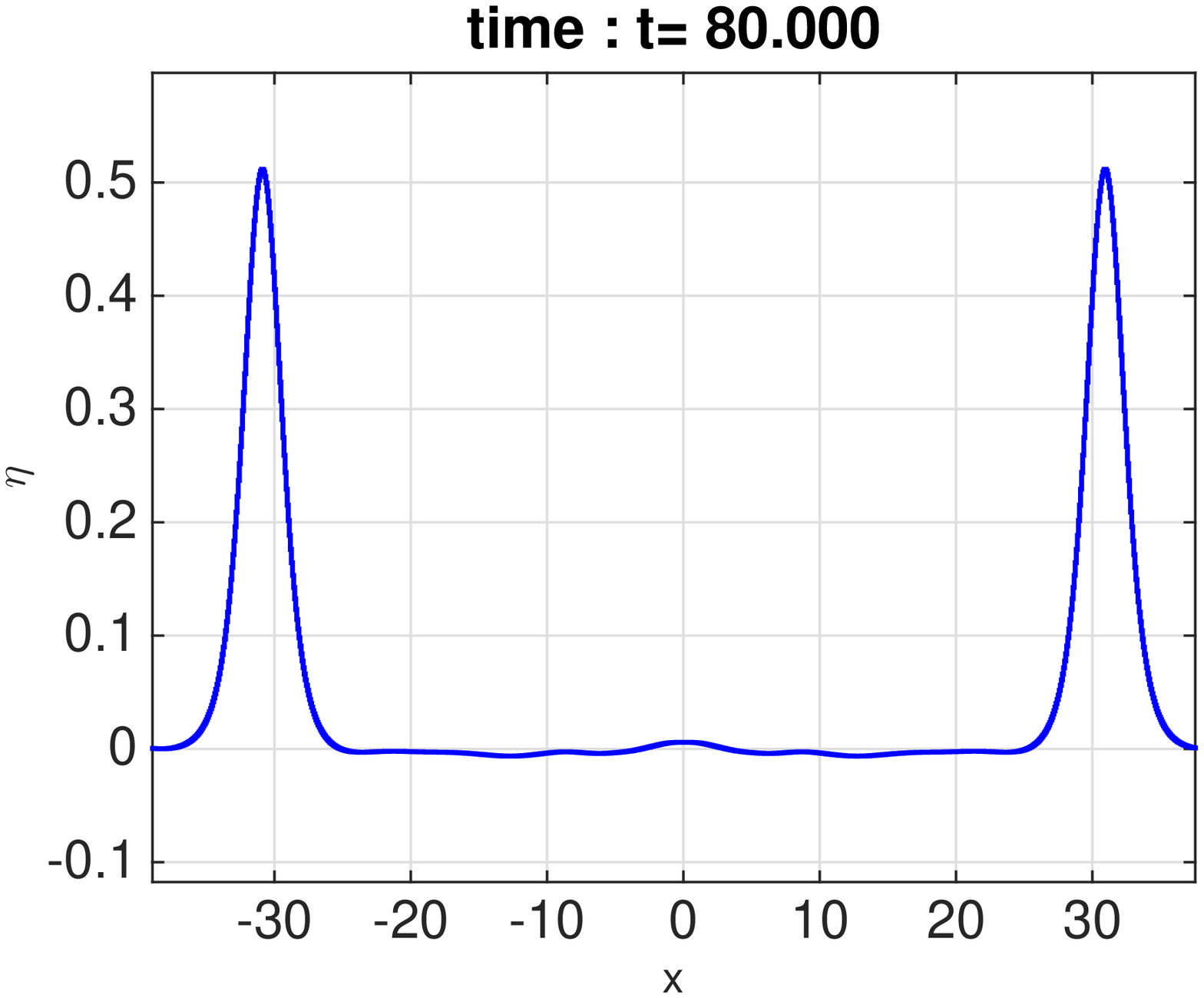,width=\linewidth}
\end{minipage}\hspace*{-0.2cm}
\begin{minipage}[b]{.5\linewidth}
\centering\epsfig{figure=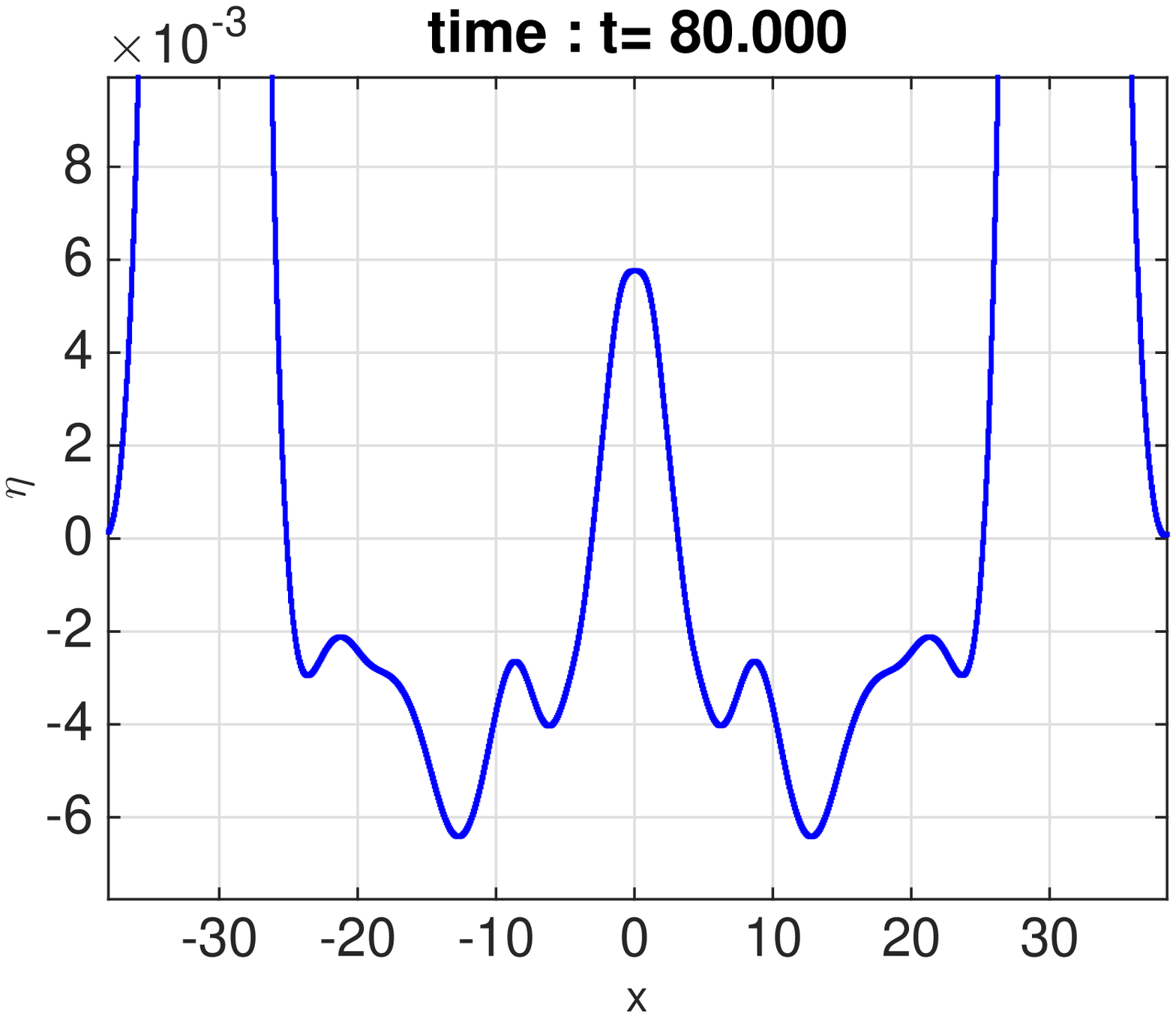,width=\linewidth}
\end{minipage}\hspace*{-0.2cm}
\caption{Dispersive tails following the traveling waves and created by the head-on collision case $(J)$\label{Dispersive_tail_fig}}
\end{minipage}
\end{figure}

\newpage

\appendix
\section{Consistency error}
\label{Section_consistency}
\subsection{Consistency error $\epsilon_{1}^{n}$.}
By definition of the consistency error, one has
\begin{equation*}
\begin{split}
\epsilon_{1}^{n}&=\left(  I-bD_{+}D_{-}\right)  \left(  \frac{\eta_{\Delta
}^{n+1}-\eta_{\Delta}^{n}}{\Delta t}\right)  +\left(  I+aD_{+}D_{-}\right)
D\left(  \theta u_{\Delta}^{n+1}+(1-\theta)u_{\Delta}^{n}\right)  +D\left(
\eta_{\Delta}^{n}u_{\Delta}^{n}\right) \\
& -\frac{\tau_{1}}{2}\Delta
xD_{+}D_{-}\eta_{\Delta}^{n}.
\end{split}
\end{equation*}
We define
\[
\left\{
\begin{split}
&  E_{\mathrm{time}}=\left(  I-bD_{+}D_{-}\right)  \left(  \frac{\eta_{\Delta
}^{n+1}-\eta_{\Delta}^{n}}{\Delta t}\right)  ,\\
&  E_{\mathrm{linear}}=\left(  I+aD_{+}D_{-}\right)  D\left(  \theta
u_{\Delta}^{n+1}+(1-\theta)u_{\Delta}^{n}\right)  ,\\
&  E_{\mathrm{non\ linear}}=D\left(  \eta_{\Delta}^{n}u_{\Delta}^{n}\right)
,\\
&  E_{\mathrm{viscosity}}=\frac{\tau_{1}}{2}\Delta xD_{+}D_{-}\eta
_{\Delta}^{n}.
\end{split}
\right.
\]

We define $Q=[x_{j}, x_{j+1}]\times[t^{n}, t^{n+1}]$. We will only develop the
non linear part and enonce the results for the other parts. By Taylor
expansions and Cauchy-Schwarz inequality, one has
\begin{equation*}
\begin{split}
E_{\mathrm{time}}&\lesssim\frac{1}{\Delta t\Delta x}\int_{x_{j}}^{x_{j+1}}%
\int_{t^{n}}^{t^{n+1}}\partial_{t} \eta(s, y)-b\partial_{x}^{2}\partial
_{t}\eta(s, y)dsdy \\
&+\sqrt{\frac{\Delta t}{\Delta x}}||\partial_{t}^{2}%
\eta||_{ L^{2}(Q)}+b\Delta x\sqrt{\frac{\Delta x}{\Delta t}}%
||\partial_{x}^{4}\partial_{t} \eta||_{ L^{2}(Q)}+b\sqrt{\frac{\Delta
t}{\Delta x}}||\partial_{x}^{2}\partial_{t}^{2}\eta||_{ L^{2}(Q)}.
\end{split}
\end{equation*}
In the same way, we develop the $E_{\mathrm{linear}}$-term.
\[%
\begin{split}
E_{\mathrm{linear}}  &  \lesssim\frac{1}{\Delta t\Delta x}\int_{x_{j}%
}^{x_{j+1}}\int_{t^{n}}^{t^{n+1}}\partial_{x}u(s, y)+a\partial_{x}^{3}u(s,
y)dsdy+\theta\sqrt{\frac{\Delta t}{\Delta x}}||\partial_{x}\partial
_{t}u||_{ L^{2}(Q)}\\
&+\Delta x\sqrt{\frac{\Delta x}{\Delta t}}%
||\partial_{x}^{3}u||_{ L^{2}(Q)}  +|a|\Delta x\sqrt{\frac{\Delta x}{\Delta t}}||\partial_{x}^{5}%
u||_{ L^{2}(Q)}+|a|\theta\sqrt{\frac{\Delta t}{\Delta x}}%
||\partial_{t}\partial_{x}^{3}u||_{ L^{2}(Q)}\\
&+|a|\Delta x\sqrt
{\frac{\Delta x}{\Delta t}}||\partial_{x}^{5}u||_{ L^{2}(Q)}.
\end{split}
\]
We will develop the non linear part. We denote $K$ the following function on $[0,1]$
$$
K(\nu)=\int_{x_{j}}^{x_{j+1}}\int_{t^{n}}^{t^{n+1}}\int_{x_{j}}^{x_{j+1}}%
\int_{t^{n}}^{t^{n+1}}\eta(s, y+\nu\Delta x)u(t, x+\nu\Delta x)dxdydsdt.
$$
Thus,
\begin{align*}
&K^{\prime}(\nu)    =\int_{x_{j}}^{x_{j+1}}\int_{t^{n}}^{t^{n+1}}\int_{x_{j}%
}^{x_{j+1}}\int_{t^{n}}^{t^{n+1}}\Delta x\partial_{x}\eta(s, y+\nu\Delta
x)u(t, x+\nu\Delta x)\\
&+\Delta x\eta(s, y+\nu\Delta x)\partial_{x} u(t,
x+\nu\Delta x)dxdydsdt\\
&  =\int_{x_{j}}^{x_{j+1}}\int_{t^{n}}^{t^{n+1}}\int_{x_{j}}^{x_{j+1}}%
\int_{t^{n}}^{t^{n+1}}\Delta x\partial_{x}\eta(s, y+\nu\Delta x)\left[  u(s,
x+\nu\Delta x)+\int_{s}%
^{t}\partial_{t} u(\tau, x+\nu\Delta x)d\tau\right]\\
&  +\Delta x\partial_{x} u(t, x+\nu\Delta x)\left[  \eta(t, y+\nu\Delta
x)+\int_{t}^{s}\partial_{t}
\eta(\tau, y+\nu\Delta x)d\tau\right]  dxdydsdt\\
&  =\int_{x_{j}}^{x_{j+1}}\int_{t^{n}}^{t^{n+1}}\int_{x_{j}}^{x_{j+1}}%
\int_{t^{n}}^{t^{n+1}}\Delta x\partial_{x}\eta(s, y+\nu\Delta x)\left[  u(s,
y+\nu\Delta x)\right.\\
&\left.-(y-x)\partial_xu(s,x+\nu\Delta x)-\int_{x}^{y}\partial_{x}^2 u(s, z+\nu\Delta x)(y-z)dz+\int_{s}%
^{t}\partial_{t} u(\tau, x+\nu\Delta x)d\tau\right] \\
&  +\Delta x\partial_{x} u(t, x+\nu\Delta x)\left[  \eta(t, x+\nu\Delta
x)-(x-y)\partial_x\eta(t, y+\nu\Delta x)\right.\\
&\left.-\int_{y}^{x}\partial_{x}^2 \eta(t, z+\nu\Delta x)(x-z)dz+\int_{t}^{s}\partial_{t}
\eta(\tau, y+\nu\Delta x)d\tau\right]  dxdydsdt\\
&  =\Delta x^{2}\Delta t\int_{x_{j}}^{x_{j+1}}\int_{t^{n}}^{t^{n+1}}%
\partial_{x}(\eta u)(s, y+\nu\Delta x)dsdy\\
&  -\Delta x\Delta t\int_{x_{j}}^{x_{j+1}}\int_{t^{n}}^{t^{n+1}}\int_{x_{j}%
}^{x_{j+1}}\int_{x}^{y}\partial_{x}\eta(s, y+\nu\Delta x)\partial_{x}^2u(s,
z+\nu\Delta x)(y-z)dzdxdyds\\
&  -\Delta x\Delta t\int_{x_{j}}^{x_{j+1}}\int_{t^{n}}^{t^{n+1}}\int_{x_{j}%
}^{x_{j+1}}\int_{y}^{x}\partial_{x}u(t, x+\nu\Delta x)\partial_{x}^2\eta(t,
z+\nu\Delta x)(x-z)dzdxdyds\\
&  +\Delta x\int_{x_{j}}^{x_{j+1}}\int_{t^{n}}^{t^{n+1}}\int_{x_{j}}^{x_{j+1}%
}\int_{t^{n}}^{t^{n+1}}\int_{s}^{t}\partial_{x}\eta(s, y+\nu\Delta
x)\partial_{t}u(\tau, x+\nu\Delta x)d\tau dxdyds dt\\
&  +\Delta x\int_{x_{j}}^{x_{j+1}}\int_{t^{n}}^{t^{n+1}}\int_{x_{j}}^{x_{j+1}%
}\int_{t^{n}}^{t^{n+1}}\int_{t}^{s}\partial_{x}u(t, x+\nu\Delta x)\partial
_{t}\eta(\tau, y+\nu\Delta x)d\tau dxdyds dt
\end{align*}
Moreover,
\begin{align*}%
&  K^{\prime\prime}(\nu)=\Delta x^2\Delta t\int_{x_{j}}^{x_{j+1}}\int_{t^{n}}^{t^{n+1}}\left[2\Delta
x\left(  \partial_{x}u\partial_{x}\eta\right)  (s, y+\nu\Delta x)\right.\\
&\left.+\Delta
x\left(  u\partial_{x}^{2}\eta\right)  (s, y+\nu\Delta x)+\Delta x\left(
\eta\partial_{x}^{2}u\right)  (s, y+\nu\Delta x)\right]dsdy\\
&  +\Delta x\Delta t\int_{x_{j}}^{x_{j+1}}\int_{t^{n}}^{t^{n+1}}\int_{x_{j}%
}^{x_{j+1}}\int_{y}^{x}\left[\Delta x\partial_{x}^{2}\eta(s, y+\nu\Delta
x)\partial_{x}u(s, z+\nu\Delta x)\right.\\
&\left.+\Delta x\partial_{x}\eta(s, y+\nu\Delta
x)\partial_{x}^{2}u(s, z+\nu\Delta x)\right]dzdxdyds\\
&  +\Delta x\Delta t\int_{x_{j}}^{x_{j+1}}\int_{t^{n}}^{t^{n+1}}\int_{x_{j}%
}^{x_{j+1}}\int_{x}^{y}\left[\Delta x\partial_{x}^{2}u(t, s+\nu\Delta x)\partial
_{x}\eta(t, z+\nu\Delta x)\right.\\
&\left.+\Delta x\partial_{x}u(t, x+\nu\Delta x)\partial
_{x}^{2}\eta(t, z+\nu\Delta x)\right]dzdxdyds\\
&  +\Delta x\int_{x_{j}}^{x_{j+1}}\int_{t^{n}}^{t^{n+1}}%
\int_{x_{j}}^{x_{j+1}}\int_{t^{n}}^{t^{n+1}}\int_{s}^{t}\left[\Delta x\partial
_{x}^{2}\eta(s, y+\nu\Delta x)\partial_{t}u(\tau, x+\nu\Delta x)\right.\\
&\left.+\Delta
x\partial_{x}\eta(s, y+\nu\Delta x)\partial_{tx}u(\tau, x+\nu\Delta x)\right]d\tau
dxdydsdt\\
&  +\Delta x\int_{x_{j}}^{x_{j+1}}\int_{t^{n}}^{t^{n+1}}%
\int_{x_{j}}^{x_{j+1}}\int_{t^{n}}^{t^{n+1}}\int_{t}^{s}\left[\Delta x\partial
_{x}^{2}u(t, x+\nu\Delta x)\partial_{t}\eta(\tau, y+\nu\Delta x)\right.\\
&\left.+\Delta
x\partial_{x}u(t, x+\nu\Delta x)\partial_{tx}\eta(\tau, y+\nu\Delta x)\right]d\tau
dxdydsdt.
\end{align*}
We have the same type of equality for $K'''$.
Applying once again the Cauchy-Schwarz inequality gives
\begin{align*}
&  |K^{\prime\prime\prime}(\nu)|\lesssim\Delta x^4\Delta t||\partial_{x}^2u||_{ L^{2}%
(Q)}||\partial_{x} \eta||_{ L^{2}(Q)}+\Delta x^4\Delta t||\partial_{x}u||_{ L^{2}%
(Q)}||\partial_{x}^2 \eta||_{ L^{2}(Q)}\\
&+\Delta x^4\Delta t||u||_{ L^{2}%
(Q)}||\partial_{x}^3 \eta||_{ L^{2}(Q)}+\Delta x^4\Delta t||\partial_{x}^3u||_{ L^{2}%
(Q)}|| \eta||_{ L^{2}(Q)}\\
&+\Delta x^5\Delta t||\partial_{x}u||_{ L^{2}%
(Q)}||\partial_{x}^3 \eta||_{ L^{2}(Q)}+\Delta x^5\Delta t||\partial_{x}^2u||_{ L^{2}%
(Q)}||\partial_{x}^2 \eta||_{ L^{2}(Q)}\\
&+\Delta x^5\Delta t||\partial_{x}^3u||_{ L^{2}%
(Q)}||\partial_{x}\eta||_{ L^{2}(Q)}+\Delta x^4\Delta t^2||\partial_{xt}u||_{ L^{2}%
(Q)}||\partial_{x}^2 \eta||_{ L^{2}(Q)}\\
&+\Delta x^4\Delta t^2||\partial_{t}u||_{ L^{2}%
(Q)}||\partial_{x}^3 \eta||_{ L^{2}(Q)}+\Delta x^4\Delta t^2||\partial_{t}\partial_x^2u||_{ L^{2}%
(Q)}||\partial_{x} \eta||_{ L^{2}(Q)}\\
&+\Delta x^4\Delta t^2||\partial_{x}^3u||_{ L^{2}%
(Q)}||\partial_{t} \eta||_{ L^{2}(Q)}+\Delta x^4\Delta t^2||\partial_{x}^2u||_{ L^{2}%
(Q)}||\partial_{xt} \eta||_{ L^{2}(Q)}\\
&+\Delta x^4\Delta t^2||\partial_{x}u||_{ L^{2}%
(Q)}||\partial_{t}\partial_x^2 \eta||_{ L^{2}(Q)}
\end{align*}
Thus, the $E_{\mathrm{non\ linear}}$-term rewrites
\begin{align*}%
&  E_{\mathrm{non\ linear}}\lesssim\frac{1}{2\Delta x^{3}\Delta t^{2}}\left(
K(1)-K(-1)\right) \\
& =\frac{1}{2\Delta x^{3}\Delta t^{2}}\left(  2K^{\prime
}(0)+\int_{0}^{1}K^{\prime\prime\prime}(w) \frac{(1-w)^2}{2}dw+\int_{0}^{1}K^{\prime\prime\prime
}(-w)\frac{(-1-w)^2}{2}dw\right) \\
&  \leq\frac{1}{\Delta x\Delta t}\int_{x_{j}}^{x_{j+1}}\int_{t^{n}}^{t^{n+1}%
}\partial_{x}(\eta u)(s, y)dsdy+\Delta x\sqrt{\frac{\Delta x}{\Delta t}}||\partial
_{x}\eta||_{ L^{\infty}(Q)}||\partial_{x}^2u||_{ L^{2}(Q)}\\
&+\Delta x\sqrt{\frac{\Delta x}{\Delta t}}||\partial
_{x}^2\eta||_{ L^{2}(Q)}||\partial_{x}u||_{ L^{\infty}(Q)}+\sqrt{\frac{\Delta t}{\Delta x}}||\partial_{x}\eta||_{ L^{\infty}(Q)}||\partial_{t}u||_{ L^{2}(Q)}  \\
&+\sqrt{\frac{\Delta t}{\Delta x}}||\partial_{t}\eta||_{ L^{2}%
(Q)}||\partial_{x}u||_{ L^{\infty}(Q)}+\Delta x\sqrt{\frac{\Delta x}{\Delta t}}||\partial_{x}^2u||_{ L^{2}%
(Q)}||\partial_{x} \eta||_{ L^{\infty}(Q)}\\
&+\Delta x\sqrt{\frac{\Delta x}{\Delta t}}||\partial_{x}u||_{ L^{\infty}%
(Q)}||\partial_{x}^2 \eta||_{ L^{2}(Q)}+\Delta x\sqrt{\frac{\Delta x}{\Delta t}}||u||_{ L^{\infty}%
(Q)}||\partial_{x}^3 \eta||_{ L^{2}(Q)}\\
&+\Delta x\sqrt{\frac{\Delta x}{\Delta t}}||\partial_{x}^3u||_{ L^{2}%
(Q)}|| \eta||_{ L^{\infty}(Q)}+\Delta x^2\sqrt{\frac{\Delta x}{\Delta t}}||\partial_{x}u||_{ L^{\infty}%
(Q)}||\partial_{x}^3 \eta||_{ L^{2}(Q)}\\
&+\Delta x^2\sqrt{\frac{\Delta x}{\Delta t}}||\partial_{x}^2u||_{ L^{\infty}%
(Q)}||\partial_{x}^2 \eta||_{ L^{2}(Q)}+\Delta x^2\sqrt{\frac{\Delta x}{\Delta t}}||\partial_{x}^3u||_{ L^{2}%
(Q)}||\partial_{x}\eta||_{ L^{\infty}(Q)}\\
&+\Delta x\sqrt{\Delta x\Delta t}||\partial_{xt}u||_{ L^{2}%
(Q)}||\partial_{x}^2 \eta||_{ L^{\infty}(Q)}+\Delta x\sqrt{\Delta x\Delta t}||\partial_{t}u||_{ L^{2}%
(Q)}||\partial_{x}^3 \eta||_{ L^{\infty}(Q)}\\
&+\Delta x\sqrt{\Delta x\Delta t}||\partial_{t}\partial_x^2u||_{ L^{2}%
(Q)}||\partial_{x} \eta||_{ L^{\infty}(Q)}+\Delta x\sqrt{\Delta x\Delta t}||\partial_{x}^3u||_{ L^{\infty}%
(Q)}||\partial_{t} \eta||_{ L^{2}(Q)}\\
&+\Delta x\sqrt{\Delta x\Delta t}||\partial_{x}^2u||_{ L^{\infty}%
(Q)}||\partial_{xt} \eta||_{ L^{2}(Q)}+\Delta x\sqrt{\Delta x\Delta t}||\partial_{x}u||_{ L^{\infty}%
(Q)}||\partial_{t}\partial_x^2 \eta||_{ L^{2}(Q)}.
\end{align*}
Finally, one has
\[
E_{\mathrm{viscosity}}\lesssim\frac{\tau_{1}\sqrt{\Delta x}}{2\sqrt{\Delta
t}}||\partial_{x}^{2}\eta||_{ L^{2}(Q)}.
\]
Then, when we sum up all the previous results, we obtain
\begin{small}
\begin{align*}
&  \epsilon_{1}^{n}=\frac{1}{\Delta t\Delta x}\int_{x_{j}}^{x_{j+1}}%
\int_{t^{n}}^{t^{n+1}}\partial_{t} \eta(s, y)-b\partial_{x}^{2}\partial
_{t}\eta(s, y)+\partial_{x}u(s, y)+a\partial_{x}^{3}u(s, y)+\partial_{x}(\eta
u)(s, y)dsdy\\
&  +\sqrt{\frac{\Delta t}{\Delta x}}||\partial_{t}^{2}\eta||_{ L%
^{2}(Q)}+b\Delta x\sqrt{\frac{\Delta x}{\Delta t}}||\partial_{x}^{4}%
\partial_{t} \eta||_{ L^{2}(Q)}+b\sqrt{\frac{\Delta t}{\Delta x}%
}||\partial_{x}^{2}\partial_{t}^{2}\eta||_{ L^{2}(Q)}\\
&+\theta
\sqrt{\frac{\Delta t}{\Delta x}}||\partial_{x}\partial_{t}u||_{ L%
^{2}(Q)}+\Delta x\sqrt{\frac{\Delta x}{\Delta t}}||\partial_{x}^{3}%
u||_{ L^{2}(Q)}\\
&  +|a|\Delta x\sqrt{\frac{\Delta x}{\Delta t}}||\partial_{x}^{5}%
u||_{ L^{2}(Q)}+|a|\theta\sqrt{\frac{\Delta t}{\Delta x}}%
||\partial_{t}\partial_{x}^{3}u||_{ L^{2}(Q)}+|a|\Delta x\sqrt
{\frac{\Delta x}{\Delta t}}||\partial_{x}^{5}u||_{ L^{2}(Q)}\\
&+\Delta x\sqrt{\frac{\Delta x}{\Delta t}}||\partial
_{x}\eta||_{ L^{\infty}(Q)}||\partial_{x}^2u||_{ L^{2}(Q)}+\Delta x\sqrt{\frac{\Delta x}{\Delta t}}||\partial
_{x}^2\eta||_{ L^{2}(Q)}||\partial_{x}u||_{ L^{\infty}(Q)}\\%
&+\sqrt{\frac{\Delta t}{\Delta x}}||\partial_{x}\eta||_{ L^{\infty}%
(Q)}||\partial_{t}u||_{ L^{2}(Q)} +\sqrt{\frac{\Delta t}{\Delta x}}||\partial_{t}\eta||_{ L^{2}%
(Q)}||\partial_{x}u||_{ L^{\infty}(Q)}\\%
&+\Delta x\sqrt{\frac{\Delta x}{\Delta t}}||\partial_{x}^2u||_{ L^{2}%
(Q)}||\partial_{x} \eta||_{ L^{\infty}(Q)}+\Delta x\sqrt{\frac{\Delta x}{\Delta t}}||\partial_{x}u||_{ L^{\infty}%
(Q)}||\partial_{x}^2 \eta||_{ L^{2}(Q)}\\
&+\Delta x\sqrt{\frac{\Delta x}{\Delta t}}||u||_{ L^{\infty}%
(Q)}||\partial_{x}^3 \eta||_{ L^{2}(Q)}+\Delta x\sqrt{\frac{\Delta x}{\Delta t}}||\partial_{x}^3u||_{ L^{2}%
(Q)}|| \eta||_{ L^{\infty}(Q)}\\
&+\Delta x^2\sqrt{\frac{\Delta x}{\Delta t}}||\partial_{x}u||_{ L^{\infty}%
(Q)}||\partial_{x}^3 \eta||_{ L^{2}(Q)}+\Delta x^2\sqrt{\frac{\Delta x}{\Delta t}}||\partial_{x}^2u||_{ L^{\infty}%
(Q)}||\partial_{x}^2 \eta||_{ L^{2}(Q)}\\
&+\Delta x^2\sqrt{\frac{\Delta x}{\Delta t}}||\partial_{x}^3u||_{ L^{2}%
(Q)}||\partial_{x}\eta||_{ L^{\infty}(Q)}+\Delta x\sqrt{\Delta x\Delta t}||\partial_{xt}u||_{ L^{2}%
(Q)}||\partial_{x}^2 \eta||_{ L^{\infty}(Q)}\\
&+\Delta x\sqrt{\Delta x\Delta t}||\partial_{t}u||_{ L^{2}%
(Q)}||\partial_{x}^3 \eta||_{ L^{\infty}(Q)}+\Delta x\sqrt{\Delta x\Delta t}||\partial_{t}\partial_x^2u||_{ L^{2}%
(Q)}||\partial_{x} \eta||_{ L^{\infty}(Q)}\\
&+\Delta x\sqrt{\Delta x\Delta t}||\partial_{x}^3u||_{ L^{\infty}%
(Q)}||\partial_{t} \eta||_{ L^{2}(Q)}+\Delta x\sqrt{\Delta x\Delta t}||\partial_{x}^2u||_{ L^{\infty}%
(Q)}||\partial_{xt} \eta||_{ L^{2}(Q)}\\
&+\Delta x\sqrt{\Delta x\Delta t}||\partial_{x}u||_{ L^{\infty}%
(Q)}||\partial_{t}\partial_x^2 \eta||_{ L^{2}(Q)}+\frac{\tau_{1}%
\sqrt{\Delta x}}{2\sqrt{\Delta t}}||\partial_{x}^{2}\eta||_{ L^{2}%
(Q)}.
\end{align*}
\end{small}
We recognize the initial equation on the first line. We recall the relation
\begin{small}
\begin{equation}
\sum_{j\in\mathbb{Z}}\Delta x||f||^{2}_{ L^{2}(Q)}=\Delta x\int
_{t^{n}}^{t^{n+1}}\int_{\mathbb{R}}|f(s, y)|^{2}dyds\leq\Delta x\Delta
t\underset{t\in[0,T]}{\mathrm{sup}}||f(t,.)||_{ L^{2}(\mathbb{R})}%
^{2}=\Delta x\Delta t||f(t,.)||_{ L^{\infty}_{t} L^{2}_{x}%
}^{2}. \label{RAPPEL}%
\end{equation}
\end{small}
Eventually, when we compute the $\ell^{2}_{\Delta}$-norm, we obtain
\begin{align*}
&  ||\epsilon_{1}^{n}||_{\ell^{2}_{\Delta}}^{2}=\Delta t^{2}||\partial_{t}%
^{2}\eta||^{2}_{ L^{\infty}_{t} L^{2}_{x}}+b^{2}\Delta
x^{4}||\partial_{x}^{4}\partial_{t} \eta||^{2}_{ L^{\infty}%
_{t} L^{2}_{x}}+b^{2}\Delta t^{2}||\partial_{x}^{2}\partial_{t}%
^{2}\eta||^{2}_{ L^{\infty}_{t} L^{2}_{x}}\\
&+\theta^{2}\Delta
t^{2}||\partial_{x}\partial_{t}u||^{2}_{ L^{\infty}_{t} L%
^{2}_{x}}+\Delta x^{4}||\partial_{x}^{3}u||^{2}_{ L^{\infty}%
_{t} L^{2}_{x}} +|a|^{2}\Delta x^{4}||\partial_{x}^{5}u||^{2}_{ L^{\infty}%
_{t} L^{2}_{x}}\\
&+|a|^{2}\theta^{2}\Delta t^{2}||\partial_{t}%
\partial_{x}^{3}u||^{2}_{ L^{\infty}_{t} L^{2}_{x}}
+|a|^{2}\Delta x^{4}||\partial_{x}^{5}u||^{2}_{ L^{\infty}%
_{t} L^{2}_{x}}+\Delta x^4||\partial_x\eta||_{L^{\infty}_tL^{\infty}_x}||\partial_x^2u||_{L^{\infty}_tL^2_x}\\
&+\Delta x^4||\partial_x^2\eta||_{L^{\infty}_tL^{2}_x}||\partial_xu||_{L^{\infty}_tL^{\infty}_x}+\Delta t^2||\partial_x\eta||_{L^{\infty}_tL^{\infty}_x}||\partial_tu||_{L^{\infty}_tL^2_x}\\
&+\Delta t^2||\partial_xu||_{L^{\infty}_tL^{\infty}_x}||\partial_t\eta||_{L^{\infty}_tL^2_x}+\Delta x^4||\partial_{x}^2u||_{ L^{\infty}_tL^{2}_x%
}||\partial_{x} \eta||_{ L^{\infty}_tL^{\infty}_x}\\
&+\Delta x^4||\partial_{x}u||_{ L^{\infty}_tL^{\infty}_x}||\partial_{x}^2 \eta||_{ L^{\infty}_tL^{2}_x}+\Delta x^4||u||_{ L^{\infty}_tL^{\infty}_x}||\partial_{x}^3 \eta||_{ L^{\infty}_tL^{2}_x}\\
&+\Delta x^4||\partial_{x}^3u||_{ L^{\infty}_tL^{2}_x}|| \eta||_{ L^{\infty}_tL^{\infty}_x}+\Delta x^6||\partial_{x}u||_{ L^{\infty}_tL^{\infty}_x}||\partial_{x}^3 \eta||_{ L^{\infty}_tL^{2}_x}\\
&+\Delta x^6||\partial_{x}^2u||_{ L^{\infty}_tL^{\infty}_x}||\partial_{x}^2 \eta||_{ L^{\infty}_tL^{2}_x}+\Delta x^6||\partial_{x}^3u||_{L^{\infty}_t L^{2}_x}||\partial_{x}\eta||_{ L^{\infty}_tL^{\infty}_x}\\
&+\Delta x^4\Delta t^2||\partial_{xt}u||_{L^{\infty}_t L^{2}_x}||\partial_{x}^2 \eta||_{ L^{\infty}_tL^{\infty}_x}+\Delta x^4\Delta t^2||\partial_{t}u||_{L^{\infty}_t L^{2}_x}||\partial_{x}^3 \eta||_{L^{\infty}_t L^{\infty}_x}\\
&+\Delta x^4\Delta t^2||\partial_{t}\partial_x^2u||_{ L^{\infty}_tL^{2}_x}||\partial_{x} \eta||_{ L^{\infty}_tL^{\infty}_x}+\Delta x^4\Delta t^2||\partial_{x}^3u||_{ L^{\infty}_tL^{\infty}_x}||\partial_{t} \eta||_{ L^{\infty}_tL^{2}_x}\\
&+\Delta x^4\Delta t^2||\partial_{x}^2u||_{L^{\infty}_t L^{\infty}_x}||\partial_{xt} \eta||_{ L^{\infty}_tL^{2}_x}+\Delta x^4\Delta t^2||\partial_{x}u||_{ L^{\infty}_tL^{\infty}_x}||\partial_{t}\partial_x^2 \eta||_{L^{\infty}_t L^{2}_x}  \\
&+\left(  \tau_{1}\right)  ^{2}\Delta x^{2}||\partial_{x}^{2}\eta
||^{2}_{ L^{\infty}_{t} L^{2}_{x}}.
\end{align*}

\noindent Thus, one has
\begin{equation*}||\epsilon_1^n||_{\ell^2_{\Delta}}\leq \left\{\begin{split}
&C(\Delta t+\Delta x),\text{\ \ if\ }\tau_1\neq0,\\
&C(\Delta t+\Delta x^2),\text{\ if\ }\tau_1=0,\end{split}\right.
\end{equation*}
with $C$ a constant depending of $u$, $\eta$ and their derivatives. 

In some cases, $D_{+}\epsilon^{n}_{1}$ is needed. To obtain an upper bound, we
perform the same computations and find%

\begin{align*}
&  ||D_{+}\epsilon_{1}^{n}||_{\ell^{2}_{\Delta}}^{2}\lesssim\Delta
x^{2}||\partial_{x}^{2}\partial_{t}\eta||^{2}_{ L^{\infty}%
_{t} L^{2}_{x}}+\Delta t^{2}||\partial_{x}\partial_{t}^{2}\eta
||^{2}_{ L^{\infty}_{t} L^{2}_{x}}+b^{2}\Delta x^{4}%
||\partial_{x}^{4}\partial_{t} \eta||^{2}_{ L^{\infty}_{t}%
 L^{2}_{x}}\\
 &+b^{2}\Delta t^{2}||\partial_{x}^{3}\partial_{t}^{2}%
\eta||^{2}_{ L^{\infty}_{t} L^{2}_{x}} +\Delta x^{2}||\partial_{x}^{3}u||^{2}_{ L^{\infty}_{t}%
 L^{2}_{x}}+|a|^{2}\Delta x^{2}||\partial_{x}^{5}u||^{2}%
_{ L^{\infty}_{t} L^{2}_{x}}+\theta^{2}\Delta t^{2}%
||\partial_{x}^{2}\partial_{t}u||^{2}_{ L^{\infty}_{t} L%
^{2}_{x}}\\
&+|a|^{2}\theta^{2}\Delta t^{2}||\partial_{t}\partial_{x}^{4}%
u||^{2}_{ L^{\infty}_{t} L^{2}_{x}}+\Delta x^{4}||\partial
_{x}^{4}u||^{2}_{ L^{\infty}_{t} L^{2}_{x}} +|a|^{2}\Delta x^{4}||\partial_{x}^{6}u||^{2}_{ L^{\infty}%
_{t} L^{2}_{x}}\\
&+\Delta x^{2}||\partial_{x}^{2}\eta||^{2}%
_{ L^{\infty}_{t} L^{2}_{x}}||\partial_{x}u||_{ L%
^{\infty}(Q)}+\Delta x^{2}||\partial_{x}\eta||^{2}_{ L^{\infty}%
_{t} L^{\infty}(Q)}||\partial_{x}^{2}u||^{2}_{ L^{\infty}%
_{t} L^{2}_{x}}\\
&  +\Delta t^{2}||\partial_{x}^{2}\eta||^{2}_{ L^{\infty}%
_{t} L^{\infty}_{x}}||\partial_{t}u||^{x}_{ L^{\infty}%
_{t} L^{2}_{x}}+\Delta t^{2}||\partial_{x}\eta||^{2}_{ L%
^{\infty}_{t} L^{\infty}_{x}}||\partial_{tx}u||^{2}_{ L%
^{\infty}_{t} L^{2}_{x}}\\
&+\Delta t^{2}|\partial_{x}^{2}u||^{2}%
_{ L^{\infty}_{t} L^{\infty}_{x}}||\partial_{t}\eta
||^{2}_{ L^{\infty}_{t} L^{2}_{x}} +\Delta t^{2}||\partial_{x}u||^{2}_{ L^{\infty}_{t} L%
^{\infty}_{x}}||\partial_{tx}\eta||^{2}_{ L^{\infty}_{t} L%
^{2}_{x}}\\
&+\Delta x^{2}||\partial_{x}^{2}u||^{2}_{ L^{\infty}%
_{t} L^{2}_{x}}||\partial_{x} \eta||^{2}_{ L^{\infty}%
_{t} L^{\infty}_{x}}  +\Delta x^{2}||\partial_{x}u||^{2}_{ L^{\infty}_{t} L%
^{\infty}_{x}}||\partial_{x}^{2} \eta||^{2}_{ L^{\infty}_{t}%
 L^{2}_{x}}\\
 &+\Delta x^{2}||u||^{2}_{ L^{\infty}_{t}%
 L^{\infty}_{x}}||\partial_{x}^{3}\eta||^{2}_{ L^{\infty}%
_{t} L^{2}_{x}} +\Delta x^{2}||\eta||^{2}_{ L^{\infty}_{t} L^{\infty}_{x}%
}||\partial_{x}^{3}u||^{2}_{ L^{\infty}_{t} L^{2}_{x}}\\
&+\Delta
x^{4}||\partial_{x}^{3}\eta||^{2}_{ L^{\infty}_{t} L^{2}_{x}%
}||\partial_{x}u||^{2}_{ L^{\infty}_{t} L^{\infty}_{x}}+\Delta
x^{4}||\partial_{x}^{2}u||^{2}_{ L^{\infty}_{t} L^{2}_{x}%
}||\partial_{x}^{2}\eta||^{2}_{ L^{\infty}_{t} L^{\infty}_{x}%
}\\
&  +\Delta x^{4}||\partial_{x}\eta||^{2}_{ L^{\infty}_{t}%
 L^{\infty}_{x}}||\partial_{x}^{3}u||^{2}_{ L^{\infty}%
_{t} L^{2}_{x}}+\Delta t^{2}\Delta x^{2}||\partial_{x}^{3}\eta
||^{2}_{ L^{\infty}_{t} L^{\infty}_{x}}||\partial_{t}%
u||^{2}_{ L^{\infty}_{t} L^{2}_{x}}\\
&+\Delta t^{2}\Delta
x^{2}||\partial_{x}^{2}\eta||^{2}_{ L^{\infty}_{t} L^{\infty
}_{x}}||\partial_{t}\partial_{x}u||^{2}_{ L^{\infty}_{t} L%
^{2}_{x}}  +\Delta t^{2}\Delta x^{2}||\partial_{x}\eta||^{2}_{ L^{\infty}%
_{t} L^{\infty}_{x}}||\partial_{t}\partial_{x}^{2}u||^{2}%
_{ L^{\infty}_{t} L^{2}_{x}}\\
&+\Delta t^{2}\Delta x^{2}%
||\partial_{t}\eta||^{2}_{ L^{\infty}_{t} L^{2}_{x}}%
||\partial_{x}^{3}u||^{2}_{ L^{\infty}_{t} L^{\infty}_{x}%
}+\Delta t^{2}\Delta x^{2}||\partial_{t}\partial_{x}\eta||^{2}_{ L%
^{\infty}_{t} L^{2}_{x}}||\partial_{x}^{2}u||^{2}_{ L^{\infty
}_{t} L^{\infty}_{x}}\\
&  +\Delta t^{2}\Delta x^{2}||\partial_{x}u||^{2}_{ L^{\infty}%
_{t} L^{\infty}_{x}}||\partial_{t}\partial_{x}^{2}\eta||^{2}%
_{ L^{\infty}_{t} L^{2}_{x}}+\left(  \tau_{1}\right)
^{2}\Delta x^{2}||\partial_{x}^{3}\eta||^{2}_{ L^{\infty}%
_{t} L^{2}_{x}}.
\end{align*}

\subsection{Consistency error $\epsilon_{2}^{n}$.}
By definition of the consistency error, on has
\begin{equation*}
\begin{split}
\epsilon_{2}^{n}&=\left(  I-dD_{+}D_{-}\right)  \left(  \frac{u_{\Delta}%
^{n+1}-u_{\Delta}^{n}}{\Delta t}\right)  +\left(  I+cD_{+}D_{-}\right)
D\left(  \theta\eta_{\Delta}^{n+1}+(1-\theta)\eta_{\Delta}^{n}\right)
+D\left(  \frac{\left(  u_{\Delta}^{n}\right)  ^{2}}{2}\right) \\
& -\frac
{\tau_{2}}{2}\Delta xD_{+}D_{-}u_{\Delta}^{n}.
\end{split}
\end{equation*}
We adapt the previous computations with $(d, c, \eta_{\Delta}^{n}, u_{\Delta
}^{n})$ instead of $(b, a, u_{\Delta}^{n}, \eta_{\Delta}^{n})$. The only
difference is concerning the non linear term
\[
K_{2}(\nu)=\frac{1}{2}\int_{x_{j}}^{x_{j+1}}\int_{t^{n}}^{t^{n+1}}\int_{x_{j}%
}^{x_{j+1}}\int_{t^{n}}^{t^{n+1}}u(s, y+\nu\Delta x)u(t, x+\nu\Delta
x)dxdydsdt.
\]

So one has

\[%
\begin{split}
&  E_{\mathrm{non\ linear}}\lesssim\frac{1}{\Delta x\Delta t}\int_{x_{j}}^{x_{j+1}}\int_{t^{n}}^{t^{n+1}%
}\partial_{x}(\frac{u^2}{2})(s, y)dsdy+\Delta x\sqrt{\frac{\Delta x}{\Delta t}}||\partial
_{x}u||_{ L^{\infty}(Q)}||\partial_{x}^2u||_{ L^{2}(Q)}%
\\
&+\sqrt{\frac{\Delta t}{\Delta x}}||\partial_{x}u||_{ L^{\infty}%
(Q)}||\partial_{t}u||_{ L^{2}(Q)}  +\Delta x\sqrt{\frac{\Delta x}{\Delta t}}||\partial_{x}^2u||_{ L^{2}%
(Q)}||\partial_{x} u||_{ L^{\infty}(Q)}\\
&+\Delta x\sqrt{\frac{\Delta x}{\Delta t}}||u||_{ L^{\infty}%
(Q)}||\partial_{x}^3 u||_{ L^{2}(Q)}+\Delta x^2\sqrt{\frac{\Delta x}{\Delta t}}||\partial_{x}u||_{ L^{\infty}%
(Q)}||\partial_{x}^3 u||_{ L^{2}(Q)}\\
&+\Delta x^2\sqrt{\frac{\Delta x}{\Delta t}}||\partial_{x}^2u||_{ L^{\infty}%
(Q)}||\partial_{x}^2 u||_{ L^{2}(Q)}+\Delta x\sqrt{\Delta x\Delta t}||\partial_{xt}u||_{ L^{2}%
(Q)}||\partial_{x}^2 u||_{ L^{\infty}(Q)}\\
&+\Delta x\sqrt{\Delta x\Delta t}||\partial_{t}u||_{ L^{2}%
(Q)}||\partial_{x}^3 u||_{ L^{\infty}(Q)}+\Delta x\sqrt{\Delta x\Delta t}||\partial_{t}\partial_x^2u||_{ L^{2}%
(Q)}||\partial_{x} u||_{ L^{\infty}(Q)}.
\end{split}
\]

The consistency error verifies
\begin{align*}
&  \epsilon_{2}^{n}\lesssim\frac{1}{\Delta t\Delta x}\int_{x_{j}}^{x_{j+1}%
}\int_{t^{n}}^{t^{n+1}}\left[\partial_{t} u(s, y)-d\partial_{x}^{2}\partial_{t}u(s,
y)+\partial_{x}\eta(s, y)+c\partial_{x}^{3}\eta(s, y)\right.\\
&\left.+\partial_{x}\frac{u^{2}%
}{2}(s, y)\right]dsdy  +\sqrt{\frac{\Delta t}{\Delta x}}||\partial_{t}^{2}u||_{ L^{2}%
(Q)}+d\Delta x\sqrt{\frac{\Delta x}{\Delta t}}||\partial_{x}^{4}\partial_{t}
u||_{ L^{2}(Q)}\\
&+d\sqrt{\frac{\Delta t}{\Delta x}}||\partial_{x}%
^{2}\partial_{t}^{2}u||_{ L^{2}(Q)}+\theta\sqrt{\frac{\Delta t}{\Delta
x}}||\partial_{x}\partial_{t}\eta||_{ L^{2}(Q)}+\Delta x\sqrt
{\frac{\Delta x}{\Delta t}}||\partial_{x}^{3}\eta||_{ L^{2}(Q)}\\
&  +|c|\Delta x\sqrt{\frac{\Delta x}{\Delta t}}||\partial_{x}^{5}%
\eta||_{ L^{2}(Q)}+|c|\theta\sqrt{\frac{\Delta t}{\Delta x}}%
||\partial_{t}\partial_{x}^{3}\eta||_{ L^{2}(Q)}\\
&+|c|\Delta
x\sqrt{\frac{\Delta x}{\Delta t}}||\partial_{x}^{5}\eta||_{ L^{2}%
(Q)}+\Delta x\sqrt{\frac{\Delta x}{\Delta t}}||\partial_{x}^2u||_{ L^{2}%
(Q)}||\partial_{x}u||_{ L^{\infty}(Q)}\\
&  +\sqrt{\frac{\Delta t}{\Delta x}}||\partial_{x}u||_{ L^{\infty}%
(Q)}||\partial_{t}u||_{ L^{2}(Q)} +\Delta x\sqrt{\frac{\Delta x}{\Delta t}}||\partial_{x}^2u||_{ L^{2}%
(Q)}||\partial_{x} u||_{ L^{\infty}(Q)}\\
&+\Delta x\sqrt{\frac{\Delta x}{\Delta t}}||u||_{ L^{\infty}%
(Q)}||\partial_{x}^3 u||_{ L^{2}(Q)}+\Delta x^2\sqrt{\frac{\Delta x}{\Delta t}}||\partial_{x}u||_{ L^{\infty}%
(Q)}||\partial_{x}^3 u||_{ L^{2}(Q)}\\
&+\Delta x^2\sqrt{\frac{\Delta x}{\Delta t}}||\partial_{x}^2u||_{ L^{\infty}%
(Q)}||\partial_{x}^2 u||_{ L^{2}(Q)}\\
&+\Delta x\sqrt{\Delta x\Delta t}||\partial_{xt}u||_{ L^{2}%
(Q)}||\partial_{x}^2 u||_{ L^{\infty}(Q)}+\Delta x\sqrt{\Delta x\Delta t}||\partial_{t}u||_{ L^{2}%
(Q)}||\partial_{x}^3 u||_{ L^{\infty}(Q)}\\
&+\Delta x\sqrt{\Delta x\Delta t}||\partial_{t}\partial_x^2u||_{ L^{2}%
(Q)}||\partial_{x} u||_{ L^{\infty}(Q)}+\frac{\tau_{2}\sqrt{\Delta x}}{2\sqrt{\Delta
t}}||\partial_{x}^{2}u||_{ L^{2}(Q)}.
\end{align*}
Finally, one has, thanks to the relation \eqref{RAPPEL}
\begin{align*}
&  ||\epsilon_{2}^{n}||_{\ell^{2}_{\Delta}}^{2}\lesssim\Delta t^{2}%
||\partial_{t}^{2}u||^{2}_{ L^{\infty}_{t} L^{2}_{x}}%
+d^{2}\Delta x^{4}||\partial_{x}^{4}\partial_{t} u||^{2}_{ L^{\infty
}_{t} L^{2}_{x}}+d^{2}\Delta t^{2}||\partial_{x}^{2}\partial_{t}%
^{2}u||^{2}_{ L^{\infty}_{t} L^{2}_{x}}\\
&+\theta^{2}\Delta
t^{2}||\partial_{x}\partial_{t}\eta||^{2}_{ L^{\infty}_{t}%
 L^{2}_{x}}+\Delta x^{4}||\partial_{x}^{3}\eta||^{2}_{ L%
^{\infty}_{t} L^{2}_{x}} +|c|^{2}\Delta x^{4}||\partial_{x}^{5}\eta||^{2}_{ L^{\infty}%
_{t} L^{2}_{x}}\\
&+|c|^{2}\theta^{2}\Delta t^{2}||\partial_{t}%
\partial_{x}^{3}\eta||^{2}_{ L^{\infty}_{t} L^{2}_{x}}%
+|c|^{2}\Delta x^{4}||\partial_{x}^{5}\eta||^{2}_{ L^{\infty}%
_{t} L^{2}_{x}}+\Delta x^{4}||\partial_{x}^2u||^{2}_{ L^{\infty
}_{t} L^{2}_{x}}||\partial_{x}u||^{2}_{ L^{\infty}%
_{t} L^{\infty}_{x}}\\
&  +\Delta t^{2}||\partial_{x}u||^{2}_{ L^{\infty}_{t} L%
^{\infty}_{x}}||\partial_{t}u||^{2}_{ L^{\infty}_{t} L^{2}%
_{x}}+\Delta x^4||\partial_{x}^2u||_{L^{\infty}_t L^{2}_x}||\partial_{x} u||_{L^{\infty}_t L^{\infty}_x}\\
&+\Delta x^4||u||_{ L^{\infty}_tL^{\infty}_x}||\partial_{x}^3 u||_{ L^{\infty}_tL^{2}_x}+\Delta x^6||\partial_{x}u||_{ L^{\infty}_tL^{\infty}_x}||\partial_{x}^3 u||_{L^{\infty}_t L^{2}_x}\\
&+\Delta x^6||\partial_{x}^2u||_{ L^{\infty}_tL^{\infty}_x}||\partial_{x}^2 u||_{ L^{\infty}_tL^{2}_x}+\Delta x^4\Delta t^2||\partial_{xt}u||_{L^{\infty}_t L^{2}_x}||\partial_{x}^2 u||_{ L^{\infty}_tL^{\infty}_x}\\
&+\Delta x^4\Delta t^2||\partial_{t}u||_{ L^{\infty}_tL^{2}_x}||\partial_{x}^3 u||_{ L^{\infty}_tL^{\infty}_x}+\Delta x^4\Delta t^2||\partial_{t}\partial_x^2u||_{ L^{\infty}_tL^{2}_x}||\partial_{x} u||_{L^{\infty}_t L^{\infty}_x}\\
&+\left(  \tau_{2}%
\right)  ^{2}\Delta x^{2}||\partial_{x}^{2}u||^{2}_{ L^{\infty
}_{t} L^{2}_{x}}.
\end{align*}

As for the $\epsilon_1^n$ case, there exists a constant $C$ depending on $u$, $\eta$ and their derivatives such that
\begin{equation*}||\epsilon_2^n||_{\ell^2_{\Delta}}\leq \left\{\begin{split}
&C(\Delta t+\Delta x),\text{\ \ if\ }\tau_2\neq0,\\
&C(\Delta t+\Delta x^2),\text{\ if\ }\tau_2=0.\end{split}\right.
\end{equation*}

For $D_{+}\epsilon^{n}_{2}$, the results are similar to those for
$D_{+}\epsilon_{1}^{n}$.

\subsubsection*{Acknowledgement} This work was performed within the framework of the LABEX MILYON (ANR-10- LABX0070) of Universit\'e de Lyon, within the program "Investissements d'Avenir" (ANR-11-IDEX-0007) operated by the French National Research Agency (ANR).

\bibliography{Biblio}

\begin{thebibliography}{10}

\bibitem{Amick1984}
C.~J. Amick.
\newblock {Regularity and uniqueness of solutions to the Boussinesq system of
  equations}.
\newblock {\em J. Nonlinear Sci.}, 54(2):231--247, 1984.

\bibitem{Anh2010}
C.~T. Anh.
\newblock {On the Boussinesq/Full dispersion systems and Boussinesq/Boussinesq
  systems for internal waves}.
\newblock {\em Nonlinear Analysis: Theory, Methods \& Applications},
  72(1):409--429, 2010.

\bibitem{AntonopoulosDougalis2012n}
D.~C. Antonopoulos and V.~A. Dougalis.
\newblock {Notes on error estimates for Galerkin approximations of the
  'classical' Boussinesq system and related hyperbolic problems}.
\newblock {\em arxiv}, 2010.

\bibitem{AntonopoulosDougalis2012}
D.~C. Antonopoulos and V.~A. Dougalis.
\newblock {Numerical solution of the classical {B}oussinesq system}.
\newblock {\em Mathematics and Computers in Simulation}, 82(6):984--1007, 2012.

\bibitem{AntonopoulosDougalis2016}
D.~C. Antonopoulos and V.~A. Dougalis.
\newblock Error estimates for the standard {G}alerkin-finite element method for
  the {S}hallow {W}ater equations.
\newblock {\em Mathematics of Computation}, 85:1143--1182, 2016.

\bibitem{BonaChen1998}
J.~L. Bona and M.~Chen.
\newblock {A Boussinesq system for two-way propagation of nonlinear dispersive
  waves}.
\newblock {\em Physica D: Nonlinear Phenomena}, 116(1):191--224, 1998.

\bibitem{BonaChen2016}
J.~L. Bona and M.~Chen.
\newblock {Singular Solutions of a Boussinesq System for Water Waves}.
\newblock {\em J. Math. Study}, 49(3):205--220, 2016.

\bibitem{BonaSmith1976}
J.~L. Bona and R.~Smith.
\newblock {A model for the two-way propagation of water waves in a channel}.
\newblock {\em Mathematical Proceedings of the Cambridge Philosophical
  Society}, 79(1):167--182, 1976.

\bibitem{Burtea2016b}
C.~Burtea.
\newblock {Long time existence results for bore-type initial data for
  BBM-Boussinesq systems}.
\newblock {\em Journal of Differential Equations}, 261(9):4825--4860, 2016.

\bibitem{Burtea2016a}
C.~Burtea.
\newblock {New long time existence results for a class of Boussinesq-type
  systems}.
\newblock {\em Journal de Math\'ematiques Pures et Appliqu\'ees},
  106(2):203--236, 2016.

\bibitem{CourtesLagoutiereRousset2017}
F.~Lagouti\`ere C.~Court\`es and F.~Rousset.
\newblock {Error estimates of finite difference schemes for the Korteweg-de
  Vries equation}.
\newblock {\em to appear in IMA Journal of Numerical Analysis, arXiv preprint
  arXiv:1712.02291, \url{https://arxiv.org/abs/1712.02291}}, 2017.

\bibitem{Chazel2007}
F.~Chazel.
\newblock {Influence of bottom topography on long water waves}.
\newblock {\em ESAIM: Mathematical Modelling and Numerical Analysis},
  41(4):771--799, 2007.

\bibitem{Chen1998}
M.~Chen.
\newblock {Exact Traveling-Wave Solutions to Bidirectional Wave Equations}.
\newblock {\em Internat. J. Theoret. Phys.}, 37(5):1547--1567, 1998.

\bibitem{Chen2003}
M.~Chen.
\newblock {Equations for bi-directional waves over an uneven bottom}.
\newblock {\em Mathematics and Computers in Simulation}, 62(1):3--9, 2003.

\bibitem{Chen2009}
M.~Chen.
\newblock {Numerical investigation of a two-dimensional Boussinesq system}.
\newblock {\em Discrete Contin. Dyn. Syst}, 28(4):1169--1190, 2009.

\bibitem{Clark1987}
D.~S. Clark.
\newblock {Short proof of a discrete Gr\"onwall inequality}.
\newblock {\em Discrete Applied Mathematics}, 16:279--281, 1987.

\bibitem{AntonopoulusDugalisMitsotakis2009}
V.~A.~Dougalis D.~C.~Antonopoulos and D.~E. Mitsotakis.
\newblock {Galerkin approximations of periodic solutions of Boussinesq
  systems}.
\newblock {\em Bulletin of the Greek Mathematical Society}, 57:13--30, 2010.

\bibitem{AntonopoulusDugalisMitsotakis2010}
V.~A.~Dougalis D.~C.~Antonopoulos and D.~E. Mitsotakis.
\newblock {Numerical solution of Boussinesq systems of the Bona--Smith family}.
\newblock {\em Applied numerical mathematics}, 60(4):314--336, 2010.

\bibitem{Daripa_Dash_2003}
P.~Daripa and R.~K. Dash.
\newblock {A class of model equations for bi-directional propagation of
  capillary-gravity waves}.
\newblock {\em Internat. J. Engrg. Sci}, 41:201--218, 2003.

\bibitem{LinaresPilodSaut2012}
D.Pilod F.~Linares and J.-C. Saut.
\newblock {Well-Posedness of strongly dispersive two-dimensional surface wave
  Boussinesq systems}.
\newblock {\em SIAM Journal on Mathematical Analysis}, 44(6):4195--4221, 2012.

\bibitem{SautWangXu2015}
C.~Wang J.-C.~Saut and L.~Xu.
\newblock {The Cauchy problem on large time for surface waves type Boussinesq
  systems II}.
\newblock {\em arXiv preprint arXiv:1511.08824,
  \url{https://arxiv.org/abs/1511.08824}}, 2015.

\bibitem{BonaChenSaut2002}
M.~Chen J.~L.~Bona and J.-C. Saut.
\newblock {Boussinesq equations and other systems for small-amplitude long
  waves in nonlinear dispersive media. I: Derivation and linear theory}.
\newblock {\em Journal of Nonlinear Science}, 12(4):283--318, 2002.

\bibitem{BonaChenSaut2004}
M.~Chen J.~L.~Bona and J.-C. Saut.
\newblock {Boussinesq equations and other systems for small-amplitude long
  waves in nonlinear dispersive media: II. The nonlinear theory}.
\newblock {\em Nonlinearity}, 17(3):925--952, 2004.

\bibitem{BonaColinLannes2005}
T.~Colin J.~L.~Bona and D.~Lannes.
\newblock {Long Wave Approximations for Water Waves}.
\newblock {\em Archive for Rational Mechanics and Analysis}, 178(3):373--410,
  2005.

\bibitem{BonaDougalisMitsotakis2007}
V.~A.~Dougalis J.~L.~Bona and D.~E. Mitsotakis.
\newblock {Numerical solution of KdV--KdV systems of Boussinesq equations: I.
  The numerical scheme and generalized solitary waves}.
\newblock {\em Mathematics and Computers in Simulation}, 74(2):214--228, 2007.

\bibitem{BonaDougalisMitsotakis2008}
V.~A.~Dougalis J.~L.~Bona and D.~E. Mitsotakis.
\newblock {Numerical solution of Boussinesq systems of KdV--KdV type: II.
  Evolution of radiating solitary waves}.
\newblock {\em Nonlinearity}, 21(12):2825--2848, 2008.

\bibitem{Lannes2013livre}
D.~Lannes.
\newblock {\em {The water waves problem}}.
\newblock Mathematical surveys and monographs, 188, 2013.

\bibitem{MingSautZhang2012}
J.-C.~Saut M.~Ming and P.~Zhang.
\newblock {Long-time existence of solutions to Boussinesq systems}.
\newblock {\em SIAM Journal on Mathematical Analysis}, 44(6):4078--4100, 2012.

\bibitem{Peregrine1966}
D.~H. Peregrine.
\newblock {Calculations of the development of an undular bore}.
\newblock {\em Journal of Fluid Mechanics}, 25(2):321--330, 1966.

\bibitem{SautXu2012}
J.-C. Saut and L.~Xu.
\newblock {The Cauchy problem on large time for surface waves Boussinesq
  systems}.
\newblock {\em Journal de math\'ematiques pures et appliqu\'ees},
  97(6):635--662, 2012.

\bibitem{Schonbek1981}
M.~E. Schonbek.
\newblock {Existence of solutions for the Boussinesq system of equations}.
\newblock {\em Journal of Differential Equations}, 42(3):325--352, 1981.

\bibitem{Tadmor_2003}
E.~Tadmor.
\newblock {Entropy stability theory for difference approximations of nonlinear
  conservation laws and related time-dependent problems}.
\newblock {\em Acta Numerica}, 12:451--512, 2003.

\bibitem{DougalisMitsotakisSaut2007}
D.~E.~Mitsotakis V.~A.~Dougalis and J.-C. Saut.
\newblock {On some Boussinesq systems in two space dimensions: Theory and
  numerical analysis}.
\newblock {\em ESAIM: Mathematical Modeling and Numerical Analysis},
  41(5):825--854, 2007.

\end{thebibliography}

\end{document}